\documentclass[letterpaper,10pt]{amsart}
\makeatletter
\newcommand{\printaddresses}{%
  \@setaddresses              % print addresses and emails here
  \let\enddoc@text\relax      % don't print them again at the end
}
\makeatother

\usepackage{tikz, tikz-cd, stmaryrd}
\usepackage[colorlinks,backref=page]{hyperref}
\hypersetup{citecolor={blue}}
\usetikzlibrary{arrows}
\usepackage{subcaption} 
\usepackage[font=footnotesize,labelfont=bf]{caption}
\usepackage{blkarray}
\usepackage{enumitem}
\usepackage{cleveref}
\usepackage{ bbold }
\usepackage{xcolor}
\usepackage{soul,dirtytalk}
\usepackage{latexsym,array,delarray,epsfig,setspace, mathtools,amssymb,mathrsfs}

\definecolor{light}{gray}{.75}
\definecolor{med}{gray}{.5}
\definecolor{dark}{gray}{.25}

\newtheorem{theorem}{Theorem}
\numberwithin{theorem}{section}
\newtheorem{proposition}[theorem]{Proposition}

\newtheorem{corollary}[theorem]{Corollary}
\newtheorem{lemma}[theorem]{Lemma}

\theoremstyle{definition}
\newtheorem{definition}[theorem]{Definition}
\newtheorem{remark}[theorem]{Remark}

\newtheorem{example}[theorem]{Example}

\newtheorem{notation}[theorem]{Notation}
\newcommand{\C}{{\mathbb C}}

\newcommand{\M}{\mathbb{M}}
\newcommand{\Q}{{\mathbb Q}}
\newcommand{\R}{{\mathbb R}}
\newcommand{\Z}{{\mathbb Z}}
\renewcommand{\P}{\textup{P}}

\newcommand{\Aa}{\mathcal{A}}
\newcommand{\K}{{\mathbb K}}
\newcommand{\B}{\mathbb{B}}
\newcommand{\bb}{\mathbb{b}}

\newcommand{\HH}{\mathcal{H}}
\newcommand{\PP}{\mathcal{P}}
\renewcommand{\O}{\mathcal{O}}

\renewcommand{\S}{\textup{S}}
\newcommand{\Sp}{\textup{Sp}}

\newcommand{\ba}{\mathbf{a}}
\newcommand{\bc}{\mathbf{c}}

\newcommand{\bo}{\mathbf{1}}
\newcommand{\bt}{\mathbf{t}}
\newcommand{\bu}{\mathbf{u}}
\newcommand{\bv}{\mathbf{v}}
\newcommand{\bw}{\mathbf{w}}
\newcommand{\bx}{\mathbf{x}}

\newcommand{\bze}{\mathbf{0}}

\newcommand{\w}{{\sf w}}
\renewcommand{\v}{{\sf v}}

\newcommand{\fv}{\mathfrak{v}}
\newcommand{\fw}{\mathfrak{w}}

\newcommand{\supp}{\textup{supp}}
\newcommand{\ptconv}{\textup{p-conv}}

\newcommand{\Spec}{\textup{Spec}}
\newcommand{\Proj}{\textup{Proj}}

\newcommand{\Hom}{\textup{Hom}}

\newcommand{\Cl}{\textup{Cl}}

\newcommand{\ord}{\textup{ord}}

\newcommand{\NN}{\mathcal{N}}
\newcommand{\MM}{\mathcal{M}}

\newcommand{\gr}{\textup{gr}}

\renewcommand{\b}{{\bf b}}

\newcommand{\exampleqed}{%
{   
\renewcommand{\qedsymbol}{$\lozenge$}%
\qed%
}
}

\title{Geometric families of degenerations from mutations of polytopes}

\author{Laura Escobar}
\address{Mathematics Department, University of California, Santa Cruz%, 1156 High Street Santa Cruz, CA 95064, USA
} 
\email{lauraescobar@ucsc.edu}

\author{Megumi Harada}
\address{Dept.\ of Mathematics and Statistics, McMaster University%, 1280 Main Street West, Hamilton, Ontario L8S 4K1, Canada
}
\email{haradam@mcmaster.ca}

\author{Christopher Manon} 
\address{Department of Mathematics, University of Kentucky%, 719 Patterson Office Tower, Lexington, Kentucky, 40506-0027, USA
} 
\email{chris.manon@gmail.com}

\date{\today}


\begin{document}

\maketitle

\begin{abstract} 
We introduce the notion of a \textbf{polyptych lattice}, which encodes a collection of lattices related by piecewise linear bijections. We also define and study \textbf{PL (polyptych-lattice) polytopes}. In certain situations, such a PL polytope associated to a polyptych lattice encodes a compactification of an affine variety whose coordinate ring can be equipped with a valuation into a certain semialgebra associated to the polyptych lattice. We show that aspects of the geometry of the compactification can be understood combinatorially; for instance, under some hypotheses, the resulting compactifications are arithmetically Cohen-Macaulay, and have finitely generated class group and finitely generated Cox rings.   
\end{abstract} 

{
  \hypersetup{linkcolor=black}
  \tableofcontents
}

\section{Introduction} 

The theory of toric geometry famously provides a dictionary between combinatorics and geometry, leading to an interplay that enriches both areas. Many areas of research within Combinatorial Algebraic Geometry seek to generalize this combinatorics-geometry dictionary of toric geometry in one way or another. Examples include, but are not limited to, toric degenerations, $T$-varieties, Newton-Okounkov bodies, cluster varieties, and integrable systems.

A common phenomenon observed in several instances of these generalizations is the presence of a collection of polytopes which are related by piecewise linear maps. (These piecewise linear maps are frequently called \textit{mutations}.) For instance, in the theory of Newton-Okounkov bodies, Kaveh and the third author observed in \cite{KavehManon-Siaga} that the Newton-Okounkov bodies of a variety $X$ with coordinate ring $\Aa$ arising from valuations on $\Aa$ with a common Khovanskii basis $\mathcal{B}$ can be organized by the maximal prime cones of the tropical variety corresponding to $X$ and $\mathcal{B}$. The first two authors then took this observation in \cite{KavehManon-Siaga} one step further, and showed in \cite{EscobarHarada} that when two maximal prime cones share a codimension-$1$ face, then the two corresponding Newton-Okounkov bodies are related by a piecewise linear transformation of a particularly simple form. In the context of cluster varieties and cluster algebras, Rietsch and Williams \cite{RW} showed that Newton-Okounkov bodies of $Gr_k(\C^n)$ are related by cluster mutations. In related work, Bossinger, Cheung, Magee, and N\'{a}jera-Chavez \cite{BCMNC} recently showed that Newton-Okounkov bodies of compactifications of certain cluster varieties are also related by cluster mutations. 

Motivated by the above, we introduce in this paper a new combinatorial notion called a \textbf{polyptych\footnote{The word ``polyptych'' is chosen to be a Greek analogue of the word ``manifold."}  lattice}, and initiate a study of its associated convex and algebraic geometry. This theory builds the notion of mutations directly into the definitions, and gives algebraic-geometric interpretations to these mutations. In particular, our constructions simultaneously generalize analogous objects in the theories of toric varieties, tropical varieties, and cluster varieties, among others.  We also note that our theory appears to be closely related to the ``broken line convex geometry'' of Frias-Medina and Magee \cite{FriasMedina-Magee}.

Our main theorems are the following. The first statement below, which is \Cref{compactifications} in Section~\ref{section: compactifications}, is analogous to the compactification of the torus $(\mathbb{G}_m^*)^n$ given by a projective toric variety, with boundary divisors corresponding to the defining half-spaces of a polytope; we provide a compactification of an affine variety associated to our polyptych-lattice data. Precise definitions are given in Section ~\ref{section: compactifications}. 

\medskip
\noindent \textbf{Theorem A.} 
Let $\MM$ be a finite polyptych lattice of rank $r$ over $\Z$ with a fixed choice of strict dual $\Z$-pair $(\MM,\NN, \v, \w)$. Let $(\Aa_\MM,\fv)$ be a detropicalization of $\MM$ with convex adapted basis $\B$. 
Let $\PP = \cap_{i=1}^{\ell} \HH_{\w(n_i), a_i} \subset \MM_\R$ be a full-dimensional PL polytope. Suppose $n_i \in \NN, n_i \neq 0$, the $\HH_{\w(n_i), a_i}$ are pairwise distinct, and $a_i \in \Z_{< 0}$ for all $i \in [\ell]$. Suppose also that for each $n_i$ there exists a coordinate chart $\alpha_i \in \pi(\MM)$ on which $n_i$ is linear, and, the intersection of the boundary of $\pi_{\alpha_i}(\HH_{\w(n_i),a_i})$ with $\pi_{\alpha_i}(\PP)$ is a facet of $\pi_{\alpha_i}(\PP)$. Then the affine scheme (over $\K$) $U_\MM := \Spec(\Aa_\MM)$ can be realized as a dense, open subscheme of $X_{\Aa_\MM}(\PP)$, where the complement of $U_\MM$ in $X_{\Aa_\MM}(\PP)$ is a union of divisors $D_i$ which are in bijection with $n_i$, $i \in [\ell]$.  

\medskip

We can additionally deduce algebraic-geometric properties of our compactifications, in \Cref{prop: normality of compactification}, \Cref{theorem: Cohen Macaulay}, \Cref{theorem: finitely generated Cox rings}. We summarize these results as follows. We note that Theorem B(3) can be viewed as a first step in determining when our compactifications are Mori dream spaces. 

\medskip

\noindent \textbf{Theorem B.} 
Let the assumptions and notation be as in the Theorem above. Then: 
\begin{enumerate} 
\item $X_{\Aa_\MM}(\PP)$ is arithmetically Cohen-Macaulay, 
\item if $\Aa_\MM$ is normal, then $X_{\Aa_\MM}(\PP)$ is normal, 
\item if $\Aa_\MM$ is a unique factorization domain, then $X_{\Aa_\MM}(\PP)$ has finitely generated class group and a finitely generated Cox ring. 
\end{enumerate} 

\medskip

Moreover, as we suggested above, we give a systematic geometric interpretation of mutations of polytopes and the wall-crossing phenomenon of Newton-Okounkov bodies, by exhibiting an explicit family $\{\mathcal{X}_\alpha\}$ of toric degenerations of a single variety, namely $X_{\Aa_\MM}(\PP)$, where each member $\mathcal{X}_\alpha$ of the family has as its central fiber the toric variety associated to a polytope which is mutation-related to the others in the family and has an interpretation as a Newton-Okounkov body. Moreover, these mutation-related polytopes are encoded in a single combinatorial object, i.e., the ``polyptych lattice polytope'' $\PP$ (defined in Section~\ref{sec_PL_polytopes}). We have the following (\Cref{theorem: NO body}, \Cref{theorem: family of degenerations}). Definitions and details are in Section~\ref{section: compactifications}. 

\medskip

\noindent \textbf{Theorem C.} 
Let the assumptions and notation be as in the Theorem above. Assume $\PP$ is integral. Let $\nu_{\alpha,\bar{\rho}}: \Aa^\PP_\MM \to \Z^{\mathrm{rank}(\MM)}\times \Z$ be the valuation defined in Section~\ref{subsec: compactification definition}. Then:
\begin{enumerate} 
\item there exists a toric degeneration $\mathcal{X} \to \Spec(\K[t])$ with generic fiber isomorphic to $X_{\Aa_\MM}(\PP)$ and central fiber isomorphic to the toric variety associated to the (classical) integral polytope $\Delta(\Aa_\MM, \nu_{\alpha,\bar{\rho}})$, the Newton-Okounkov body of $\Aa_\MM$ with respect to $\nu_{\alpha,\bar{\rho}}$, 
\item the Newton-Okounkov body $\Delta(\Aa_\MM, \nu_{\alpha,\bar{\rho}})$ is combinatorially encoded in the data of polyptych-lattice polytope $\PP$. 
\end{enumerate}

Our polyptych-lattice theory, and our main theorems above, have already had some impact in related research areas. Fujita and Higashitani used our polyptych-lattice theory to show that the transfer map mutations of marked chain-order polytopes for the Gelfand–Tsetlin poset can be seen as a combinatorial shadow of our compactifications and families of toric degenerations in \cite{FujitaHigashitani}. Oda showed in \cite[Theorem 3]{Oda2025} that the compactifications of the rank-$2$ polyptych lattices studied in our companion paper \cite{CookEscobarHaradaManon2024} yield all index-one normal projective $\Q$-factorial log Calabi–Yau surface pairs having reduced boundary admitting an effective ample class and a nontrivial torus action.

\bigskip

We now describe the contents of this paper in some more detail. We first define in \Cref{sec: PL} the central combinatorial object of this paper, a \textbf{polyptych lattice}, which is a collection $\MM$ of rank-$r$ lattices for some positive integer $r$, related to each other by mutations, cf.\ \Cref{definition: polyptych lattice}. We may think of a polyptych lattice (PL) as the appropriate combinatorial setup which can serve as the receptacle for the data of a family of mutation-related Newton-Okounkov bodies. 
We then define in \Cref{sec: space of points} the \textbf{space of points $\Sp(\MM)$} of a polyptych lattice $\MM$, which is a certain subset of the functions on $\MM$, cf.\ \Cref{definition: space of points}. The definition is motivated in part from the study of convex functions on vector spaces obtained as minima of a finite collection of linear functions.  
Using $\Sp(\MM)$, we define two polyptych lattice analogues of the classical notion of convexity, one of which involves ``broken lines'' (an analogue of a notion appearing in cluster theory, and developed in the paper of Frias-Medina and Magee \cite{FriasMedina-Magee}); we also define a \textbf{PL half-space}. We are also able to define a semialgebra $S_\MM$ which is canonically associated to any polyptych lattice; our concept generalizes the classical semialgebra of integral polytopes (with respect to the operations ``convex hull of union'' and ``Minkowski sum'').

In \Cref{sec: duals} we develop a theory of duality for polyptych lattices, and use it to study the convex geometry of polyptych lattices in Section~\ref{sec_PL_polytopes}. Imposing more structure on the pairings between two polyptych lattices allows us to connect more tightly the piecewise-linear structure on both sides, yielding our definition of a \textbf{strict dual pair}. We then define both the notions of a \textbf{PL polytope} $\PP$, obtained as intersections of PL half-spaces, and (in the presence of, and with respect to, a strict dual) its \textbf{dual PL polytope} $\PP^\vee$. We also define a polyptych lattice analogue to classical Gorenstein-Fano polytopes, cf.\ \Cref{definition: PL Gorenstein Fano}.

In \Cref{sec_detrop}, we connect the combinatorics to algebraic geometry, which will lead to our main theorems. In the classical case of the torus $T \cong (\C^*)^r$, its coordinate ring is the Laurent polynomial ring $\C[x_1^{\pm},\cdots,x_r^{\pm}]$, and there is a map $\fv$ assigning to each $f \in \C[x_1^{\pm}, \cdots, x_r^{\pm}]$ its Newton polytope, i.e. $\fv(f) := \mathrm{Newt}(f)$, an integral polytope in $\Z^r \otimes \R \cong \R^r$. The set of integral polytopes forms an idempotent semialgebra $\mathcal{S}$ (cf.\ \Cref{subsec: semialgebra}) and the map $\fv$ above can be interpreted as a valuation with values in $\mathcal{S}$. In \Cref{def_valuation} we define (quasi)valuations $\fv$ with values in idempotent semialgebras in more general situations; an important case is when the codomain is the canonical semialgebra $S_\MM$ associated to a polyptych lattice $\MM$, which was mentioned above. In the presence of a strict dual, we are able to use the semialgebra-valued valuations to produce full-rank valuations which generalize the weight (quasi)valuations (arising from maximal prime cones in a tropical variety) introduced by Kaveh and the third author in \cite{KavehManon-Siaga}. In part motivated by this, we call certain pairs $(\Aa_\MM, \fv: \Aa_\MM \to S_\MM)$ a \textbf{detropicalization} of $\MM$ (cf.\ \Cref{def_detrop} for the precise statement). As we already mentioned above, our constructions generalize and systematize the wall-crossing phenomena of Newton-Okounkov bodies discussed in \cite{EscobarHarada}.  This discussion then leads directly to our construction of compactifications, and the proofs of our main theorems (Theorems A,B,C above), in Section~\ref{section: compactifications}.

\bigskip

We already mentioned above that our theory is a simultaneous generalization of toric, tropical, and cluster geometry, among others. Indeed, the theory of mirror symmetry, and relatedly, the theory of cluster varieties were key motivations for undertaking this work. We now take a moment to discuss these connections in some more detail. 

 In Section \ref{sec_PL_polytopes} we initiate the study of polytopes in polyptych lattices. In forthcoming work \cite{EscobarHaradaManon-GF} we introduce the notion of a Gorenstein-Fano PL polytope in a polyptych lattice; in an appropriate setting as in our Section~\ref{section: compactifications}, the corresponding compactification (projective variety) associated to such a polytope is Fano. We anticipate that hypersurfaces in these varieties will provide a new source of Calabi-Yau varieties, generalizing hypersurfaces in Fano toric varieties.  Along these lines, it is natural to ask if the famous results of Batyrev and Borisov \cite{BatyrevBorisov} generalize to the polyptych setting.

Another connection to cluster varieties comes from considering functions $\Psi_\PP$ on polyptych lattices. Namely, under some technical hypotheses, one can associate to a PL polytope $\PP\subset\MM_\R$ the function $\Psi_\PP := \min\{\v(m) \mid m\in\PP\cap\MM\}$ on $\NN$ and $\PP$ is recovered from the analogous function on the dual side as $\PP=\{m : \Psi_{\PP^\vee}(m)\ge -1\}$. These  
$\Psi_\PP, \Psi_{\PP^\vee}$ can be interpreted as a polyptych analogue of \textit{mirror superpotentials} for their associated compactifications $X_{\Aa_\MM}(\PP), X_{\Aa_\NN}(\PP^\vee)$, generalizing the form of mirror symmetry found by Rietsch and Williams for Grassmannians in \cite{RW} as well as the much more general case of Gross, Hacking, Keel, and Kontsevich in \cite{GHKK}.
In forthcoming work with Frias-Medina and Magee \cite{EFMHMM} we make the connection with cluster geometry concrete by showing that a lattice equipped with a skew-symmetric bilinear form gives rise to a self-dual polyptych lattice.  In this way, many cluster algebras can be understood as detropicalizations of self-dual polyptych lattices, where the dual pairing can be viewed as a tropicalization of theta reciprocity. There were already hints in this direction in the work of  
Cheung, Magee, Mandel, and Muller \cite{CheungMageeMandelMuller} and Bossinger \cite{Bossinger}.

A new theory needs concrete examples to illustrate its worth. First, in a companion note to this paper \cite{CookEscobarHaradaManon2024} we give detailed computations for a family of rank-$2$ examples, and give an explicit computation of the Cox ring of one of our compactifications. In \Cref{sec-Example} of this paper, we give a family of higher-rank examples, denoted $\MM_{d,r}$ and parametrized by two integers $d \geq 2, r \geq 2$. (In fact, $\MM_{d,r}$ can also be realized as a distinguished subcomplex of the tropical variety of the ideal defining the detropicalization $\Aa_{d,r}$ given in \Cref{subsec: detrop MMdr}.) We also show that $\MM_{d,r}$ and $\MM_{r,d}$ (note the switch) are a strict dual pair. Moreover, we give an explicit example of a chart-Gorenstein-Fano PL polytope $\PP$ in $\MM_{d,r}$, which therefore implies that their associated classical polytopes $P_\alpha$ in each coordinate chart are Gorenstein-Fano in the classical sense. Moreover, Theorem C implies that these classical Gorenstein-Fano polytopes $P_\alpha$ are not only combinatorially related by mutation, but also geometrically related in the sense that the compactification $X_{\Aa_{\MM_{d,r}}}(\PP)=\mathrm{Proj}(\Aa^\PP_{\MM_{d,r}})$ degenerates in a flat family to each of the toric varieties $X(P_\alpha)$ associated to these polytopes.
Finally, because our detropicalization 
$\Aa_{\MM_{d,r}}$ is a UFD, our compactification $\mathrm{Proj}(\Aa_{\MM_{d,r}}^\PP)$ has finitely generated class group and finitely generated Cox ring.
In \cite{EscobarHaradaManon-Gcone} we give an algorithm for constructing a new dual pair of detropicalizable polyptych lattices, based on the data of two pairs of strictly-dual and detropicalizable polyptych lattices, together with a PL cone in each (which satisfy conditions analogous to the classical Gorenstein situation). 
  In this way we can produce a large number of examples of dual pairs of detropicalizable polyptych lattices; in particular, we can build examples of finite type polyptych lattices which do not arise from cluster algebras.

We close this introduction with a brief sample of further questions inspired by the constructions in this paper. Ilten has suggested to us that it would be of interest to find a polyptych-lattice interpretation of Mukai varieties and their toric degenerations, as developed in e.g.\ \cite{ChristophersenIlten2014, ChristophersenIlten2016}.  Moreover, Ilten \cite{Ilten2012, Ilten2011} has constructed projective $T$-varieties of complexity $1$ which simultaneously degenerate to the toric varieties associated to two mutation-equivalent lattice polytopes. Thus, it is natural to ask the following: 
can Ilten's complexity-$1$ $T$-variety be realized as one of our compactifications $X_{\Aa_\MM}(\PP)$ of $\Spec(\Aa_\MM)$, for suitable choices of polyptych data? 
We also note that our constructions can exhibit different toric Fano varieties as arising as degenerations of a single compactification $\mathrm{Proj}(\Aa_\MM^\PP)$; thus, we may ask what this implies about the moduli space of (toric) Fano varieties. Finally, it is also natural to ask if a form of Batyrev-Borisov mirror symmetry holds for complete intersections in $X_{\Aa_\MM}(\PP)$ when $\PP$ is a chart-Gorenstein-Fano PL polytope (cf.\ \Cref{definition: PL Gorenstein Fano}). We intend to pursue these and other questions in future work.

\subsection*{Statement on AI Use.} 
The authors acknowledge the use of AI tools to assist with proofreading this manuscript. Generative AI was not used to formulate the theorems or to construct the proofs.

\subsection*{Acknowledgements.} 
We thank Alessio Corti, Bosco Frias-Medina, and Timothy Magee for interest in our project and for useful conversations. We are grateful to Silas Vriend for pointing out errors in our definitions in previous versions of this manuscript and suggesting the name ``polyptych lattice''. We also thank Federico Castillo for helping us with the theory of polytopes. We thank Eugene Gorsky for explaining braid varieties (and their cluster structures) to us, and we thank Nathan Ilten for suggesting the question about Mukai varieties. We also thank the anonymous referees for numerous constructive comments and questions, all of which improved the manuscript.

LE was supported in part by a Fields Research Fellowship from the Fields Institute for Research in the Mathematical Sciences. LE is also supported by an NSF CAREER grant DMS-2142656 and DMS-2521270. MH was supported by a Canada Research Chair Award (Tier 2) and NSERC Discovery Grant 2019-06567. CM is supported by NSF DMS grant 2101911 and 2501468. 

This material is also based upon work supported by the National Science Foundation under Grant No. DMS-1928930, while the authors were in residence at the Simons Laufer Mathematical Sciences Institute in Berkeley, California, during the summer of 2025.

\section{Polyptych lattices}\label{sec: PL}

We begin by introducing the central object in this manuscript. For this purpose, we use (and generalize) many notions from classical convex geometry; our notation and definitions mainly follow that of \cite{Cox_Little_Schenck}.

First, we set some terminology. Let $V$ be a finite-dimensional real vector space, with a fixed isomorphism $V \cong \R^n$. In this manuscript, we say that a \textbf{fan} $\Sigma$ in $V$ is a finite collection of non-empty rational polyhedral cones $\Sigma =\{C\}$ such that every nonempty face of a cone is also a cone in $\Sigma$, and, the intersection of any two cones in $\Sigma$ is a face of both.\footnote{Some authors do not require that the cones in a fan are rational and polyhedral. Note that our usage differs from that of \cite{Cox_Little_Schenck}, wherein it is additionally assumed that the cones are strongly convex.} 
The \textbf{support} of a fan $\Sigma$ is $\lvert \Sigma \rvert := \cup_{C \in \Sigma} C$. A fan $\Sigma$ is \textbf{complete} if its support is all of $V$. 
Next we define our notion of piecewise-linearity. Let $V,W$ be real vector spaces. For the purposes of this manuscript\footnote{In the literature, there are more general definitions of piecewise-linearity, but we wish to restrict to a special case.}, a map $\psi: V \to W$ is said to be \textbf{piecewise linear} if $\psi$ is continuous, and, there exists a complete fan $\Sigma$ in $V$ such that for each cone $C$ in $\Sigma$, the restriction $\psi \vert_C$ of $\psi$ to $C$ is $\R$-linear in $C$. Now suppose $F$ is a subring of $\R$ and $M$ and $N$ are finite-rank free $F$-modules. We will say a cone in $M \otimes_F \R$ is \textbf{$F$-rational} if it can be defined by inequalities with coefficients which are $\Q$-linear combinations of elements of $F$, and an $F$-rational fan is then a finite collection of non-empty $F$-rational polyhedral cones satisfying the same conditions as a fan $\Sigma$ in $V$ above. (When $F=\Z$, $F$-rationality reduces to the standard notion of rationality.) We also say that a function $\psi: M \to N$ is \textbf{piecewise $F$-linear} if there exists a complete $F$-rational fan $\Sigma$ in $M \otimes_F \R$ such that for each cone $C \in \Sigma$, the restriction $\psi \vert_{C \cap M}$ is $F$-linear. We sometimes drop the reference to the base field $F$ when it is understood by context.

With the above terminology in place, we can make the following definition.

\begin{definition}\label{definition: polyptych lattice}
Let $r$ be a positive integer and let $\Z \subseteq F \subseteq \R$ be a subring of $\R$. A \textbf{polyptych lattice (PL) of rank $r$ over $F$} is a pair $\mathcal{M} := (\{M_\alpha\}_{\alpha \in \mathcal{I}}, \{\mu_{\alpha,\beta}: M_\alpha \to M_\beta\}_{\alpha,\beta \in \mathcal{I}})$ consisting of a collection $\{M_\alpha\}_{\alpha \in \mathcal{I}}$ of free $F$-modules, each of rank $r$ and indexed by a set $\mathcal{I}$, and a collection of piecewise-linear maps $\mu_{\alpha,\beta}: M_\alpha \to M_\beta$ for every pair $(\alpha,\beta)$ of indices, satisfying the following conditions:
\begin{enumerate} 
\item $\mu_{\alpha,\alpha} = \mathrm{Id}_{M_\alpha}$ is the identity map for all $\alpha \in \mathcal{I}$, 
\item $\mu_{\alpha,\beta} = \mu_{\beta,\alpha}^{-1}$ for all pairs $\alpha,\beta \in \mathcal{I}$, and 
\item $\mu_{\beta,\gamma} \circ \mu_{\alpha,\beta} = \mu_{\alpha,\gamma}$ for all triples $\alpha,\beta,\gamma \in \mathcal{I}$. 
\end{enumerate} 
Note in particular that the requirement (2) above implies that all the maps $\mu_{\alpha,\beta}$ are invertible. We call the maps $\mu_{\alpha,\beta}$ \textbf{mutations}, and we call  $M_\alpha$ a \textbf{chart} (or the \textbf{$\alpha$-th chart}) of $\mathcal{M}$. When $\mathcal{I}$ is finite, we say $\mathcal{M}$ is a \textbf{finite} polyptych lattice. 
\exampleqed
\end{definition} 

 \begin{remark} 
In this manuscript, we focus our attention on \textbf{finite} polyptych lattices. 
\exampleqed
\end{remark}

Given a polyptych lattice $\MM$ of rank $r$ over $F \subset \R$ and $F'$ with $F \subseteq F' \subseteq \R$, we may define a polyptych lattice $\MM \otimes_F F'$ of rank $r$ over $F'$ by tensoring all charts with $F'$ and using the same mutation maps (or more precisely, their natural extensions to $M_\alpha \otimes_F F'$). We let $\MM_{F'}$ denote the resulting polyptych lattice over $F'$. We will most often be concerned with the cases where the rings in question are $\Z,\Q$, or $\R$.

In analogy with the definition of manifolds in differential topology, given the data of a polyptych lattice $\MM = (\{M_\alpha\}_{\alpha \in \mathcal{I}}, \{\mu_{\alpha,\beta}: M_\alpha \to M_\beta\})$ of rank $r$ over $F$, we now build an associated topological space. First, for each $\alpha \in \mathcal{I}$ we fix an identification of $F$-modules $M_\alpha \cong F^r \subset \R^r$ and equip $M_\alpha$ with the subspace topology with respect to the standard Euclidean topology on $\R^r$. (It is straightforward to check that this is independent of the choice of isomorphism $M_\alpha \cong F^r$.) By slight abuse of notation we denote also by $\MM$ the quotient space 
\begin{equation}\label{eq: def elements of MM} 
\MM := \bigsqcup_{\alpha \in \mathcal{I}} M_\alpha \bigg/ \sim
\end{equation} 
where the equivalence relation is defined by $m_\alpha \sim m_\beta$, for $m_\alpha \in M_\alpha, m_\beta \in M_\beta$, precisely when $\mu_{\alpha,\beta}(m_\alpha)=m_\beta$. This construction then equips $\MM$ with the associated quotient topology. Note also that any equivalence class $m \in \MM$ has a unique representative in each chart $M_\alpha$. This motivates the following definitions.

\begin{definition}\label{def: element of MM} 
Let $\MM$ be a polyptych lattice of rank $r$ over $F$ as in Definition~\ref{definition: polyptych lattice}. An \textbf{element} of $\MM$ is an equivalence class in the quotient space in \Cref{eq: def elements of MM}. The \textbf{$\alpha$-th chart map,} or \textbf{$\alpha$-th coordinate (map),} is the association 
\begin{equation*}\label{eq: def alpha coordinate} 
\pi_\alpha: \MM \to M_\alpha, \quad m \mapsto m_\alpha,
\end{equation*}
taking an equivalence class $m \in \MM$ to its unique representative $m_\alpha$ in $M_\alpha$. We call $\pi_\alpha(m)$ the \textbf{$\alpha$-th chart representative}, or \textbf{$\alpha$-th coordinate}, of $m \in \MM$. 
\exampleqed
\end{definition} 

\begin{notation} 
Motivated by the above definition, we sometimes denote the indexing set $\mathcal{I}$ associated to a polyptych lattice $\MM = (\{M_\alpha\}_{\alpha \in \mathcal{I}}, \{\mu_{\alpha,\beta}\}_{\alpha,\beta \in \mathcal{I}})$ by $\pi(\MM) := \mathcal{I}$. 
\end{notation} 

\begin{remark}\label{remark: zero in MM}
Since all mutation maps are by definition piecewise-linear, and in particular send $0$ to $0$, this implies that there is a unique element of $\MM$ with $\alpha$-th chart representative equal to $0 \in M_\alpha$, for all $\alpha$. We will denote this element as $0_\MM$. 
\exampleqed
\end{remark}

\begin{example}\label{ex: running example}
We give an example of a rank-$2$ polyptych lattice over $\Z$, defined with $2$ charts. More precisely, let $\MM=(\{M_1,M_2\},\{\mu_{1,2}\})$, where $M_1=M_2=\Z^2$ (equipped with the standard basis $\{\varepsilon_1,\varepsilon_2\}$) and $\mu_{12}:M_1\to M_2$ is given in coordinates with respect to the standard basis by 
\begin{equation*}
\mu_{12}(x,y)=(\min\{0,y\}-x,y)
.\end{equation*}
So an element of $\MM$ is an equivalence class $m$ in $M_1 \sqcup M_2/\sim$ where $\pi_1(m)=(x,y) \in \Z^2 \cong M_1$ and $\pi_2(m)=(\min\{0,y\}-x,y) \in \Z^2$. A concrete example is $m \in \MM$ with $\pi_1(m)=(1,-1)$ and $\pi_2(m)=(-2,-1)$. 
\exampleqed
\end{example}

  Unlike the situation of a classical lattice, there does not exist in general a well-defined operation of addition on (the set of elements of) $\MM$, since the mutation maps $\mu_{\alpha,\beta}$ are only piecewise-linear and not necessarily linear. More precisely, given two elements $m , m' \in \MM$, the sum $\pi_\alpha(m)+\pi_\alpha(m')$ may not be identified via $\mu_{\alpha,\beta}$ to the sum $\pi_\beta(m)+\pi_\beta(m')$. In light of this, we set the following notation. For $m,m' \in \MM$, and $\alpha \in \pi(\MM) = \mathcal{I}$, we define 
    \begin{equation}\label{eq: addition in chart}
m +_\alpha m' := \pi_\alpha^{-1}(\pi_\alpha(m) + \pi_\alpha(m'))
    \end{equation}
   which we think of as ``addition in the chart $M_\alpha$''.

Next we define a PL analogue of the notion of a convex rational polyhedral cone. 

{
\begin{definition}\label{definition: real elements}
For a polyptych lattice $\MM$ over $F$ we define $\MM_\R$ to be the polyptych lattice over $\R$ defined by the charts $M_\alpha\otimes_F \R$ for $\alpha \in \mathcal{I}$ equipped with the natural extension of the mutation maps $\mu_{\alpha, \beta}: M_\alpha \otimes_F \R \to M_{\beta}\otimes_F \R$. 
\end{definition}
}

\begin{definition}\label{definition: PL cone} 
Let $\MM$ be a polyptych lattice over $F$. 
A \textbf{PL cone over $F$} is a subset $\mathcal{C} $ of $\MM_{\R}$ such that $\pi_\alpha(\mathcal{C}) \subseteq M_\alpha \otimes_{F} \R$ is an $F$-rational polyhedral cone for each $\alpha \in \pi(\MM)$ (cf.\ \cite[Definition 1.2.1, Definition 1.2.14]{Cox_Little_Schenck}).  
\exampleqed
\end{definition}

The \textbf{dimension} of a PL cone $\mathcal{C}$ is the dimension of any chart image $\pi_\alpha(\mathcal{C})$; it is easily seen that this is independent of the choice of $\alpha$. Given a PL cone $\mathcal{C}$, a \textbf{face} $\mathcal{C}'$ of $\mathcal{C}$ is a subset of $\mathcal{C}$ such that $\pi_\alpha(\mathcal{C}')$ is a face of $\pi_\alpha(\mathcal{C})$ for all $\alpha \in \pi(\MM)$. A \textbf{facet} of a PL cone $\mathcal{C}$ is a face of dimension $\dim(\mathcal{C})-1$. (See \cite[\textsection 1.2]{Cox_Little_Schenck} for the corresponding classical notions.) Any face of $\mathcal{C}$ is itself a PL cone, as can easily be checked. The following is a  natural analogue of the classical notion of a fan \cite[\textsection 3.1]{Cox_Little_Schenck}. (We note that we do not require our PL cones to be strongly convex.) 

\begin{definition}\label{definition: PL fan}
Let $\MM$ be a polyptych lattice over $F$. A \textbf{PL fan over $F$ in $\MM_{\R}$} is a finite collection $\Sigma $ of PL cones in $\MM_{\R}$ such that: 
\begin{enumerate} 
\item for every $\mathcal{C} \in \Sigma$ and every $\alpha \in \pi(\MM)$, the chart image $\pi_\alpha(\mathcal{C})$ is an $F$-rational polyhedral cone, 
\item for every $\mathcal{C} \in \Sigma$, each face of $\mathcal{C}$ is also in $\Sigma$, 
\item for all $\mathcal{C},\mathcal{C}' \in \Sigma$, the intersection $\mathcal{C} \cap \mathcal{C}'$ is a face of each, (and hence also in $\Sigma$). 
\end{enumerate}
The \textbf{support} of a PL fan is $\lvert \Sigma \rvert := \cup_{\mathcal{C} \in \Sigma} \mathcal{C}$. A PL fan in $\MM_{\R}$ is \textbf{complete} if $\lvert \Sigma \rvert = \MM_{\R}$. A PL fan $\Sigma'$ \textbf{refines} a PL fan $\Sigma$ if every $\mathcal{C}' \in \Sigma'$ is contained in a PL cone of $\Sigma$ and $\lvert \Sigma' \rvert = \lvert \Sigma \rvert$.  
\exampleqed
\end{definition}

The following is immediate from the definitions. We record it for reference. 

\begin{lemma}\label{lemma: PL fan projection}
Let $\MM$ be a finite polyptych lattice over $F$ and let $\Sigma $ be a PL fan over $F$ in $\MM_\R$. Then for any $\alpha \in \mathcal{I}$, the images $\{\pi_\alpha(\mathcal{C})\mid \mathcal{C}\in\Sigma\}$ form a fan, denoted $\pi_\alpha(\Sigma)$, in $M_\alpha \otimes_{F} \R$. Moreover, the fans $\pi_\alpha(\Sigma)$ for varying $\alpha$ are related by the mutation maps of $\MM$.  
\end{lemma}

 When $\MM$ is a finite polyptych lattice, it turns out that the piecewise linearity of the mutations $\mu_{\alpha,\beta}$ defining $\MM$ allow us to define a PL fan naturally associated to $\MM$, as we now explain. Indeed, by definition of piecewise linearity, for any pair $(\alpha,\beta) \in \mathcal{I}^2$, there exists a minimal fan $\Sigma(\MM, \alpha,\beta)$ in $M_\alpha \otimes_{F} \mathbb{R}$ such that, for each cone $C \in \Sigma(\MM, \alpha, \beta)$, the restriction $\mu_{\alpha,\beta} \vert_C: C \to M_\beta\otimes\mathbb{R}$ is $\mathbb{R}$-linear. Now let $\alpha$ be fixed. Let $\Sigma(\MM, \alpha)$ denote the common refinement of all $\Sigma(\MM,\alpha,\beta)$ as $\beta$ ranges over the finite set $\mathcal{I}=\pi(\MM)$. This is a fan in $M_\alpha \otimes_{F} \R$ which, by construction, has the property that for any cone $C$ of $\Sigma(\MM, \alpha)$ and any $\beta \in \mathcal{I}$, the  mutation $\mu_{\alpha,\beta}$ restricts to $C$ to be linear. 
Now let $\MM_\R = \bigcup_{C \in \Sigma(\MM,\alpha)} \pi_\alpha^{-1}(C)$ be the decomposition of $\MM_\R$ into preimages of the cones in $\Sigma(\MM,\alpha)$. We call this decomposition the \textbf{mutation fan of $\MM$}, and denote it by $\Sigma(\MM)$. 
This terminology is justified by the following lemma.

\begin{lemma}\label{lemma: fan of MM}
Let $\MM$ be a finite polyptych lattice over $F$, let $\alpha \in \mathcal{I}$, and let $\Sigma(\MM,\alpha)$ be defined as above. Then: 
\begin{enumerate}
\item[(a)] for $C \in \Sigma(\MM,\alpha)$, the preimage $\mathcal{C} := \pi_\alpha^{-1}(C) \subset \MM_\R$ is a PL cone over $F$, 
\item[(b)] the decomposition $\MM_\R = \bigcup_{C \in \Sigma(\MM,\alpha)} (\mathcal{C} := \pi_\alpha^{-1}(C))$ is a complete PL fan over $F$, denoted $\Sigma(\MM)$,
\item[(c)] for any $\beta, \gamma \in \mathcal{I}$, the mutation $\mu_{\beta,\gamma}: M_\beta \to M_\gamma$ is linear on $\pi_\beta(\mathcal{C}) = \pi_\beta(\pi_\alpha^{-1}(C))$ for any $\mathcal{C} \in \Sigma(\MM)$, and $\Sigma(\MM)$ is the minimal PL fan with this property, and 
\item[(d)] $\Sigma(\MM)$ is independent of the choice of $\alpha$. 
\end{enumerate} 
\end{lemma}

\begin{proof} 
For (a), it suffices to show that for any $\beta \in \mathcal{I}$, the image $\pi_\beta \pi_\alpha^{-1}(C)$ is a $F$-rational polyhedral cone in $M_\beta \otimes_{F} \R$. If $\beta=\alpha$, the claim is immediate. For any $\beta \neq \alpha$, we know $\pi_\beta \circ \pi_{\alpha}^{-1} = \mu_{\alpha, \beta}$, and by the definition of $\Sigma(\MM,\alpha)$, the mutation is linear (and invertible) on $C$. Moreover, by definition of a polyptych lattice, the mutation is a map of $F$-lattices. It follows that the image $\mu_{\alpha,\beta}(C)$ must also be $F$-rational and polyhedral. This proves (a). By similar arguments, and using that $\Sigma(\MM,\alpha)$ is a (classical) fan in $M_\alpha \otimes_{F}\R$ (note $\Sigma(\MM,\alpha)$ is a finite collection of cones because $\MM$ is assumed finite), it also follows that the decomposition $\MM = \cup_{C \in \Sigma(\MM,\alpha)} \pi_\alpha^{-1}(C)$ is a PL fan. It is complete because $\Sigma(\MM,\alpha)$ is complete by construction. This proves (b). 
Next we prove (c), for which we need to show that for any PL cone $C \in \Sigma(\MM,\alpha)$ and any $\beta,\gamma \in \mathcal{I}$, the mutation $\mu_{\beta,\gamma}: M_\beta \to M_\gamma$ is linear on $\pi_\beta(\pi_{\alpha}^{-1}(C)) = \mu_{\alpha,\beta}(C)$. To see this, we observe that by conditions (2) and (3) of \Cref{definition: polyptych lattice} we have $\mu_{\beta,\gamma} = \mu_{\alpha,\gamma} \circ \mu_{\beta,\alpha} = \mu_{\alpha,\gamma} \circ \mu_{\alpha,\beta}^{-1}$. Thus it suffices to show that $\mu_{\alpha,\beta}^{-1}$ is linear on $\mu_{\alpha,\beta}(C)$, and that $\mu_{\alpha,\gamma}$ is linear on $C$. The first statement is true since $\mu_{\alpha,\beta}$ is both linear and invertible on $C$ by assumption, hence its inverse is also linear; the second claim is by assumption on $C$. Now we claim that $\Sigma(\MM)$ is minimal among PL fans with the stated property in (c). To see this, suppose $\Sigma'$ is a PL fan with this property. We want to show that any PL cone $\mathcal{C}' \in \Sigma'$ is contained in a PL cone of $\Sigma(\MM)$, for which it suffices to show $\pi_\alpha(\mathcal{C}')$ is contained in a cone of $\Sigma(\MM,\alpha)$. By assumption, for any $\beta \in \mathcal{I}$, $\mu_{\alpha,\beta}$ is linear on $\pi_\alpha(\mathcal{C}')$, i.e., $\pi_\alpha(\mathcal{C}')$ is contained in a cone of linearity of $\mu_{\alpha,\beta}$. We denote this cone by $C_{\alpha,\beta}$. Repeating this argument for all $\beta \in \mathcal{I}$, we obtain $\pi_\alpha(\mathcal{C}') \subset \cap_{\beta \in \mathcal{I}} C_{\alpha,\beta}$.  The cones of $\Sigma(\MM,\alpha)$ are precisely these intersections, so the claim is proved. Finally, the condition given in (c) is independent of $\alpha$, so claim (d) follows. 
\end{proof}

\begin{example}
    For the $\MM$ described in \Cref{ex: running example}, the mutation fan $\Sigma(\MM)$ is given by the (preimage under $\pi_1$ of) the two half-spaces $\{y \geq 0\}$ and $\{y \leq 0\}$ in $M_1 \otimes \R$.
    \exampleqed
\end{example}

 As we noted above, we cannot in general give a well-defined notion of addition on $\MM$. However, by construction of $\Sigma(\MM)$, addition is uniquely determined on PL cones of $\Sigma(\MM)$. More precisely, we have the following.

 \begin{lemma}\label{lemma: addition well def on cones}
Let $\MM$ be a finite polyptych lattice over $F$ and let $\Sigma(\MM)$ be the associated mutation fan as defined above. Let $\mathcal{C} \in \Sigma(\MM)$. Then the operations of addition and scalar multiplication by non-negative scalars are well-defined on $\mathcal{C}$, i.e., 
\begin{itemize} 
\item for $m,m' \in \mathcal{C}$, the addition $m +_\alpha m' \in \mathcal{C}$ in the chart $\alpha$ is independent of $\alpha \in \mathcal{I}$ and, 
\item for $\lambda \in F_{\geq 0}, m \in \mathcal{C}$, the scalar multiplication of $m$ by $\lambda$, defined by $\pi_\alpha^{-1}(\lambda \cdot \pi_\alpha(m))$, is independent of $\alpha \in \mathcal{I}$. 
\end{itemize}
 \end{lemma}

 \begin{proof} 
Let $m,m' \in \mathcal{C}$. We wish to show that $m+_\alpha m' = m +_\beta m'$ for all pairs $\alpha,\beta \in \mathcal{I}$. Note that by construction of $\Sigma(\MM)$ and \Cref{lemma: fan of MM}(c), the mutation $\mu_{\alpha,\beta}$ restricts to be linear on $\pi_\alpha(\mathcal{C})$. This means 
$$
\mu_{\alpha,\beta}(\pi_\alpha(m)+\pi_\alpha(m')) = \mu_{\alpha,\beta}(\pi_\alpha(m))+\mu_{\alpha,\beta}(\pi_\alpha(m')) = \pi_\beta(m) + \pi_\beta(m').$$
This in turn implies, since $\mu_{\alpha,\beta} = \pi_\beta \circ \pi_\alpha^{-1}$, that $m+_\alpha m' = m+_\beta m'$, as was to be shown. Moreover, since $\mathcal{C}$ is a PL cone, $\pi_\alpha(\mathcal{C})$ is a classical cone for any $\alpha$, from which it follows that $\pi_\alpha(m)+\pi_\alpha(m') \in \pi_\alpha(\mathcal{C})$, i.e., $m +_\alpha m' \in \mathcal{C}$. The well-definedness of scalar multiplication follows from the piecewise $F$-linearity of any mutation map, which implies that scalar multiplication by non-negative scalars commutes with mutation. 
 \end{proof}

 \begin{notation}\label{notation: scalar mult on cones}
By \Cref{lemma: addition well def on cones} there is a well-defined notion of scalar multiplication by non-negative scalars on any element of $\MM$. For $\lambda \in F_{\geq 0}$ and $m \in \MM$, we denote the result here and below by $\lambda \cdot m$ or more simply, $\lambda \, m$. The operation extends in a straightforward manner to $\MM_\R$ and multiplication by scalars in $\R_{\geq 0}$, and we use the same notation in this case. Moreover, for $m,m' \in \mathcal{C}$ for a cone $\mathcal{C}$ in $\Sigma(\MM)$, we also have a well-defined addition, and we sometimes denote this by $m+m' \in \mathcal{C}$. 
\exampleqed
 \end{notation}

 Our next step is to define functions on polyptych lattices which are piecewise linear. 

\begin{definition}\label{definition: piecewise linear on PL}
Let $\MM = (\{M_\alpha\}_{\alpha \in \mathcal{I}}, \{\mu_{\alpha,\beta}: M_\alpha \to M_\beta\}_{\alpha,\beta \in \mathcal{I}})$ be a polyptych lattice over $F$. A function $\psi: \MM \to F$ on the set of elements of $\MM$ is said to be \textbf{piecewise $F$-linear} if $\psi \circ \pi_\alpha^{-1}: M_\alpha \to F$ is piecewise $F$-linear for each $\alpha \in \mathcal{I}$.  We denote the natural extension of $\psi$ to $\MM_\R$ by $\psi_\R: \MM_\R \to \R$. 
\exampleqed
\end{definition} 

\begin{remark}\label{notation: functions as tuples of functions} 
Any function $\psi: \MM \to S$ for any set $S$ can be equivalently encoded by translating it to a function the $\alpha$-th chart, i.e., $\psi_\alpha := \psi \circ \pi_{\alpha}^{-1}$, for any $\alpha \in \pi(\MM)$. Note that these satisfy the compatibility relations $\psi_\beta \circ \mu_{\alpha,\beta} = \psi_\alpha$ for all $\alpha,\beta \in \mathcal{I}$. 
\exampleqed
\end{remark}

We will pay particular attention to those piecewise linear functions which are convex. 
In the classical context, recall that a function $\psi: V \to \R$ from a real vector space $V$ to $\R$ is \textbf{convex} if 
\begin{equation}\label{eq: definition convex}
\psi(t \, m + (1-t)m') \geq t \, \psi(m) + (1-t)\psi(m')
\end{equation}
for all $m, m' \in V$ and $t \in [0,1]$ \cite[Definition 6.1.4]{Cox_Little_Schenck}. Based on this, we make the following definition.  

\begin{definition}\label{definition: convex}
Let $\MM$ be a polyptych lattice over $F$. 
A piecewise linear function $\psi$ on $\MM_\R$ is \textbf{convex} if $\psi_\alpha := \psi \circ \pi_{\alpha}^{-1}: M_\alpha \otimes_{F}\R \to \R$ is convex in the sense of \Cref{eq: definition convex} for all $\alpha \in \pi(\MM)$.
\exampleqed
\end{definition}

The last definition of this section is a notion of a direct product of polyptych lattices. 

\begin{definition} 
Let $\Z \subset F \subset \R$ be a subring of $\R$. 
Let $\MM = (\{M_\alpha\}_{\alpha \in \mathcal{I}}, \{\mu_{\alpha,\beta}: M_{\alpha} \to M_\beta\}_{\alpha,\beta \in \mathcal{I}})$ and $ \MM' = (\{M'_{\alpha'}\}_{\alpha' \in \mathcal{I}'}, \{\mu'_{\alpha',\beta'}: M'_{\alpha'} \to M'_{\beta'}\})$ be two polyptych lattices of rank $r$ and $r'$ respectively, over $F$. We define the \textbf{product polyptych lattice $\MM \times \MM'$} to be  the polyptych lattice of rank $r+r'$ over $F$ given by charts $M_\alpha \times M'_{\alpha'}$ indexed by $\mathcal{I} \times \mathcal{I}'$ and mutation maps $\mu_{\alpha,\beta} \times \mu'_{\alpha',\beta'}: M_\alpha \times M'_{\alpha'} \to M_\beta \times M'_{\beta'}$ for pairs $(\alpha,\alpha'),(\beta,\beta')$ in $\mathcal{I} \times \mathcal{I'}$. It is straightforward to check that this satisfies the axioms of a polyptych lattice. Note that if both $\MM$ and $\MM'$ are finite, then $\MM \times \MM'$ is also a finite polyptych lattice.
\exampleqed
\end{definition} 

\begin{example}
    Let $\MM=(\{M_1,M_2\},\{\mu_{1,2}\})$ and $\MM'=(\{M_1',M_2'\},\{\mu_{1,2}'\})$ both be the polyptych lattice from \Cref{ex: running example}.
    The product polyptych lattice $\MM\times \MM'$ will then have rank $4$ as well as four charts: $M_1\times M_1'$, $M_1\times M_2'$, $M_2\times M_1'$, and $M_2\times M_2'$.
     The mutation maps are as follows
     \[
    \begin{tikzcd}[column sep=4em]
M_1\times M_1' \arrow[r, "\mu_{1,1}\times \mu'_{1,2}"] \arrow[d, "\mu_{1,2}\times \mu'_{1,1}"]
& M_1\times M_2' \arrow[d, "\mu_{1,2}\times \mu'_{2,2}"] \\
M_2\times M_1' \arrow[r, "\mu_{2,2}\times \mu'_{2,2}"]
& M_2\times M_2'. 
\end{tikzcd}
    .\]
    For example, 
    $$\mu_{1,1}\times \mu'_{1,2}((x,y),(x',y'))=((x,y),\mu'_{1,2}(x',y'))=((x,y),(\min\{0,y\}-x,y)).$$
\end{example}

\begin{remark}
    We can extend the notion of product to a product of finitely many polyptych lattices.
    Concretely, given $\MM^1,\ldots,\MM^k$ a finite list of polyptych lattices as below,
    $$\MM^j = (\{M^j_\alpha\}_{\alpha \in \mathcal{I}_j}, \{\mu_{\alpha,\beta}^j: M^j_{\alpha} \to M^j_\beta\}_{\alpha,\beta \in \mathcal{I}_j})$$
    then the product is the polyptych lattice
    \begin{equation*}
        \MM^1\times\cdots \MM^k
        =(\{M_{\alpha_1}^1\times \cdots M_{\alpha_k}^k\}_{(\alpha_1,\ldots, \alpha_k)\in \mathcal{I}_1\times\cdots\times \mathcal{I}_k}
        , \{\mu_{\alpha_1,\beta_1}^1\times\cdots\times\mu_{\alpha_k,\beta_k}^k\mid \forall j,\ \alpha_j,\beta_j\in \mathcal{I}_j \}
        ).
    \end{equation*}
\end{remark}

\section{The space of points of a polyptych lattice}\label{sec: space of points}

In this section we introduce a key concept which is the first step toward a theory of duality for polyptych lattices as well as the extension of a notion of polytopes to polyptych lattices.

\subsection{Definition and first properties}

 We begin with the main definition of this section. We use the notation $m+_\alpha m'$ and $\lambda m$ introduced in \Cref{eq: addition in chart} and \Cref{notation: scalar mult on cones} respectively. 

\begin{definition}\label{definition: space of points}
Let $\MM = (\{M_\alpha\}_{\alpha \in \mathcal{I}}, \{\mu_{\alpha,\beta}:M_\alpha \to M_\beta\}_{\alpha,\beta \in \mathcal{I}}\})$ be a polyptych lattice over $F$. A \textbf{point of $\MM$} is a function $p: \MM \to F$ such that 
\begin{equation}\label{eq: def point min} 
p(m) + p(m') = \min\{p(m +_\alpha m') \,\mid\, \alpha \in \pi(\MM) \} \, \, \textup{ for all } \, m, m' \in \MM
\end{equation} 
and
\begin{equation}\label{eq: def point F homog}
p(\lambda m) = \lambda p(m) \, \, \textup{ for all } \, m \in \MM, \lambda \in F_{\geq 0}. 
\end{equation} 
The set of all such $p: \MM \to F$ is called \textbf{the space of points of $\MM$} and denoted $\Sp(\MM)$. For $F \subset F'$, an \textbf{$F'$-point of $\MM$} is a point of $\MM \otimes_F F'$. We denote the set of $F'$-points as $\Sp_{F'}(\MM)$. (For the base ring $F$, we often drop the subscript.) 
\exampleqed
\end{definition} 

We will primarily be concerned with the space of points $\Sp(\MM)$ for a polyptych lattice $\MM$ defined over $\Z$ and its $\R$-points $\Sp_{\R}(\MM)$. 

\begin{remark}\label{remark: motivation for points}
One motivation for the terminology of ``points'' comes from algebraic geometry. A (closed) $\C$-point on an affine scheme over $\C$ is a morphism from the coordinate ring to the ``ground field'' $\C$. In an analogous way, we will see later that a point of $\MM$ (over $F$) induces a morphism of semialgebras over $F_{\geq 0}$ from the canonical $F_{\geq 0}$-semialgebra $S_\MM$ associated to $\MM$ (cf.\ \Cref{definition: canonical semialgebra MM}) to the ``ground $F_{\geq 0}$-semialgebra'' $\overline{F}$ equipped with semialgebraic operations $+$ and $\min$, see \Cref{lemma-pointsaresemialgebramaps}. From the theory of schemes over semirings (see e.g.\ \cite{Giansiracusa}), a point in the sense of \Cref{definition: space of points} can be interpreted as an $F$-point in the scheme $\mathrm{Spec}(S_\MM)$. 
\exampleqed
\end{remark}

\begin{remark}\label{remark: Fgeq0 closed SpMM}
It follows easily from the definition that $\Sp(\MM)$ is closed under (pointwise) $F_{\geq 0}$-scalar multiplication.
\exampleqed
\end{remark} 

We begin with an illustration of \Cref{definition: space of points} with our simple running example. 

\begin{example}\label{ex_points} 
Let $\MM$ be the rank-$2$ polyptych lattice over $\Z$ given in \Cref{ex: running example}. For this example, for computational clarity we temporarily notate an element $m$ of $\MM$ by listing the pair $(\pi_1(m),\pi_2(m)) \in \Z^2 \times \Z^2$.  
Now consider the following elements of $\MM$, $\mathbb{e}_1:=(\varepsilon_1,-\varepsilon_1)$, $\mathbb{e}_2:=(\varepsilon_2,\varepsilon_2)$, and $\mathbb{e}'_2:=(-\varepsilon_2,-\varepsilon_1-\varepsilon_2)$. 
If we choose values for $p(\mathbb{e}_1)$, $p(\mathbb{e}_2)$, and $p(\mathbb{e}'_2)$ such that $p(\mathbb{e}_2)+p(\mathbb{e}'_2)=\min\{0,p(\mathbb{e}_1)\}$, we can extend $p$ to a function on $\MM$ as follows,
\begin{equation}\label{eq_ex_point}
    p((x,y),\mu_{12}(x,y))=
    \begin{cases}
        xp(\mathbb{e}_1)-yp(\mathbb{e}'_2), & y\le 0
        \\
        xp(\mathbb{e}_1)+yp(\mathbb{e}_2), & y\ge 0
    \end{cases}
.\end{equation}
It is not hard to check explicitly that the above definition of $p$ satisfies the conditions to be a point in $\Sp(\MM)$, and moreover, that all points of $\MM$ are of this form. We refer to \cite{CookEscobarHaradaManon2024} for more details. 
\exampleqed
\end{example}

Next, we show that a point $p \in \Sp(\MM)$ is piecewise $F$-linear, and in particular, are linear on the PL cones $\mathcal{C}$ of $\Sigma(\MM)$.

\begin{lemma}\label{lemma: points linear on faces of SigmaM}
 Let $\MM$ be a finite polyptych lattice over $F$, and let $F \subset F'$. Let $p \in \Sp(\MM)$. Then $p$ is linear on the PL cones $\mathcal{C}$ of $\Sigma(\MM)$, and in particular is piecewise $F$-linear on $\MM$. 
Moreover, $p$ naturally extends to a piecewise $F'$-linear function on $\MM_{F'}$ which is linear on the PL cones of $\Sigma(\MM)$, and, this extension is in $\Sp_{F'}(\MM)$. 
\end{lemma} 

\begin{proof} 
Let $\mathcal{C} \in \Sigma(\MM)$. We need to show $p_\alpha := p \circ \pi_\alpha^{-1}$ is linear on $\pi_\alpha(\mathcal{C})$ for all $\alpha \in \pi(\MM)$, but we know from \Cref{lemma: addition well def on cones} that addition and $F_{\geq 0}$-multiplication are well-defined on $\mathcal{C}$. In fact, by \Cref{lemma: addition well def on cones} we have $m +_\alpha m' = m +_\beta m'$ for any $\alpha,\beta \in \mathcal{I}$, so the minimum on the RHS of \Cref{eq: def point min} is a minimum over a singleton set. From this and \Cref{eq: def point F homog} it follows that for any $m,m' \in \mathcal{C} \cap \MM$ and $\lambda, \mu \in F_{\geq 0}$, we have $p(\lambda m + \mu m') = \lambda p(m) + \mu p(m')$. Thus $p$ is linear on $\mathcal{C}$. For the second claim, it suffices to prove the case $F=\Z$ and $F'=\R$. 
 It is straightforward that $p$ can be naturally extended to $\MM_{\R}$, and this extension is piecewise $\R$-linear for each cone in $\Sigma(\MM)$. To prove continuity  it suffices to see that the linear maps $p_\alpha \vert_{\pi_{\alpha}(\mathcal{C})}$ agree on intersections of cones. But these intersections are spanned by elements of $\MM$, and the values on these elements are determined by the original function $p$. Hence they must agree, so the extension of $p$ to $\MM_\R$ is continuous. 
To see that the extension $p$ is in $\Sp_{\R}(\MM)$, we must show it satisfies \Cref{eq: def point min} for all $m,m' \in \MM_{\R}$. We argue in two steps. We first claim if \Cref{eq: def point min} holds over $\Z$ then it holds over $\Q$. This is because for any $m,m' \in \MM_{\Q}$ there exists $\lambda \in \Z$ such that $\lambda m, \lambda m' \in \MM$, and then the equation reduces to that in $\MM$. Next, if \Cref{eq: def point min} holds over $\Q$ then it holds over $\R$ since we already know $p$ is continuous, so both LHS and RHS extend to continuous functions and they are equal on a dense subset, hence equal everywhere. Since the original $p$ is in $\Sp_{\Z}(\MM)$ by assumption and hence satisfies \Cref{eq: def point min}, the result follows. 
\end{proof}

  Based on \Cref{lemma: points linear on faces of SigmaM}, we may make the following definition. 
 
\begin{definition}\label{definition: L_C}
Let $\MM$ be a polyptych lattice over $F$ and let $\mathcal{C} \subset \MM_\R$ be a PL cone in $\Sigma(\MM)$.   
We denote by $L_{\mathcal{C}}: \Sp(\MM) \to \Hom(\mathcal{C}\cap \MM, F)$ (respectively $L_{\mathcal{C}}: \Sp_\R(\MM) \to \Hom(\mathcal{C}, \R)$) the map sending $p: \MM \to F$ (resp. its natural extension to $p: \MM_\R \to \R$) to its restriction to $\mathcal{C}$, i.e. $p \vert_{\mathcal{C}\cap \MM}: \mathcal{C} \cap \MM \to F$ (resp. $p \vert_{\mathcal{C}}: \mathcal{C} \to \R$). This is well-defined by \Cref{lemma: points linear on faces of SigmaM}. 
\exampleqed
\end{definition}

 In some cases, the map $L_{\mathcal{C}}$ above is injective. 

\begin{lemma}\label{lem-convex-11}
Let $\MM$ be a finite polyptych lattice and let $\mathcal{C}$ be a maximal-dimensional PL cone in $\Sigma(\MM)$. 
The restriction map $L_{\mathcal{C}}: \Sp(\MM) \to \Hom(\mathcal{C} \cap \MM, F)$ (respectively $L_{\mathcal{C}}: \Sp_\R(\MM) \to \Hom(\mathcal{C}, \R)$) is injective, i.e., $p$ is determined uniquely by its values on $\mathcal{C} \cap \MM$ (resp. $\mathcal{C}$). 
\end{lemma}

\begin{proof}
The argument given here is for $\Sp(\MM) \to \Hom(\mathcal{C} \cap \MM,F)$ but the argument is similar for $\Sp_\R(\MM)$. 
   We wish to show that the value of $p$ on any element of $\MM$ is determined by its values on $\mathcal{C} \cap \MM$.   
   Fix any $v \in \MM$. If $v$ is in $\mathcal{C}$, there is nothing to show. Suppose $v \not \in \mathcal{C}$. Since $\MM$ is finite, there are only finitely many coordinate chart images of $v$, making it possible to pick $w \in \mathcal{C}\cap \MM$ such that $\pi_\alpha(v) + \pi_\alpha(w) \in \pi_\alpha(\mathcal{C})$ (i.e., $v +_\alpha w \in \mathcal{C}\cap \MM$), for all $\alpha \in \mathcal{I}$. Since $p$ is a point, we have that $p(v) + p(w) = \min\{p(v +_\alpha w) \, \mid \, \alpha \in \pi(\MM) = \mathcal{I}\}$. Since the RHS and $p(w)$ are both determined by $p \vert_{\mathcal{C}\cap \MM}$, so is $p(v)$, as was to be shown. 
\end{proof}

It should be remarked that the property of being determined entirely by its values on (maximal-dimensional) faces of $\Sigma(\MM)$ is very special to points in $\Sp(\MM)$; it is  certainly not true in general of arbitrary piecewise-linear functions on $\MM$. This illustrates the strength of the assumption~\eqref{eq: def point min} in \Cref{definition: space of points}.

We now relate the definition of points with the notion of convexity. 

\begin{lemma}\label{lemma: point PL concave}
Let $\MM$ be a finite polyptych lattice over $F$ and let $p \in \Sp(\MM)$. Then the natural extension $p: \MM_\R \to \R$ is convex in the sense of \Cref{definition: convex}. 
\end{lemma}

\begin{proof}
We need to show that $p_\alpha := p \circ \pi_\alpha^{-1}: M_\alpha \otimes_F \R \to \R$ satisfies \Cref{eq: definition convex} for all $\alpha \in \mathcal{I}$. 
Let $m, m' \in \MM_\R$ and $\alpha \in \mathcal{I}$, and $\pi_\alpha(m)=u, \pi_\alpha(m')=u' \in M_\alpha \otimes \R$. We have 
\begin{equation*}
    \begin{split} 
    p_\alpha(t \, u + (1-t) \, u') & = p_\alpha(t\, \pi_\alpha(m) + (1-t) \, \pi_\alpha(m')) \\
     & = p(\pi_\alpha^{-1}(t \, \pi_\alpha(m) + (1-t) \, \pi_\alpha(m'))) \\
     & = p(t \, m +_\alpha (1-t) \, m') \\
     & \geq p(t \, m) + p((1-t) \, m') \\ 
     & = t \, p(m) + (1-t) \, p(m') \\
     & = t \, p_\alpha(u) + (1-t) \, p_\alpha(u')
    \end{split} 
\end{equation*}
where we use the definition of $+_\alpha$, and for the inequality above, we use \Cref{lemma: points linear on faces of SigmaM} to conclude $p \in \Sp_\R(\MM)$, so $p$ satisfies \Cref{eq: def point min}. 
The argument is valid for any $\alpha$, so the claim follows. 
\noindent
\end{proof}

The following elementary observation will be useful in Section~\ref{sec_PL_polytopes}.

\begin{lemma}\label{lemma: nonzero point negative}
Let $\mathcal{M}$ be a finite polyptych lattice over $F$, let $p\in\mathrm{Sp}_\mathbb{R}(\mathcal{M})$
with $p\neq 0$, and let $U\subseteq\mathcal{M}_\mathbb{R}$ be a neighborhood of $0_{\mathcal{M}}$.
Then there exists $m\in U$ with $p(m)<0$.
\end{lemma}

\begin{proof}
Fix $\alpha \in \pi(\mathcal{M})$ and let $p_\alpha := p \circ \pi_\alpha^{-1} : M_\alpha \otimes \mathbb{R} \to \mathbb{R}$.
By Lemma~\ref{lemma: point PL concave} applied to $\mathcal{M}_\mathbb{R}$, $p_\alpha$ is convex. Moreover, $p_\alpha(0) = 0$ by Equation~\eqref{eq: def point F homog}.
Suppose for a contradiction that $p_\alpha \geq 0$ everywhere. Then for any $y$ we have $0=p_\alpha(0) \geq \frac{1}{2}p_\alpha(y)+\frac{1}{2}p_\alpha(-y) \geq 0$. This in turn means $p_\alpha(y) = 0$ since $p_\alpha(y) + p_\alpha(-y) = 0$ and $p_\alpha(y), p_\alpha(-y)\ge 0$. Thus $p_\alpha \equiv 0$,
contradicting $p \neq 0$. We conclude that $p_\alpha(y) < 0$ for some $y$. Since $\pi_\alpha$ is a homeomorphism
with $\pi_\alpha(0_{\mathcal{M}}) = 0$, we have $\lambda y \in \pi_\alpha(U)$ for sufficiently small
$\lambda > 0$, and $m := \pi_\alpha^{-1}(\lambda y) = \lambda\,\pi_\alpha^{-1}(y) \in U$ satisfies
$p(m) = \lambda p_\alpha(y) < 0$ by Notation~\ref{notation: scalar mult on cones} and Equation~\eqref{eq: def point F homog}.
\end{proof}

We saw in \Cref{lemma: points linear on faces of SigmaM} that a point $p: \MM \to F$ is piecewise $F$-linear when viewed on any chart $M_\alpha$. However, it may sometimes be the case that there exists some chart $M_\alpha$ on which $p_\alpha$ is actually linear, i.e., it has only one domain of linearity. We make the following definition.

\begin{definition}
Following notation as above, and for $\alpha \in \mathcal{I}$, 
we let $\Sp(\MM,\alpha) \subset \Sp(\MM)$ denote the subset of points of $\MM$ such that $p_\alpha: M_\alpha \to F$ is $F$-linear. 
\exampleqed
\end{definition}

These subsets $\Sp(\MM,\alpha)$ provide additional convex-geometric structure on the space of points. We begin to see this with the following proposition. 

\begin{proposition}\label{prop-conesinthepoints}
With notation as above, let $p: \MM \to F$ be a piecewise $F$-linear function on $\MM$. Then: 
\begin{enumerate} 
\item[(1)] $p$ is in $\Sp(\MM,\alpha)$ if and only if (the natural extension) $p: \MM_\R \to \R$ is convex, and, $p_\alpha: M_\alpha \to F$ is $F$-linear. 
\item[(2)] $p$ is in $\Sp(\MM,\alpha)$ if and only if $p$ is in $\Sp(\MM)$ and, for all $m, m' \in \MM$, the minimum in the RHS of \Cref{eq: def point min} 
is achieved by $\alpha \in \pi(\MM)$.  
\end{enumerate} 
Moreover, 
viewed as functions on $\MM$, $\Sp(\MM,\alpha)$ is closed under $F_{\geq 0}$-scalar multiplication and addition. 
\end{proposition}

\begin{proof}
We first prove (1). To prove one direction, suppose $p \in \Sp(\MM,\alpha)$. Then $p_\alpha: M_\alpha \to F$ is linear by definition of $\Sp(\MM,\alpha)$ and \Cref{lemma: point PL concave} implies $p$ is convex. Conversely, suppose that $p: \MM \to F$ is piecewise $F$-linear, its natural extension $p: \MM_\R \to \R$ is convex, and $p_\alpha := p \circ \pi_\alpha^{-1}: M_\alpha \to F$ is $F$-linear. We need to show that $p$ is a point, i.e., $p$ satisfies \Cref{eq: def point min} and \Cref{eq: def point F homog}.  Let $m, m' \in \MM$ and $\beta \in \mathcal{I}$. Set the notation $u = \pi_\beta(m)$ and $u'=\pi_\beta(m')$. Since $p_\beta$ is convex by assumption, 
\begin{equation}\label{eq: p beta convex}
p(m) + p(m') = p_\beta(u) + p_{\beta}(u') \leq p_{\beta}(u+u') = p(m +_{\beta} m').
\end{equation} 
This argument applies to all $\beta \in \mathcal{I}$, so we conclude 
\begin{equation}\label{eq: min over all beta}
p(m) + p(m') \leq \min_{\beta \in \mathcal{I}} \{p(m +_{\beta} m') \}. 
\end{equation}
On the other hand, $p_\alpha$ is linear by assumption, so for $\alpha \in \mathcal{I}$ the computation in \Cref{eq: p beta convex} yields $p(m)+p(m')= p(m +_\alpha m')$ which implies that $p(m +_\alpha m')$ achieves the minimum in the RHS of \Cref{eq: min over all beta} and also that $p(m)+p(m') = \min\{p(m +_\beta m') \, \mid \, \beta \in \mathcal{I} \}$. Thus $p$ satisfies \Cref{eq: def point min}. Since $p$ is piecewise $F$-linear, it also satisfies \Cref{eq: def point F homog}. Thus $p \in \Sp(\MM)$, proving (1). 
A similar argument proves (2).

Now we prove the last claim. We need to show that for any $p,q \in \Sp(\MM,\alpha)$ and any $\lambda,\mu \in F_{\geq 0}$,  we have $\lambda p+ \mu q \in \Sp(\MM,\alpha)$. By (1) we only need to show that $\lambda p+ \mu q$ is convex, and, linear on $M_\alpha$. Any non-negative combination of convex functions is convex, and any linear combination of linear functions is linear, so the result follows. 
\end{proof}

Given a point $p \in \Sp(\MM)$, it is not necessarily the case that there exists a chart on which $p_\alpha$ is linear. If such a chart exists for each $p \in \Sp(\MM)$, this means that the subsets $\Sp(\MM,\alpha)$ cover $\Sp(\MM)$. We make the following definition. 

\begin{definition}\label{definition: full PL}
Let $\MM$ be a finite polyptych lattice over $F$, with charts $\{M_\alpha\}_{\alpha \in \mathcal{I}}$. 
If $\Sp(\MM) = \bigcup_{\alpha \in \mathcal{I}} \Sp(\MM,\alpha)$, then we say that $\MM$ is a \textbf{full} polyptych lattice. 
\exampleqed
\end{definition} 

\begin{example} 
The 2-vertex, rank-2 polyptych lattice $\MM$ in \Cref{ex_points} is full, since any point in $\Sp(\MM)$ is linear in either $M_1$ or $M_2$, as can easily be checked. 
\exampleqed
\end{example}

The last lemma of this section concerns the space of points of a product polyptych lattice. 

\begin{lemma}\label{lemma: points of product PL}
Let $\MM, \MM'$ be finite polyptych lattices over $F$ and let $\MM \times \MM'$ be the associated product polyptych lattice. Let $\mathrm{pr}: \MM \times \MM' \to \MM$ and $\mathrm{pr}': \MM \times \MM' \to \MM'$ denote the natural projection maps on the corresponding sets of elements. Then $\Sp(\MM \times \MM') \cong \Sp(\MM) \times \Sp(\MM')$ as sets, where the bijection $\Sp(\MM)\times \Sp(\MM') \to \Sp(\MM \times \MM')$ is given by $(p,p') \mapsto p \circ \mathrm{pr} + p' \circ \mathrm{pr}'$. Moreover, this bijection restricts to $\Sp(\MM,\alpha) \times \Sp(\MM',\alpha') \to \Sp(\MM \times \MM', (\alpha,\alpha'))$ for all $\alpha \in \pi(\MM), \alpha' \in \pi(\MM')$. 
\end{lemma} 

\begin{proof} 
We first show that the map $(p,p') \mapsto p \circ \mathrm{pr} + p' \circ \mathrm{pr}'$ is well-defined, i.e., the image is a point of $\MM \times \MM'$. Note that by definition of the product polyptych lattice, for $(m_1,m'_1), (m_2,m'_2) \in \MM \times \MM'$, and a choice of chart $(\alpha,\alpha') \in \mathcal{I} \times \mathcal{I}'$, the addition in the chart $M_\alpha \times M'_{\alpha'}$ is given by $(m_1,m'_1) +_{(\alpha,\alpha')} (m_2,m'_2) = (m_1+_\alpha m_2, m'_1 +_{\alpha'} m'_2)$.
Hence a point $q \in \Sp(\MM \times \MM')$ must satisfy 
\begin{equation}\label{eq: point of MM times MM'}
q(m_1,m'_1) + q(m_2,m'_2) = \min\{q(m_1+_\alpha m_2, m'_1 +_{\alpha'} m'_2) \, \mid \, \alpha\in \mathcal{I}, \alpha' \in \mathcal{I}'\}.
\end{equation}
For $q = p \circ \mathrm{pr} + p' \circ \mathrm{pr}'$, we have that the LHS of \Cref{eq: point of MM times MM'} is 
$p(m_1)+p'(m'_1)+p(m_2)+p'(m'_2)
$ whereas the RHS is $\min\{p(m_1+_\alpha m_2)+p'(m'_1+_{\alpha'} m'_2) \, \mid \, \alpha \in \mathcal{I}, \alpha' \in\mathcal{I}'\}$. From the definition of $\Sp(\MM)$ and $\Sp(\MM')$ it follows that LHS and RHS are equal. Moreover, since both $p$ and $p'$ are $F_{\geq 0}$-homogeneous, so is $q$. It is also straightforward to see that the map is injective. To complete the proof, it now suffices to show that any point of $\MM \times \MM'$ is of the form $p \circ \mathrm{pr} + p' \circ \mathrm{pr}'$. Since both $p,p'$ are piecewise $F$-linear by \Cref{lemma: points linear on faces of SigmaM} we know $p(0_\MM)=0=p'(0_{\MM'})$. Note also that $(m,m')=(m+_\alpha 0_\MM, 0+_{\alpha'} 0_{\MM'}) = (m,0_{\MM'}) +_{(\alpha,\alpha')} (0_\MM,m')$ for all $m\in \MM,m'\in\MM'$, and $(\alpha,\alpha') \in \mathcal{I} \times \mathcal{I}'$. By \Cref{eq: point of MM times MM'} it then follows that 
$$
q(m,0_{\MM'}) + q(0_\MM,m') = \min\{q(m+_\alpha 0_\MM, 0_{\MM'}+_{\alpha'} m') \mid \alpha \in \mathcal{I}, \alpha' \in \mathcal{I}'\} = q(m,m')
$$
for all $m \in \MM, m' \in \MM'$. In other words, $q$ is determined entirely by its values on $\MM \otimes \{0_{\MM'}\}$ and $\{0_\MM\} \times \MM'$, and, $q = \left( q \vert_{\MM \times \{0_{\MM'}\}} \right) \circ \mathrm{pr} + \left( q \vert_{\{0_\MM\} \times \MM'} \right) \circ \mathrm{pr}'$. From \Cref{eq: point of MM times MM'} it immediately follows that both of the restrictions $q \vert_{\MM \times \{0_{\MM'}\}}$ and $q \vert_{\{0_\MM\} \times \MM'}$ satisfy \Cref{eq: def point min} on $\MM$ and $\MM'$ respectively. Moreover, since $q$ is $F_{\geq 0}$-homogeneous, so are the restrictions. This proves the bijection, and the last claim of the lemma is straightforward.  
\end{proof}

%%%%%%%%%%
\subsection{Convexity and half-spaces in polyptych lattices} \label{subsec_PL_halfspaces_basics}

Our next goal is to define analogues of the notion of convexity in the setting of polyptych lattices. (In \Cref{sec_PL_polytopes} we develop these ideas further to a theory of PL polytopes.) In fact, there are two natural definitions of convexity in the PL setting; both are useful. We begin with what we call chart-convexity. 

\begin{definition}\label{def_M_convex} 
Let $\MM$ be a polyptych lattice over $F$ and $\PP$ a subset of $\MM_\R$. 
We say that $\PP$ is \textbf{chart-convex} if each $\pi_\alpha(\PP)$ is a (classical) convex set in $M_\alpha \otimes \R \cong \R^r$ for all $\alpha \in \pi(\MM)$. 
\exampleqed
\end{definition} 

Sets that are chart-convex can be characterized by looking at subsets of $\MM$ that map to a (classical) straight line in some coordinate chart. To make this precise, we give the following definition. 

\begin{definition}\label{def: broken line}
We say that a subset $\ell \subset \MM_\R$ is a \textbf{broken line} if there exists $\alpha \in \pi(\MM)$ such that $\pi_\alpha(\ell) \subset M_\alpha \otimes \R$ is a straight line (possibly not including endpoints) in the classical sense. We say that a \textbf{broken line $\ell$ has endpoints $m, m' \in \MM_\R$} if $\pi_\alpha(\ell)$ has endpoints $\pi_\alpha(m),\pi_\alpha(m')$. 
\exampleqed
\end{definition}

\begin{example} 
Continuing with \Cref{ex: running example}, the vertical line segment in $M_1 \otimes \R$ connecting the two points $(0,1)$ and $(0,-1)$ maps under $\mu_{1,2}$ to the union of two line segments in $M_2 \otimes \R$, the first connecting $(0,0)$ to $(0,1)$, and the second connecting $(0,0)$ to $(-1,-1)$. The subset $\ell \in \MM_\R$ which maps to these two sets under $\pi_1$ and $\pi_2$ respectively is a broken line in the sense of \Cref{def: broken line}, since its chart image in $M_1 \otimes \R$ is a straight line in the classical sense. 
\exampleqed
\end{example} 

In the classical setting, a convex set is, by definition, characterized by the fact that it contains the straight line connecting any two of its points. The following proposition is the analogue in our setting, and gives a characterization of chart-convexity in terms of broken lines.

\begin{proposition}\label{prop-convex}
Let $\MM$ be a polyptych lattice over $F$. 
A set $\PP \subseteq \MM_{\R}$ is chart-convex if and only if, for any $m,m' \in \PP$, and any broken line $\ell$ with endpoints $m, m'$,  the broken line $\ell$ is contained in $\PP$.
\end{proposition}

\begin{proof}
First suppose $\PP$ is chart-convex.  Suppose $\ell$ is a broken line in $\MM_\R$ with endpoints $m, m' \in \PP$. By definition of broken lines, there exists $\alpha \in \pi(\MM)$ with $\pi_\alpha(\ell)$ a straight line in $M_\alpha \otimes \R$, with endpoints $\pi_\alpha(m)$ and $\pi_\alpha(m')$.  By the chart-convexity of $\PP$, we know $\pi_\alpha(\PP)$ is convex in $M_\alpha \otimes \R$, and since $\pi_\alpha(m),\pi_\alpha(m') \in \pi_\alpha(\PP)$, we conclude $\pi_\alpha(\ell) \subset \pi_\alpha(\PP)$. Since $\pi_\alpha$ is a bijection, this implies $\ell \subset \PP$.

Conversely, suppose $\PP$ has the property that any broken line with endpoints $m, m' \in \PP$ must lie in $\PP$.  Let $\alpha \in \pi(\MM)$ and consider $\pi_\alpha(m), \pi_\alpha(m') \in \pi_\alpha(\PP)$.  Let $\ell$ be the preimage in $\MM$ under $\pi_\alpha$ of the straight line $\ell'$ in $M_\alpha \otimes \R$ with endpoints $\pi_\alpha(m), \pi_\alpha(m')$. Then $\ell$ is a broken line, hence by assumption is contained in $\PP$. Thus $\ell'=\pi_\alpha(\ell)$ is contained in $\pi_\alpha(\PP)$, and it follows that $\pi_\alpha(\PP)$ is convex in the classical sense. Since $\alpha \in \pi(\MM)$ was arbitrary, we conclude that $\PP \subseteq \MM_\R$ is chart-convex. 
\end{proof}

The following is immediate from \Cref{prop-convex}. 

\begin{corollary}\label{corollary: intersection of chart-convex}
Any intersection of chart-convex sets is chart-convex.
\end{corollary}

Now we proceed to define our second notion of convexity for polyptych lattices. We begin with a generalized notion of half-spaces.

\begin{definition}\label{definition: M half space} 
Let $\MM$ be a polyptych lattice over $F$ and let 
$p \in \Sp(\MM)$ be a point of $\MM$. Let $a \in F$. We define the \textbf{PL-half-space (over $F$) with threshold $a$ associated to $p$} as follows: 
\begin{equation*}\label{eq: def M half space} 
\HH_{p, a} := \{ m \in \MM_\R \, \mid \, p(m) \geq a\} \subset \MM_\R. 
\end{equation*} 
(When the base ring is clear from context, we sometimes drop the $F$.) 
\exampleqed
\end{definition}

\begin{example}\label{ex_half_space} 
We continue with \Cref{ex_points}.
Consider $p\in \Sp(\MM)$ given by $p(\mathbb{e}_1)=-1$, $p(\mathbb{e}_2)=0$, and $p(\mathbb{e}'_2)=-1$.
By \eqref{eq_ex_point} from \Cref{ex_points}, we have that $\pi_1(\HH_{p,-1})=\{(x,y)\in M_1\mid -x+\min\{0,y\}\ge -1\}$ and  $\pi_2(\HH_{p,-1})=\mu_{12}(\pi_1(\HH_{p,-1}))=\{(u,v)\in M_2\mid u\ge -1\}$. The latter is a classical half-space in $M_2 \otimes \R$, while the former is not. Both of the chart images of $\HH_{p,-1}$ are depicted in \Cref{fig_half_spaces}.
\begin{figure}[h]
    \centering
    \begin{tikzpicture}
    \filldraw[fill=gray!50,draw=none] (-.5*3.05,-.5*3.05)--(-.5*2,-.5*3.05)--(.5*1,.5*0)--(.5*1,.5*3.05)--(-.5*3.05,.5*3.05);
    % loop over the lattice points
    \foreach \i in {-3,...,3}
      \foreach \j in {-3,...,3}{
        \filldraw[black] (.5*\i,.5*\j) circle(.5pt);
      };
    \draw[gray,<->] (-1.5,0)--(1.5,0);
    \draw[gray,<->] (0,-1.5)--(0,1.5);
    \draw (-.5*2,-.5*3)--(.5*1,.5*0)--(.5*1,.5*3);
    \end{tikzpicture}
    \hspace{5cm}
    \begin{tikzpicture}
    \filldraw[fill=gray!50,draw=none] (.5*3.05,-.5*3.05)--(-.5*1,-.5*3.05)--(-.5*1,.5*3.05)--(.5*3.05,.5*3.05);
    % loop over the lattice points
    \foreach \i in {-3,...,3}
      \foreach \j in {-3,...,3}{
        \filldraw[black] (.5*\i,.5*\j) circle(.5pt);
      };
    \draw[gray,<->] (-1.5,0)--(1.5,0);
    \draw[gray,<->] (0,-1.5)--(0,1.5);
    \draw (-.5*1,-.5*3)--(-.5*1,.5*3);
    \end{tikzpicture}
    \caption{The two chart images of the $\MM$-half-space $\HH_{p,-1}$ from \Cref{ex_half_space}. On the left is $\pi_1(\HH_{p,-1})$ and on the right is $\pi_2(\HH_{p,-1})$. }
    \label{fig_half_spaces}
\end{figure}
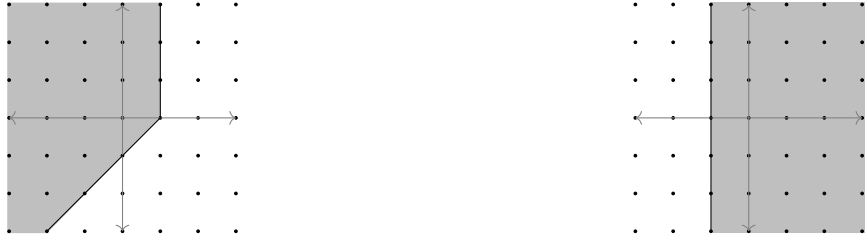
\exampleqed
\end{example}

We now show that PL-half-spaces are chart-convex.

\begin{proposition}\label{prop_halfconvex}
Let $\MM$ be a polyptych lattice over $F$, and let $a \in F$ and $p\in\Sp(\MM)$. Then the PL half-space $\HH_{p, a} \subset \MM_\R$ is chart-convex. 
\end{proposition}

\begin{proof}
By \Cref{prop-convex}, it suffices to show that any broken line $\ell$ with endpoints $m,m' \in \HH_{p,a}$ is contained in $\HH_{p,a}$. For such an $\ell$, let $\alpha \in \pi(\MM)$ be such that $\pi_\alpha(\ell)$ is a (classical) straight line segment in $M_\alpha \otimes \R$ with endpoints $\pi_\alpha(m), \pi_\alpha(m')$. 
For any $q \in \ell$, there exist some $s, t \in \R_{\geq 0}$ such that $s + t = 1$ and $s \, \pi_\alpha(m) + t \, \pi_\alpha(m') = \pi_\alpha(q)$. Since $p$ is a point in $\Sp(\MM)$, its extension to $\MM_\R$ is convex by \Cref{lemma: point PL concave}. Thus we have
\begin{equation*}
\begin{split}
a = (s + t)a = sa + ta & \leq s \, p(m) + t \, p(m') \quad \textup{ since $m,m' \in \HH_{p,a}$} \\
& = s \, p_\alpha(\pi_\alpha(m)) + t \, p_\alpha(\pi_\alpha(m')) \\
& \leq p_\alpha( s \pi_\alpha(m) + t \pi_\alpha(m')) \quad \textup{ by convexity of $p_\alpha$} \\
& = p_\alpha(\pi_\alpha(q))\\
& = p(q)
\end{split}
\end{equation*}
so $q \in \HH_{p, a}$. Thus $\ell$ is entirely contained in $\HH_{p,a}$, and since $\ell$ was arbitrary, the claim follows. 
\end{proof}

The notion of PL-half-spaces provides a natural extension of the half-space description of a classical convex set, which we do now.
Also, we can now give our second notion of convexity in the PL setting. 

\begin{definition}\label{def_pt_convex hull}\label{def: point convex}
Let $\MM$ be a finite polyptych lattice over $F$ and let $S \subseteq \MM_{\R}$. 
 The \textbf{point-convex hull (over $F$)} $\ptconv_{F}(S)$ of $S$ is the intersection of all PL half-spaces $\HH_{p, a}$ over $F$ containing $S$, i.e.,
 \begin{equation*}\label{eq: def point convex hull}
 \ptconv_F(S) := \bigcap_{S \subset \HH_{p,a}} \HH_{p,a}.
 \end{equation*} 
 We say that $S$ is \textbf{point-convex (over $F$)} if $S=\ptconv_F(S)$. 
 \exampleqed
\end{definition}

The following lemma relates the two notions of convexity.

\begin{lemma}\label{lemma: point convex implies chart-convex}
Let $\MM$ be a finite polyptych lattice over $F$ and let $S \subset \MM_\R$. If $S$ is point-convex over $F$, then $S$ is chart-convex. 
\end{lemma} 

\begin{proof} 
If $S$ is point-convex then $S = \ptconv_F(S) = \cap_{S \subset \HH_{p,a}} \HH_{p,a}$ is an intersection of PL half-spaces. The claim then follows from \Cref{prop_halfconvex} and \Cref{corollary: intersection of chart-convex}. 
\end{proof}

We use the following in the sequel. The proof is a definition check. 

\begin{lemma}\label{lemma: S subset T point convex}
Let $\MM$ be a finite polyptych lattice over $F$ and let 
$T \subseteq S \subset \MM_\R$. Then $\ptconv_F(T) \subset \ptconv_F(S)$.  
In particular, if $S$ is point-convex, then $\ptconv_F(T) \subseteq S$. 
\end{lemma} 

We end the section with a useful characterization of point-convex hulls. 

\begin{lemma}\label{lem-pconvex-1}
Let $\MM$ be a finite polyptych lattice of rank $r$ over $F$. Let $S$ be a compact subset of $\MM_\R$. Then $\ptconv_{\R}(S)=\{m\in \MM_{\R} \mid \forall p \in \Sp_{\R}(\MM),\ p(m) \geq \min\{p(s) \mid s \in S\}\}$.
    \end{lemma} 

\begin{proof} 
Note first that the minimum $\min\{p(s) \mid s \in S\}$ is achieved for any $p \in \Sp_{\R}(\MM)$ because $p$ is continuous by assumption and $S$ is closed and bounded in $\MM_\R$, hence compact. 
First suppose that $m \in \MM_\R$. If there exists $p \in \Sp_\R(\MM)$ with $p(m) < \min\{p(s) \mid s \in S\}$ then we can find $a \in \R$ with $m \notin \HH_{p, a}$ but $S \subset \HH_{p, a}$. Thus $m$ is not in the point-convex hull of $S$. 
Conversely, if for all $p \in \Sp_\R(\MM)$ we have $p(m) \geq \min\{p(s) \mid s \in S\}$, then if $S \subset \HH_{p, a}$ we must have $p(m) \geq a$, so $m \in \HH_{p, a}$, and $m$ is in the point-convex hull of $S$. 
\end{proof}

We will further study convexity in the PL setting in \Cref{sec_PL_polytopes}.

\subsection{The canonical semialgebra}\label{subsec: semialgebra}

We now take a moment to introduce terminology which we use frequently below. 
In this manuscript, a \textbf{commutative semiring}\footnote{The literature is inconsistent in the terminology for these objects, as the authors of e.g.\ \cite{Golan, Kazimierz} also testify.}  $\mathcal{S}$ is a set satisfying all the axioms of a commutative ring with a multiplicative identity, \emph{except} the axiom of existence of additive inverses. We denote addition by $\oplus$ and multiplication by $\odot$ respectively, with corresponding identities $\infty$ and $0$. Due to the lack of additive inverses, we must add an extra axiom, namely that $a \odot \infty = \infty \odot a = \infty$ for all $a \in \mathcal{S}$.   A semiring is said to be \textbf{idempotent} if $a \oplus a = a$ for all $a \in \mathcal{S}$. 

\begin{remark}\label{remark: partial order of idempotent semialgebra}
An idempotent commutative semiring can be equipped with a natural partial order, defined by $a\le b$ if and only if $a\oplus b=a$. Note that $\infty$ is the largest element under this partial order (this justifies the notation). This is not a total ordering in general. 
\exampleqed 
\end{remark}

The following examples are central to this manuscript. 

\begin{example}\label{example: OMM as semialgebra}
Let $X$ be any set. Then the $\R$-vector space of $\R$-valued functions $\mathrm{Fun}(X,\R)$ with domain $X$ can be equipped with the structure of an idempotent commutative semiring, with operations given by 
\begin{equation*}\label{eq: semialgebra operations on O_MM}
f \odot g := f + g, \quad \textup{ and } \quad 
f \oplus g := \min\{f,g\}
\end{equation*}
where the operations on the RHS are the natural operations on functions. We also formally add an element denoted $\infty$ which serves as the additive identity under $\oplus$, namely, $f\oplus \infty = \min\{f,\infty\} = f$ for all $f \in \mathrm{Fun}(X,\R)$. We define $f \odot \infty = \infty$. Note that $\mathrm{Fun}(X,\R)$ is indeed idempotent, since $f \oplus f = \min\{f,f\} = f$ for all $f \in \mathrm{Fun}(X,\R)$. Any subset of $\mathrm{Fun}(X,\R)$ which is closed under $+$ and $\min$ naturally inherits an idempotent commutative semiring structure. The case of main interest for us will be $X = \MM_\R$, and we consider the idempotent commutative sub-semiring  $\O_\MM \subset \mathrm{Fun}(\MM,\R)$, defined as the set of all $\R$-valued piecewise-linear functions on $\MM_\R$. 
\exampleqed
\end{example} 

\begin{example}\label{ex: Malpha as semialgebra}
Let $\Gamma$ be a totally ordered abelian group. Let $\overline{\Gamma} := \Gamma \cup \{\infty\}$ denote $\Gamma$ with an element $\infty$ formally added, where we define $a<\infty$ and $a+\infty=\infty$ for all $a \in \Gamma$. It is straightforward to see that $\overline{\Gamma}$ can be equipped with an idempotent commutative semiring structure with respect to the operations $m \odot m' := m+m'$, where the RHS is the abelian group operation on $\overline{\Gamma}$, and $m \oplus m' := \min\{m, m'\}$ where the minimum is taken with respect to the given total order on $\overline{\Gamma}$. As in \Cref{example: OMM as semialgebra} it is clear that $(\overline{\Gamma}, \odot, \oplus)$ is idempotent, since $m \oplus m = \min\{m,m\}=m$ for all $m \in \overline{\Gamma}$. There are three main examples of $\overline{\Gamma}$ which we consider in this paper. The first two are $\overline{\Z} := \Z \cup \{\infty\}$ with the usual order $<$, and $\overline{\Z^r}:=\Z^r \cup \{\infty\}$ with the standard lex order $<_{lex}$ on $\Z^r$.\footnote{We take standard lex order to mean that $(a_1,\cdots,a_r) >_{lex} (b_1,\cdots,b_r)$ if there exists $\ell \in [r]$ such that $a_i=b_i$ for all $i<\ell$ and $a_\ell > b_\ell$ with respect to the standard order on $\Z$.} The third is constructed by an identification with $(\overline{\Z^r}, <_{lex})$ or $(\overline{F^r}, <_{lex})$. Specifically, let $\MM$ be a finite polyptych lattice of rank $r$ over $F$, let $\alpha \in \pi(\MM)$, and $M_\alpha$ its coordinate chart. We know $M_\alpha$ is isomorphic to $F^r$ by definition of a polyptych lattice; fix such an isomorphism $\psi: M_\alpha \to F^r$. (For instance, $\psi$ may be defined using a choice of ordered basis of $M_\alpha$.) We may then equip $M_\alpha$ with the total order induced from $\psi$ and the lex order on $F^r$. Then $\overline{M_\alpha}:=M_\alpha \cup \{\infty\}$ is an idempotent commutative semiring. We will return to this last example in \Cref{sec_detrop}. 
\exampleqed
\end{example}

\begin{example}\label{example: lattice polytope semialg}
Fix a positive integer $r>0$ and consider the set of integral polytopes in $\Z^r$. We also include the empty set $\emptyset$. 
This set can be equipped with the structure of an idempotent commutative semiring, where the operation $\oplus$ is ``convex hull of the union'' and $\odot$ is ``Minkowski sum''. The empty set $\emptyset$ takes on the role of the additive identity $\infty$. 
\exampleqed
\end{example} 

Let $\Z \subset F \subset \R$ be a subring of $\R$ containing $\Z$. It is straightforward to see that, with respect to the usual operations $+$ and $\times$ in $\R$, the subset $F_{\geq 0} := \{\lambda \in F \mid \lambda \geq 0\}$ is a commutative semiring, with additive identity $0$ and multiplicative identity $1$. For what follows, we use some terminology from \cite{PareigisRohrl}. It can be checked that $\bar{F} := (F \cup \{\infty\}, +)$ has the structure of a semimodule over the semiring $F_{\geq 0}$, with semimodule operation $(\lambda, a) \mapsto \lambda a = \lambda \times a$ for $\lambda \in F_{\geq 0}, a \in F$ where $\times$ denotes the usual multiplication in $F \subset \R$, and $(\lambda, \infty) \mapsto \infty$ for any $\lambda \in F_{>0}$, and $(0,\infty) \mapsto 0$. This structure extends the canonical $\Z_{\geq 0}$-semimodule structure that exists on any commutative monoid $(M,\odot)$ (given by $(n,m) \mapsto m\odot \cdots \odot m$ for any $n \in\Z_{\geq 0}$ and $m\in M$). Motivated by this, in this manuscript we say that an idempotent commutative semiring $(S, \oplus, \odot)$ is a \textbf{semialgebra over $F_{\geq 0}$}, or a \textbf{$F_{\geq 0}$-semialgebra}, if it is equipped with an operation $F_{\geq 0} \times S \to S, (\lambda, a) \mapsto \lambda a$, which makes $(S, \odot)$ into an $F_{\geq 0}$-semimodule (extending the canonical $\Z_{\geq 0}$-semimodule structure), and which is compatible with the operation $\oplus$ (i.e., $(\lambda, a\oplus b)=(\lambda,a)\oplus(\lambda,b)$ for all $\lambda \in F_{\geq 0}, a,b\in S$). With this terminology, it can be checked that $(\bar{F}, \mathrm{min}, +)$ has a canonical $F_{\geq 0}$-semialgebra structure.

\begin{example}\label{example: Fgeq0 semialgebra structures on basic examples} 
The idempotent commutative semirings given in Examples~\ref{example: OMM as semialgebra} and~\ref{ex: Malpha as semialgebra} are $F_{\geq 0}$-semialgebras. Indeed, for both $\mathrm{Fun}(X,\R)$ and $\mathcal{O}_\MM$ it is straightforward to check that scalar multiplication by $F_{\geq 0}$, given by pointwise multiplication of functions, defines a $F_{\geq 0}$-semialgebra structure. For $\overline{\Z^r}$ and $\overline{M_\alpha}$ discussed in Example~\ref{ex: Malpha as semialgebra}, the natural scalar multiplication $(\lambda,m)\mapsto \lambda m$ gives them the structure of a $\Z_{\geq 0}$-semialgebra and $F_{\geq 0}$-semialgebra respectively. \exampleqed 
\end{example}

Now we define the notion of the canonical semialgebra over $F_{\geq 0}$, denoted $S_\MM$, associated to a polyptych lattice $\MM$.  We first introduce the underlying additive semigroup.\footnote{By a semigroup, we mean a set $S$ equipped with an associative binary operation and an identity element. This follows the terminology of \cite{Cox_Little_Schenck}.} Specifically, we let $\langle \MM \rangle$ denote the semigroup obtained by first taking the free abelian idempotent semigroup (with binary operation denoted $\oplus$) generated by the elements of $\MM$, and then imposing the relation that
    for finite subsets $\mathcal{S}, \mathcal{S}'\subseteq\MM$, we have 
    $$
    \bigoplus_{m \in \mathcal{S}} m = \bigoplus_{m' \in \mathcal{S}'} m' \quad \textup{ if } \quad \ptconv_F(\mathcal{S}) = \ptconv_F(\mathcal{S}').
    $$
We formally add an element to $\langle \MM \rangle$, denoted $\infty$, which serves as the additive identity. It follows that $\langle \MM \rangle$ can be considered a semimodule in the sense of \cite[Definition 1.3]{PareigisRohrl}, over the Boolean semifield $\mathcal{B} := \{0,1\}$.\footnote{The operations on the two-element set $\mathcal{B}$ are defined so that $0$ is the additive identity and $1\oplus 1=1$ and the product operation is the usual one, with $0\cdot0=0=0\cdot 1, 1 \cdot 1=1$.} 

Next, we define a product operation, which we will denote by $\star$, on $\langle \MM \rangle$. 
For $m,m' \in \MM$, let $\Upsilon(m, m') := \{m +_\alpha m' \mid \alpha \in \pi(\MM)\}$.

\begin{definition}\label{definition: star product}
Let $\MM$ be a polyptych lattice and let $\langle \MM \rangle$ be as above. For $m_1,m_2 \in \MM$, we define 
\begin{equation*}\label{eq: def star product}
m_1\star m_2 := \bigoplus_{m \in \Upsilon(m_1, m_2)} m, 
\end{equation*} 
and for $m \in \MM$ we define $m \star \infty = \infty$; we then extend the definition of $\star$ to all combinations $\bigoplus_{m \in \mathcal{S}} m$ in $\langle \MM \rangle$ by distributivity over $\oplus$. 
\exampleqed
\end{definition}

To see that the $\star$ product gives $\langle \MM \rangle$ the structure of an idempotent commutative semiring, we need the following. 

\begin{theorem}
The binary operation $\star$ of \Cref{definition: star product} is associative and commutative. 
\end{theorem}

\begin{proof}
Let $m_1,m_2,m_3 \in \MM$. We want to prove $(m_1 \star m_2)\star m_3 = m_1 \star (m_2 \star m_3)$. Unravelling definitions, the LHS is the sum over all elements in the set 
$$
\mathcal{S} = \{m \, \mid \, \exists  \, \alpha,\beta \in \mathcal{I}, \pi_\alpha(m) = \pi_\alpha(\pi_\beta^{-1}(\pi_\beta(m_1)+\pi_\beta(m_2))) + \pi_\alpha(m_3) \, \textup{ in } \, M_\alpha\} 
$$
whereas the RHS is the sum over 
$$
\mathcal{S}' = \{ m' \, \mid \, \exists \, \gamma, \delta \in \mathcal{I}, \pi_\gamma(m') = \pi_\gamma(m_1) + \pi_\gamma(\pi_\delta^{-1}(\pi_\delta(m_2) + \pi_\delta(m_3))) \, \textup{ in } \, M_\gamma \}.
$$
To prove the desired equality it now suffices to show that the point-convex hulls of $\mathcal{S}$ and $\mathcal{S}'$ coincide. For this, it in turn suffices to show that for any $p \in \Sp(\MM)$ and any $a \in F$, we have $\mathcal{S} \subset \HH_{p,a}$ if and only if $\mathcal{S}' \subset \HH_{p,a}$. To see this, first note that it follows from \Cref{eq: def point min} in the definition of points and the construction of the sets $\mathcal{S}$ and $\mathcal{S}'$ that 
\begin{equation}\label{eq: assoc of star}
    p(m_1) + p(m_2) + p(m_3) = \min\{p(m) \, \mid \, m \in \mathcal{S}\} = \min\{p(m') \, \mid \, m' \in \mathcal{S}'\}. 
\end{equation}
On the other hand, by definition of $\HH_{p,a}$, we have that $\mathcal{S} \subset \HH_{p,a}$ exactly if $p(m) \geq a$ for all $m \in \mathcal{S}$, i.e., $\min\{p(m) \, \mid \, m \in \mathcal{S}\} \geq a$. A glance at \Cref{eq: assoc of star} immediately yields that $\mathcal{S} \subseteq \HH_{p,a}$ if and only $\mathcal{S}' \subseteq \HH_{p,a}$. Thus $\ptconv_F(\mathcal{S}) = \ptconv_F(\mathcal{S}')$ as desired. For the case when some of the $m_i$ are equal to $\infty$, the associativity follows easily from the definition of $\star$. Thus $\star$ is associative on $\langle \MM \rangle$. Finally, the operation $+_\alpha$ is commutative for all $\alpha \in \mathcal{I}$, so it is straightforward that $\star$ is also commutative. 
\end{proof}

\begin{remark}\label{remark: zero is star product identity}
Recall from \Cref{remark: zero in MM} that there is an element $0_\MM$ in $\MM$ such that $m +_\alpha 0_\MM=m$ for all $m \in \MM$ and $\alpha \in \mathcal{I}=\pi(\MM)$. Thus $\Upsilon(m,0_\MM) = \{m\}$ and so $m \star 0_\MM = m$ for any $m \in \MM$, i.e., $0_\MM$ is the multiplicative identity with respect to the $\star$ product. This is analogous to $0$ being the ``multiplicative'' identity with respect to the semiring $\Z \cup \{\infty\}$ with operations $\odot = +$ and $\oplus =\min$. 
\exampleqed
\end{remark} 

Recall from \Cref{lemma: addition well def on cones} (and \Cref{notation: scalar mult on cones}) that there is a well-defined scalar multiplication operation $F_{\geq 0} \times \MM \to \MM$ which we denote by $(\lambda, m) \mapsto \lambda m$. The following is straightforward. 

\begin{lemma}\label{lemma: SMM is a semialgebra}
The commutative monoid $(\langle \MM \rangle, \star)$ can be equipped with an $F_{\geq 0}$-semimodule structure $F_{\geq 0} \times \langle \MM \rangle \to \langle \MM \rangle$ defined by $(\lambda, \oplus_{m \in \mathcal{S}} m) \to \oplus_{m \in \mathcal{S}} \lambda m$, and $(\lambda, \infty) \mapsto \infty$ for $\lambda >0$, and $(0,\infty)\mapsto 0$. Moreover, this semimodule structure is compatible with the operation $\oplus$ on $\langle \MM \rangle$. In particular, $(S_\MM, \oplus, \star)$ is a semalgebra over $F_{\geq 0}$. 
\end{lemma}

\begin{definition}\label{definition: canonical semialgebra MM}
The \textbf{canonical $F_{\geq 0}$-semialgebra} $S_\MM$ is the idempotent commutative semiring with underlying set $\langle \MM\rangle$, equipped with the additive operation $\oplus$ with additive identity $\infty$, and the associative commutative product operation $\star$ of \Cref{definition: star product}. Moreover, we equip $S_\MM$ with the $F_{\geq 0}$-semialgebra structure given in \Cref{lemma: SMM is a semialgebra}. 
(We will frequently refer to $S_\MM$ as the ``canonical semialgebra'' instead of the ``canonical idempotent semialgebra over $F_{\geq 0}$''.) 
\exampleqed
\end{definition}

\begin{example}\label{example: canonical semialgebra of trivial PL}
The trivial polyptych lattice of rank $r$ over $\Z$ consists of a single chart $M_\alpha\simeq \Z^r$.
For this polyptych lattice, the reader may check that the canonical semialgebra $S_\MM$ coincides with the idempotent commutative semiring of integral polytopes in $\Z^r$ from \Cref{example: lattice polytope semialg}, where the $\Z_{\geq 0}$-semimodule structure is given by the usual dilation operation. 
\exampleqed
\end{example} 

The following key result will allow us to build concrete valuations in Section~\ref{sec_detrop}.  It also motivates the original definition of points; see \Cref{remark: motivation for points}. We use the natural map $\MM \to S_\MM$ sending $m$ to its equivalence class in $S_\MM$. 

\begin{proposition}\label{lemma-pointsaresemialgebramaps}
Let $\MM$ be a finite polyptych lattice over $F$ and let $S_\MM$ denote the canonical semialgebra over $F_{\geq 0}$ of $\MM$. Then there is a one-to-one correspondence between the set of $F_{\geq 0}$-semialgebra morphisms $S_\MM \to (\bar{F}, \min, +)$ taking finite values on $S_\MM \setminus \{\infty\}$, and the set of points $\Sp(\MM)$. The correspondence is given by restricting an $F_{\geq 0}$-semialgebra homomorphism $S_\MM \to \bar{F}$ to $\MM$. 
\end{proposition}

\begin{proof} 
To prove the claim, we will first show that if $\tilde{p}: (S_\MM, \oplus, \star) \to (\bar{F}, \min, +)$ is a morphism of $F_{\geq 0}$-semialgebras, then the restriction $p := \tilde{p} \vert_{\MM}$ to $\MM$ is a point. First observe that since $\tilde{p}$ takes finite values on $S_\MM \setminus \{\infty\}$, the restriction $p$ is well-defined as a map $p: \MM \to F$. We now check the conditions to be a point. Since $\tilde{p}$ is a morphism of semirings, for any $m, m' \in \MM$ we must have $\tilde{p}(m \star m') = \tilde{p}(m)+\tilde{p}(m')$. Here the LHS is $\tilde{p}(m \star m')=\tilde{p}(\oplus_{\alpha \in \pi(\MM)} m+_\alpha m') = \min\{\tilde{p}(m+_\alpha m') \, \mid \, \alpha \in \pi(\MM)\}$ where the first equality is the definition of $\star$ and the second is because $\tilde{p}$ is a map of semirings. Since $m+_\alpha m'$ is an element of $\MM$ for all $\alpha$ it follows that $p$ satisfies~\eqref{eq: def point min}. We also know that $\tilde{p}$ is a morphism of $F_{\geq 0}$-semimodules, so $\tilde{p}$ is compatible with the scalar multiplication by $F_{\geq 0}$ on $S_\MM$ and $\bar{F}$. This implies that for all $m \in \MM$ and $\lambda \in F_{\geq 0}$ we have $p(\lambda m) = \tilde{p}(\lambda m) = \lambda \tilde{p}(m) = \lambda p(m)$, which is condition~\eqref{eq: def point F homog}. Hence $p \in \Sp(\MM)$, as required. 

Now we prove that any point $p \in \Sp(\MM)$ can be uniquely extended to a morphism $\tilde{p}: S_\MM \to \bar{F}$ of $F_{\geq 0}$-semialgebras. 
To see this, recall first that 
the underlying set of the canonical semialgebra $S_\MM$ consists of formal sums $\oplus_{m \in \mathcal{S}} m$, in addition to the formal element $\infty$. We extend $p$ to all of $S_\MM$ by defining $\tilde{p}(\oplus_{m \in \mathcal{S}} m) := \min_{m \in \mathcal{S}}\{p(m)\}$ and $\tilde{p}(\infty)=\infty$. We claim this is well-defined, i.e., respects the relations on $S_\MM$. First, we have $\tilde{p}(m \oplus m)=\tilde{p}(m)$ since $\min\{\tilde{p}(m),\tilde{p}(m)\}=\min\{p(m),p(m)\}=p(m)=\tilde{p}(m)$ for $m \in \MM$. 
Next suppose $\mathcal{S}, \mathcal{S}'$ are two finite subsets of $\MM$ with $\ptconv_{F}(\mathcal{S})=\ptconv_{F}(\mathcal{S}')$. We need to show $\tilde{p}(\oplus_{m \in \mathcal{S}} m) = \tilde{p}(\oplus_{m' \in \mathcal{S}'} m')$, or equivalently, $\min\{p(m) \mid m \in \mathcal{S}\} = \min\{p(m') \mid m' \in \mathcal{S}'\}$. Suppose the equality does not hold; without loss of generality we may suppose the LHS is greater than the RHS. Then there exists $a \in F$ such that $\mathcal{S} \subset \HH_{p,a}$ but $\mathcal{S}' \not \subset \HH_{p,a}$, but this implies $\ptconv_{F}(\mathcal{S}) \neq \ptconv_{F}(\mathcal{S}')$, which is a contradiction. Hence $p(\oplus_{m \in \mathcal{S}} m) = p(\oplus_{m' \in \mathcal{S}'} m')$. We also have $\tilde{p}(m \oplus \infty) = \min\{\tilde{p}(m),\infty\} = \tilde{p}(m)$ and $\tilde{p}(\infty \oplus \infty)=\min\{\infty,\infty\} = \infty=\tilde{p}(\infty)$. This shows that our definition of $\tilde{p}$ respects the relations on $\langle \MM \rangle$, and it takes the $\oplus$ operation to $\min$ by construction. Now we claim it takes the $\star$ operation to $+$. Indeed we have for any $m,m' \in \MM$ that $\tilde{p}(m \star m') = \tilde{p}(\oplus_{m'' \in \Upsilon(m,m')} m'') = \min\{p(m+_\alpha m') \mid \alpha \in \pi(\MM)\} = p(m)+p(m')=\tilde{p}(m)+\tilde{p}(m')$ where in the last equality we have used the property \eqref{eq: def point min} of points. We also have $\tilde{p}(m \star \infty) = \tilde{p}(\infty) = \infty = \tilde{p}(m)+\infty$. This shows that $\tilde{p}$ is a semiring map. Finally, the fact that $\tilde{p}$ is an $F_{\geq 0}$-semialgebra morphism, i.e. that it respects scalar multiplication by $F_{\geq 0}$, is precisely the condition~\eqref{eq: def point F homog}. 
This completes the proof. 
\end{proof}

We end the section with a result about canonical semialgebras associated to a product of polyptych lattices. 
Before stating the lemma, we must give a definition. Let $\MM, \MM'$ be finite polyptych lattices with associated canonical semialgebras $S_\MM, S_{\MM'}$. 
We first take the tensor product $S_{\MM} \otimes_{\mathcal{B}} S_{\MM'}$ of $\mathcal{B}$-semimodules as defined in \cite[\textsection 6]{PareigisRohrl}; this is a $\mathcal{B}$-semimodule. To equip it with a product operation, by abuse of notation also denoted $\star$, we define 
\begin{equation*}
    (\psi_1 \otimes \phi_1) \star (\psi_2 \otimes \phi_2) := (\psi_1 \star \psi_2) \otimes (\phi_1 \star \phi_2) \quad \forall \psi_1,\psi_2 \in S_\MM, \phi_1, \phi_2 \in S_{\MM'}
\end{equation*}
where on the RHS, the $\star$ denotes the corresponding operations on $S_\MM$ and $ S_{\MM'}$; we then extend $\star$ by distributivity. It can be checked that this makes $S_\MM \otimes_{\mathcal{B}} S_{\MM'}$ into an idempotent commutative semiring. An $F_{\geq 0}$-semialgebra structure can be given by defining $(\lambda, m \otimes m') \mapsto (\lambda m) \otimes (\lambda m')$ and extending. 
With the above definition in hand, we can state and prove the following. 

\begin{lemma}\label{lemma: SMM SMMprime} 
Let $\MM,\MM'$ be two finite polyptych lattices. Then there is a natural map of $F_{\geq 0}$-semialgebras $\kappa: S_{\MM} \otimes_{\mathcal{B}} S_{\MM'} \to S_{\MM \times \MM'}$ taking $\infty$ to $\infty$ and $m \otimes m' \mapsto (m,m')$. 
\end{lemma} 

\begin{proof} 
As a first step, we consider a morphism of the underlying semigroups, $\langle \MM \rangle \otimes_{\mathcal{B}} \langle \MM' \rangle \to \langle \MM\times \MM' \rangle$. To specify such a morphism, by \cite[\textsection 6]{PareigisRohrl} it suffices to give a $\mathcal{B}$-balanced map $\langle \MM \rangle \times \langle \MM' \rangle \to \langle \MM \times \MM' \rangle$. It is straightforward to check that the map sending $(\oplus_{m \in \mathcal{S}} m, \oplus_{m' \in \mathcal{S}'} m') \mapsto \bigoplus_{(m,m') \in \mathcal{S} \times \mathcal{S}'} (m,m')$ (and, any element with $\infty$ in either or both factor(s) is sent to $\infty$), is $\mathcal{B}$-balanced in the sense of \cite[Definition 6.1]{PareigisRohrl}. Thus the above map corresponds to a morphism $\langle \MM \rangle \otimes_{\mathcal{B}} \langle \MM' \rangle \to \langle \MM\times \MM' \rangle$ of $\mathcal{B}$-semimodules. Moreover, it is straightforward to check that the above map sends $\infty$ to $\infty$ and $m \otimes m'$ to $(m,m')$. We denote this morphism by $\kappa$. 

It remains to check that $\kappa$ also respects the product operation on both sides. By distributivity, it suffices to check that $\kappa((m_1 \otimes m'_1) \star (m_2 \otimes m'_2)) = \kappa(m_1 \otimes m'_1) \star \kappa(m_2 \otimes m'_2)$. (Here, by abuse of notation, we are using $\star$ to denote the product operation on both $S_\MM \otimes_{\mathcal{B}} S_{\MM'}$ and on $S_{\MM \times \MM'}$.) The LHS is 
\begin{equation*}
    \begin{split}
  \kappa((m_1 \otimes m'_1) \star (m_2 \otimes m'_2)) &= \kappa((m_1 \star m_2) \otimes (m'_1 \star m'_2)) \, \textup{ by definition of $\star$ on $S_\MM \otimes_{\mathcal{B}} S_{\MM'}$} \\
  &= \kappa((\oplus_\alpha (m_1 +_\alpha m_2)) \otimes (\oplus_{\alpha'} (m'_1 +_{\alpha'} m'_2))) \\
  & = \kappa\left( \oplus_{\alpha,\alpha'} (m_1 +_\alpha m_2) \otimes (m'_1 +_{\alpha'} m'_2 \right)) \\
  & = \bigoplus_{\alpha,\alpha'} (m_1 +_\alpha m_2, m'_1 +_{\alpha'} m'_2) 
    \end{split}
\end{equation*}
On the RHS we have 
\begin{equation*}
    \begin{split} 
\kappa(m_1 \otimes m'_1) \star \kappa(m_2 \otimes m'_2) &= (m_1,m'_1) \star (m_2, m'_2) \, \, \textup{ by definition of $\kappa$} \\
& = \bigoplus_{\alpha,\alpha'} ((m_1,m'_1) +_{(\alpha,\alpha')} (m_2,m'_2)) \, \, \textup{ by definition of $\star$ on $S_{\MM \times \MM'}$}  \\
& = \bigoplus_{\alpha,\alpha'} (m_1+_\alpha m_2, m'_1 +_{\alpha'} m'_2) \, \, \textup{ by definition of $\MM \times \MM'$} \\
    \end{split} 
\end{equation*}
so the LHS and RHS agree. Finally, $m \otimes m' \mapsto (m,m')$ respects the $F_{\geq 0}$-scalar multiplication by construction. The claim follows.  
\end{proof} 

%%%%%%%%%%%%%
\section{Dual pairings of polyptych lattices}\label{sec: duals}

Our next step is to define a \textbf{(strict) dual pairing} between polyptych lattices, in analogy with the classical situation of a lattice and its dual lattice.

\begin{definition} 
\label{def_dual}
Let $\MM, \NN$ be finite polyptych lattices of rank $r$ over $F$ with associated mutation fans $\Sigma(\MM)$ and $\Sigma(\NN)$ respectively. Let $F \subseteq F' \subseteq \R$. We say that a pair of continuous maps $\v: \MM_\R \to \Sp_\R(\NN)$ and $\w: \NN_\R \to \Sp_\R(\MM)$ a \textbf{strict dual $F'$-pairing} if: 
\begin{enumerate} 
\item[(1)] $\v$ and $\w$ restrict respectively to maps $\v: \MM_{F'} \to \Sp_{F'}(\NN)$ and $\w: \NN_{F'} \to \Sp_{F'}(\MM)$ (by abuse of notation we denote the restrictions also by $\v$ and $\w$),
\item[(2)] $\v(m)(n) = \w(n)(m)$ for all $n \in \NN_{F'}, m \in \MM_{F'}$, 
\item[(3)] $\v: \MM_{F'} \to \Sp_{F'}(\NN)$ and $\w: \NN_{F'} \to \Sp_{F'}(\MM)$ are both bijections, and
\item[(4)] the preimages $\v^{-1}\Sp_{\R}(\NN,\gamma)$ (respectively $\w^{-1}\Sp_{\R}(\MM,\alpha)$) are precisely the maximal-dimensional cones of $\Sigma(\MM)$ (respectively $\Sigma(\NN)$), as $\gamma$ ranges over $\pi(\NN)$ (respectively, $\alpha$ ranges over $\pi(\MM)$), giving a bijection between the index set $\pi(\NN)=\mathcal{J}$ of charts of $\NN$ and the set of maximal-dimensional faces of $\Sigma(\MM)$ (respectively, between $\pi(\MM)$ and the maximal-dimensional faces of $\Sigma(\NN)$). 
\end{enumerate} 
In the above setting, 
we say that $(\MM, \NN, \v,\w)$ is a \textbf{strict dual ($F'$)-pair} of polyptych lattices; 
more informally, we also refer to $\NN_{F'}$ as a \textbf{strict ($F'$-)dual to $\MM$}.
If $(\MM,\NN,\v,\w)$ satisfies only the axioms (1)-(3), we call it a \textbf{dual $(F')$-pair}, and $\NN_{F'}$ a \textbf{($F'$-) dual to} $\MM$. If $\MM$ has a strict dual (respectively dual) pairing with itself, we say that $\MM$ is \textbf{strictly self-dual} (respectively \textbf{self-dual}). 
\exampleqed
\end{definition}

The following records some immediate consequences of the definition.  

\begin{lemma}\label{lemma: bijection charts and faces for duals and dilations and PL on cones} 
Let $(\MM, \NN, \v, \w)$ be two finite polyptych lattices over $F$ equipped with a dual $F'$-pairing for some $F \subseteq F' \subseteq \R$. Then for any cone $C$ in $\Sigma(\MM)$ (resp. $\Sigma(\NN)$), the map $\v$ (resp. $\w$) is linear on $C \cap \MM_{F'}$ (resp. $C \cap \NN_{F'}$). 
If, in addition, the dual pairing is strict, then both $\MM$ and $\NN$ are full in the sense of \Cref{definition: full PL}. 
\end{lemma} 

\begin{proof}
For the first claim, let $C$ be a maximal-dimensional cone in $\Sigma(\MM)$. Let $m,m' \in C \cap \MM_{F'}$ and $\lambda,\mu \in F'_{\geq 0}$. We wish to show $\v(\lambda m + \mu m')=\lambda \v(m) + \mu \v(m')$. It would suffice to show $\v(\lambda m + \mu m')(n)=\lambda \v(m)(n) + \mu \v(m')(n)$ for all $n \in \NN_{F'}$. By the property (2) of dual pairings, it in turn suffices to check $\w(n)(\lambda m + \mu m') = \lambda \w(n)(m) + \mu \w(n)(m')$, but this holds since $\w(n)$ is an element of $\Sp_{F'}(\MM)$, so by \Cref{lemma: points linear on faces of SigmaM} it is $F$-linear on $C$. The second claim follows because the union of the maximal faces of $\Sigma(\MM)$ is $\MM$, hence by property (4) of \Cref{def_dual}, the union $\bigcup_\beta \v^{-1}(\Sp(\NN,\beta))$ is all of $\MM$. Since $\v$ is a bijection by property (3) of \Cref{def_dual}, this implies $\bigcup_\beta \Sp(\NN, \beta)$ is all of $\Sp(\NN)$, i.e., $\NN$ is full. The argument is similar for $\MM$. 
\end{proof}

\begin{lemma}\label{eq: strict duals Fgeq0 linear}
Suppose $(\MM, \NN, \v, \w)$ is a strict dual $F'$-pair. Then $\v$ and $\w$ are $F'_{\geq 0}$-linear, i.e., $$\v(\lambda m)=\lambda \v(m), \w(\lambda n)=\lambda \w(n)$$ for all $\lambda \in F'_{\geq 0}, m \in \MM_{F'}, n \in \NN_{F'}$. 
\end{lemma} 

\begin{proof} 
The $F'_{\geq 0}$-scalar multiplication on the space of points is given by pointwise multiplication of functions (here we think of a point as a function $\MM_{F'} \to F'$ or $\NN_{F'} \to F'$). The $F'_{\geq 0}$-scalar multiplication operation on $\MM_{F'}, \NN_{F'}$ was observed in \Cref{notation: scalar mult on cones}. The claim of the lemma now follows from property (2) of strict duals. Indeed, for any $\lambda \in F'_{\ge0}, m \in \MM_{F'}, n \in \NN_{F'}$ we have $\v(\lambda m)(n)=\w(n)(\lambda m)=\lambda \w(n)(m) = \lambda \v(m)(n)$ which implies $\v(\lambda m)=\lambda \v(m)$, and the argument is similar to see $\w(\lambda n)=\lambda \w(n)$. 
\end{proof}

Some examples are in order. 

\begin{example}\label{ex: trivial dual pair}
Let $\MM$ be the trivial polyptych lattice of rank $r$ over $\Z$ first discussed in \Cref{example: canonical semialgebra of trivial PL}. We can think of $\MM$ as a single rank-$r$ $\Z$-lattice $M$, and its strict dual $\NN$ may be taken to be the classical dual lattice $N=\Hom(M,\Z)$.
The pairings $\v(m)(n)=\w(n)(m)$ are the usual pairings between a lattice $M$ and its dual $\Hom(M,\Z)$.
\exampleqed
\end{example} 

\begin{example}\label{ex_dual}
Returning to our running example, it can be seen that the $\MM$ from \Cref{ex_points} is self-($\Z$-)dual. (The maps we give below are defined over $\Z$ but easily extend to $\R$; for brevity we stick to $\Z$ coefficients.) 
In \Cref{ex_points}, we saw that $\Sp(\MM)$ is in bijection with the set $T :=\{(a,b,b')\in\Z^3\mid b+b'=\min\{0,a\}\}$.
In the notation of \Cref{ex_points}, we have $a=p(\mathbb{e}_1)$, $b=p(\mathbb{e}_2)$, $b'=p(\mathbb{e}'_2)$.
Let $((x,y),\mu_{12}(x,y)) = ((x,y),(x',y))$ denote an element of $\MM$. Then since $x'=\min\{0,y\}-x$ by definition of $\mu_{12}$,  it follows that the map
\begin{equation*}
\MM\to T,\qquad \w((x,y),(x',y))=(y,x,x')
\end{equation*}
is well-defined and easily seen to be a bijection.
By using the bijection between $T$ and $\Sp(\MM)$ in \Cref{ex_points} we obtain a bijection (by abuse of notation also denoted $\w$)  $\w:\MM\to \Sp(\MM)$ given by
\begin{equation}\label{eqeq_dual_pairing}
\w:\MM\to \Sp(\MM),\qquad \w((x,y),(x',y))((u,v),(u',v)) :=\begin{cases}
    uy+vx, & v\ge 0\\
    uy-vx', & v\le 0.
\end{cases}
\end{equation}
If we wish to express $\w$ purely in terms of $M_1$ coordinates we may equivalently write 
\begin{equation}\label{eqeq_dual_pairing 1}
\w:\MM\to \Sp(\MM),\qquad \w((x,y),(x',y))((u,v),(u',v))=\begin{cases}
    uy+vx, & v\ge 0\\
    uy-v(\min\{0,y\}-x), & v\le 0.
\end{cases}
\end{equation}
Similarly we may also write (in terms of $M_2$ coordinates) 
\begin{equation}\label{eq_dual_pairing 2}
\begin{split}
\w:\MM\to \Sp(\MM),\qquad \w((x,y),(x',y))((u,v),(u',v)) &=\begin{cases}
    (\min\{0,v\}-u')y+v(\min\{0,y\} - x'), & v\ge 0\\
    (\min\{0,v\}-u')y-vx', & v\le 0 
\end{cases} \\ 
& = \begin{cases}
    -y u' + (\min\{0,y\}-x')v, \quad v \geq 0 \\
    -yu' + (y-x')v, \quad v \leq 0.
\end{cases}
\end{split}
\end{equation}
We claim that $\MM$ is self-dual, with respect to the map $\w$ defined above and setting $\v = \w$. To prove this, we must check the axioms of \Cref{def_dual}. Property (1) is clear, by definition. Property (2) is a straightforward check by taking cases, and we already saw (3) in the discussion above. We now prove (4). From \Cref{eqeq_dual_pairing 1} it follows that $\w((x,y),(x',y))$ is linear on $M_1$ precisely when $y \geq 0$, i.e. $\w^{-1}(\Sp(\MM,1)) = \{y \geq 0\}$ and similarly, from \Cref{eq_dual_pairing 2} $\w^{-1}(\Sp(\MM,2)) = \{y \leq 0\}$. These are precisely the $2$ domains of linearity of $M_1$ and correspond to the maximal-dimensional cones of $\Sigma(\MM)$, so we are done. 
\exampleqed
\end{example}

Duals and strict duals behave well with respect to products of polyptych lattices. 

\begin{lemma}\label{lemma: duals of product PL}
Let $F \subseteq F' \subseteq \R$. Let $(\MM,\NN, \v, \w)$ and $(\MM',\NN', \v', \w')$ be two strict dual $F'$-pairs of finite polyptych lattices over $F$. Then their direct product $(\MM \times \MM', \NN \times \NN', \v \times \v', \w \times \w')$, where we have used the isomorphism of \Cref{lemma: points of product PL} to identify $\Sp(\MM) \times \Sp(\MM')$ with $\Sp(\MM \times \MM')$ (and similarly for the dual), is a strict dual $F'$-pair of polyptych lattices. Moreover, the same statement holds when ``strict dual'' is replaced with ``dual''. 
\end{lemma}

\begin{proof} 
The proof is a straightforward definition check which we leave to the reader. We note only that to check axiom (4) for strict duality, it is useful to use the last claim in \Cref{lemma: points of product PL}. 
\end{proof} 

In the presence of a dual, it turns out we can also give a convenient characterization of point-convex hulls. We will use this in the arguments below.

\begin{proposition}\label{lem-pconvex-2}
Let $\MM$ be a finite polyptych lattice and $S$ be a compact subset of $\MM_\R$. If there exists a dual $\R$-pairing $(\MM,\NN, \v, \w)$, then $\ptconv_{\R}(S)=\{m\in \MM_\R \mid \v(m)\ge \min\{\v(s) \mid s \in S\}\, \textup{as functions on} \, \NN_\R \}$.
  \end{proposition}

\begin{proof}
The duality between $\MM_\R$ and $\NN_\R$ gives us that for $m\in\MM_\R$,
\begin{align*}
    \v(m)\ge \min\{\v(s) \mid s \in S\} 
    & \qquad\Longleftrightarrow\qquad 
    \forall n\in\NN_\R,\ \v(m)(n)\ge \min\{\v(s)(n) \mid s \in S\}
    \\ &
    \qquad\Longleftrightarrow\qquad 
    \forall n\in\NN_\R,\ \w(n)(m)\ge \min\{\w(n)(s) \mid s \in S\}
        \\ &
    \qquad\Longleftrightarrow\qquad 
    \forall p\in\Sp_\R(\MM),\ p(m)\geq \min\{p(s) \mid s \in S\}
.\end{align*}
The result now follows from \Cref{lem-pconvex-1}. 
\end{proof}

Recall from \Cref{example: Fgeq0 semialgebra structures on basic examples} that the piecewise linear $\R$-valued functions $\O_\MM$ on $\MM_\R$ can be equipped with an $F_{\geq 0}$-semialgebra structure. 
We now introduce a subsemialgebra $P_\MM$ of $\O_\MM$. Note that by \Cref{lemma: points linear on faces of SigmaM} we know that any point in $\Sp_{F'}(\MM)$ is piecewise $F'$-linear, which naturally extends to $\Sp_\R(\MM) \subset \O_\MM$. 

\begin{definition}\label{definition: point algebra}
Let $\MM$ be a finite polyptych lattice over $F$. Let $F \subseteq F' \subseteq \R$. As discussed above, we view $\Sp_{F'}(\MM)$ as a subset of $\O_\MM$. We define the \textbf{point semialgebra $P_{\MM_{F'}}$ of $\MM$} to be the subsemialgebra over $F_{\geq 0}$ of $(\O_\MM, +, \min)$ generated by $\Sp_{F'}(\MM)$. (We also add the element $\infty$ to $P_{\MM_{F'}}$.) Note that $P_{\MM_{F'}}$ is closed under $F'_{\geq 0}$-scalar multiplication by \Cref{remark: Fgeq0 closed SpMM}. 
\exampleqed
\end{definition}

\begin{example} 
We continue with the trivial polyptych lattice in \Cref{ex: trivial dual pair}. 
In this case, the point algebra $P_\MM$ consists of piecewise linear functions on $M$ obtained by taking minima and sums of the linear functions $N=\Hom(M,\Z) = \Sp(\MM)$. In fact, any convex piecewise-linear functions on $M$ can be obtained by taking finite min-combinations of linear functions, so in this example, $\P_\MM$ is precisely the set of piecewise-linear convex functions $M \to \Z$. 
\exampleqed
\end{example}

In the presence of a dual pairing, the point subsemialgebra introduced in \Cref{definition: point algebra} turns out to be isomorphic to the canonical semialgebra of its dual. More precisely, we have the following. 

\begin{proposition}\label{prop-salgebras}
Let $\MM,\NN$ be finite polyptych lattices over $F$,  $F \subseteq F' \subseteq \R$, and $(\MM, \NN, \v, \w)$ be a strict dual $F'$-pair. 
Let $S_{\MM_{F'}}, S_{\NN_{F'}}$ denote the canonical semialgebras of $\MM_{F'}$ and $\NN_{F'}$ respectively. Then $\v$ and $\w$ extend to $F_{\geq 0}$-semialgebra isomorphisms (which by abuse of notation we also denote by $\w$ and $\v$): 
\begin{equation*}\label{eq: SN isom to PM and vice versa}
\w: \S_{\NN_{F'}} \to \P_{\MM_{F'}}  \ \ \ \  \v: \S_{\MM_{F'}} \to \P_{\NN_{F'}}.
\end{equation*}

\end{proposition}

\begin{proof}
The assertion is clearly symmetric so it suffices to prove the claim for $\w: \NN_{F'} \to \Sp_{F'}(\MM)$. We extend the definition of $\w$ to $\langle \NN_{F'} \rangle$ by $\w(\bigoplus_{n \in \mathcal{S}} n) := \min\{\w(n) \mid n \in \mathcal{S}\}$ for a finite subset $\mathcal{S}$ of $\NN_{F'}$, and, $\w(\infty)=\infty$. We must first show this is well-defined. The claim $\w(n \oplus n) = \w(n)$ follows immediately since $\min\{\w(n),\w(n)\} = \w(n)$ for any function. Next we must show that if $\mathcal{S}, \mathcal{S'}$ are two finite subsets of $\NN_{F'}$ with the same point-convex hull, then $\w(\bigoplus_{n \in \mathcal{S}} n) = \w(\bigoplus_{n' \in \mathcal{S}'} n')$ or equivalently, $\min\{ \w(n) \mid n \in \mathcal{S}\} = \min\{ \w(n') \mid n' \in \mathcal{S}'\}$. To see this, we have the equivalences 
\begin{equation*}
    \begin{split}
        \ptconv_{F'}(\mathcal{S}) &= \ptconv_{F'}(\mathcal{S}')  \Leftrightarrow \forall p \in \Sp_{F'}(\NN), a \in F', \, \mathcal{S} \subset \HH_{p,a} , \textup{ iff } \mathcal{S}' \subset \HH_{p,a} \\
        & \Leftrightarrow\forall p \in \Sp_{F'}(\NN), a \in F', \, \min\{p(n) \mid n \in \mathcal{S}\} \geq a \, \textup{ iff } \, \min\{p(n') \mid n' \in \mathcal{S}'\}  \geq a  \\ 
        & \Leftrightarrow  \forall p \in \Sp_{F'}(\NN), \, \min\{p(n) \mid n \in \mathcal{S}\} = \min\{p(n') \mid n' \in \mathcal{S}'\} \\ 
        & \Leftrightarrow \forall m \in \MM_{F'}, \, \min\{\v(m)(n) \mid n \in \mathcal{S}\} = \min\{\v(m)(n') \mid n' \in \mathcal{S}'\} \, \textup{ since $\v$ is bijective } \\ 
        & \Leftrightarrow \forall m \in \MM_{F'}, \, \min\{\w(n)(m) \mid n \in \mathcal{S}\} = \min\{\w(n')(m) \mid n' \in \mathcal{S}'\} \, \textup{ since $\v, \w$ form a pairing} \\ 
        & \Leftrightarrow \min\{\w(n) \mid n \in \mathcal{S}\} = \min\{\w(n') \mid n' \in \mathcal{S}'\} , \textup{ as functions on $\MM_{F'}$} 
    \end{split}
\end{equation*} 
as desired. Thus the extension $\w: \langle \NN_{F'}\rangle \to \P_{\MM_{F'}}$ is well-defined, and $\w$ is a semigroup homomorphism by definition. We claim $\w$ is a bijection. Indeed, by hypothesis, $\w$ is a bijection on the generating sets, so $\w$ is certainly surjective, and the argument above shows that $\w$ is injective. 

Next we check that $\w$ takes the star product to the addition operation in $\mathcal{O}_\MM$, i.e., $\w(n \star n') = \w(n) + \w(n')$ for any $n, n' \in \NN_{F'}$. We have 
\begin{equation*} 
\begin{split}
\w(n \star n')(m) & = \w(\bigoplus_\alpha n +_\alpha n')(m) \, \textup{ by definition of $\star$} \\ 
& = \min\{ \w(n +_\alpha n')(m) \mid \alpha \in \pi(\NN) \} \\ 
& = \min \{ \v(m)(n+_\alpha n') \mid \alpha \in \pi(\NN)\} \\
& = \v(m)(n) + \v(m)(n') \quad \textup{ since $\v(m)$ is a point} \\
& = \w(n)(m) + \w(n')(m)\\
& = (\w(n)+\w(n'))(m) \\
\end{split} 
\end{equation*}
for arbitrary $m \in \MM_{F'}$, as required. 
Finally, it follows from \Cref{eq: strict duals Fgeq0 linear} that $\v$ and $\w$ also respect the homogeneity with respect to $F_{\geq 0}$-scalar multiplication, so they are $F_{\geq 0}$-semialgebra morphisms. This completes the proof. 
\end{proof}

\begin{example}\label{ex: SMM and PNN iso in trivial case}
Recall that in \Cref{example: canonical semialgebra of trivial PL} we saw that the canonical semialgebra $S_\MM$ for the trivial polyptych lattice over $\Z$ is the $\Z_{\geq 0}$-semialgebra of integral polytopes. In this case, the map $S_\MM \to P_\NN$ is the morphism taking an integral polytope $P$ to its support function $\varphi_P: u \mapsto \min\{\langle m, u \rangle \, \mid \, m \in P\}$ as defined in e.g.\ \cite[p.186]{Cox_Little_Schenck}. Our PL analogue of this classical construction will make its appearance in \Cref{definition: support function}. 
\exampleqed
\end{example}

\subsection{A polyptych lattice structure on $\Sp(\MM)$}\label{subsec: PL SpM}

Let $\MM$ and $\NN$ be finite polyptych lattices of rank $r$ over $F$, and suppose that $(\MM,\NN, \v: \MM_{\R} \to \Sp_{\R}(\NN), \w: \NN_{\R} \to \Sp_{\R}(\MM))$ form a strict dual $F'$-pair for $F \subseteq F' \subseteq \R$. If $F' \subseteq \R$ is a \textit{subfield} of $\R$, we will show below that the set $\Sp_{F'}(\MM)$ can be endowed with a PL structure over $F'$, in the sense that there exists a polyptych lattice over $F'$ for which $\Sp_{F'}(\MM)$ can be identified with its set of elements. In fact, in this case, this polyptych lattice will be strongly isomorphic (in a sense to be defined precisely below) to $\NN_{F'}$. Moreover, in the case when $F=\Z$, we may then use this isomorphism to endow a subset of $\Sp_{F'}(\MM)$ with a polyptych lattice $\Z$-structure such that it is strongly isomorphic to $\NN$ over $\Z$.

 To begin, we make use of property (4) of \Cref{def_dual} for strict dual pairs. Let $\gamma \in \pi(\NN)$ and 
\begin{equation}\label{eq: def Cj}
C_\gamma := \v^{-1}\Sp_{\R}(\NN, \gamma) \subset \MM_{\R}.
\end{equation} 
By property (4), $C_\gamma$ is one of the maximal-dimensional faces of $\Sigma(\MM)$, and as $\gamma$ ranges over $\pi(\NN)$, we obtain all such faces. 
Let $L_{C_\gamma}: \Sp(\MM) \to \Hom(C_\gamma \cap \MM, F)$ be the restriction map of \Cref{definition: L_C}. 
The point of the next result is that the restriction map $L_{C_\gamma}$ and the map $\w: \NN \to \Sp(\MM)$ can be used to define coordinate chart maps on $\Sp(\MM)$ which are compatible with the chart maps of $\NN$.

\begin{proposition}\label{prop-evaluationmap}
Let $\MM,\NN$ be finite polyptych lattices of rank $r$ over $F$, let $F \subseteq F' \subseteq \R$, and let $(\MM, \NN, \v, \w)$ be a strict dual $F'$-pair. Let $\gamma \in \pi(\NN)$ and consider $C_\gamma$ as above. Then there exists an injective $F'$-linear map $\w_\gamma: N_\gamma \otimes F' \to \Hom(C_\gamma \cap \MM_{F'}, F')$ making the following diagram commute: 
\begin{equation}\label{eq: w gamma def}
\begin{tikzcd}
\NN \otimes F' \arrow[r, "\pi_\gamma"] \arrow[d, "\w"]
& N_\gamma \otimes F' \arrow[d, "\w_\gamma"] \\
\Sp_{F'}(\MM) \arrow[r, "L_{C_\gamma}"]
& \Hom(C_\gamma \cap \MM_{F'}, F'). 
\end{tikzcd}
\end{equation}
Moreover, if $F'$ is a field, then $\w_\gamma$ is an $F'$-linear isomorphism of $F'$-vector spaces. 
\end{proposition}

\begin{proof}
    We first define the map $\w_\gamma$ as $L_{C_\gamma} \circ \w \circ \pi_\gamma^{-1}$. This is well-defined since $\pi_\gamma$ is an isomorphism, and it follows immediately that \Cref{eq: w gamma def} commutes. 
    
    We now show that $\w_\gamma$ is $F'$-linear. We begin by showing that for $u,u' \in N_\gamma \otimes F'$, we have $\w_\gamma(u)+\w_\gamma(u')=\w_\gamma(u+u')$. Let $m_1, \ldots, m_r \in C_\gamma \cap \MM_{F'}$ form a maximal $\R$-linearly independent set. By assumption on $C_\gamma$, we know $\v(m_k) \in \Sp_{F'}(\NN, \gamma)$, i.e., $\v(m_k)$ is $F'$-linear on the coordinate chart $N_\gamma \otimes F'$, for all $k$. Hence $\v(m_k)(\pi_\gamma^{-1}(u)) + \v(m_k)(\pi_\gamma^{-1}(u')) = \v(m_k)(\pi_\gamma^{-1}(u+u'))$ for each $k$, and duality of $\v,\w$ in turn implies
    \begin{equation}\label{eq: wj for k} \w(\pi_\gamma^{-1}(u))(m_k) + \w(\pi_\gamma^{-1}(u'))(m_k) = \w(\pi_\gamma^{-1}(u+u'))(m_k).
    \end{equation}
    for all $k$. By definition $\w$ takes values in $\Sp_{F'}(\MM)$ and hence by \Cref{lemma: points linear on faces of SigmaM} its images are $F'$-linear on $C_\gamma \cap \MM_{F'}$ (since $C_\gamma$ is a cone in $\Sigma(\MM)$). The $\{m_k\}$ are a maximally linearly independent set, so the equality in \Cref{eq: wj for k} for all $k$ implies $\w(\pi_\gamma^{-1}(u))+\w(\pi_\gamma^{-1}(u'))=\w(\pi_\gamma^{-1}(u+u'))$ as functions on $C_\gamma \cap \MM_{F'}$, i.e., $L_{C_\gamma}(\w(\pi_\gamma^{-1}(u))) + L_{C_\gamma}(\w(\pi_\gamma^{-1}(u'))) = L_{C_\gamma}(\w(\pi_\gamma^{-1}(u+u')))$, which by definition of $\w_\gamma$ yields $\w_\gamma(u)+\w_\gamma(u') = \w_\gamma(u+u')$. 
    Next, we argue that $\pi_\gamma^{-1}(\lambda \, u) = \lambda \pi_\gamma^{-1}(u)$ for any non-negative scalar $\lambda \in F'_{\geq 0}$ and any $u \in N_\gamma \otimes F'$. Indeed, since mutation maps are piecewise-linear, it follows that $\pi_\gamma$ (and $\pi_\gamma^{-1}$) commutes with non-negative dilations, so $\pi_\gamma^{-1}(\lambda \, u) = \lambda \pi_\gamma^{-1}(u)$. By \Cref{lemma: bijection charts and faces for duals and dilations and PL on cones} we know that $\w$ also commutes with dilations, and it is straightforward that restriction to $C_\gamma \cap \MM_{F'}$ does as well. Hence $\w_\gamma$ also commutes with non-negative dilations, as required. Finally, we wish to show that $\w_\gamma(\lambda \, u)=\lambda\, \w_\gamma(u)$ for $\lambda < 0, \lambda \in F'$. 
    To see this, observe that we already know $\w_\gamma(0)=0$ by taking $\lambda=0$ for the dilation factor. Thus we know that $\w_\gamma(\lambda \, n + (-\lambda) n) = \w_\gamma(0) = 0$. By what was shown above, we therefore conclude $0 = \w_\gamma(\lambda \, n + (-\lambda) n) = \w_\gamma(\lambda \, n) + \w_\gamma((-\lambda)n) = \w_\gamma(\lambda \, n) + (-\lambda) \w_\gamma(n)$, since $-\lambda $ is non-negative. But this implies $\w_\gamma(\lambda \, n) = - (- \lambda) \w_\gamma(n) = \lambda \w_\gamma(n)$, as desired. Hence $\w_\gamma(\lambda \, n)= \lambda \w_\gamma(n)$ for \emph{any} $\lambda \in F'$, concluding the proof that $\w_\gamma$ is $F'$-linear on $N_\gamma$.
    
    For the last claim, note that by \Cref{lem-convex-11} we know that $L_{C_\gamma}$ is one-to-one, and both $\pi_\gamma^{-1}$ and $\w$ are bijections. This implies that $\w_\gamma$ is also one-to-one. In the case that $F'$ is a field, both $N_\gamma \otimes F'$ and $\Hom(C_\gamma \cap \MM_{F'}, F')$ are $F'$-vector spaces of dimension $r$, so an injective linear map between them must be a linear isomorphism. 
\end{proof}

For the remainder of this section, we assume that $F'$ is a field, so that by \Cref{prop-evaluationmap} the maps $\w_\gamma$ are $F'$-linear isomorphisms. In this context, we now return to our stated goal of equipping $\Sp_{F'}(\MM)$ with the structure of a polyptych lattice, and identifying it with $\NN_{F'}$. 
In fact, this follows as a straightforward consequence of \Cref{prop-evaluationmap}.
However, for the latter step, we need first to define a notion that allows us to identify two polyptych lattices.

\begin{definition}
Let $\MM = (\{M_\alpha\}_{\alpha \in \mathcal{I}}, \{ \mu_{\alpha,\beta}: M_\alpha \to M_\beta \}_{\alpha,\beta \in \mathcal{I}}), \MM' = (\{M'_{\alpha'}\}_{\alpha' \in \mathcal{I}'}, \{\mu'_{\alpha', \beta'}: M'_{\alpha'} \to M'_{\beta'}\}_{\alpha',\beta' \in \mathcal{I}'})$ be two polyptych lattices of rank $r$ over $F$. 
A \textbf{strong isomorphism of polyptych lattices} from $\MM$ to $\MM'$ is a tuple of maps $(\Phi: \mathcal{I} \to \mathcal{I}', \{\Phi_\alpha: M_\alpha \to M'_{\Phi(\alpha)} \}_{\alpha \in \mathcal{I}})$ where $\Phi$ is a bijection, and each $\Phi_\alpha$ is a linear isomorphism such that the diagrams 
\begin{equation}\label{eq: linear morphism of PL}
\begin{tikzcd}
M_\alpha \arrow[r, "\mu_{\alpha,\beta}"] \arrow[d, "\Phi_\alpha"]
& M_{\beta} \arrow[d, "\Phi_{\beta}"] \\
M_{\Phi(\alpha)}' \arrow[r, "\mu_{\Phi(\alpha),\Phi(\beta)}' "]
& M_{\Phi(\beta)}'
\end{tikzcd}
\end{equation}
commute for all $\alpha,\beta \in \mathcal{I}$. 
When such $\Phi$ and $\Phi_\alpha$ exist, we say that $\MM$ and $\MM'$ are \textbf{strongly isomorphic as polyptych lattices}. 
\exampleqed
\end{definition}

\begin{remark}
There are other natural candidates for a notion of isomorphism of polyptych lattices. For instance, and roughly speaking, if some mutations in the definition of a polyptych lattice are actually linear (not just piecewise linear), then in some sense, those mutations give redundant information, and the polyptych lattice with such mutations omitted should be viewed as the ``same'' as the original polyptych lattice. Discussing these ideas is beyond the scope of this manuscript. We intend to develop them in future work. 
\exampleqed
\end{remark}

We may now state and prove the following consequence of \Cref{prop-evaluationmap}.

\begin{proposition}\label{cor-fulldimiso}
Let $\MM,\NN$ be finite polyptych lattices of rank $r$ over $F$, let $F \subseteq F' \subseteq \R$, and let $(\MM, \NN, \v, \w)$ be a strict dual $F'$-pair. Assume $F'$ is a field. Then: 
\begin{enumerate} 
\item[(1)] 
The set of points $\Sp_{F'}(\MM)$ can be equipped with the structure of a polyptych lattice of rank $r$ over $F'$, with charts $\{\Hom(C_\gamma \cap \MM_{F'}, F')\}_{\gamma\in \pi(\NN)}$, chart maps $L_{C_\gamma}: \Sp_{F'}(\MM) \to \Hom(C_\gamma \cap \MM_{F'},F')$, and mutation maps $L_{C_{\delta}} \circ L_{C_\gamma}^{-1}$. In particular, $\pi(\Sp(\MM)) = \pi(\NN) = \mathcal{J}$. 

\item[(2)] The map $\Phi: \pi(\NN) \to \pi(\Sp(\MM))$ given by the identity, i.e. $\Phi(\gamma)=\gamma$ for all $\gamma \in \mathcal{J}$, and the linear isomorphisms $\w_\gamma: N_\gamma \otimes F' \to \Hom(C_\gamma \cap \MM_{F'}, F')$ for $\gamma \in \mathcal{J}$, give a strong isomorphism between $\NN_{F'}$ and $\Sp_{F'}(\MM)$. In particular, $\Sp_{F'}(\MM)$ and $\NN_{F'}$ are strongly isomorphic. 
\end{enumerate} 
\end{proposition}

\begin{proof} 
We begin with (1). The properties of \Cref{definition: polyptych lattice} follow immediately from the definition of the mutations as $\mu'_{\gamma, \delta}:= L_{C_{\delta}} \circ L_{C_\gamma}^{-1}$, so the only claim to be proved is that there exists a complete fan in $\Hom(C_\gamma, \R)$ such that all mutations $\mu'_{\gamma,\delta}$ are linear on any cone of the fan. We will take the fan with maximal-dimensional cones given by $L_{C_\gamma}(\Sp_\R(\MM,\alpha))$ as $\alpha$ varies over $\pi(\MM)$. Since $\MM$ is full, by Lemma \ref{lemma: bijection charts and faces for duals and dilations and PL on cones} the $\Sp(\MM,\alpha)$ cover $\Sp(\MM)$, so this does define a complete fan. Moreover, each $L_{C_\gamma}(\Sp_\R(\MM,\alpha))$ is an $F'$-rational polyhedral cone, since by strict duality, $\w^{-1}(\Sp_\R(\MM,\alpha))$ is a cone of $\Sigma(\MM)$, and in particular is $F$-rational and polyhedral. This implies $\pi_\gamma(\w^{-1}(\Sp_\R(\MM,\alpha))$ is also $F$-rational and polyhedral, and since $\w_\gamma$ is $F'$-linear by \Cref{prop-evaluationmap}, we conclude from the diagram in \Cref{eq: w gamma def} that $L_{C_\gamma}(\Sp_\R(\MM,\alpha))$ is $F'$-rational polyhedral. Now we claim that for any $\gamma,\delta \in \pi(\NN)=\mathcal{J}$ and any $\alpha \in \pi(\MM)$, the mutation $\mu'_{\gamma,\delta}$ is linear on $L_{C_\gamma}(\Sp_{F'}(\MM,\alpha))$. 
To see this, suppose $p,q \in \Sp_{F'}(\MM,\alpha)$ and consider $L_{C_\gamma}(p)$ and $L_{C_\gamma}(q)$. We want to show $\mu'_{\gamma,\delta}(L_{C_\gamma}(p)+L_{C_\gamma}(q)) = \mu'_{\gamma,\delta}(L_{C_\gamma}(p)) + \mu'_{\gamma,\delta}(L_{C_\gamma}(q))$. We compute both sides. By definition of the mutation, the RHS is $L_{C_\delta}(p)+L_{C_\delta}(q)$. 
We know from \Cref{prop-conesinthepoints} that $\Sp_{F'}(\MM,\alpha)$ is closed under addition, so $p+q \in \Sp_{F'}(\MM,\alpha)$, and hence $L_{C_\delta}(p+q)=L_{C_\delta}(p)+L_{C_\delta}(q)$. By the same argument, $L_{C_\gamma}(p)+L_{C_\gamma}(q)=L_{C_\gamma}(p+q)$. Now we can see that the LHS is $\mu'_{\gamma,\delta}(L_{C_\gamma}(p+q))= L_{C_\delta}(p+q)$, which is the RHS. Finally, the mutations respect scalar multiplication by non-negative scalars  in $F'$ since the restriction maps $L_{C_\gamma}$ for any $\gamma$ commute with non-negative dilations. 
This proves (1). 

To prove (2) we must check that the following diagrams 
\begin{equation*}\label{eq: strong iso of NN with SpMM}
\begin{tikzcd}
    N_\gamma \otimes F' \arrow[r, "\mu_{\gamma,\delta}"] \arrow[d, "\w_\gamma"]
& N_\delta \otimes F' \arrow[d, "\w_\delta"] \\
\Hom(C_\gamma \cap \MM_{F'},F') \arrow[r, "L_{\delta} \circ L_\gamma^{-1} "]
& \Hom(C_\delta \cap \MM_{F'}, F') 
\end{tikzcd}
\end{equation*}
commute for all $\gamma$ and $\delta$. This is a straightforward computation using the definition of $\w_\gamma$ and $\w_\delta$. 
\end{proof}

\begin{corollary}\label{prop-strictpair-unique}
Let $F$ be a field and let $\MM$ be a finite polyptych lattice over $F$. Suppose $\NN, \NN'$ are both strict dual $F$-pairs to $\MM$. Then $\NN$ and $\NN'$ are strongly isomorphic.      
\end{corollary}

\begin{proof}
Both $\NN$ and $\NN'$ are strongly isomorphic to $\Sp_F(\MM)$ equipped with the PL structure constructed in \Cref{cor-fulldimiso} and strong isomorphism can easily be seen to be transitive. 
\end{proof}

We conclude this section by giving a construction of a lattice structure on the polyptych lattice $\Sp_{F'}(\MM)$ in the case when $F'$ is a field, $F=\Z$, and $F\subseteq F'$. More specifically, the proof of \Cref{prop-evaluationmap} shows that the diagram \Cref{eq: w gamma def} commutes. Now consider the image $\w(\NN)$ in $\Sp_{F'}(\MM)$, and also the image $L_{C_\gamma}(\w(\NN))$ in $\Hom(C_\gamma \cap \MM_{F'}, F')$. Then we get a commutative diagram \begin{equation}\label{eq: w gamma def integral}
\begin{tikzcd}
\NN  \arrow[r, "\pi_\gamma"] \arrow[d, "\w"]
& N_\gamma  \arrow[d, "\w_\gamma"] \\
\w(\NN) \arrow[r, "L_{C_\gamma}"]
& L_{C_\gamma}(\w(\NN)) 
\end{tikzcd}
\end{equation}
where each corner of the above diagram is a subset of the corresponding corner in \Cref{eq: w gamma def}. By construction, the left vertical arrow is a bijection, and so is the bottom horizontal arrow. Since $\pi_\gamma$ is a bijection by hypothesis on polyptych lattices, we conclude that the right vertical arrow $\w_\gamma$ in \Cref{eq: w gamma def integral} is a $\Z$-linear isomorphism. 
In particular, $L_{C_\gamma}(\w(\NN))$ is a free $\Z$-lattice of rank $r$ for each $\gamma$. We may now define a PL structure, over $\Z$, on the subset $\w(\NN)$ of $\Sp_{F'}(\MM)$ by considering the maps $L_{C_\gamma}: \w(\NN) \to L_{C_\gamma}(\w(\NN))$ for each $\gamma \in \pi(\NN)=\mathcal{J}$ and mutation maps given by $L_{C_\delta} \circ L_{C_\gamma}^{-1}: L_{C_\gamma}(\w(\NN)) \to L_{C_\delta}(\w(\NN))$ for all $\delta,\gamma \in \pi(\NN)$. By the same argument as in the proof of \Cref{prop-evaluationmap}, this gives rise to a PL structure over $\Z$ on $\w(\NN)$. By construction, $\w(\NN)$ is a subset of $\Sp_{F'}(\MM)$ and may be thought of as a lattice structure inside $\Sp_{F'}(\MM)$.

%%%%%%%%%%
\section{Polyptych lattice polytopes and their duals} \label{sec_PL_polytopes}

In \Cref{subsec_PL_halfspaces_basics} we initiated a theory of convex geometry in the PL setting. In this section we develop this further by introducing PL polytopes and their duals. 

{
\begin{definition} \label{def-PL polytope}
We say $\PP \subset \MM_{\R}$ is a \textbf{PL polyhedron over $F$} if it is a finite intersection of PL half-spaces over $F$, i.e., 
$$\PP = \bigcap_{i=1}^\ell \HH_{p_i, a_i}$$
for some collection of points $p_i \in \Sp(\MM)$ and $a_i \in F$. We say a PL polyhedron over $F$ is a \textbf{PL polytope over $F$} if it is compact. 
\exampleqed
\end{definition} 
}

\begin{example}\label{ex_M_polytope}
We continue with our running example from \Cref{ex_half_space}.
Consider the three points determined by the following table:
\begin{center}
    \begin{tabular}{|l|l|l|}
    \hline
        $p_1(\mathbb{e}_1)=-1$ & $p_1(\mathbb{e}_2)=0$ & $p_1(\mathbb{e}'_2)=-1$ 
        \\ \hline
        $p_2(\mathbb{e}_1)=0$ & $p_2(\mathbb{e}_2)=1$ & $p_2(\mathbb{e}'_2)=-1$ 
        \\ \hline
        $p_3(\mathbb{e}_1)=1$ & $p_3(\mathbb{e}_2)=-1$ & $p_3(\mathbb{e}'_2)=1$ 
        \\ \hline
    \end{tabular}
.\end{center}
The PL polytope over $\Z$ obtained by intersecting the three PL half-spaces $\PP=\HH_{p_1,-1}\cap \HH_{p_2,-1}\cap \HH_{p_3,-1}$ is depicted in \Cref{fig_polytopes}.
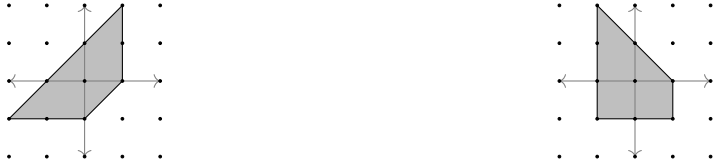
\begin{figure}[h]
    \centering
    \begin{tikzpicture}
    \filldraw[fill=gray!50] (-.5*2,-.5*1)--(0,-.5*1)--(.5*1,0)--(.5*1,.5*2)--(-.5*2,-.5*1);
    \draw[gray,<->] (-1,0)--(1,0);
    \draw[gray,<->] (0,-1)--(0,1);
    % loop over the lattice points
    \foreach \i in {-2,...,2}
      \foreach \j in {-2,...,2}{
        \filldraw[black] (.5*\i,.5*\j) circle(.5pt);
      };
    \end{tikzpicture}
    \hspace{5cm}
    \begin{tikzpicture}
    \filldraw[fill=gray!50] (-.5*1,.5*2)--(-.5*1,-.5*1)--(.5*1,-.5*1)--(.5*1,0)--(-.5*1,.5*2)
    ;
    \draw[gray,<->] (-1,0)--(1,0);
    \draw[gray,<->] (0,-1)--(0,1);
    % loop over the lattice points
    \foreach \i in {-2,...,2}
      \foreach \j in {-2,...,2}{
        \filldraw[black] (.5*\i,.5*\j) circle(.5pt);
      };
    \end{tikzpicture}
    \caption{The two chart images of the PL polytope $\PP$ from \Cref{ex_M_polytope}. On the left is $\pi_1(\PP)$ and on the right is $\pi_2(\PP)$. }
    \label{fig_polytopes}
\end{figure}
\exampleqed
\end{example}

\begin{lemma}\label{rem_vert_superset}\label{lemma: PL polytope is point convex}
Let $\MM$ be a finite polyptych lattice over $F$ and let $\PP \subset \MM_{\R}$ be a PL polytope over $F$.  Then: (1) for any $\alpha \in \pi(\MM)$, the chart image $\pi_\alpha(\PP)$ is a (classical) polytope in $M_\alpha \otimes \R$, so in particular $\PP$ is chart-convex, and (2) $\PP$ is point-convex over $\R$. 
\end{lemma}

    \begin{proof} 
    First we prove (1). 
    Let $\PP = \cap_{k=1}^\ell \HH_{p_k,a_k}$. We know from \Cref{lemma: point PL concave} that a point in $\Sp(\MM)$ is piecewise-linear and convex on any chart. Thus, on any chart $M_\alpha$, $(p_k)_\alpha$ is defined as the minimum of a finite collection of linear functions. In particular, $\pi_\alpha(\HH_{p_k,a_k})$ is a finite intersection of (classical) half-spaces in $M_\alpha \otimes \R$. Since this is true of all $k$, it follows that $\pi_\alpha(\PP)$ is a compact finite intersection of (classical) half-spaces and hence a (classical) polytope in $M_\alpha \otimes \R$. Polytopes are convex, and since this is true of all $\alpha$, $\PP$ is chart-convex. Next we prove (2). By definition $\ptconv_\R(\PP) = \cap_{\PP \subset \HH_{p,a}} \HH_{p,a}$. Note that if $\PP = \cap_{k=1}^\ell \HH_{p_k,a_k} \subset \HH_{p,a}$, then $\left(\cap_{k=1}^\ell \HH_{p_k,a_k}\right) \cap \HH_{p,a} = \cap_{k=1}^\ell \HH_{p_k,a_k}$. This implies that in the intersection $\cap_{\PP \subset \HH_{p,a}} \HH_{p,a}$ we may omit all $\HH_{p,a}$ which are not equal to any of the $\HH_{p_k,a_k}$ without altering the result. Thus $\ptconv_\R(\PP)=\cap_{k=1}^\ell \HH_{p_k,a_k} = \PP$, as was to be shown. 
    \end{proof} 

    \Cref{lemma: PL polytope is point convex}(1) allows us to define a notion of a vertex of a PL polytope.

\begin{definition}\label{def_vertices}
Let $\MM$ be a finite polyptych lattice over $F$ and let $\PP$ be a PL polytope over $F$. An element $m \in \MM$ is a \textbf{vertex} of $\PP$ if it is a vertex of the image of $\PP$ in some coordinate chart $M_\alpha \otimes \R$. More precisely, the \textbf{set of vertices $V(\PP)$} of $\PP$ is defined to be 
\begin{equation*}\label{eq: def vertices of P}
V(\PP) := \{ m \in \MM_{\R}  \, \mid \, \exists \alpha \in \pi(\MM), \, \, \pi_\alpha(m)\,  \textup{ is a vertex of } \, \pi_\alpha(\PP)\}. 
\end{equation*} 
\end{definition}

\begin{remark}\label{remark: vertices finite}
The finiteness assumptions on $\MM$ and $\PP$ imply that $V(\PP)$ is a finite set. 
\exampleqed
\end{remark}

\begin{example}\label{example: M polytope vertices}
Continuing with our running example in \Cref{ex_M_polytope}, for the $\MM$-polytope $\PP$ depicted in \Cref{fig_polytopes} it is straightforward to see that  
\begin{equation}\label{eq_ex_vertices}
    V(\PP)=\{(\varepsilon_1,-\varepsilon_1),(\varepsilon_1+2\varepsilon_2,-\varepsilon_1+2\varepsilon_2),(-2\varepsilon_1-\varepsilon_2,\varepsilon_1-\varepsilon_2),(-\varepsilon_1,\varepsilon_1),(-\varepsilon_2,-\varepsilon_1-\varepsilon_2)\}
.\end{equation} 
Note that $(\varepsilon_1,-\varepsilon_1)$ is a vertex of $\pi_1(\PP)$ but not of $\pi_2(\PP)$.
\exampleqed
\end{example}

 As a first step, we prove an analogue of a classical fact concerning polytopes. 

\begin{proposition}\label{poly_is_conv_verts}
Let $\MM$ be a finite polyptych lattice over $F$ and $\PP \subset \MM_\R$ a PL polytope over $F$. Then $\PP=\ptconv_\R(V(\PP))$.  
\end{proposition}

\begin{proof}
By definition of point-convex hulls, we have that $\ptconv_\R(V(\PP)) = \cap_{V(\PP) \subset \HH_{p,a}} \HH_{p,a}$ and similarly for $\ptconv_\R(\PP)$. Since $V(\PP) \subset \PP$, it follows that any PL half-space containing $\PP$ contains $V(\PP)$, so any $\HH_{p,a}$ appearing in the intersection for $\ptconv_\R(\PP)$ appears in that for $\ptconv_\R(V(\PP))$. It follows that $\ptconv_\R(V(\PP)) \subset \ptconv_\R(\PP)=\PP$, where the last equality follows from \Cref{lemma: PL polytope is point convex}(2).  It now suffices to show the other inclusion, $\PP \subset \ptconv_\R(V(\PP))$. For this it would suffice to show $\pi_\alpha(\PP) \subset \pi_\alpha(\HH_{p,a})$ for any PL half-space $\HH_{p,a}$ with $V(\PP) \subset \HH_{p,a}$. We know $\pi_\alpha(\PP)$ is a classical compact polytope in $M_\alpha \otimes \R$ and $\pi_\alpha(\HH_{p,a})$ is a finite intersection of classical half-spaces in $M_\alpha \otimes \R$. By definition of $V(\PP)$, all vertices of $\pi_\alpha(\PP)$ are contained in $\pi_\alpha(V(\PP))$. By hypothesis on $\HH_{p,a}$, we know $\pi_\alpha(V(\PP)) \subset \pi_\alpha(\HH_{p,a})$ so all vertices of $\pi_\alpha(\PP)$ are contained in $\pi_\alpha(\HH_{p,a})$. Since $\pi_\alpha(\HH_{p,a})$ is itself a closed convex set and contains all vertices of $\pi_\alpha(\PP)$, it also contains its (classical) convex hull, namely $\pi_\alpha(\PP)$. This completes the proof. 
\end{proof}

 In the setting of strictly dualizable finite polyptych lattices, we may derive further properties of PL polytopes and of point-convex hulls. 
For instance, we saw in \Cref{lemma: PL polytope is point convex} that any PL polytope has coordinate chart images which are classical polytopes. In \Cref{prop_polytopes_are_convex}, we obtain a converse, under additional hypotheses. 

\begin{proposition} \label{prop_polytopes_are_convex}
Let $\MM$ be a finite polyptych lattice over $F$, and let $(\MM, \NN, \v, \w)$ be a strict dual $F'$-pair for $F'$ a field. 
If $\mathcal{Q} \subset \MM_{\R}$ has the property that $\pi_\alpha(\mathcal{Q}) \subset M_\alpha \otimes \R$ is an $F'$-rational polytope for all $\alpha \in \mathcal{I}$, then $\mathcal{Q}$ is a PL polytope, i.e., $\mathcal{Q}$ is a finite intersection of PL half-spaces over $F'$.
\end{proposition} 

\begin{proof} 
For simplicity we assume in this proof that $\mathcal{Q}$ is full-dimensional; the argument is similar for the smaller-dimensional cases by first restricting to an appropriate subspace. For a maximal-dimensional cone $\mathcal{C}$ in $\Sigma(\MM)$, consider the intersection $\mathcal{Q} \cap \mathcal{C}$. For $\alpha \in \pi(\MM)$, by hypothesis we know $\pi_\alpha(\mathcal{Q})$ is a classical polytope in $M_\alpha \otimes \R$. Let $K$ be the preimage under $\pi_\alpha$ of a facet of $\pi_\alpha(\mathcal{Q})$ and consider $K \cap \mathcal{C}$. Allowing $\alpha$ and $K$ and $\mathcal{C}$ to vary, let $\mathcal{F}$ denote the set of all such subsets $K \cap \mathcal{C}$ such that their images under the charts $\pi_\alpha$ are nonempty and codimension $1$. 

For each subset of the form $K \cap \mathcal{C}$ in $\mathcal{F}$, we now construct a corresponding PL half-space over $F'$ as follows. 
Since $\mathcal{Q}$ is $F'$-rational and $K \cap \mathcal{C}$ is nonempty and codimension $1$, it follows that there exists $\rho \in \Hom(\mathcal{C} \cap \MM_{F'},F')$ (which we may extend naturally to $\Hom(\mathcal{C}, \R)$) such that $K \cap \mathcal{C}$ lies in $\rho^{-1}(a)$ for some $a \in F'$ and $\mathcal{Q} \cap \mathcal{C}$ lies in $\rho^{-1}([a,\infty))$. 
Since $F'$ is a field, by \Cref{prop-evaluationmap} we know that the restriction $L_{\mathcal{C}}:\Sp_{F'}(\MM) \to \Hom(\mathcal{C} \cap \MM_{F'},F')$ is surjective, and we know $\w$ is bijective, so there exists $n \in \NN_{F'}$ such that the restriction of $\w(n)$ to $\mathcal{C}$ coincides with $\rho$. We now claim $\mathcal{Q} \subset \HH_{\w(n),a}$. To see this, recall by \Cref{lemma: bijection charts and faces for duals and dilations and PL on cones} that $\MM$ is full, so $\w(n)$ is linear on some chart $M_\alpha$.  
 It follows that $\pi_\alpha(\HH_{\w(n),a})$ is a classical half-space in $M_\alpha$, and by our choice of $\rho$ and $n$, $\pi_\alpha(\mathcal{Q})$ lies in $\pi_\alpha(\HH_{\w(n),a})$, implying $\mathcal{Q} \subset \HH_{\w(n),a}$ as claimed.

In the above paragraph we constructed a half-space $\HH_{\w(n),a}$ for each set of the form $K \cap \mathcal{C}$ in $\mathcal{F}$. For the argument below we denote this association by $K \cap \mathcal{C} \mapsto \HH_{K \cap \mathcal{C}}$. We have already seen that  $\mathcal{Q} \subset \cap_{K\cap \mathcal{C} \in \mathcal{F}} \HH_{K \cap \mathcal{C}}$. 
Now let $\alpha \in \pi(\MM)$. Let $H_{\alpha, K \cap \mathcal{C}}$ denote the classical half-space in $M_\alpha \otimes \R$ containing $\pi_\alpha(\mathcal{Q})$ and whose boundary contains $\pi_\alpha(K \cap \mathcal{C})$.  
 By \Cref{corollary: PL half space in classical half space}, since $H_{\alpha, K \cap \mathcal{C}}$ is defined by a linear function on $M_\alpha \otimes \R$ which agrees with $\w(n)_\alpha$ on $\pi_\alpha(\mathcal{C})$, we have $\pi_\alpha(\HH_{K \cap \mathcal{C}}) \subseteq H_{\alpha,K \cap \mathcal{C}}$.    Thus $\pi_\alpha(\mathcal{Q}) \subseteq  \bigcap_{K \cap \mathcal{C} \in \mathcal{F}} \pi_\alpha(\HH_{ K \cap \mathcal{C}}) \subseteq \bigcap_{K \cap \mathcal{C} \in \mathcal{F}} H_{\alpha,K \cap \mathcal{C}}$. The last expression includes an intersection over all facets of $\pi_\alpha(\mathcal{Q})$, so the inclusion must be an equality, i.e., $\pi_\alpha(\mathcal{Q}) = \bigcap_{K \cap \mathcal{C} \in \mathcal{F}} \pi_\alpha( \HH_{K \cap \mathcal{C}}) = \bigcap_{K \cap \mathcal{C} \in \mathcal{F}}  H_{\alpha,K \cap \mathcal{C}}$.  This proves that $\mathcal{Q} = \cap_{K \cap \mathcal{C} \in \mathcal{F}} \HH_{K \cap \mathcal{C}}$ is an intersection of PL half-spaces. 
Moreover, in the setting that $\MM$ is a finite polyptych lattice, then $\Sigma(\MM)$ is finite, there are finitely many polytopes $\pi_\alpha(\mathcal{Q})$ to be considered and each $\pi_\alpha(\mathcal{Q})$ has finitely many facets, so $\mathcal{F}$ is finite and $\mathcal{Q}$ can be realized as a finite intersection. 
\end{proof}

We now turn to some technical lemmas that are needed in later sections.

\begin{lemma}\label{lemma: P with interior zero}
Let $\MM$ be a finite polyptych lattice over $\Z$ and let $(\MM,\NN,\v,\w)$ be a strict dual $\R$-pair. Then there exists a PL polytope $\PP \subset \MM_\R$ over $\Z$ which contains $0_\MM$ in its interior. 
\end{lemma}

\begin{proof} 
Let $\mathcal{C}$ be a maximal-dimensional cone in $\Sigma(\MM)$ and consider $\Hom(\mathcal{C},\R)$. Choose a finite set $S_\mathcal{C} = \{p^\mathcal{C}_1, p^\mathcal{C}_2, \cdots, p^\mathcal{C}_{\ell_\mathcal{C}}\}$ of elements of $\Hom(\mathcal{C},\Q)$ and parameters $a^\mathcal{C}_1,\cdots,a^\mathcal{C}_{\ell_\mathcal{C}} \in \Q_{<0}$ so that the subset $\cap_{i=1}^{\ell_\mathcal{C}} \{p^\mathcal{C}_i \geq a^\mathcal{C}_i\}$ in $\mathcal{C}$ is compact in $\mathcal{C}$. Note such $p^\mathcal{C}_i, a^\mathcal{C}_i$ always exist by classical linear algebra on the cone $\mathcal{C}$, and by the choice of $a^\mathcal{C}_i < 0$, we know $0_\MM \in \mathcal{C}$ must lie in $\cap_{i=1}^{\ell_\mathcal{C}} \{p^\mathcal{C}_i \geq a^\mathcal{C}_i\}$. Again by \Cref{prop-evaluationmap} we know $\Sp_{\Q}(\MM) \to \Hom(\mathcal{C} \cap \MM_\Q, \Q)$ is surjective, so there exist points $\tilde{p}^\mathcal{C}_1, \cdots, \tilde{p}^\mathcal{C}_{\ell_\mathcal{C}}$ in $\Sp_\Q(\MM)$ that restrict to $p^\mathcal{C}_1, \cdots, p^\mathcal{C}_{\ell_\mathcal{C}}$ on $\mathcal{C}$. By rescaling as necessary we may also assume $\tilde{p}^{\mathcal{C}}_j \in \Sp(\MM)$ and $a_j \in \Z_{<0}$. By construction, $\PP_\mathcal{C} := \cap_{i=1}^{\ell_\mathcal{C}} \HH_{\tilde{p}^\mathcal{C}_i, a^\mathcal{C}_i}$ contains $0_\MM$ in its interior. 
Repeating this construction for all maximal-dimensional cones $\mathcal{C}$, consider $\PP := \cap_\mathcal{C} \PP_\mathcal{C}$ where the intersection is over all maximal-dimensional cones of $\Sigma(\MM)$. Since $\MM$ is finite, this is a finite intersection, and $\PP$ again contains $0_\MM$ in its interior. Moreover, $\PP$ is compact since by construction $\PP$ is contained in a compact set in each of the finitely many cones $\mathcal{C}$. Hence $\PP$ is a PL polytope. This concludes the proof. 
\end{proof}

In \Cref{lemma: pt convex hulls lattice points finite} 
we obtain some finiteness results for point-convex hulls.

\begin{lemma}\label{lemma: pt convex hulls lattice points finite}
Under the same assumptions as \Cref{lemma: P with interior zero}, we have: 
\begin{enumerate} 
\item[(1)] if $S \subset \MM$ is finite, then $\ptconv_\R(S) \cap \MM$ is finite, and, 
\item[(2)] if $S = \{m\} \subset \MM$ has a single element, then $\ptconv_\R(S) = \{m\}$. 
\end{enumerate}
\end{lemma}

\begin{proof} 
We begin with (1). Since $\MM$ is defined over $\Z$, in order to see that $\ptconv_\R(S) \cap \MM$ is finite, it suffices to show that $\ptconv_\R(S)$ is compact. Since $\ptconv_\R(S)$ is closed by construction, it in turn suffices to show that it is contained in a compact set. 
Recall that by \Cref{lemma: P with interior zero} we know there exists a compact PL polytope $\PP = \cap_{i=1}^\ell \HH_{p_i, a_i} \subset \MM_\R$ with $a_i<0$ such that $0_\MM$ is in its interior. Then for $N >> 0$ a sufficiently large integer, any finite set $S$ lies in $N \cdot \PP := \cap_{i=1}^{\ell} \HH_{p_i, Na_i}$. Since a PL polytope is point-convex, by \Cref{lemma: S subset T point convex} it follows that $\ptconv_\R(S) \subset N \cdot \PP$. Moreover, since $\PP$ is compact, so is $N \cdot \PP$. This proves (1). 

Now we show (2). For each $\mathcal{C}$ a maximal-dimensional cone in $\Sigma(\MM)$ with $m \in \mathcal{C}$, by an argument similar to the proof of \Cref{lemma: P with interior zero}, we may find $\tilde{p}^\mathcal{C}_1, \cdots, \tilde{p}^\mathcal{C}_{\ell_\mathcal{C}}$ in $\Sp_\R(\MM)$ and scalars $a_i^\mathcal{C}$ such that $\left(\cap_{i=1}^{\ell_\mathcal{C}} \HH_{\tilde{p}^\mathcal{C}_i, a^\mathcal{C}_i}\right) \cap \mathcal{C} = \{m\}$. Define $\PP$ to be the intersection of all such $\HH_{\tilde{p}^\mathcal{C}_i, a^\mathcal{C}_i}$ as $\mathcal{C}$ ranges over those max-dimensional cones with $m \in \mathcal{C}$. We claim $\PP=\{m\}$. By construction $\PP \cap \mathcal{C} = \{m\}$ for all $\mathcal{C}$ containing $\{m\}$. Since a PL polytope is chart-convex by \Cref{lemma: point convex implies chart-convex}, it cannot contain a point $m'$ where $m' \not \in \mathcal{C}$ for $\mathcal{C}$ containing $m$ (since a broken line connecting $m'$ and $m$ cannot be contained in $\PP$). Thus $\PP=\{m\}$ and thus $\ptconv_\R(\{m\}) \subset \PP = \{m\}$ but $\{m\} \subset \ptconv_\R(\{m\})$ by definition, so the two sets are equal. 
\end{proof}

We now define a notion of a \textbf{dual} of a PL polytope, in those cases when $\MM$ has a strict dual. 
We need a preliminary definition. 

\begin{definition}\label{definition: support function} 
Let $\MM$ be a finite polyptych lattice over $F$ and let $(\MM, \NN, \v, \w)$ be a strict dual $F$-pair. Let $\PP \subset \MM_\R$ be a PL polytope over $F$. 
The \textbf{support function} $\psi_\PP:\NN_{\R} \to \R$ associated to $\PP \subset \MM_\R$ is defined as 
\begin{equation}\label{eq: def support function}
\psi_\PP(-) := \min\{\v(u)(-) \mid u \in \PP\}.
\end{equation} 
\exampleqed
\end{definition} 

Note that since $\PP$ is compact, then for any $n\in\NN_\R$ the minimum in \Cref{eq: def support function} exists.
This function is the  generalization, to the setting of polyptych lattices, of the definition of a support function of a classical polytope (see e.g.\ \cite[Proposition 4.2.14]{Cox_Little_Schenck}).
The following lemma tells us that $\psi_\PP$ can be computed by taking a minimum of a finite set -- namely, a minimum over the vertices $V(\PP)$ of $\PP$.

\begin{lemma}\label{lem_psidelta_finite} 
In the setting of \Cref{definition: support function}, 
the support function $\psi_\PP: \NN_\R \to \R$ satisfies 
\begin{equation}\label{eq: psiP finite min}
\psi_\PP=\bigoplus_{s\in V(\PP)}\v(s).
\end{equation}
In particular, $\psi_\PP$ is continuous, and $\psi_\PP$ is an element of $\P_{\NN_\R}$.
\end{lemma}

\begin{proof}
Since $V(\PP)\subseteq\PP$ and $\psi_\PP = \min_{s \in \PP}\{\v(s)\}$ by definition, certainly $\psi_\PP\le\bigoplus_{s\in V(\PP)}\v(s)$.
To show the reverse inequality, let $n\in\NN_\R$ be arbitrary. By compactness of $\PP$, the minimum in the RHS of \Cref{eq: def support function} is achieved, so there exists $u_\circ\in\PP$ such that $\psi_\PP(n)=\v(u_\circ)(n)$ for the fixed choice of $n$. 
From \Cref{poly_is_conv_verts} we know that $\PP = \ptconv_\R(V(\PP))$ so $u_\circ \in \ptconv_\R(V(\PP))$. Then by \Cref{lem-pconvex-2} we know
\begin{equation*}
\psi_\PP(n)=\v(u_\circ)(n)\ge\bigoplus_{s\in V(\PP)}\v(s)(n), \end{equation*}
as functions on $\NN_\R$. This proves \Cref{eq: psiP finite min}.  Finally, since $V(\PP)$ is finite and each $\v(s)$ is an element of $\Sp_\R(\NN)$, it follows immediately that $\psi_\PP$ is continuous and, by \Cref{definition: point algebra}, $\psi_\PP$ is in $\P_\NN$. 
\end{proof}

\begin{example} 
Continuing with our running example, we listed the vertices of the PL polytope $\PP$ in \Cref{example: M polytope vertices}. From that computation it now follows that $\psi_\PP$ is the function
\begin{equation*}
    \begin{split} 
\psi_{\PP}(-)  = \min\{&\v((-2,-1),(1,-1))(-), 
\v((0,-1),(-1,-1))(-), \v((1,0),(-1,0))(-),  \\
& 
\v((1,2),(-1,2))(-), \v((-1,0),(1,0))(-)\}
\end{split} 
\end{equation*} 
where $\v$ is the function 
$$
\v((x,y),(x',y))((u,v),(u',v)) = 
\begin{cases} 
uy+vx, \quad v \geq 0, \\
uy - v(\min\{0,y\}-x), \quad v \leq 0. 
\end{cases} 
$$
\exampleqed
\end{example}

The following lemma shows how to recover the PL polytope $\PP$ from its support function $\psi_\PP$.

\begin{lemma}\label{lem_support_to_poly}
In the setting of \Cref{definition: support function}, we have \begin{equation}\label{eq: PP from PsiP}
\PP=\{u\in \MM_\R \mid \forall n\in \NN_\R, \v(u)(n)\ge \psi_\PP(n)\}.
\end{equation}
\end{lemma}

\begin{proof}
It follows from the definition of support function that the LHS is contained in the RHS. To see the opposite containment, let $u$ be contained in the RHS. Let $\PP = \cap_{i=1}^\ell \HH_{p_i,a_i} = \cap_{i=1}^\ell \HH_{\w(n_i),a_i}$ where $n_i$ are such that $\w(n_i)=p_i$. (Such choices $n_i$ exist since $\w$ is bijective.) Then
$$
\w(n_i)(u) =\v(u)(n_i) \geq \psi_\PP(n_i) = \min\{\v(u')(n_i) \mid u' \in \PP\} = \min\{\w(n_i)(u') \mid u' \in \PP\} \geq a 
$$
where the last inequality is because $\PP \subset \HH_{\w(n_i),a_i}$ for all $i$. Hence $u \in \HH_{\w(n_i),a_i}$ for all $i$, i.e., $u \in \PP$. 
\end{proof}

We are now prepared to define the dual PL polytope associated to certain PL polytopes, with respect to a choice of strict dual.

\begin{definition}\label{definition: dual convex set}
Let $\MM$ be a finite polyptych lattice over $F$, let $(\MM, \NN, \v: \MM_\R \to \Sp_\R(\NN), \w: \NN_\R \to \Sp_\R(\MM))$ be a strict dual $F$-pair, and let $\PP \subset \MM_\R$ be a PL polytope over $F$ containing $0_\MM \in \MM_\R$ in its interior, with corresponding support function $\psi_{\PP}$.
The \textbf{dual PL polytope $\PP^\vee$ to $\PP$ with respect to the strict dual $\NN$} is defined to be 
\begin{equation*}\label{eq: def dual Pcheck}
\PP^\vee := \{n \in \NN_\R \, \mid \, \psi_\PP(n) \geq -1\} \subset \NN_\R.
\end{equation*}
\exampleqed
\end{definition}

The following lemma justifies the terminology as one can compare it with the definition of the dual of a polytope in the classical sense \cite[\textsection 2.2]{Cox_Little_Schenck}.

 \begin{lemma}\label{lemma: first properties of dual}
 In the setting of \Cref{definition: dual convex set}, the set $\PP^\vee$ is a PL polytope in $\NN_\R$. Specifically, $\PP^\vee$ can be expressed as 
 \begin{equation*}\label{eq: Pvee as N polytope}
  \PP^\vee = \bigcap_{m \in V(\PP)} \HH_{\v(m),-1}
 \end{equation*}
 and $\PP^\vee$ is compact in $\NN_\R$. 
\end{lemma}

\begin{proof} 
We compute 
\begin{equation*}
\begin{split}
\PP^\vee & := \{n \in \NN_{\R} \, \mid \, \psi_{\PP}(n) \geq -1\} \\ 
& = \{n \in \NN_{\R} \, \mid \, \bigoplus_{m\in V(\PP)} \v(m)(n)  \geq -1\} \, \textup{ by \Cref{lem_psidelta_finite}} \\
& = \bigcap_{m \in V(\PP)} \{n \in \NN_{\R} \mid \v(m)(n) \geq -1\} \\
& = \bigcap_{m \in V(\PP)} \HH_{\v(m), -1} 
\end{split}
\end{equation*}
and the last equality is a finite intersection since we know (cf.~\Cref{remark: vertices finite}) that $V(\PP)$ is finite.

Next we show that $\mathcal{P}^\vee$ is compact. 
Since $\mathcal{P}^\vee$ is an intersection of PL half-spaces, it remains to show that $\PP^\vee$ is bounded. Suppose for a contradiction that it is unbounded. Being an intersection of PL half-spaces, this implies that there exists some positive ray that lies entirely in $\mathcal{P}^\vee$, i.e. there exists $n \neq 0, n \in \NN_\R$, such that $n \in \mathcal{P}^\vee$ and also $\lambda n \in \mathcal{P}^\vee$ for any $\lambda \in \R_{\geq 0}$. This means that $\v( m)( \lambda n)  \geq -1$ for all $m \in V(\PP)$ and $\lambda \in \R_{> 0}$.  Dividing, we see that $\v( m)( n)  \geq -\frac{1}{\lambda}$ for all $\lambda \in \R_{> 0}$ and all $m \in V(\PP)$, which in turn implies $\v( m)( n)  \geq 0$ for all $m \in V(\PP)$. Since $V(\PP)$ are the vertices of $\PP$, if $\v( m)( n)  \geq 0$ for $m \in V(\PP)$ then $\v( m)( n)=\w(n)(m)  \geq 0$ for all $m \in \PP$, by Proposition~\ref{poly_is_conv_verts}. This contradicts Lemma~\ref{lemma: nonzero point negative} since $n\neq 0$ implies $\w(n)\neq 0$.
\end{proof}

We also observe the following. 

\begin{lemma}
    Let $\MM$ be a finite polyptych lattice over $F$ and let $(\MM, \NN, \v, \w)$ be a strict dual $F$-pair. Let $S = \{m_1,\cdots,m_\ell\} \subset \MM$ be a finite subset of $\MM$. Suppose that $\PP := \cap_{i=1}^{\ell} \mathcal{H}_{\v(m_i),-1}$ is full-dimensional, compact, and contains $0_\NN$ in its interior. Then $\ptconv(S) = \PP^\vee$. 
\end{lemma}

\begin{proof} 
We first claim $\ptconv(S) \subset \PP^\vee$. Since $\PP^\vee$ is point-convex, to see the inclusion it suffices to show $m_i \in \PP^\vee$ for each $m_i \in S$. By definition of $\PP^\vee$ it suffices to show $\psi_\PP(m_i) = \min\{\w(u)(m_i) \mid u \in \PP\} = \min\{\v(m_i)(u) \mid u \in \PP\} \geq -1$ where the first equality is the definition of $\psi_\PP$ and the second equality follows from the axioms of strict duality. But the last equality holds by definition of $\PP$. 

To complete the proof we need to show the opposite inclusion, $\PP^\vee \subset \ptconv(S)$. Let $m \in \PP^\vee$. By definition of $\PP^\vee$ and $\psi_\PP$, this means $\w(u)(m) =\v(m)(u) \geq -1$ for all $u \in \PP$. By \Cref{lem-pconvex-1} we know $\ptconv(S) = \{m \in \MM_\R \mid \forall p \in \Sp_\R(\MM), p(m) \geq \min\{p(m_i)\mid i \in [\ell]\}\}$. Using strict duality it therefore suffices to show $\w(n)(m) \geq \min_{i \in [\ell]}\{\w(n)(m_i)\}$ for all $n \in \NN_\R$. Suppose not, i.e. there exists nonzero $n_0 \in \NN_\R$ with $\v(m_i)(n_0)=\w(n_0)(m_i) > \v(m)(n_0)$ for all $i$. Note that since $0_\NN$ is in the interior of $\PP$ by assumption, it follows that for any nonzero $n \in \NN_\R$, we must have $\min_{i \in [\ell]}\{\v(m_i)(n)\} < 0$. Thus $\v(m)(n_0) < 0$, which means we may find some scalar $\lambda>0$ such that $\v(m_i)(\lambda n_0) \geq -1 >  \v(m)(\lambda n_0)$ for all $i \in [\ell]$. This implies $\lambda n_0 \in \PP$ but then $m \not \in \PP^\vee$. This is a contradiction. Hence $m \in \ptconv(S)$ as desired. 
\end{proof}

\begin{example}\label{ex_dual_polytope} 
We now compute the dual PL polytope $\PP^\vee$ to the PL polytope $\PP$ from \Cref{ex_M_polytope} utilizing the dual pairing described in \Cref{ex_dual}.
(For this example, $\MM$ is self-dual, so $\MM=\NN$. However, for clarity, we will use $\NN$ when we are referring to the objects needed to define $\PP^\vee$.)
By \Cref{lemma: first properties of dual}, we have $\PP^\vee=\bigcap_{m\in V(\PP)} \HH_{\v(m), -1}$, where $\HH_{\v(m), -1}=\{n\in\NN_\R\mid \v(m)(n)\ge -1\}$.
Recalling $V(\PP)$ from \eqref{eq_ex_vertices} and using \eqref{eqeq_dual_pairing}, we have that for $i=1,2$, $\pi_\alpha(\PP^\vee)$ is given by the following conditions:
\begin{center}\renewcommand{\arraystretch}{1.2}
    \begin{tabular}{|r|r|r|}
    \hline
    $m\in V(\PP)$ & $\pi_1(\HH_{\v(m), -1})$ & $\pi_2(\HH_{\v(m), -1})$ \\
    \hline \hline
        $(\varepsilon_1,-\varepsilon_1)$ & $y\ge-1$ & $y\ge-1$ 
        \\ \hline
        $(\varepsilon_1+2\varepsilon_2,-\varepsilon_1+2\varepsilon_2)$ & $2x+y\ge-1$ & $-2x'+\min\{y,3y\}\ge-1$ 
        \\ \hline
       $(-2\varepsilon_1-\varepsilon_2,\varepsilon_1-\varepsilon_2)$ & $-x+\min\{-y,-2y\}\ge-1$ & $x'-2y\ge-1$ 
        \\ \hline
        $(-\varepsilon_1,\varepsilon_1)$ & $-y\ge-1$ & $-y\ge-1$
        \\ \hline
       $(-\varepsilon_2,-\varepsilon_1-\varepsilon_2)$ & $-x+\min\{0,y\}\ge-1$ & $x'\ge-1$ 
        \\ \hline
    \end{tabular}
.\end{center}
We depict the resulting chart images in \Cref{fig_dual_polytope}.
Notice that not all chart images of $\PP^\vee$ are integral polytopes in the classical sense.
\begin{figure}[h]
    \centering
    \begin{tikzpicture}
    \filldraw[fill=gray!50] (.5*0,-.5*1)--(.5*1,.5*0)--(-.5*1,.5*1)--(.5*0,-.5*1);
    \draw[gray,<->] (-1.5,0)--(1.5,0);
    \draw[gray,<->] (0,-1)--(0,1);
    % loop over the lattice points
    \foreach \i in {-2,...,2}
      \foreach \j in {-2,...,2}{
        \filldraw[black] (.5*\i,.5*\j) circle(.5pt);
      };
    \end{tikzpicture}
    \hspace{5cm}
    \begin{tikzpicture}
    \filldraw[fill=gray!50] (-.5*1,-.5*1)--(.5*.5,0)--(.5*1,.5*1)--(-.5*1,0)--(-.5*1,-.5*1)
    ;
    \draw[gray,<->] (-1,0)--(1,0);
    \draw[gray,<->] (0,-1)--(0,1);
    % loop over the lattice points
    \foreach \i in {-2,...,2}
      \foreach \j in {-2,...,2}{
        \filldraw[black] (.5*\i,.5*\j) circle(.5pt);
      };
    \end{tikzpicture}
    \caption{The two chart images of the PL polytope $\PP^\vee$ in $\NN_\R$ from \Cref{ex_dual_polytope}. On the left is $\pi_1(\PP^\vee)$ and on the right is $\pi_2(\PP^\vee)$. Note that $\pi_2(\PP^\vee)$ is not an integral polytope in the classical sense.}
    \label{fig_dual_polytope}
\end{figure}
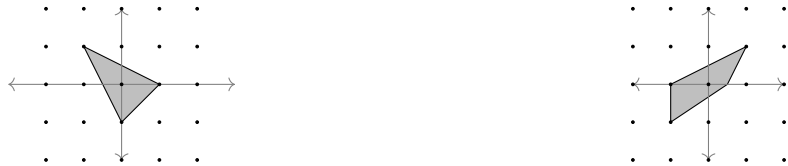
\exampleqed
\end{example}

For the rest of this section, we assume that $\MM$ is a polyptych lattice over $\Z$. 
\Cref{ex_dual_polytope} suggests that the PL analogue of the classical notion of an integral polytope may be somewhat subtle. 
We make the following definition. 

\begin{definition} 
Let $\MM$ be a finite polyptych lattice over $\Z$ and 
let $\PP \subset \MM_\R$ be a PL polytope over $\Z$. We say that $\PP$ is an \textbf{integral PL polytope} if $\pi_\alpha(\PP)$ is an integral polytope in $M_\alpha \otimes \R$ (i.e., all its vertices are in $M_\alpha$) for every $\alpha \in \pi(\MM)$.
\exampleqed
\end{definition} 

Thus, for example, $\PP^\vee$ in \Cref{fig_dual_polytope} is not an integral PL polytope, whereas the PL polytope $\PP$ in \Cref{fig_polytopes} is integral. The following is straightforward; we use it in \Cref{sec-Example}.  

\begin{lemma}\label{lem_lattice_intercrit}
Let $\MM$ be a finite polyptych lattice over $\Z$. 
    Let $\PP$ be a PL polytope in $\MM_\R$. If $\pi_\alpha(\PP)\cap\pi_\alpha(\mathcal{C})$ is an integral polytope for all $\alpha\in\pi(\MM)$ and all maximal-dimensional cones $\mathcal{C}\in\Sigma(\MM)$, then $\PP$ is integral.
\end{lemma}

\begin{proof} 
The image $\pi_\alpha(\PP)$ (respectively $\pi_\alpha(\mathcal{C})$) is a classical polytope (respectively polyhedral cone) in $M_\alpha \otimes \R$ so it suffices to check that if $v$ is a vertex of a polytope $P$ in a vector space $V$ and $\mathcal{C}$ is a full-dimensional polyhedral cone with $v \in \mathcal{C}$, then $v$ is a vertex of $P \cap \mathcal{C}$. A characterization of vertices of a polytope is that there exists a linear function $L$ on $V$ such that $\{v\} = \{x \in P: L(x) \geq L(y), \, \textup{ for all } \, y \in P\}$. Using this characterization it is straightforward to see that $v$ is also a vertex of $P \cap \mathcal{C}$. 
\end{proof}

We also define the following PL analogue of a Gorenstein-Fano polytope. 

\begin{definition}\label{definition: PL Gorenstein Fano} 
Let $\MM$ be a finite polyptych lattice over $\Z$. We say that a PL polytope $\PP \subset \MM_\R$ over $\Z$ is \textbf{chart-Gorenstein-Fano} if $\PP$ is an integral PL polytope, and, its PL half-space representation is of the form 
$$
\PP = \bigcap_{i=1}^{\ell} \HH_{p_i, -1}
$$
where $p_i \in \Sp(\MM)$ and $a_i=-1$ for all $i \in [\ell]$. 
\exampleqed
\end{definition}

The following lemma gives motivation for the terminology, and, indicates why a study of chart-Gorenstein-Fano PL polytopes would be of interest. Recall that a \emph{classical} polytope $P$ in $M \otimes \R$ (for a classical lattice $M \cong \Z^r$) is said to be Gorenstein-Fano\footnote{In \cite{Cox_Little_Schenck} these are called \textbf{reflexive polytopes}.}  if it is full-dimensional, integral, and has a facet presentation as $\cap_{F} H_{u_F,-1}$ where $F$ ranges over facets of $P$, $u_F$ are elements of the dual lattice $\mathrm{Hom}(M,\Z)$, and $H_{u_F,-1} := \{m \in M \otimes \R \, \mid \, u_F(m) \geq -1\} \subset M \otimes \R$.

\begin{lemma} 
Let $\MM$ be a finite polyptych lattice over $\Z$ and let $\PP \subset \MM_\R$ be a full-dimensional PL polytope over $\Z$. Then $\PP$ is chart-Gorenstein-Fano if and only if, for any $\alpha \in \pi(\MM)$, the chart image $\pi_\alpha(\PP) \subset M_\alpha \otimes \R$ is Gorenstein-Fano in the classical sense. 
\end{lemma} 

\begin{proof} 
We begin by showing the ``only if'' direction. Suppose $\PP$ is chart-Gorenstein-Fano. 
Let $\alpha \in \pi(\MM)$. 
It suffices to show that $\pi_\alpha(\PP)$ is an integral polytope with a half-space representation $\cap_{i=1}^N H_{u_i, -1}$ for some $N$ and $u_i$ for $i \in [N]$, where the $u_i$ are elements of the lattice dual to $M_\alpha$. Firstly, the fact that $\pi_\alpha(\PP)$ is integral follows directly from the definition of chart-Gorenstein-Fano. Secondly, since $\pi_\alpha$ is a bijection, the chart image $\pi_\alpha(\PP)$ is the intersection $\cap_{i=1}^\ell \{ v \in M_\alpha \otimes \R \, \mid \, (p_i \circ \pi_\alpha^{-1})(v) \geq -1\}$. We know from \Cref{lemma: points linear on faces of SigmaM} that $p_i$ is piecewise linear and from \Cref{lemma: point PL concave} that $p_i$ is convex. Thus $p_i \circ \pi_\alpha^{-1}$ is a finite min-combination of $\Z$-linear functions on $M_\alpha$. In other words, there exists an integer $\ell_{i,\alpha}>0$ and $u_{i,\alpha,j}$ for $1 \leq j \leq \ell_{i,\alpha}$ in the dual lattice to $M_\alpha$ such that $\pi_\alpha(\mathcal{H}_{p_i,-1}) = \{ v \in M_\alpha \otimes \R \, \mid \, (p_i \circ \pi_\alpha^{-1})(v) \geq -1\} = \cap_{j=1}^{\ell_{i,\alpha}} H_{u_j, -1}$, where $H_{u_j,-1}$ is the classical half-space in $M_\alpha \otimes \R$ determined by $u_j$ and the parameter $-1$. Note that the parameter remains unchanged since $p_i \circ \pi_\alpha^{-1} = \min_{j\in [\ell_{i,\alpha}]} u_{i,\alpha,j}$ is required to be $\geq -1$, so each $u_{i,\alpha,j}$ must be $\geq -1$. Assembling these $u_{i,\alpha,j}$ over all $i$ we obtain that $\pi_\alpha(\PP) = \cap_{i=1}^{\ell} \cap_{j =1}^{\ell_{i,\alpha}} H_{u_{i,\alpha,j},-1}$. This is a classical Gorenstein-Fano polytope. 

Now we show the ``if'' direction. Suppose $\PP$ has the property that each chart image is Gorenstein-Fano. By definition of PL polytopes we know that $\PP$ may be written as $\cap_{i=1}^\ell \mathcal{H}_{p_i, a_i}$ for some $p_i \in \Sp(\MM)$ and $a_i \in \Z$, and we may assume that this representation is minimal. It would suffice to see that each $a_i$ is strictly negative, and, that $\frac{1}{\lvert a_i \rvert}p_i$ is an element of $\Sp(\MM)$. The fact that each $\pi_\alpha(\PP)$ is classically Gorenstein-Fano implies that $0 \in M_\alpha \otimes \R$ is contained in the interior of $\pi_\alpha(\PP)$ for all $\alpha$, which implies $0_\MM$ is contained in the interior of $\PP$. This is only possible if $a_i<0$ for all $i$. Now we claim $\frac{1}{\lvert a_i\rvert} p_i \in \Sp(\MM)$. Note that it is immediate that $\frac{1}{\lvert a_i\rvert} p_i$ is in $\Sp_\R(\MM)$, since the conditions~\eqref{eq: def point min} and~\eqref{eq: def point F homog} remain true after scalar multiplication. Hence we only need to show that $\frac{1}{\lvert a_i\rvert}p_i$ maps elements in $\MM$ to $\Z$.  In fact, to see this, it would suffice to show that there exists a full-dimensional cone $\mathcal{C}$ in $\Sigma(\MM)$ such that $\frac{1}{\lvert a_i\rvert}p_i$ maps elements in $\mathcal{C} \cap \MM$ to $\Z$. This is because the proof of \Cref{lem-convex-11} shows that if a point in $\Sp_\R(\MM)$ takes values in $\Z$ on $\mathcal{C} \cap \MM$ for a full-dimensional $\mathcal{C}$ in $\Sigma(\MM)$, then it takes values in $\Z$ in all of $\MM$.  By minimality of the representation $\PP = \cap_{i=1}^{\ell} \mathcal{H}_{p_i,a_i}$, there must exist some full-dimensional $\mathcal{C}$ on which $\mathcal{H}_{p_i,a_i}$ defines a codim-$1$ facet of $\PP \cap \mathcal{C}$ and hence of $\pi_\alpha(\PP)$ for any $\alpha$. We know $p_i$ is linear on $\mathcal{C}$ and by hypothesis, $\pi_\alpha(\PP)$ is Gorenstein-Fano, so we conclude $\frac{1}{\lvert a_i\rvert}p_i$ is integral on $\mathcal{C}$, as desired. 
\end{proof} 

We refer the reader to \cite{CookEscobarHaradaManon2024} for further examples of chart-Gorenstein-Fano PL polytopes.

%-------------------------------------%%%%%%%%%%%%%%
\section{Detropicalizations of polyptych lattices}\label{sec_detrop} 

In this section, we make the connection between the combinatorial data introduced in the previous sections and (algebraic) geometry. The geometry is then further developed in \Cref{section: compactifications}, where we build compact projective varieties corresponding to a choice of PL polytope $\PP \subset \MM_\R$, in a manner similar to the construction of a compactification of the torus $(\C^*)^n$ arising from a choice of classical polytope. Building this compactification in \Cref{section: compactifications}, however, requires a choice of an algebra, denoted $\Aa_\MM$, which fills the role played by the Laurent polynomial ring in the classical toric geometry situation. This algebra $\Aa_\MM$ (and a choice of valuation with domain $\Aa_\MM$) is the topic of this section.

We assume throughout this section that \textbf{$\K$ is an algebraically closed field}. 
We begin by recalling the definition of a quasivaluation $\fv: \Aa \to \mathcal{S}$, with values in an idempotent semialgebra. The concept is an analogue of a classical discrete valuation defined on a field.  We have the following (see e.g.~\cite{KavehManon-PL}).

\begin{definition}\label{def_valuation}
Let $\Aa$ be a Noetherian $\K$-algebra which is an integral domain. Let $(\mathcal{S}, \odot, \oplus)$ be an idempotent semialgebra.  
    We say a map $\fv: \Aa \to \mathcal{S}$ is a \textbf{quasivaluation with values in $\mathcal{S}$} if we have:
\begin{enumerate}
    \item[(1)] $\fv(fg) \ge \fv(f)\odot\fv(g)$, for all $f,g \in \Aa$,
    \item[(2)] $\fv(f+g)\ge \fv(f)\oplus\fv(g)$, for all $f,g \in \Aa$,
    \item[(3)] $\fv(cf)=\fv(f)$, for all $c\in \K^*$ and $f \in \Aa$, and 
    \item[(4)] $\fv(0)=\infty$, and, $0$ is the only element in $\Aa$ that maps to $\infty$. 
\end{enumerate}
We say that $\fv$ is a \textbf{valuation with values in $\mathcal{S}$} if the inequality in axiom (1) is in fact an equality, i.e., $\fv(fg)=\fv(f) \odot \fv(g)$. 
\exampleqed 
\end{definition}

There are several particularly significant cases of idempotent semialgebras which appear as the co-domains of our (quasi)valuations in this paper. Firstly, the canonical semialgebras $S_\MM$ of \Cref{definition: canonical semialgebra MM} appear as codomains of valuations in our theory of detropicalizations, begun in \Cref{def_detrop} below. Secondly, in the strictly dualizable setting, the related case of piecewise-linear functions $\O_\NN$ as discussed in \Cref{example: OMM as semialgebra}, or its point subsemialgebra $P_\NN$ of \Cref{definition: point algebra}, also occurs -- due to the identification in \Cref{prop-salgebras}. Thirdly, the $\Z^r$-valued weight (quasi)valuations in e.g.\ \cite{KavehManon-Siaga, EscobarHarada} are examples where the codomain is a totally ordered abelian group (viewed as an idempotent semialgebra as explained in \Cref{ex: Malpha as semialgebra}), and we will see similar examples below.

\begin{remark}\label{remark: giansiracusa}
The notion of (quasi)valuations with values in an idempotent semialgebra is certainly not new. To cite just two examples, it plays a prominent role in the work of Giansiracusa--Giansiracusa \cite{Giansiracusa} on tropical algebraic geometry, and, the idea is used by Kaveh and the third author in \cite{KavehManon-MathZ, KavehManon-PL, KavehManon-PL-part2} to study toric vector bundles, toric principal bundles, and toric flat families. One distinction between previous uses of these valuations and our use of them in this paper is that, in the tropical context such as in \cite{Giansiracusa}, one generally starts with the algebra $\Aa$ (often thought of as the coordinate ring of a variety $X$), and then asks about valuations from $\Aa$ to different targets $\mathcal{S}$. In our setting, the point of view is the opposite, as will become clear in \Cref{def_detrop}; we start with a fixed $\MM$ and its canonical idempotent semialgebra $S_\MM$, and then ask for choices of algebras $\Aa_\MM$, and of valuations on $\Aa_\MM$, with target $S_\MM$. 
\exampleqed
\end{remark}

In the case when the codomain of a (quasi)valuation is a totally ordered abelian group $\overline{\Gamma}=\Gamma \cup \{\infty\}$, as in the third class of examples mentioned above, we can construct corresponding objects which are useful in our analysis. 
We keep this discussion brief and refer the reader to e.g.~\cite{KavehManon-Siaga} for details. The following discussion is valid for any quasivaluation, not just valuations, so we present it in this generality. Suppose $\Aa$ is a $\K$-algebra equipped with a quasivaluation $\fv: \Aa \to \overline{\Gamma}$. Given $i \in \Gamma$ we define the $\K$-vector space
\begin{equation*}\label{eq: def valuation filtration}
    F_i := F_{\fv \geq i} := \{f \in \Aa \, \mid \, \fv(f) \geq i\}
\end{equation*}
and denote by $\mathcal{F}_{\fv} := \{F_{\fv \geq i}\}_{i \in \Gamma}$ the filtration defined by the $F_{\fv \geq i}$. 
This is a decreasing filtration, namely if $j \leq i$ then $F_j \supseteq F_i$. It follows immediately from the definitions that $\mathcal{F}_{\fv}$ is multiplicative, i.e. $F_{\fv \geq i} \cdot F_{\fv \geq j} \subset F_{\fv \geq i+j}$ for all $i,j \in \Gamma$. 
 We define the \textbf{associated graded algebra} $\mathrm{gr}_{\fv}(\Aa) := \mathrm{gr}_{\mathcal{F}_{\fv}}(\Aa)$ of $\Aa$ with respect to $\fv$ by 
 \begin{equation*}
     \mathrm{gr}_{\fv}(\Aa) := \bigoplus_{i \in \Gamma} F_{\fv \geq i}/F_{\fv > i}
 \end{equation*}
 where $F_{\fv > i} := \cup_{\ell > i} F_{\fv \geq \ell}$ and the multiplication on $\mathrm{gr}_{\fv}(\Aa)$ is defined by 
 \begin{equation*}
     F_{\fv \geq i}/F_{\fv > i} \times F_{\fv \geq j}/F_{\fv > j} \to F_{\fv \geq i+j}/F_{\fv > i+j},\quad ([g],[h]), \mapsto [gh]
 \end{equation*}
and then extending linearly. 
We say that $\fv$ has \textbf{one-dimensional leaves} if each quotient $F_{\fv \geq i}/F_{\fv > i}$ has at most dimension $1$ as a $\K$-vector space.  
We also recall that the \textbf{value semigroup} of $\fv$, denoted $S(\Aa,\fv)$, is the image in $\Gamma$ of $\Aa\setminus\{0\}$ under $\fv$. The \textbf{rank} of a valuation, denoted $\mathrm{rank}(\fv)$, is the rank of its value semigroup, i.e., the rank of the abelian subgroup generated by $S(\Aa,\fv)$ in ${\Gamma}$. We say that $\fv$ has \textbf{full rank}, or that $\fv$ is a \textbf{full rank valuation}, if $\mathrm{rank}(\fv) = \dim(\Aa)$.

We now come to a key definition -- that of a \textbf{detropicalization} of a polyptych lattice. 

\begin{definition}\label{def_detrop}
Let $\MM$ be a finite polyptych lattice of rank $r$ over $F$ with associated canonical idempotent semialgebra $S_\MM$. Let $\Aa_\MM$ be a Noetherian $\mathbb{K}$-algebra which is an integral domain and $\fv: \Aa_\MM \to S_\MM$ a valuation with values in $S_\MM$. We say that the pair $(\Aa_\MM, \fv)$ is a \textbf{detropicalization of $\MM$} if every element of $\MM$ is in the image of $\fv$, and the Krull dimension of $\Aa_\MM$ equals the rank $r$ of $\MM$. \exampleqed
\end{definition}

The motivation for the terminology comes from the historical context described in \Cref{remark: giansiracusa}. 
As before, in the case of the trivial polyptych lattice, we recover a familiar object from toric geometry. 

\begin{example}\label{example: detrop trivial PL}
Let $\MM$ be the trivial polyptych lattice of rank $r$ over $\Z$. Recall from \Cref{example: canonical semialgebra of trivial PL} that in this case, the canonical semialgebra $S_\MM$ is the semialgebra of integral polytopes. Then the reader may check that we may choose a detropicalization of $\MM$ to be the pair of $\Aa_\MM = \K[x_1^{\pm}, x_2^{\pm}, \cdots, x_r^{\pm}]$, the Laurent polynomial ring in $r$ variables, equipped with the valuation $\fv: \K[x_1^{\pm}, \cdots, x_r^{\pm}] \to S_\MM$ which takes a Laurent polynomial $f$ to its Newton polytope $\mathrm{Newt}(f)$. 
\exampleqed
\end{example}

We make the following definition,
which 
 is the PL analogue of the notion of ``adapted basis'' from \cite[Definition 2.27]{KavehManon-Siaga}.

\begin{definition}\label{def_CAB}
Let $\MM$ be a finite polyptych lattice over $F$. Let $(\Aa_\MM,\fv: \Aa_\MM \to S_\MM)$ be a detropicalization of $\MM$. We say that a $\K$-vector space basis $\B$ of $\Aa_\MM$ is a \textbf{convex adapted basis} for $\fv: \Aa_\MM \to S_\MM$ if 
\begin{enumerate} 
\item[(1)] $\fv(\sum \lambda_i \bb_i) = \bigoplus_i \fv(\bb_i)$, for any finite collection $\lambda_i \in \K^*$ and $\bb_i \in \B$, and
\item[(2)] $\fv(\bb) \in \MM \subset S_\MM$ for all $\bb \in \B$.\exampleqed
\end{enumerate}
\end{definition}

\begin{remark} 
From the assumption (2) on (quasi)valuations in \Cref{def_valuation}, it follows that for \textit{any} basis of $\Aa_\MM$ we always have $\fv(\sum_i \lambda_i \bb_i) \geq \bigoplus_i \fv(\bb_i)$. The strength of the assumption (1) in \Cref{def_CAB} is that the inequality is in fact an equality. We will comment on the power of the other assumption (2) in \Cref{remark: CAB axiom}. 
\exampleqed
\end{remark}

In order to obtain sharper results about detropicalizations, it will be useful to restrict to the case when $\MM$ has a strict dual.   \textbf{Thus, for the remainder of this section, we let $(\MM, \NN, \v, \w)$ be a fixed choice of strict dual pair of polyptych lattices, so in particular, we assume that $\MM$ is strictly dualizable.}  
In this setting, a detropicalization $(\Aa_\MM, \fv: \Aa_\MM \to S_\MM)$ may be viewed by \Cref{prop-salgebras} as a map $\fv: \Aa_\MM \to \P_\NN$, and by the properties of the isomorphism $S_\MM \to \P_\NN$ (see \Cref{prop-salgebras}), a convex adapted basis $\B$ has the properties that $\fv(\sum \lambda_i \bb_i) = \min_i\{\fv(\bb_i)\}$ and $\fv(\bb) \in \Sp(\NN)$ for all $\bb \in \B$. We frequently take this point of view in the discussion that follows.

\begin{example}\label{example: detrop running example}
We continue with our running example. Namely, consider the (strictly) self-dual polyptych lattice in \Cref{ex_dual}. 
Let $\Aa=\C[x_1,x_2,t^{\pm}]/\langle x_1x_2-1-t \rangle$. 
In this example we define a valuation $\fv$ so that $(\Aa,\fv)$ is a detropicalization of $\MM$. We also give a convex adapted basis. 
We keep the details brief, since the reader can find the full example in \cite{CookEscobarHaradaManon2024}.
We also remark that $\Aa$ is a cluster algebra of type $A_1$ with one frozen variable, i.e.~the cluster algebra associated to the quiver with 2 vertices: one corresponding to a cluster variable and the other corresponding to a frozen variable,  see e.g.~\cite[\textsection 2.7 and \textsection 3.1]{FominWilliamsZelevinsky}.

We define the valuation $\fv$ by specifying an additive basis $\B$ of $\Aa$, defining $\fv$ on the basis elements, and then extending the definition to all of $\Aa$.
The basis is given by 
    \begin{equation*}
        \B=\{\overline{x_1^{u_1}x_2^{u_2}t^w}\mid u_1,u_2\in \Z_{\ge 0},\ w\in\Z,\ \min\{u_1,u_2\}=0\}
        \subset \Aa
    .\end{equation*}
We define $\fv$ on $\B$ using the self-dual pairing $\w$, as follows: 
\begin{equation*}
        \fv:\B\to\Sp(\NN),\qquad \fv(\overline{x_1^{u_1}x_2^{u_2}t^w}):=
        \w((w,u_2-u_1),(-w-u_1,u_2-u_1)) \in \Sp(\NN).
\end{equation*}
We have now specified the piecewise linear function $\fv(\mathbb{b}) \in \Sp(\NN) \subset \P_\NN$ for each basis element $\mathbb{b}$ in $\B$. Since $\B$ forms an additive basis of $\Aa$, any element in $\Aa$ can be written uniquely as $\sum_i \lambda_i \mathbb{b}_i$ with all $\lambda_i\neq 0$. We then define the valuation $\fv: \Aa \to \P_\NN$ by 
\begin{equation*}
        \fv:\Aa\to\P_\NN,\qquad
        \fv\left(\sum \lambda_i \mathbb{b}_i\right) := \bigoplus_i \fv(\mathbb{b}_i), \, \; \quad \textup{ and} \, \, \fv(0) := \infty 
,\end{equation*}
where the $\bigoplus$ denotes the min-operation in $\P_\NN$.
The pair $(\Aa,\fv)$ is a detropicalization of $\MM$ and $\B$ is a convex adapted basis for $\fv$, as is shown in detail in \cite{CookEscobarHaradaManon2024}. \exampleqed
\end{example}

Our first result holds in the presence of strict duals; it relates valuations with values in $P_\NN$ with valuations with values in a totally ordered abelian group such as $\overline{M_\alpha}$ (as in \Cref{ex: Malpha as semialgebra}), thus providing a link between the semialgebra-valued valuations of \Cref{def_valuation}, and the more classical valuations taking values in totally ordered groups.

\begin{proposition}\label{lem-canonicaltototal}
Let $\MM$ be a finite polyptych lattice of rank $r$ over $F$ and $(\MM,\NN,\v,\w)$ a strict dual $F$-pair. 
Let $\Aa$ be a Noetherian $\K$-algebra which is an integral domain and let $\fv: \Aa \to \P_\NN$ be a valuation with values in $\P_\NN$ in the sense of \Cref{def_valuation}. Let $\alpha \in \pi(\MM)$ be a choice of coordinate chart of $\MM$ and $\tilde{\rho} \subset C_\alpha$ be a choice of ordered basis for the cone $C_\alpha := \w^{-1}(\Sp_\R(\MM,\alpha)) \subset \NN_\R$. Then there exists a total order on $M_\alpha$ and a corresponding idempotent semialgebra $\overline{M_\alpha}$ (as described in \Cref{ex: Malpha as semialgebra}), such that the following hold. 
\begin{enumerate} 
\item[(1)] Associated to $\fv$ and the choice of $\alpha$ and $\tilde{\rho}$, there exists a valuation $\fv_{\alpha, \tilde{\rho}}: \Aa \to \overline{M}_{\alpha}$.
\item[(2)] If $(\Aa = \Aa_\MM,\fv: \Aa_\MM \to \P_\NN)$ is, in addition, a detropicalization in the sense of \Cref{def_detrop}, then $\fv_{\alpha, \tilde{\rho}}$ has full rank. 
\end{enumerate} 
\end{proposition}

The following general lemma is useful for the proof of \Cref{lem-canonicaltototal}. The proof is straightforward and uses the definition of the partial order on an idempotent semialgebra given in \Cref{remark: partial order of idempotent semialgebra}.

\begin{lemma}\label{lemma: semialg hom and valuations}
Let the notation be as in \Cref{def_valuation}. Suppose that $\fv: \Aa \to \mathcal{S}$ is a valuation with values in $\mathcal{S}$ and $F: \mathcal{S} \to \mathcal{S}'$ is a homomorphism of idempotent semialgebras. Then the composition $F \circ \fv$ is a valuation with values in $\mathcal{S}'$. 
\end{lemma}

We can now prove \Cref{lem-canonicaltototal}.

\begin{proof}[Proof of \Cref{lem-canonicaltototal}]
We begin with (1). We will construct $\fv_{\alpha, \tilde{\rho}}$ by composing $\fv$ with a sequence of idempotent semialgebra maps. By \Cref{lemma: semialg hom and valuations} the result will be a valuation. We first need to set some notation. Recall from \Cref{lemma: points linear on faces of SigmaM} that for $p \in \Sp(\NN)$, its restriction $p \vert_{C_\alpha \cap \NN}$ is linear on $C_\alpha$. As in \Cref{prop-evaluationmap} we denote this restriction by $L_{C_\alpha}: \Sp(\NN) \to \Hom(C_\alpha \cap \NN,F)$. 
Since $\tilde{\rho} :=\{\rho_1,\cdots,\rho_r\} \subset C_\alpha$ is a basis contained in $C_\alpha$ we may consider $C_{\tilde{\rho}} := \mathrm{span}_{\R_{\geq 0}}\{\rho_1,\cdots,\rho_r\} \subset C_\alpha$ and restrict further to $C_{\tilde{\rho}} \cap \NN$. 
We denote this by $\iota_{\tilde{\rho}}: \Hom(C_\alpha \cap \NN, F) \to \Hom(C_{\tilde{\rho}}\cap \NN, F)$; since $\tilde{\rho}$ is a basis, $\iota_{\tilde{\rho}}$ is an isomorphism. Finally, using $\tilde{\rho}$ to identify $C_{\tilde{\rho}}$ with $F^r_{\geq 0}$, we also have the isomorphism $\phi_{\tilde{\rho}}: \Hom(C_{\tilde{\rho}} \cap \NN, F) \to \Hom(F^r_{\geq 0}, F)$. Note that each of these three $\Hom$-spaces can be identified with $F^r$ by evaluating against the basis $\tilde{\rho}$ or with the standard basis, denoted $\overline{\varepsilon}$, of $F^r$, and all relevant diagrams commute. 

We can now construct a sequence of semialgebra maps. 
Given $C\subset \NN$, let $\O^\oplus_C$ denote the algebra consisting of the piecewise linear functions that are finite min-combinations of elements of $\Hom(C\cap\NN,F)$.
Similarly, $\O^\oplus_{F^r_{\geq 0}}$ denotes the algebra consisting of the piecewise linear functions that are finite min-combinations of elements of $\Hom(F^r_{\geq 0},F)$.
Note that all the above maps extend naturally to algebra maps which, by abuse of notation,  denote as $L_{C_\alpha}: P_\NN \to \O^{\oplus}_{C_\alpha}, \iota_{\tilde{\rho}}: \O^{\oplus}_{C_\alpha} \to \O^{\oplus}_{C_{\tilde{\rho}}}, \phi_{\tilde{\rho}}: \O^{\oplus}_{C_{\tilde{\rho}}} \to \O^{\oplus}_{F^r_{\geq 0}}$. It is straightforward that these are morphisms of idempotent semialgebras.
Let $\sf Im$ denote the image of $P_\NN$ under the composition $\phi_{\tilde{\rho}} \circ \iota_{\tilde{\rho}} \circ L_{C_\alpha}$.
Now we define a map $\Phi: {\sf Im} \to \overline{F^r}$ and then prove that it is a semialgebra morphism, where $F^r$ is equipped with the lex order and $\overline{F^r}$ is the associated idempotent semialgebra as described in \Cref{ex: Malpha as semialgebra}. To define $\Phi$, note that any element ${\sf f}\in \sf Im$ can be written uniquely as ${\sf f} = \oplus_{p \in S} \phi_{\tilde{\rho}} \circ \iota_{\tilde{\rho}} \circ L_{C_\alpha}(p)$ for some finite minimal set $S \subset \Sp(\NN)$. 
Let $\mathrm{ev}_{\overline{\varepsilon}}: \Hom(F^r_{\geq 0}, F) \to F^r$ denote the isomorphism mentioned above, given by evaluation against the standard basis vectors of $F^r_{\geq 0}$. We may now define 
$$
\Phi: {\sf Im} \to \overline{F^r},\qquad
\Phi({\sf f}) := \min_{p \in S}\{ \mathrm{ev}_{\overline{\varepsilon}} \circ \phi_{\tilde{\rho}} \circ \iota_{\tilde{\rho}} \circ L_{C_\alpha}(p)\}
,$$
where the minimum is with respect to the lex order on $F^r$. We now claim that $\Phi$ is a semialgebra morphism. To prove this, we must first show that for ${\sf f} = \oplus_{p \in S} \phi_{\tilde{\rho}} \circ \iota_{\tilde{\rho}} \circ L_{C_\alpha}(p)$ and ${\sf g} = \oplus_{p' \in S'} \phi_{\tilde{\rho}} \circ \iota_{\tilde{\rho}} \circ L_{C_\alpha}(p')$, we have $\Phi(\min\{{\sf f,g}\}) = \min\{\Phi({\sf f}),\Phi({\sf g})\}$, where the $\min$ on the LHS is that of functions, and on the RHS is that of lex order. 
We have the equation $\min\{{\sf f,g}\} = \oplus_{p'' \in S \cup S'} \phi_{\tilde{\rho}} \circ \iota_{\tilde{\rho}} \circ L_{C_\alpha}(p'')$. However, this is not necessarily the unique minimal expression as a min-combination of linear functions. 
To obtain the result, it would suffice to see that the $m_{min}\in S\cup S'$ giving the lex-minimal element in the preceding equation appears in the unique minimal expression of $\min\{f,g\}$ as a min-combination of linear functions.
This follows from \Cref{lemma: min of functions and lex min}. Thus $\Phi(\min\{{\sf f,g}\}) = \min\{\Phi({\sf f}),\Phi({\sf g})\}$. A similar argument shows that $\Phi({\sf f}+{\sf g})=\Phi({\sf f})+\Phi({\sf g})$. Thus $\Phi$ is a semialgebra morphism. 

We may thus define a valuation with values in $F^r$ via the composition $\Phi \circ \phi_{\tilde{\rho}} \circ \iota_{\tilde{\rho}} \circ L_{C_\alpha} \circ \fv: \Aa \to \overline{F^r}$. 
By the commuting diagram in \Cref{prop-evaluationmap}, if ${\sf f} = \oplus_{p \in S} \phi_{\tilde{\rho}} \circ \iota_{\tilde{\rho}} \circ L_{C_\alpha}(p)$ is the unique expression for ${\sf f}\in \sf Im$, then $\Phi({\sf f})=\min_{m \in T}\{ \mathrm{ev}_{\overline{\varepsilon}} \circ \phi_{\tilde{\rho}} \circ \iota_{\tilde{\rho}} \circ {\v_\alpha}(m)\}$, where $T := \pi_\alpha(\v^{-1}(S)) \subset M_\alpha$ is finite.
Since the map $\mathrm{ev}_{\overline{\varepsilon}} \circ \phi_{\tilde{\rho}} \circ \iota_{\tilde{\rho}} \circ \v_\alpha: M_\alpha \to F^r$ is injective, we may equip $M_\alpha$ with a total order via this map, and define $\fv_{\alpha,\tilde{\rho}}$ by defining $\fv_{\alpha,\tilde{\rho}}(f)$ for $f \in \Aa \setminus\{0\}$ to be the unique preimage under $\mathrm{ev}_{\overline{\varepsilon}} \circ \phi_{\tilde{\rho}} \circ \iota_{\tilde{\rho}} \circ \v_\alpha$ of $\Phi \circ \phi_{\tilde{\rho}} \circ \iota_{\tilde{\rho}} \circ L_{C_\alpha} \circ \fv(f)$. This is also a valuation by construction.  This proves (1).

To prove (2), it suffices to show that if $\fv: \Aa_\MM \to \P_\NN$ additionally has the property that every element of $\Sp(\NN) \cong \MM$ is in the image of $\fv$, then $\fv_{\alpha,\tilde{\rho}}$ is full rank. If every element of $\Sp(\NN)$ is in the image of $\fv$, then by the commuting diagram in \Cref{eq: w gamma def} it follows that the image of $\fv_{\alpha,\tilde{\rho}}$ spans a submodule of full rank. This proves (2). 
\end{proof}

We will make use of a different (but equivalent, as we show below) description of the valuation $\fv_{\alpha,\tilde{\rho}}$ in our arguments in \Cref{section: compactifications}, as recorded in the following result.  

\begin{proposition}\label{prop: appendix connection}
In the setting of \Cref{lem-canonicaltototal}, assume that $(\Aa=\Aa_\MM, \fv: \Aa_\MM \to \P_\NN)$ is a detropicalization and that there exists $\B$ a convex adapted basis of $(\Aa_\MM,\fv)$. Then there exist $r$ many $\overline{F}$-valued valuations $\fv_{\alpha,\tilde{\rho},1}, \cdots, \fv_{\alpha,\tilde{\rho},r}$, such that 
\begin{equation}\label{eq: rank r val is wedge of r rank 1 val}
\fv_{\alpha,\tilde{\rho}} = \fv_{\alpha,\tilde{\rho},1} \circledast \cdots \circledast \fv_{\alpha,\tilde{\rho},r}
\end{equation}
where $\circledast$ is described in \Cref{sec: appendix valuations}. 
\end{proposition} 

\begin{proof} 
Let $\ell \in [r]$. Let $\varepsilon_\ell$ be the $\ell$-th standard basis vector in $F^r_{\geq 0}$. Let $\Phi_\ell: \O^{PL}_{F^r_{\geq 0}} \to \overline{F}$ be defined by $\Phi_\ell(f) := \min_{m \in T}\{\mathrm{ev}_{\varepsilon_\ell} \circ \phi_{\tilde{\rho}} \circ \iota_{\tilde{\rho}} \circ \v_\alpha(m)\}$ where $\mathrm{ev}_{\varepsilon_\ell}$ is the evaluation map at $\varepsilon_\ell$, and the $\min$ on the RHS is with respect to the total order on $F$. An argument similar to that given for \Cref{lem-canonicaltototal} shows that this is a semialgebra morphism, so we may define $r$ many $\overline{F}$-valued valuations by $\fv_{\alpha,\tilde{\rho},\ell} := \Phi_\ell \circ \phi_{\tilde{\rho}} \circ \iota_{\tilde{\rho}} \circ L_{C_\alpha} \circ \fv: \Aa_\MM \to \overline{F}$. To see \Cref{eq: rank r val is wedge of r rank 1 val} it would suffice to show that $\B$ is an adapted basis for both LHS and RHS, and that the two sides agree when evaluated on any basis element $\bb \in \B$. Both of these statements follow from the definitions of detropicalization, convex adapted bases, the LHS as given in the proof of \Cref{lem-canonicaltototal}, and $\circledast$ given in \Cref{sec: appendix valuations}. 
\end{proof}

 The valuations $\fv_{\alpha,\tilde{\rho}}$ in \Cref{lem-canonicaltototal} are also compatible with mutations in a natural sense.

\begin{proposition}\label{prop-mutation-elements}
In the setting of \Cref{lem-canonicaltototal}, assume that $(\Aa=\Aa_\MM, \fv: \Aa_\MM \to \P_\NN)$ is a detropicalization and that there exists $\B$ a convex adapted basis of $(\Aa_\MM, \fv)$. Let $\fv_{\alpha,\tilde{\rho}}$ be the valuation constructed in \Cref{lem-canonicaltototal}.  Then $\fv_{\alpha, \tilde{\rho}}(\bb) = \pi_\alpha (\v^{-1} (\fv(\bb)))$ for any $\bb\in\B$, and, $\fv_{\alpha, \tilde{\rho}}(\B) = \pi_\alpha(\v^{-1}(\fv(\Aa\setminus\{0\})))$ is independent of the choice of $\tilde{\rho}$. 
 Moreover, for all $\alpha'\in\pi(\MM)$ and any choice of bases $\tilde{\rho}, \tilde{\rho}'$, we have that $\mu_{\alpha, \alpha'}\fv_{\alpha, \tilde{\rho}}(\bb) = \fv_{\alpha',\tilde{\rho}'}(\bb)$. 
\end{proposition}

\begin{proof}
Let $\mathbb{b}\in\B$. By assumption (2) of \Cref{def_CAB}, we know $\fv(\mathbb{b})=(\v \circ \pi_\alpha^{-1})(m)$ for some $m\in M_\alpha$. Now by the commutative diagram \Cref{eq: w gamma def} and the definition of $\fv_{\alpha,\tilde{\rho}}$ we have $\fv_{\alpha,\tilde{\rho}}(\bb) = m$. 
The first claim of the proposition follows. For the second assertion, note that by assumption (1) of \Cref{def_CAB} we know that for any $f \in \Aa \setminus \{0\}$ we have $\fv(f) = \bigoplus_i \fv(\bb_i)$ for some $\bb_i \in \B$. It follows that any $\fv_{\alpha,\tilde{\rho}}(f)$ for $f \in \Aa_\MM \setminus \{0\}$ is equal to $\fv_{\alpha,\tilde{\rho}}(\bb)$ for some $\bb \in \B$ and the assertion follows. The last claim follows since, by what we have just shown, $\mu_{\alpha,\alpha'}\fv_{\alpha,\tilde{\rho}}(\bb) = \mu_{\alpha,\alpha'}\pi_\alpha(\v^{-1}(\fv(\b))) = \pi_{\alpha'}(\v^{-1}(\fv(\b))) = \fv_{\alpha',\tilde{\rho}'}(\bb)$. 
\end{proof}

\begin{remark}\label{remark: CAB axiom}
As can be seen from the above proof, the independence of $\fv_{\tilde{\rho}}(\bb)$ from the choice of basis holds because of assumption (2) in \Cref{def_CAB}, which stipulates that for $\bb\in\B$, its value under the valuation $\fv(\bb)$ is a single point in $\Sp(\NN)$, rather than a minimum of a collection of points. If its value were a non-trivial min-combination of points, then from the definition of $\psi_{\tilde{\rho}}$ it follows that we must take a minimum with respect to the lex order, which depends on $\tilde{\rho}$.  
\exampleqed
\end{remark}

Let $\B$ be a subset of $\Aa_\MM$. Our next goal is to show an equivalence between the conditions that $\B$ is a convex adapted basis, and, that $\B$ maps bijectively to $\MM$ (or equivalently $\Sp(\NN)$). This is the content of \Cref{corollary: CAB bijective with MM} and \Cref{prop-detropbasis}. This equivalence gives us an effective tool for constructing convex adapted bases. \textbf{For this discussion, we assume that $\MM$ is a polyptych lattice defined over $\Z$.}

 In the arguments below, it is useful to know when valuations have one-dimensional leaves, because -- among other things -- this allows us to detect convex adapted bases. For this purpose, Abhyankar's inequality (see e.g. \cite[Theorem 6.6.7]{HunekeSwanson} for details) is useful; below, we give the statement using our notation and in the form appropriate for our situation. 

\begin{theorem}[Abhyankar's inequality]\label{theorem: Abhyankar}
Let $\Aa$ be a Noetherian $\K$-algebra which is an integral domain and let $\fv: \Aa \to \Gamma$ be a valuation where $\Gamma$ is a totally ordered abelian group. Let $\tilde{\fv}: \mathcal{K}(\Aa) \to \Gamma$ denote the natural extension of $\fv$ to the fraction field of $\Aa$ by defining $\tilde{\fv}(f/g) := \fv(f)-\fv(g)$ for $f,g \in \Aa, g \neq 0$. Let $\Aa_{\mathfrak{p}}$ denote the localization of $\Aa$ at the prime ideal $\mathfrak{p} := \{f \in \Aa: \fv(f) > 0\}$ with maximal ideal $m_{\mathfrak{p}} = \mathfrak{p}\Aa_{\mathfrak{p}}$ and residue field $Q(\mathfrak{p}) := \Aa_{\mathfrak{p}}/m_{\mathfrak{p}}$. Let $V := \{\frac{f}{g} \in \mathcal{K}(\Aa)^* \, \mid \, \fv(f)-\fv(g) \geq 0\} \cup \{0\} \subset \mathcal{K}(\Aa)$ be the valuation ring of $\fv$ with maximal ideal $m_V := \{ \frac{f}{g} \in \mathcal{K}(\Aa)^* \, \mid \, \fv(f)-\fv(g) > 0 \} \cup \{0\}$ and residue field $k(\fv)$. Let $\Gamma_\fv$ denote the value group of the valuation domain associated to $\fv$. Let $\mathrm{rat.rk}(\fv) := \dim_{\Q}(\Gamma_\fv \otimes_{\Z} \Q)$. Then 
\begin{enumerate}
\item $\mathrm{rat.rk}(\fv) + \mathrm{tr.deg}_{Q(\mathfrak{p})}(k(\fv)) \leq \dim(\Aa_\mathfrak{p})$, where $\dim(\Aa_\mathfrak{p})$ denotes the Krull dimension of $\Aa_\mathfrak{p}$, and 
\item if $\mathrm{rat.rk}(\fv) + \mathrm{tr.deg}_{Q(\mathfrak{p})}(k(\fv)) = \dim(\Aa_\mathfrak{p})$, then $k(\fv)$ is finitely generated over $Q(\mathfrak{p})$. 
\end{enumerate} 
\end{theorem}

It follows from Abhyankar's inequality that the full-rank valuations have one-dimensional leaves. The precise statement is as follows. 

\begin{proposition}\label{proposition: full rank one dim leaves}
Let $\Aa$ be a Noetherian $\K$-algebra which is an integral domain. Suppose that $\fv: \Aa \to \Gamma$ is a full rank valuation, i.e.~$\mathrm{rank}(\fv)$ equals the Krull dimension of $\Aa$. Then $\fv$ has one-dimensional leaves. 
\end{proposition} 

\begin{proof} 
If  $\mathrm{rank}(\fv)=\dim(\Aa)$, then $\mathrm{rat.rk}(\fv)=\dim(\Aa)$. Since $\dim(\Aa) \geq \dim(\Aa_{\mathfrak{p}})$, by \Cref{theorem: Abhyankar}(1) this forces $\mathrm{tr.deg}_{Q(\fv)}(k(\fv))=0$ and $\dim(\Aa_{\mathfrak{p}})=\dim(\Aa)$. Moreover, by \Cref{theorem: Abhyankar}(2) we have $k(\fv)$ is finitely generated over $Q(\mathfrak{p})$, so $k(\fv)$ is an algebraic extension of $Q(\mathfrak{p})$. Note now that $\K$ is a subfield of $Q(\mathfrak{p})$ since $\fv$ is trivial on $\K$ and that $\mathrm{tr.deg}_\K(Q(\mathfrak{p})) = \dim(\Aa/\mathfrak{p})$. We have $\dim(\Aa/{\mathfrak{p}}) + \mathrm{height}(\mathfrak{p}) \leq \dim(\Aa)$ and $\mathrm{height}(\mathfrak{p}) = \dim(\Aa_{\mathfrak{p}})$, but the latter is equal to $\dim(\Aa)$ as we saw above. Hence $\dim(\Aa/\mathfrak{p}) = 0 = \mathrm{tr.deg}_\K(Q(\mathfrak{p})) = 0$. Since $\K$ is algebraically closed, $\K=Q(\mathfrak{p})$. Going back to the transcendence degree of $k(\fv)$ over $Q(\mathfrak{p})$, this means $k(\fv)$ is algebraic over $\K$ and again by algebraic closure we have $k(\fv)=\K$. 

Now observe that when we pass to associated graded algebras, $\mathrm{gr}_{\fv}(\Aa) \subset \mathrm{gr}_{\fv}(\mathcal{K}(\Aa))$. One can see that $\mathrm{gr}_{\fv}(\mathcal{K}(\Aa)) \cong k(\fv)[\Gamma_\fv]$, the group ring over the field $k(\fv)$ of the value semigroup $\Gamma_{\fv}$. (Since the homogeneous elements of $\mathrm{gr}_{\fv}(\Aa)$ are not zero divisors and each graded piece is $1$-dimensional, it follows that $\mathrm{gr}_{\fv}(\Aa)$ is isomorphic to the semigroup algebra over $\Gamma_{\fv}$, see e.g.\ \cite[Remark 4.13]{BrunsGubeladze}.)  Hence each graded piece $F_{\fv \geq i}/F_{\fv > i}$ of $\mathrm{gr}_\fv(\Aa)$ for $i \in \Gamma_\fv$ is a subspace of $k(\fv)$. We have seen above that $k(\fv)=\K$; this implies each graded piece is a $1$-dimensional $\K$-vector space, or equivalently, $\fv$ on $\Aa$ has one-dimensional leaves, as was to be shown. 
\end{proof} 

The following is then immediate. 

\begin{corollary}\label{corollary: v alpha rho has one dim leaves}
Let $\MM$ be a finite polyptych lattice over $\Z$ and fix a choice of strict dual pair $(\MM, \NN, \v, \w)$ over $\Z$. Let $\Aa$ be a Noetherian $\K$-algebra which is an integral domain and let $\fv: \Aa \to \P_\NN$ be a valuation with values in $\P_\NN$. 
 Let $\alpha \in \pi(\MM)$ be a choice of coordinate chart of $\MM$ and let $\tilde{\rho} \subset C_\alpha$ be a choice of ordered basis for $\NN$, where $C_\alpha := \w^{-1}(\Sp_\R(\MM,\alpha)) \subset \NN_\R$. Then the associated valuation $\fv_{\alpha,\tilde{\rho}}$ has one-dimensional leaves, and the associated graded algebra $\mathrm{gr}_{\fv_{\alpha,\tilde{\rho}}}(\Aa_\MM)$ is isomorphic to the semigroup algebra $\K[S(\Aa_\MM,\fv_{\alpha,\tilde{\rho}})]$, where $S(\Aa_\MM,\fv_{\alpha, \tilde{\rho}})$ is the value semigroup of $\fv_{\alpha,\tilde{\rho}}$.  
\end{corollary} 

\begin{proof} 
By \Cref{lem-canonicaltototal} we know that $\fv_{\alpha,\tilde{\rho}}$ has full rank, so \Cref{proposition: full rank one dim leaves} applies. Since each graded piece $F_{\fv_{\alpha,\tilde{\rho}} \geq i}/F_{\fv_{\alpha,\tilde{\rho}} > i}$ is at most $1$-dimensional, and is dimension $1$ precisely when $i \in S(\Aa_\MM,\fv_{\alpha,\tilde{\rho}})$, the second claim follows by the definition of the product structure on $\mathrm{gr}_{\fv_{\alpha,\tilde{\rho}}}(\Aa_\MM)$. 
\end{proof}

Another consequence of \Cref{proposition: full rank one dim leaves} is that a convex adapted basis $\B$ of a detropicalization $\fv: \Aa_\MM \to S_\MM \cong \P_\NN$ gives a bijection between $\B$ and $\MM \cong \Sp(\NN)$. 

\begin{corollary}\label{corollary: CAB bijective with MM}
    Let $\MM$ be a finite polyptych lattice over $\Z$ and fix a choice of strict dual pair $(\MM, \NN, \v, \w)$ over $\Z$. Let $\Aa_\MM$ be a Noetherian $\K$-algebra which is an integral domain and let $\fv: \Aa_\MM \to S_\MM \cong \P_\NN$ be a detropicalization of $\MM$. Let $\B$ be a convex adapted basis of $\fv: \Aa_\MM \to S_\MM \cong \P_\NN$. Then $\fv$ induces a bijection from $\B$ to $\MM\cong \Sp(\NN)$. 
\end{corollary}

\begin{proof} 
It follows immediately from the definitions of a detropicalization and a convex adapted basis (\Cref{def_detrop} and \Cref{def_CAB} respectively) that $\fv: \B \to \MM\cong\Sp(\NN)$ is surjective. Thus it suffices to see injectivity. For any choice of $\alpha \in \pi(\MM)$ and $\tilde{\rho}$ as in \Cref{lem-canonicaltototal} we have seen that $\fv_{\alpha,\tilde{\rho}}$ is full rank, so \Cref{proposition: full rank one dim leaves} applies. In particular, $\fv_{\alpha,\tilde{\rho}}$ has one-dimensional leaves. By \Cref{prop-mutation-elements} we know that images of convex adapted basis elements $\bb \in \B$ under $\fv$ (viewed in $S_\MM$) are related to those under $\fv_{\alpha,\tilde{\rho}}$ by the bijection $\pi_\alpha$. Since $\fv_{\alpha,\tilde{\rho}}$ has one-dimensional leaves, this implies that for any $m \in \MM$, there cannot exist two distinct basis elements $\bb,\bb'$ with $\fv(\bb)=\fv(\bb')=m$. Thus $\fv$ is injective on $\B$, and the result follows. 
\end{proof}

\Cref{prop-detropbasis} below can be viewed as a converse to the above. In particular, it allows us to give an effective criterion to construct convex adapted bases. 

\begin{proposition}\label{prop-detropbasis}
Let $\MM$ be a finite polyptych lattice of rank $r$ over $\Z$ and fix a choice of strict dual pair $(\MM, \NN, \v, \w)$ over $\Z$. Suppose $(\Aa_\MM, \fv: \Aa_\MM \to \P_\NN)$ is a detropicalization in the sense of \Cref{def_detrop}.
Let $\B \subset \Aa_\MM$ be a subset of $\Aa_\MM$ such that the restriction of $\fv$ to $\B$ induces a bijection from $\B$ to $\Sp(\NN) \subset \P_\NN$. Then $\B$ is a convex adapted basis of $\fv: \Aa_\MM \to \P_\NN$.
\end{proposition}

Before proving \Cref{prop-detropbasis} we give a preliminary lemma. Note that, by definition, in order to prove that $\B$ is a convex adapted basis, we must first show that $\B$ is a $\K$-vector space basis of $\Aa_\MM$.  The following lemma yields linear independence.

\begin{lemma}\label{lem-detropdep}
In the setting of \Cref{prop-detropbasis}, suppose that $\bb_1, \ldots, \bb_\ell \in \Aa_\MM$ have the property that $\fv(\bb_i) \in \Sp(\NN)$ for all $i$, and, $\fv(\bb_i) \neq \fv(\bb_j)$ for $i \neq j$. Then $\{\bb_1, \ldots, \bb_\ell\}$ is $\K$-linearly independent.  
\end{lemma}

\begin{proof}
Suppose, in order to obtain a contradiction, that there exists a linear relation $\sum_{i \in [\ell]} \lambda_i \bb_i = 0$ in $\Aa_\MM$. Without loss of generality we may assume that $\lambda_i \neq 0$ for all $i \in [\ell]$. This implies that for any index $i_0 \in [\ell]$ we have $\bb_{i_0} = \frac{1}{\lambda_{i_0}}\left( \sum_{k \neq i_0} \lambda_k \bb_k\right)$ and hence, by assumptions on valuations, 
\begin{equation}\label{eq: linear indep 1}
\fv(\bb_{i_0}) \geq \bigoplus_{i \neq i_0} \fv(\bb_i) = \min\{ \fv(\bb_i) \, \mid \, i \neq i_0, i \in [\ell]\}
\end{equation}
as functions on $\NN_\R$. 
Now observe that since the $\fv(\bb_i)$ are elements of $\Sp(\NN) \subset \Sp_\R(\NN)$ by assumption, they restrict to linear functions on any maximal-dimensional cone $C \subset \NN_\R$ of $\Sigma(\NN_\R)$ by \Cref{lemma: points linear on faces of SigmaM}. Fix such a cone $C$. Consider the piecewise linear function $\Psi := \min\{\fv(\bb_i) \vert_C \, \mid \, i \in [\ell]\}$ on $C$. Since the $\fv(\bb_i)$ are pairwise distinct, their restrictions to $C$ must also be pairwise distinct, by \Cref{lem-convex-11}. Let $C = \cup_j \tilde{C}_j$ be the (finite) decomposition of $C$ into (full-dimensional) subcones of linearity of $\Psi$. Fix $\tilde{C}_j$. Then on $\tilde{C}_j$ there is a unique $\fv(\bb_{i_0})$ achieving the minimum in the definition of $\Psi$, and, $\fv(\bb_i) > \fv(\bb_{i_0})$ on the interior of  $\tilde{C}_j$, for all $i \neq i_0, i \in [\ell]$. This means that $\fv(\bb_{i_0}) < \min\{\fv(\bb_k) \, \mid \, k \neq i_0, k \in [\ell]\}$
on the interior of $\tilde{C}_j$, and in particular, $\fv(\bb_{i_0}) \not \geq \{\min\{\fv(\bb_k) \, \mid \, k \neq i_0, k \in [\ell]\}$
on $C$, which contradicts \Cref{eq: linear indep 1}. Thus the $\bb_i$ must be linearly independent. 
\end{proof}

\begin{proof}[Proof of \Cref{prop-detropbasis}]
To show $\B$ is a convex adapted basis, we must first show that $\B$ is a $\K$-vector space basis of $\Aa_\MM$. \Cref{lem-detropdep} shows that $\B$ is linearly independent. We now claim that $\B$ spans $\Aa_\MM$.
To see this, let $f\in\Aa_\MM$ and consider its image $\fv(f) \in \P_\NN$; by definition of $\P_\NN$, we know $\fv(f)$ can be expressed as $\fv(f) = \bigoplus_{m \in S} \v(m)$ for some finite set $S$ in $\MM$. Now let $\alpha \in \pi(\MM)$ and $\tilde{\rho} \subset C_\alpha$ be an ordered basis for $\NN$ where $C_\alpha := \w^{-1}(\Sp_\R(\MM,\alpha))$. Then by the construction of $\fv_{\alpha,\tilde{\rho}}$ as given in the proof of \Cref{lem-canonicaltototal}, we have $\fv_{\alpha,\tilde{\rho}}(f) = \pi_\alpha(m_0)$ for some $m_0 \in S$.  Now observe that since $\fv$ maps $\B$ bijectively to $\Sp(\NN) \cong \MM$, there exists $\bb \in \B$ with $\fv(\bb) = \v(m_0)$, which together with the construction of $\fv_{\alpha,\tilde{\rho}}$ imply that $\fv_{\alpha,\tilde{\rho}}(\bb) = \pi_\alpha(m_0)$. By \Cref{corollary: v alpha rho has one dim leaves} we know $\fv_{\alpha,\tilde{\rho}}$ has one-dimensional leaves, so there exists $c_{\bb} \in \K \setminus \{0\}$ such that 
\begin{equation}\label{eq: first step subduction} 
\fv_{\alpha,\tilde{\rho}}(f- c_{\bb} \bb) > \fv_{\alpha,\tilde{\rho}}(f).
\end{equation}
From properties of valuations, we know $\fv(f-c_{\bb}\bb) \geq \fv(f) \oplus \fv(\bb)$. By construction of $S$ and $\bb$, we know $\fv(\bb)=\v(m_0)$ appears in this min expression for $\fv(f)$. Thus $\fv(f) \oplus \fv(\bb) = \fv(f)$ and we conclude $\fv(f-c_\bb \bb) \geq \fv(f)$. From this we immediately obtain 
\begin{equation}\label{eq: pt conv hulls for subduction} 
\{p \in \Sp(\NN) \, \mid \, p \geq \fv(f - c_\bb \bb)\} \subseteq \{p \in \Sp(\NN) \, \mid \, p \geq \fv(f) \}.
\end{equation}
We have already seen above that $\fv(f)\oplus \fv(\bb)=\fv(f)$, or in other words $\fv(\bb) \geq \fv(f)$. Now $\fv(\bb) \in \Sp(\NN)$ by assumption, so $\fv(\bb)$ is in the RHS of \Cref{eq: pt conv hulls for subduction}. We now claim that $\fv(\bb)$ is not contained in the LHS of \Cref{eq: pt conv hulls for subduction}. Indeed, by passing to $\fv_{\alpha,\tilde{\rho}}$ and using \Cref{eq: first step subduction}, we see that the unique minimal expression for $L_{C_\alpha} \circ \fv(f-c_\bb \bb)$ does not contain $L_{C_\alpha} \circ \fv(\bb)$. 
On the other hand, $L_{C_\alpha} \circ \fv(\bb)$ does appear in $L_{C_\alpha} \circ \fv(f)$, so we know there exists a full-dimensional subcone $C'$ of linearity contained in $C_\alpha$ on which $L_{C_\alpha} \circ \fv(f) \vert_{C'} = L_{C_\alpha} \circ \fv(\bb) \vert_{C'}$ and $L_{C_\alpha}\circ \fv(\bb)$ is the unique term achieving the minimum. In particular, by injectivity of $\fv$ on $\B$ and injectivity of $L_{C_\alpha}$, we conclude $L_{C_\alpha} \circ \fv(\bb') > L_{C_\alpha} \circ \fv(\bb)$ on $C'$, where $\bb'$ is such that $L_C \circ \fv(\bb')$ appears in $L_C \circ \fv(f-c_\bb \bb)$. 
Hence $L_C \circ \fv(\bb) \geq L_C \circ \fv(f-c_\bb \bb)$ cannot hold on $C$. This in turn implies $\fv(\bb)$ cannot appear in the LHS of \Cref{eq: pt conv hulls for subduction}, as was to be shown. 

The argument just given shows that the LHS of \Cref{eq: pt conv hulls for subduction} is strictly smaller than the RHS. Since $\fv(f)=\bigoplus_{m \in S} \v(m)$, by \Cref{lem-pconvex-2} we obtain that the RHS is $\v(\ptconv_\R(S) \cap \MM)$. Since $S$ is finite, we have that the RHS is a finite set. Then, by repeating the argument in the previous paragraph by replacing $f$ by $f- c_\bb \bb$ each time, we obtain after finitely many steps an expression $f - \sum_i c_i \bb_i$ with $\{p \in \Sp(\NN) \, \mid \, p > \fv(f - \sum_i c_i \bb_i) \} = \emptyset$, i.e. $\fv(f - \sum_i c_i \bb_i)=\infty$. By our assumption on $\fv$ in \Cref{def_valuation}, this implies $f - \sum_i c_i \bb_i=0$, i.e. $f = \sum_i c_i \bb_i$. Thus $\B$ spans $\Aa_\MM$, as desired.

We must now prove property (1) in \Cref{def_CAB}. For this purpose, we introduce temporary notation. 
For $m \in \MM$ we let $\bb_m$ denote the unique element such that $\fv(\bb_m) = \v(m)$. We now claim that if $f = \sum_{m \in S} c_m \bb_m$ for $S \subset \MM$ a finite set and $c_m \in \K^*$ for all $m \in S$, then $\fv(f) = \bigoplus_{m \in S} \fv(\bb_m)$. 
To show this equality, it would suffice to show that for any $n \in \NN$, the function value $\fv(f)(n)$ is achieved by $\fv(\bb_m)(n)$ for some $m \in S$. 
To do this, we will define a new valuation; it will depend on a fixed but arbitrary choice of $n \in \NN$.  Let $C$ be a full-dimensional cone in $\Sigma(\NN)$ containing $n$ and fix a linearly independent set $\{n_1=n, n_2, \cdots, n_r\} \subset C$. 
Note that $n$ is the first element in this ordered set. In a manner similar to the construction of $\fv_{\alpha,\tilde{\rho}}$ in the proof of \Cref{lem-canonicaltototal}, we define $\Phi: \mathcal{O}^{\oplus}_C \to \Z^r$ by $\bigoplus_{m \in S} \v(m) \mapsto \min_{m \in S} \left\{ (\v(m)(n), \v(m)(n_2),\cdots, \v(m)(n_r)) \in \Z^r \right\}$ where the minimum is with respect to lex order on $\Z^r$. 
The argument that the restriction of $\Phi$ to the image of $L_C\circ\fv$ is a semialgebra homomorphism proceeds in the same way as in \Cref{lem-canonicaltototal} so we omit details. Define the valuation $\mathfrak{w} := \Phi \circ L_C \circ \fv$. Since $\{n,n_2,\cdots, n_r\}$ is linearly independent and $L_C$ surjects onto $\Hom(C,\R)$, it follows that the valuation $\mathfrak{w}$ is full rank and hence, by \Cref{corollary: v alpha rho has one dim leaves}, has one-dimensional leaves. Now by an argument similar to that given above, there exists $\bb_0 \in \B$ and $c_0 \in \K^*$ such that $\mathfrak{w}(f)=\mathfrak{w}(\bb_0)$ and $\mathfrak{w}(f-c_0 \bb_0) > \mathfrak{w}(f)$. By the same inductive process as in the argument above, we see that $\bb_0$ appears in the unique expression of $f$ as a linear combination of the basis elements $\B$. Now suppose $\fv(f) = \bigoplus_{i=1}^\ell \v(m_i)$ for some $\{m_1,\cdots,m_\ell\}$. Then we compute
\begin{equation}\label{eq: funny valuation}
    \begin{split} 
\mathfrak{w}(f) & = \min_{i\in [\ell]}\{(\v(m_i)(n), \v(m_i)(n_2),\cdots,\v(m_i)(n_r))\} \\
& = (\fv(f)(n), \ldots)\\ 
& = \mathfrak{w}(\bb_0) \\
& = (\fv(\bb_0)(n), \ldots)
    \end{split} 
\end{equation}
where the second equality is because $\mathfrak{w}(f)=\mathfrak{w}(\bb_0)$ and the last equality follows from the fact that $\fv(\bb_0)$ is a single element in $\Sp(\NN)$ so the minimum in the definition of $\Phi$ is over a singleton set. 
Note that since the $\min$ in the RHS of the first equality is the lex order, the first entry of that RHS is $\min_{i \in [\ell]} \v(m_i)$.
From \Cref{eq: funny valuation} we see that $\fv(f)(n)=\fv(\bb_0)(n)$, i.e., the value $\fv(f)(n)$ is achieved by $\fv(\bb_0)$ where $\bb_0$ appears in the expression of $f$ as a linear combination of basis elements. This is what we wished to show, and concludes the proof. 
\end{proof}

We now record some observations about detropicalizations which will be useful in the sequel. 

\begin{proposition}\label{prop: direct product of detrop}
Let $\MM,\MM'$ be two finite polyptych lattices over $\Z$ of rank $r$ and $r'$ respectively. Let $(\Aa_{\MM},\fv: \Aa_{\MM} \to S_{\MM}), (\Aa_{\MM'}, \fv': \Aa_{\MM'} \to S_{\MM'})$ be detropicalizations of $\MM, \MM'$ with convex adapted bases $\B, \B'$, respectively. 
Let $\tilde{\fv}: \Aa_\MM \otimes \Aa_{\MM'} \to S_{\MM \times \MM'}$ denote the map obtained by composing the tensor product $\fv \otimes \fv': \Aa_{\MM} \otimes \Aa_{\MM'} \to S_{\MM} \otimes S_{\MM'}$ with the idempotent semialgebra homomorphism $S_{\MM} \otimes S_{\MM'} \to S_{\MM \times \MM'}$ of \Cref{lemma: SMM SMMprime}. Then $(\Aa_{\MM} \otimes \Aa_{\MM'}, \tilde{\fv})$ is a detropicalization of the product polyptych lattice $\MM \times \MM'$, with convex adapted basis $\tilde{\mathbb{B}} = \mathbb{B} \times \mathbb{B}'$. 
\end{proposition}

\begin{proof}
We define the tensor product map $\fv \otimes \fv': \Aa_{\MM} \otimes \Aa_{\MM'} \to S_{\MM} \otimes S_{\MM'}$ using the convex adapted basis. Concretely, for $\bb \in \B, \bb' \in \B'$, we define $(\fv \otimes \fv')(\bb \otimes \bb') := \fv(\bb) \otimes \fv'(\bb') \in S_{\MM} \otimes S_{\MM'}$, and then extend linearly with respect to the addition operation in $\Aa_{\MM} \otimes \Aa_{\MM'}$ and the $\oplus$ operation in $S_{\MM} \otimes S_{\MM'}$. It is then straightforward that $\fv \otimes \fv'$, thus defined, satisfies the conditions to be a $S_{\MM} \otimes S_{\MM'}$-valued valuation on $\Aa_\MM \otimes \Aa_{\MM'}$. From \Cref{lemma: semialg hom and valuations} we know that composing a valuation with a homomorphism of idempotent semialgebras yields a valuation, so the result follows from  \Cref{lemma: SMM SMMprime}. Moreover, the set of elements $\{\bb \otimes \bb'\}_{\bb \in \B, \bb' \in \B'}$ is a $\K$-basis of $\Aa_\MM\otimes \Aa_{\MM'}$, and is a convex adapted basis of $\tilde{\fv}$ by construction. 
\end{proof}

The following shows that given a detropicalization for a direct product polyptych lattice, it is possible to construct a detropicalization for each factor, if we restrict to the strictly dualizable case.

\begin{proposition}\label{prop: detrop of product is product of detrop}
    Let $\MM, \MM'$ be two finite polyptych lattices over $F$ of rank $r$ and $r'$ respectively. Suppose that $\MM$ and $\MM'$ are strictly dualizable, with strict duals $\NN, \NN'$ respectively. Suppose that $(\Aa_{\MM\times\MM'},\fv: \Aa_{\MM \times \MM'} \to S_{\MM \times \MM'})$ is a detropicalization of the product $\MM\times\MM'$. Then there are subalgebras $\Aa_{\MM}, \Aa_{\MM'} \subset \Aa_{\MM\times\MM'}$ and valuations $\fv_{\MM}, \fv_{\MM'}$ such that $(\Aa_{\MM}, \fv_\MM)$ is a detropicalization of $\MM$, $(\Aa_{\MM'}, \fv_{\MM'})$ is a detropicalization of $\MM'$. Moreover, we have a $\K$-algebra isomorphism $\Aa_{\MM\times \MM'} \cong \Aa_{\MM} \otimes \Aa_{\MM'}$.    
\end{proposition}

\begin{proof}
    Since $(\Aa_{\MM\times \MM'},\fv)$ is a detropicalization, there exist sets $\B, \B' \subset \Aa_{\MM\times\MM'}$ which map onto $\MM\times \{0\}$ and $\{0\}\times \MM'$ respectively.  Then by properties of valuations, the set of products $\B\cdot \B' \subset \Aa_{\MM\times\MM'}$ map onto $\MM \times \MM'$ via $\fv$, and hence $\B \cdot \B'$ is a convex adapted basis of $\Aa_{\MM\times\MM'}$ by Proposition \ref{prop-detropbasis}. Let $\Aa_{\MM},\Aa_{\MM'}$ denote the $\K$-spans of $\B, \B'$ respectively. We observe that the subsets $\MM\times\{0\}, \{0\}\times \MM' \subset \MM\times \MM'$ are convex, hence $\Aa_{\MM}$ and $\Aa_{\MM'}$ are both closed under multiplication (see \Cref{lemma: support of product of CAB elements}). Thus they are subalgebras of $\Aa_{\MM\times \MM'}$. The map $\Aa_{\MM} \otimes \Aa_{\MM'} \to \Aa_{\MM\times \MM'}$ defined by taking the tensor product basis $\B\otimes \B'$ to $\B\cdot\B'$ by $\bb_m \otimes \bb_{m'} \mapsto \bb_m \cdot \bb_{m'}$ and extending $\K$-linearly, can be seen to be an algebra morphism. In addition, by the preceding discussion, this is also a bijective map from $\B \otimes\B'$ to the convex adapted basis $\B \cdot \B'$ of $\Aa_{\MM \times \MM'}$. We may conclude that $\Aa_{\MM \times \MM'} \cong \Aa_{\MM} \otimes \Aa_{\MM'}$ as $\K$-algebras, as was to be shown. 
    
    It remains to show that $\Aa_{\MM}, \Aa_{\MM'}$ are detropicalizations of $\MM, \MM'$ respectively. To see this, we must first construct valuations $\fv_{\MM}: \Aa_{\MM} \to S_\MM$ (respectively $\fv_{\MM'}: \Aa_{\MM'} \to S_{\MM'}$) which map onto $\MM \subset S_\MM$ (respectively $\MM' \subset S_{\MM'})$, and then show that the Krull dimension of $\Aa_\MM$ (respectively $\Aa_{\MM'})$ is $r$ (respectively $r'$). Observe first that from \Cref{lemma: duals of product PL} we know that $\MM \times \MM'$ has strict dual $\NN \times \NN'$ equipped with the natural direct product identifications $\v_{\MM \times \MM'} = (\v_\MM, \v_{\MM'}): \MM \times \MM' \to \NN \times \NN'$ and similarly for $\w_\MM, \w_{\MM'}$. From this it is straightforward to see that $S_{\MM \times \MM'} \cong \P_{\NN \times \NN'}$. Using this isomorphism we may view $\fv: \Aa_{\MM \times \MM'} \to S_{\MM \times \MM'}$ as a map $\fv: \Aa_{\MM \times \MM'} \to \P_{\NN \times \NN'}$. Moreover, following the reasoning of \Cref{prop-salgebras} and using that fact that the identification $\v_{\MM \times \MM'}: \MM \times \MM' \to \NN \times \NN'$ operates separately on the two factors, it is straightforward to see that for an element $\bb_m$ in $\Aa_{\MM \times \MM'}$ and an element $n \in \NN$, we have $\fv(\bb_m)(n,0) = \v_\MM(m)(n)$. Moreover, since $\v_\MM(m)$ may be viewed as an element of $\P_\NN$ as in \Cref{prop-salgebras}, this implies that each element of $f \in \Aa_\MM$ defines an element $\fv_\MM(f) \in \P_\NN$. This association provides the desired map $\fv_\MM: \Aa_\MM \to \P_\NN \cong S_\MM$. Moreover, we have already seen that $\B$ maps onto $\MM$. Finally, since the original $\fv$ satisfies the axioms of a valuation, so does $\fv_\MM$. This completes the argument that $\fv_\MM: \Aa_\MM \to S_\MM$ is a valuation and that the image contains $\MM$.  The argument is similar for $\fv_{\MM'}$. 

To complete the proof, we must show that the Krull dimensions are as expected. We give the proof for $\MM$; the case of $\MM'$ is similar. By the same argument as given in \Cref{lem-canonicaltototal}, we may use a basis $\rho = \{n_1,\cdots,n_r\}$ of a maximal cone $C_\alpha$ in $\Sigma(\NN)$ (associated to a chart $M_\alpha$ of $\MM$), and the valuation $\fv_\MM$ constructed in the paragraph above, to define a valuation $\fv_\rho: \Aa_\MM \to F^r$ (here we consider $F^r$ to be a totally ordered abelian group; see the proof of \Cref{lem-canonicaltototal} for details). As we have seen above, since the original $\fv_{\MM \times \MM'}$ is a valuation associated to a detropicalization, we know that $\B \subset \Aa_\MM \subset \Aa_{\MM \times \MM'}$ maps onto $\MM$ which then maps bijectively to $\NN$ via $\v_\MM$. As in the proof of \Cref{lem-canonicaltototal}, from the commutative diagram in \Cref{eq: w gamma def} we may conclude that $\fv_\rho$ is of full rank $r$. Now from a combination of the fact that $\mathrm{rk}(\fv) \leq \mathrm{rat.rk}(\fv)$ \cite[Proposition 6.6.6]{HunekeSwanson} and \Cref{theorem: Abhyankar}(1), we conclude that the Krull dimension of $\Aa_\MM$ must be at least $r$. By a similar argument, $\Aa_{\MM'}$ has Krull dimension at least $r'$. But we know $\Aa_\MM \otimes \Aa_{\MM'} \cong \Aa_{\MM \times \MM'}$ and by hypothesis ($\Aa_{\MM \times \MM'}$ is a detropicalization) the Krull dimension of $\Aa_{\MM \times \MM'}$ is $r+r'$. Thus the inequalities must be equalities, and $\dim(\Aa_\MM)=r$ and $\dim(\Aa_{\MM'}) = r'$, as desired. 
\end{proof}

%-------------------------------------%%%%%%%
\section{Compactifications arising from polyptych lattice polytopes}\label{section: compactifications}
%-------------------------------------%%%%%%%%%%

We can now state and prove the main results of this paper; the combinatorics and algebra in the previous sections are motivated by the geometry discussed below. We view the results below as a beginning of a broader theory and, as stated in the Introduction, leave further developments to future work. 

Throughout this section, we assume that $\MM$ is a finite polyptych lattice of rank $r$ over $\Z$ and that $(\MM, \NN, \v, \w)$ is a strict dual $\Z$-pair. We assume throughout this section that \textbf{$\K$ is an algebraically closed field}. We also fix a choice $(\Aa_\MM, \fv: \Aa_\MM \to S_\MM \cong \P_{\NN})$ of a detropicalization of $\MM$, and assume that there exists, and fix a choice of, a convex adapted basis $\B $ of $\Aa_\MM$ with respect to $\fv$. 

\subsection{Definitions and first properties}\label{subsec: compactification definition}

We start by defining a graded $\K$-algebra which is constructed using the data of the detropicalization $(\Aa_\MM, \fv)$ and a PL polytope over $\Z$.  The notation and terminology will be reminiscent of section rings of line bundles over projective toric varieties. This is not an accident; the special case when $\MM$ is a trivial polyptych lattice over $\Z$ recovers that well-known case.

Let $\PP \subset \MM_\R$ be a full-dimensional PL polytope over $\Z$ with $\PP = \bigcap_{i=1}^\ell \HH_{\w(n_i), a_i}$ for some $a_i \in \Z_{< 0}$ and $n_i\in\NN$.  Note that since $a_i<0$ for all $i$, the interior of $\PP$ contains $0_\MM$. 
For $k \in \Z_{\geq 0}$, we define the \textbf{rescaled PL polytope} $k\PP$ by 
$$
k \PP := \bigcap_{i=1}^{\ell} \HH_{\w(n_i), ka_i}. 
$$
We also define the set  
\begin{equation}\label{eq_gamma_kdelta}
    \Gamma(\Aa_\MM, k\PP) := \{f \in \Aa_\MM \mid \fv(f) \geq \psi_{k\PP}\}
\end{equation}
where $\psi_{k\PP}: \NN_\R \to \R$ denotes the support function of the PL polytope $k\PP$ 
as defined in~\eqref{eq: def support function}. We refer to $\Gamma(\Aa_\MM, k\PP)$ as \textbf{the subspace (of $\Aa_\MM$) of level $k$ with respect to $\PP$ and $\fv$}. (\Cref{prop: sections algebra}(1), below, justifies this terminology.)

To develop these ideas further, the notion of the support of an element $f\in\Aa_\MM$ will be useful. Recall that by definition of convex adapted bases, for any $\bb \in \B$ we know that $\fv(\bb) \in \Sp(\NN)$.

\begin{definition}\label{def: support of f}
In the setting above, let $f=\sum\lambda_i\bb_i \in \Aa_\MM$ for $\lambda_i \in \K$ and $\bb_i \in \B$. The \textbf{support of $f$ (with respect to $\fv$) } is defined as $\supp(f) := \ptconv_\R\left(\{\v^{-1}(\fv(\bb_i)) \in \MM \mid \lambda_i\neq 0\}\right) \subset \MM_\R$. 
\exampleqed
\end{definition}

\begin{remark}\label{remark: equivalent support desc}
For a detropicalization $\fv: \Aa_\MM \to S_\MM$, then by \Cref{corollary: CAB bijective with MM} the valuation $\fv$ induces a bijection between $m \in \MM$ and $\B$. In the discussion below, we will sometimes denote this correspondence as $\bb_m \leftrightarrow m$. In particular, $\fv(\bb_m) = m$. (When $\fv$ is viewed as a valuation with values in $\P_\NN$ via the isomorphism $S_\MM \cong \P_\NN$, then $\fv(\bb_m)=\v(m)$.) With this notation, let $f = \sum_{m \in S} c_m \bb_m$ with $c_m \in \K^*$ for a finite set $S \subset \MM$. Then we may equivalently describe the support of $f$ as $\supp(f)=\ptconv_\R(S)$. 
  \end{remark}

The next result is an observation about the support of products of convex basis elements.   

\begin{lemma}\label{lemma: support of product of CAB elements}
Let $\MM$ be a finite polyptych lattice over $\Z$ of rank $r$. Let $(\Aa_\MM, \fv: \Aa_\MM \to S_\MM)$ be a detropicalization of $\MM$ and $\B$ a convex adapted basis for $\fv: \Aa_\MM \to S_\MM$. Let $m, m' \in \MM$ and let $\bb_m, \bb_{m'} \in \B$ denote their associated convex adapted basis elements. Then $$
\supp(\bb_m \cdot \bb_{m'}) = \ptconv_\R(\{m +_\alpha m' \, \mid \, \alpha \in \pi(\MM)\}).
$$
\end{lemma}

\begin{proof} 
Since $\B$ is a basis of $\Aa_\MM$, we may write $\bb_m \cdot \bb_{m'} = \sum_{m''} c_{m''} \bb_{m''}$ for some $c_{m''} \in \K$, and since $\B$ is a convex adapted basis, we must have $\fv(\bb_m \cdot \bb_{m'}) = \fv(\sum_{m''} c_{m''} \bb_{m''}) = \oplus_{c_{m''} \neq 0} \fv(\bb_{m''}) = \oplus_{c_{m''} \neq 0} m''$. On the other hand, since $\fv: \Aa_\MM \to S_\MM$ is a valuation, we know $\fv(\bb_m \cdot \bb_{m'}) = \fv(\bb_m) \star \fv(\bb_{m'}) = m \star m'$, where $\star$ denotes the product operation in $S_\MM$. By definition of the canonical semialgebra we have $m \star m' := \oplus_{\alpha \in \pi(\MM)} m +_\alpha m'$. Thus we have $\oplus_{c_{m''} \neq 0} m'' = \oplus_{\alpha \in \pi(\MM)} m +_\alpha m'$ in $S_\MM$, and by definition of $S_\MM$ this implies $\ptconv_\R(\{\oplus_{c_{m''} \neq 0} m''\}) = \ptconv_\R(\{\oplus_{\alpha \in \pi(\MM)} m +_\alpha m'\})$. By \Cref{remark: equivalent support desc} this implies $\supp(\bb_m \cdot \bb_{m'}) = \ptconv_\R(\{\oplus_{\alpha \in \pi(\MM)} m +_\alpha m'\})$ as required. 
\end{proof}

Our next lemma returns to the theme of PL polytopes, and in particular, relates \Cref{def: support of f} to \Cref{eq_gamma_kdelta}. 

\begin{lemma}\label{claim_support} 
In the setting above, let $\PP \subset \MM_\R$ be a PL polytope and $f \in \Aa_\MM$. Then $\fv(f) \geq \psi_{\PP}$ if and only if 
$\supp(f) \subseteq \PP$. 
\end{lemma} 

\begin{proof}
    We assume without loss of generality that $f=\sum\lambda_i\bb_i$ where  $\lambda_i\neq 0$ for all $i$ and $\bb_i \in \B$. 
    To prove the claim, first suppose that $\fv(f) \geq \psi_{\PP}$. Since $\fv(f)=\bigoplus_{i} \fv(\bb_i)$ by \Cref{def_CAB} we conclude that for all $i$ and $n\in\NN_\R$ we have $\fv(\bb_i)(n)\ge \psi_\PP(n)$. \Cref{lem_support_to_poly} then implies that for all $i$, $\v^{-1}(\fv(\bb_i))\in\PP$. Since $\PP$ is point-convex by \Cref{lemma: PL polytope is point convex}, it follows that $\supp(f)\subset \PP$. 
    Conversely, suppose $\supp(f)\subset \PP$. Then for all $i$, $\v^{-1}(\fv(\bb_i))\in\PP$. By \Cref{lem_support_to_poly} we have that for all $i$, $\fv(\bb_i)\ge \psi_\PP$, so $\fv(f) = \bigoplus_i \fv(\bb_i) \geq \psi_{\PP}$.
\end{proof}

It follows immediately from \Cref{claim_support} that 
\begin{equation}\label{eq: level k support in kP}
\Gamma(\Aa_\MM, k\PP) = \{ f \in \Aa_\MM \, \mid \, \supp(f) \subseteq k\PP \}.
\end{equation}

The sets $\Gamma(\Aa_\MM, k\PP)$ will be the graded pieces of a graded ring which we define below. For this to make sense, we need some technical properties of these sets, as recorded in the next proposition. 

\begin{proposition}\label{prop: sections algebra}
In the setting above, let $\PP \subset \MM_\R$ be a full-dimensional PL polytope with $\PP = \bigcap_{i =1}^\ell \HH_{\w(n_i), a_i}$ for $a_i \in \Z_{< 0}$ and $n_i\in\NN$. Let $k \in \Z_{\geq 0}$ and $\Gamma(\Aa_\MM, k\PP) \subset \Aa_\MM$ be the set defined in \Cref{eq_gamma_kdelta}. Then: 
\begin{enumerate} 
\item $\Gamma(\Aa_\MM, k\PP)$ is a vector space, i.e., it is closed under addition and scalar multiplication by $\K$.  
\item For any $k' \in \Z$ with $k' > k$, we have $\Gamma(\Aa_\MM, k\PP)\subseteq \Gamma(\Aa_\MM, k'\,\PP)$. 
\item Moreover, for any $k' \in \Z_{\geq 0}$, the product structure in $\Aa_\MM$ induces a map 
\[\Gamma(\Aa_\MM, k\PP)\otimes \Gamma(\Aa_\MM, k'\, \PP) \to \Gamma(\Aa_\MM, (k + k')\PP).\]
\end{enumerate}
\end{proposition}

\begin{proof}
To prove (1), first observe that if $f,g$ satisfy $\fv(f) \geq \psi_{k\PP}$ and $\fv(g) \geq \psi_{k\PP}$, respectively, then $\fv(f+g) \geq \fv(f) \oplus \fv(g) \geq \psi_{k \PP}$ as desired. We also have that $\fv(cf)=\fv(f)$ for $c \neq 0$, $c \in \K$.

We now prove (2). We have seen in \Cref{claim_support} that $\fv(f) \geq \psi_{k\PP}$ if and only if $\supp(f) \subseteq k\PP$. Notice that since $a_i<0$ for all $i$ it follows that $k \PP$ is contained in $ k' \PP$ for any $k' > k$. Hence $\supp(f) \subset k \PP$ implies $\supp(f) \subset k' \PP$ which means $\fv(f) \geq \psi_{k'\PP}$. This yields the claim.

To see (3), we first note that it follows straightforwardly from the definition of support functions and of $k\PP$ that  
$\psi_{k\PP}=k\psi_\PP$. 
 Now suppose $\fv(f)\ge \psi_{k\PP}$ and $\fv(g)\ge \psi_{k' \PP}$. Then 
\begin{equation*}
    \fv(fg)=\fv(f)+\fv(g)\ge 
    \psi_{k\PP} + \psi_{k' \PP} = 
k \psi_{\PP} + k' \psi_{\PP} = (k+k') \psi_{\PP} = \psi_{(k+k')\PP}
.\end{equation*}
Thus $fg \in \Gamma(\Aa_\MM, (k+k')\PP)$ by \Cref{claim_support}, as was to be shown. 
\end{proof}

By \Cref{prop: sections algebra} we may define the following (graded) algebra. 

\begin{definition}\label{definition: M polytope algebra}
In the setting above, 
we define the graded algebra $\Aa_\MM^\PP$ as
\begin{equation*}\label{eq: def APM}
\Aa_\MM^\PP := \bigoplus_{k \geq 0} \Gamma(\Aa_\MM, k\PP) \cdot t^k \subseteq \Aa_\MM[t]
\end{equation*}
where $t$ is a formal variable keeping track of the grading. We call $\Aa_\MM^\PP$ the \textbf{PL polytope algebra} (associated to $(\Aa_\MM,\fv)$ and $\PP$).
\exampleqed
\end{definition}

In the course of the geometric arguments to follow, it will be useful to have a perspective on the graded algebra $\Aa_\MM[t]$ in terms of direct products of polyptych lattices. We develop this here in preparation for Sections~\ref{subsec:compactification} to ~\ref{subsec: Cox rings}. The first step observes that $\Aa_\MM[t]$ is itself a detropicalization.

\begin{lemma}\label{lemma: AaMMt is detrop}
Let $\MM$ be a finite polyptych lattice over $\Z$ and suppose $(\Aa_\MM, \fv: \Aa_\MM \to S_\MM)$ is a detropicalization of $\MM$ and $\B$ is a convex adapted basis for $\Aa_\MM$ with respect to $\fv$. Let $\Z$ denote the trivial polyptych lattice of rank $1$ over $\Z$. Then the pair $(\Aa_\MM[t, t^{-1}], \tilde{\fv}: \Aa_\MM[t, t^{-1}] \to S_{\MM \times \Z})$ where the valuation $\tilde{\fv}$ with $\tilde{\fv}(\bb \cdot t^k) = (\fv(\bb), k) \in \MM \times \Z$ is constructed in \Cref{prop: direct product of detrop}, is a detropicalization of the product polyptych lattice $\MM \times \Z$ of rank $\mathrm{rank}(\MM)+1$ over $\Z$. Moreover, the set $\{\bb \cdot t^k\}_{\bb \in \B, k\in \Z}$ is a convex adapted basis of $\Aa_\MM[t,t^{-1}]$ with respect to $\tilde{\fv}$. 
\end{lemma} 

\begin{proof} 
In \Cref{example: detrop trivial PL} we saw that a detropicalization of the trivial rank-$1$ polyptych lattice $\Z$ may be chosen to be $\K[t,t^{-1}]$, with valuation sending a Laurent polynomial $\sum_k c_k t^k$ to the piecewise-linear function $\min\{ \langle k, \cdot \rangle \, \mid \, c_k \neq 0\}$ on $\R$ where we denote by $\langle \cdot, \cdot \rangle$ the standard inner product on $\R$. It is straightforward to see that the monomials $\{t^k \, \mid \, k \in \Z\}$ form a convex adapted basis for $\K[t,t^{-1}]$ with respect to this valuation. Applying \Cref{prop: direct product of detrop} to $\MM$ and $\Z$ then yields the claims. 
\end{proof}

The next lemma now follows.  

\begin{lemma}\label{lemma: rank 1 valuation}
In the setting of \Cref{lemma: AaMMt is detrop}, let $n \in \NN$ and $a \in \Z$. There is a $\overline{\Z}$-valued valuation $\tilde{\fv}_{n,a}$ on $\Aa_\MM[t]$ satisfying $\tilde{\fv}_{n,a}(\bb \cdot t^k) = \fv(\bb)(n) - ka$. 
\end{lemma}

\begin{proof} 
The claim follows from \Cref{lemma: semialg hom and valuations} and \Cref{lemma-pointsaresemialgebramaps}. 
Since $(n,-a) \in \NN \times \Z$ is an element of the dual to $\MM \times \Z$, hence a point in $\Sp(\MM \times \Z)$, we obtain $\tilde{\fv}_{n,a}$ by composing $\tilde{\fv}$ with this point.
\end{proof}

Given $\PP = \bigcap_{i =1}^\ell \HH_{\w(n_i), a_i} \subset \MM_\R$, let $\tilde{\fv}_i$ denote the valuation $\tilde{\fv}_{n_i,a_i}$ from the lemma above. It is straightforward from the definitions and \Cref{lemma: S subset T point convex} that, for $i \in [\ell]$,
\begin{equation}\label{eq: fvi support}
\supp(f) \subset \HH_{\w(n_i), ka_i} \ \Longleftrightarrow \ \fv(f)(n_i)-ka_i \geq 0 \ \Longleftrightarrow \ \tilde{\fv}_i(f \cdot t^k)\geq 0. 
\end{equation}
Motivated by this,  and in analogy with~\Cref{eq: level k support in kP}, we define 
\begin{equation*}\label{eq:def Gamma half-space}
\Gamma(\Aa_\MM, \HH_{\w(n_i), k a_i}):=\{f\in\Aa_\MM\mid \supp(f)\subset\HH_{\w(n_i), ka_i} \}
\end{equation*} 
and also define 
\begin{equation}\label{eq: def A_niai}
\Aa^{n_i,a_i}_\MM := \bigoplus_{k \in \Z_{\geq 0}} \Gamma(\Aa_\MM, \HH_{\w(n_i), ka_i}) \cdot t^k = \bigoplus_{k \in \Z_{\geq 0}} \left\{ f \cdot t^k \in \Aa_\MM[t] \, \mid \, \tilde{\fv}_i(f \cdot t^k) \geq 0 \right\}
\end{equation}
where the last equality follows from \Cref{eq: fvi support}. Now it is immediate, since $\PP = \cap_{i \in [\ell]} \HH_{\w(n_i),a_i}$, that 
\begin{equation*}\label{eq: AMP intersect A_niai}
\Aa_\MM^\PP = \bigcap_{i \in [\ell]} \Aa_\MM^{n_i,a_i} \subset \Aa_\MM[t].
\end{equation*}
It is straightforward that $\Aa_\MM^\PP \hookrightarrow \Aa_\MM^{n_i,a_i} \hookrightarrow \Aa_\MM[t]$ are inclusions of graded rings.

Next, we make some constructions using the rank-$1$ valuations constructed in \Cref{lemma: rank 1 valuation}. For each $i \in [\ell]$, consider the $\overline{\Z}$-valued valuation $\tilde{\fv}_{n_i,a_i}$ on $\Aa_\MM[t]$ constructed in \Cref{lemma: rank 1 valuation}. Since $\Aa_\MM[t]$ is an integral domain, we may consider the extension of $\tilde{\fv}_{n_i,a_i}$ to its field of fractions $\mathcal{K}(\Aa_\MM[t])$. Let $\mathcal{O}_{n_i,a_i}$ denote the associated discrete valuation ring in $\mathcal{K}(\Aa_\MM[t])$, i.e., 
\begin{equation}\label{eq: def O ni ai}
\mathcal{O}_{n_i,a_i} := \left\{ \frac{f}{g} \in \mathcal{K}(\Aa_\MM[t]) \, \mid \, \tilde{\fv}_{n_i,a_i}(f) - \tilde{\fv}_{n_i,a_i}(g) \geq 0 \right\} \subset  \mathcal{K}(\Aa_\MM[t]).
\end{equation}
Let $\mathfrak{m}_{n_i,a_i}$ denote the unique maximal ideal in $\mathcal{O}_{n_i,a_i}$, i.e., 
\begin{equation*}\label{eq: def m ni ai}
\mathfrak{m}_{n_i,a_i} = \left\{ \frac{f}{g} \in \mathcal{K}(\Aa_\MM[t]) \, \mid \, \tilde{\fv}_{n_i,a_i}(f) - \tilde{\fv}_{n_i,a_i}(g) > 0 \right\} \subset  \mathcal{K}(\Aa_\MM[t]).
\end{equation*}
We then define 
\begin{equation}\label{eq: def Pi}
P_i := \Aa_\MM^{n_i,a_i} \cap \mathfrak{m}_{n_i,a_i} = \{ f \cdot t^k \in \Aa_\MM^{n_i,a_i} \, \mid \, \fv(f)(n_i) - k a_i > 0\} \subset \Aa_\MM^{n_i,a_i}
\end{equation}
viewed as an ideal of $\Aa^{n_i,a_i}_\MM$, 
and similarly $Q_i := \Aa_\MM^\PP \cap \mathfrak{m}_{n_i,a_i}$, an ideal in $\Aa_\MM^\PP$. Note that $P_i$ and $Q_i$ are both prime ideals, by construction.

We saw in \Cref{corollary: v alpha rho has one dim leaves} that for the full-rank valuation $\fv_{\alpha,\tilde{\rho}}$, the associated graded algebra is the semigroup algebra over its value semigroup. For technical reasons, we will need below to consider $\overline{\Z}$-valued valuations in place of $\overline{\Z^r}$-valued valuations, but in such a way that their associated graded algebras are the same. 
We give the precise statement in \Cref{lemma: rank 1 modify full rank} below, for which we need some preparation. Following the notation in \Cref{lem-canonicaltototal}, by slight abuse of notation we also denote by $\tilde{\rho}$ the basis of $C_\alpha \times \R$ obtained by adding the vector $(0,1) \in C_\alpha \times \Z$ to $\tilde{\rho}$. Since $\Z$ is a trivial polyptych lattice, it is straightforward to check that it is self-dual with respect to the standard inner product pairing, so by \Cref{lemma: duals of product PL} we have $\NN \times \Z$ is the strict dual of $\MM \times \Z$. In particular $C_\alpha \times \R = (\w \times \mathrm{id})^{-1}(\Sp(\MM \times \Z, \alpha))$. Again by slight abuse of notation we denote by $\fv_{\alpha,\tilde{\rho}}: \Aa_\MM^\PP \to \overline{\Z^{\mathrm{rank}(\MM)} \times \Z}$ the full-rank valuation associated to $\alpha \in \pi(\MM \times \Z)$ and $\tilde{\rho}$, as constructed in \Cref{lem-canonicaltototal}, with domain restricted from $\Aa_\MM[t,t^{-1}]$ to the subalgebra $\Aa_\MM^\PP$. We equip $\Z^{\mathrm{rank}(\MM)} \times \Z$ with the lex order which compares the $\Z$ factor first and then orders with respect to the $\Z^{\mathrm{rank}(\MM)}$ factor. (Informally, the grading with respect to the variable $t$ ``comes first''.) Now we let $S$ temporarily denote the saturated affine semigroup of lattice points (in $M_\alpha \times \Z$) lying in the cone over $\pi_\alpha(\PP)$; more precisely, 
\begin{equation}\label{eq: def temp S}
S := \{ (m,k) \, \mid \, m \in M_\alpha \cap k \cdot \pi_\alpha(\PP), \, k \in \Z_{\geq 0} \} \subset M_\alpha \times \Z.
\end{equation}

\begin{lemma}\label{lemma: rank 1 modify full rank}
With assumptions and notation as above, (1) the value semigroup $S(\Aa_\MM^\PP, \fv_{\alpha,\tilde{\rho}})$ is equal to $S$ as defined in \Cref{eq: def temp S}, (2) $\mathrm{gr}_{\fv_{\alpha,\tilde{\rho}}}(\Aa_\MM^\PP)$ is isomorphic to the semigroup algebra $\K[S]$, and (3) there exists a $\overline{\Z}$-valued valuation $\bar{\fv}: \Aa_\MM^\PP \to \overline{\Z}$ such that $\mathrm{gr}_{\bar{\fv}}(\Aa_\MM^\PP) \cong \mathrm{gr}_{\fv_{\alpha,\tilde{\rho}}}(\Aa_\MM^\PP) \cong \K[S] = \K[S(\Aa_\MM^\PP, \fv_{\alpha,\tilde{\rho}})]$ as $\K$-algebras. 
\end{lemma} 

\begin{proof} 
Let $S(\Aa_\MM^\PP, \fv_{\alpha,\tilde{\rho}})$ denote the value semigroup of $\fv_{\alpha,\tilde{\rho}}$.  Since $(\Aa_\MM,\fv)$ is a detropicalization, $\fv$ induces a bijection from $\mathbb{B}$ to $\MM$. It follows from the definition of $\Aa_\MM^\PP$ that $S(\Aa_\MM^\PP, \fv_{\alpha,\tilde{\rho}})$ is precisely the semigroup $S$ defined above. This proves (1). The claim (2) follows from \Cref{corollary: v alpha rho has one dim leaves}. 
To see claim (3), recall from \Cref{prop: appendix connection} that 
$\fv_{\alpha,\tilde{\rho}}$ can be interpreted as $\fv_{\alpha,\tilde{\rho},1} \circledast \fv_{\alpha,\tilde{\rho},2} \circledast \cdots \circledast \fv_{\alpha,\tilde{\rho},r+1}$. 
Then setting $\overline{\fv} := 
\fv_{\alpha,\tilde{\rho},1} \boxplus \fv_{\alpha,\tilde{\rho},2} \boxplus \cdots \boxplus \fv_{\alpha,\tilde{\rho},r+1}$ 
and using \Cref{lemma: grF and grG iso} completes the proof.  
\end{proof}

We conclude this section with some observations about a distinguished collection of subalgebras of $\Aa_\MM$ (which leads naturally to subalgebras of $\Aa_\MM^\PP, \Aa_\MM^{n_i,a_i}$, etc).  Recall that the standing hypotheses of this section include that $\MM$ is a finite polyptych lattice equipped with a strict dual $\NN$, and we have fixed a choice of a convex adapted basis $\B$ for the detropicalization $\fv: \Aa_\MM \to \P_{\NN}$. 

Let $\mathcal{C}$ be a cone in $\Sigma(\MM)$. We know from \Cref{lemma: addition well def on cones} that addition of elements of $\mathcal{C} \subset \MM$ is well-defined. Since $\mathcal{C}$ is a cone, the zero element $0_\MM$ is also contained in $\mathcal{C} \cap \MM$. Thus we may view $\mathcal{C}\cap \MM$ as a semigroup.  Following \Cref{remark: equivalent support desc} we denote by $\bb_m$ the unique element in $\B$ corresponding to  $m \in \MM$. We also define $\Aa_\MM^{\mathcal{C}} := \mathrm{span}_\K\{\bb_m \, \mid \, m \in \mathcal{C}\} \subset \Aa_\MM$. We have the following. 

\begin{lemma}\label{lemma: cone subset of AaMM is subalg}
With assumptions and notation as above, the $\K$-vector subspace $\Aa_\MM^{\mathcal{C}}$ is: 
\begin{enumerate} 
\item closed under multiplication, and hence is a subalgebra of $\Aa_\MM$,
\item isomorphic to the semigroup algebra $\K[\mathcal{C} \cap \MM]$, and 
\item finitely generated. 
\end{enumerate}
\end{lemma} 

\begin{proof} 
Let $m,m' \in \mathcal{C} \cap \MM$ and consider their corresponding basis elements $\bb_m, \bb_{m'}$. Since $\B$ is a basis, we know that $\bb_m \cdot \bb_{m'} = \sum_{m'' \in \MM} c_{m''} \bb_{m''}$ for some $c_{m''} \in \K$. By \Cref{lemma: support of product of CAB elements}, we know $c_{m''} \neq 0$ only if $m'' \in \ptconv_\R(\{m +_\alpha m' \, \mid \, \alpha \in\pi(\MM)\})$. Since $m, m'$ lie in a common cone of $\Sigma(\MM)$, by \Cref{lemma: addition well def on cones} we know that $m +_\alpha m'$ is independent of the choice of $\alpha$, i.e. the set $\{m +_\alpha m'\, \mid \, \alpha \in \pi(\MM)\}$ is a singleton set and we may denote it by $m+m'$. In particular, $\bb_m \cdot \bb_{m'}$ is a non-zero scalar multiple of the single basis element $\bb_{m+m'}$. We also know that $m+m' \in \mathcal{C}$. This proves (1).  Moreover, this proves that $\Aa_\MM^{\mathcal{C}}$ is a $\Z^r$-graded algebra with non-zero graded components corresponding to elements of $\mathcal{C} \cap \MM$, and the $\K$-dimension of each graded piece is $\leq 1$. We also know $\Aa_\MM^{\mathcal{C}}$ is an integral domain so by \cite[Remark 4.13]{BrunsGubeladze} we conclude $\Aa_\MM^{\mathcal{C}}$ is isomorphic to the semigroup algebra of $\mathcal{C} \cap \MM$.  This proves (2). The last claim (3) follows by Gordan's Lemma, since $\mathcal{C}$ is a rational polyhedral cone. 
\end{proof}

\subsection{The compactification associated to a PL polytope of a detropicalization $\Aa_\MM$}\label{subsec:compactification}

In \Cref{definition: M polytope algebra} we introduced the $\Z_{\geq 0}$-graded PL polytope algebra $\Aa^\PP_\MM$. In this section we define and begin to analyze its geometric counterpart. 
Specifically, we define the \textbf{compactification of $\Spec(\Aa_\MM)$ with respect to $\PP$} as
\begin{equation*}\label{eq: def compactification wrt P}
X_{\Aa_\MM}(\PP) := \Proj(\Aa_\MM^\PP).
\end{equation*}
This terminology is justified by the next result.

\begin{theorem}\label{compactifications}
Let $\MM$ be a finite polyptych lattice of rank $r$ over $\Z$ with a fixed choice of strict dual $\Z$-pair $(\MM,\NN, \v, \w)$. Let $(\Aa_\MM,\fv)$ be a detropicalization of $\MM$ with convex adapted basis $\B$. 
Let $\PP = \cap_{i=1}^{\ell} \HH_{\w(n_i), a_i} \subset \MM_\R$ be a full-dimensional PL polytope. Suppose $n_i \in \NN, n_i \neq 0$, the $\HH_{\w(n_i), a_i}$ are pairwise distinct, and $a_i \in \Z_{< 0}$ for all $i \in [\ell]$. Suppose also that for each $n_i$ there exists a coordinate chart $\alpha_i \in \pi(\MM)$ on which $n_i$ is linear, and, the intersection of the boundary of $\pi_{\alpha_i}(\HH_{\w(n_i),a_i})$ with $\pi_{\alpha_i}(\PP)$ is a facet of $\pi_{\alpha_i}(\PP)$. Then the affine scheme (over $\K$) $U_\MM := \Spec(\Aa_\MM)$ can be realized as a dense, open subscheme of $X_{\Aa_\MM}(\PP)$, where the complement of $U_\MM$ in $X_{\Aa_\MM}(\PP)$ is a union of divisors $D_i$ which are in bijection with $n_i$, $i \in [\ell]$.  
\end{theorem}

\begin{proof}
Note that $\Proj(\Aa_\MM[t]) \cong \Spec(\Aa_\MM) = U_\MM$ as schemes.
 Hence the inclusion of graded rings $\Aa_\MM^\PP \hookrightarrow \Aa_\MM[t]$ induces a birational map $\Psi: U_\MM \dashrightarrow X^{\PP}_{\Aa_\MM}$. We wish to show that $\Psi$ is in fact an inclusion of $U_\MM$ into $X^{\PP}_{\Aa_\MM}$ as an open dense subscheme. To prove this, 
we first claim that $\Psi$ is defined everywhere on $U_\MM$. Indeed, note that $\Psi$ is defined on any basic open set of the form  $D(f \cdot t^{k}) \subset \Proj(\Aa_\MM[t])$ where $f \cdot t^{k} \in \Aa_\MM^\PP$ and $k>0$. This is because for any $\mathfrak{p} \in D(f\cdot t^k)$, by definition $f\cdot t^k \not \in \mathfrak{p}$ and since $f \cdot t^k \in \Aa_\MM^\PP$, the prime $\mathfrak{p}$ cannot contain all of $(\Aa_\MM^\PP)_+$. Hence $\Psi(\mathfrak{p}) = \mathfrak{p} \cap \Aa_\MM^\PP$ is a well-defined element of $\Proj(\Aa_\MM^\PP)$. 
From this argument we see that $\Psi$ is defined on the union of all such basic opens. Next we claim that this union is all of $U_\MM$. To see this, let $\mathfrak{p} \in U_\MM$. By definition of $\Proj$, there exists some $f = \sum c_{m,k} \bb_m t^k \in \Aa_\MM[t]$ not contained in $\mathfrak{p}$. (Here $\B = \{\bb_m\}_{m \in \MM}$ is the convex adapted basis of $(\Aa_\MM,\fv)$, and $c_{m,k} \in \K$.) From this it follows that there exists $m \in \MM, k \in \Z_{\geq 0}$ with $\bb_m t^k \not \in \mathfrak{p}$. Next, note that since $0 \in \PP$, $m$ must be contained in the dilated polytope ${k'} \cdot \PP$ for $k' \in \Z$ sufficiently large, and without loss of generality we may assume $k'>k$. Since $\mathfrak{p}$ is in $\Proj(\Aa_\MM[t])$, it cannot contain $t$, and since $\mathfrak{p}$ is prime, we conclude $\bb_m t^{k'} \not \in \mathfrak{p}$. On the other hand, $k'$ is chosen so that $\bb_m t^{k'} \in \Aa_\MM^\PP$. Thus $\mathfrak{p} \in D(\bb_m t^{k'})$, and these distinguished open sets cover all of $U_\MM$. Thus $\Psi$ is defined on all of $U_\MM$.  A similar argument shows that $\Psi$ is an injection, and thus is an inclusion of $U_\MM$ as a subscheme of $X_{\Aa_\MM}(\PP)$. 

Next we claim that the complement $X_{\Aa_\MM}(\PP) \setminus U_\MM$ is a union of divisors; in particular, $U_\MM$ is an open and dense subscheme of $X_{\Aa_\MM}(\PP)$.  Consider the degree-$1$ element $t := 1\cdot t \in \Aa_\MM[t]$.  As was seen in the previous paragraph, for any $f \in \Aa_\MM$ there exists $k>0$ such that $f \cdot t^k \in \Aa_\MM^\PP$. Hence $\left( \frac{1}{t} \Aa_\MM^\PP \right)_0 \cong \Aa_\MM$. Thus the basic open set $D(t)$ is $U_\MM$, and the complement $X_{\Aa_\MM}(\PP) \setminus U_\MM$ is the hypersurface $V(t)$; in particular, $U_\MM$ is open and dense in $X_{\Aa_\MM}(\PP)$. 

We now claim $V(t)$ is a union of divisors $D_i$ corresponding to each $i \in [\ell]$. To see this, fix $i \in [\ell]$. 
Recall from \Cref{lemma: rank 1 valuation} that the pair $(n_i, a_i)$ defines a corresponding valuation $\tilde{\fv}_{n_i,a_i}$ and that, by definition, $\Aa_{\MM}^{n_i, a_i}$ is the intersection $\Aa_\MM[t] \cap \O_{n_i,a_i}$ where $\O_{n_i,a_i}$ is defined in \Cref{eq: def O ni ai}. Let $Q_i := P_i \cap \Aa_\MM^\PP$ where $P_i$ is defined in \Cref{eq: def Pi}. By construction, $Q_i \subset \Aa_\MM^\PP$ is a prime ideal. We claim that $\mathrm{height}(Q_i)=1$. 
To see this, it would suffice to show $\dim(\Aa_\MM^{\PP}/Q_i) = \dim(\Aa_\MM^\PP)-1$ (by the general formula $\mathrm{height}(\mathfrak{p})+\dim(R/\mathfrak{p}) = \dim(R)$ for prime ideals $\mathfrak{p} \subset R$). By hypothesis, there exists a coordinate chart $\alpha_i \in \pi(\MM)$ on which $n_i$ is linear, and, the intersection of the boundary of $\pi_{\alpha_i}(\HH_{\w(n_i),a_i})$ with $\pi_{\alpha_i}(\PP)$ is a facet of $\pi_{\alpha_i}(\PP)$. 
Let $\alpha_i$ be as above and fix a basis $\tilde{\rho}$ of $C_i :=\w^{-1}(\Sp(\MM,\alpha_i))$. 
By \Cref{lemma: rank 1 modify full rank} we know there exists a $\Z$-valued valuation $\overline{\fv}$ such that its associated graded algebra $\mathrm{gr}_{\overline{\fv}}(\Aa_\MM^\PP)$ is $\K[S(\Aa^\PP_\MM, \fv_{\alpha_i,\tilde{\rho}})]$, a semigroup algebra where $S(\Aa^\PP_\MM, \fv_{\alpha_i,\tilde\rho})$ is the saturated affine semigroup of lattice points in the full-dimensional cone over $\pi_{\alpha_i}(\PP)$ in $M_{\alpha_i}\times \Z\cong\Z^{\mathrm{rank}(\MM)+1}$. Let $\mathcal{R}=\bigoplus_{j \geq 0} F_{\overline{\fv}\geq j} \tau^j$ denote the usual Rees algebra associated to the filtration corresponding to $\overline{\fv}$.
It is straightforward from its construction that $\mathcal{R}$ is torsion-free as a $\K[\tau]$-module, hence flat \cite[Corollary 6.3]{Eisenbud}. Thus the dimension $\dim(\Aa_\MM^\PP)$ of the general fiber is equal to the dimension of the special fiber, i.e., $\dim(\Aa_\MM^\PP) = \mathrm{rank}(\MM)+1$. 
Next, since $\Aa_\MM^\PP$ is, by construction, spanned by $\bb_m \cdot t^k$ with $\tilde{\fv}_{n_i,a_i}(\bb_m t^k) \geq 0$ and since $Q_i$ is the subspace spanned by $\bb_m t^k$ with $\tilde{\fv}_{n_i,a_i}(\bb_m t^k) > 0$, the equivalence classes $[\bb_m t^k]$ for $\tilde{\fv}_{n_i,a_i}(\bb_m t^k)=0$ form an additive basis for $\Aa_\MM^\PP/Q_i$. Thus we may define a valuation $\fw$ on $\Aa_\MM^\PP/Q_i$ by defining $\fw([\bb_m t^k]) := \fv(\bb_m t^k)$ for $\bb_m t^k$ with $\tilde{\fv}_{n_i,a_i}(\bb_m t^k)=0$. The image of $\fw$ is now a saturated affine semigroup $S'$ of rank $1$ less, and an argument similar to that given for $\Aa_\MM^\PP$ shows that the quotient ring flatly degenerates to a semigroup algebra $\K[S']$. This proves $\dim(\Aa_\MM^\PP/Q_i)=\dim(\Aa_\MM^\PP)-1$ and completes the proof that $\mathrm{height}(Q_i)=1$.

Finally, we claim that 
\begin{equation}\label{eq: V(tau) is a union of divisors}
\sqrt{\langle t \rangle} = Q_1 \cap \cdots \cap Q_\ell.
\end{equation} 
From this it will follow that $V(t) = D_1 \cup D_2 \cup \cdots \cup D_\ell$ where $D_i := V(Q_i)$.
The previous paragraph shows that $Q_i$ is prime of height $1$, so each $D_i$ is a divisor. Thus it remains to prove~\Cref{eq: V(tau) is a union of divisors}. We first prove that $\sqrt{\langle t \rangle} \subseteq Q_1 \cap \cdots \cap Q_\ell$;  for this, it clearly suffices to show $\sqrt{\langle t \rangle} \subseteq Q_i$ for all $i$. To see this, note that $t \in Q_i$ for all $i$ since $\tilde{\fv}_{n_i,a_i}(t) = - a_i > 0$ by assumption on $a_i$. Then $\sqrt{\langle t \rangle} \subset Q_i$ follows since $Q_i$ is prime. 
To prove the other containment, observe that the intersection of the $Q_i$ can be described as the span of certain $\bb_m \cdot t^k$'s. We claim $\sqrt{\langle t \rangle}$ is also spanned by elements of the form $\bb_m \cdot t^k$. 
Indeed, suppose $f = \sum c_{m,k} \bb_m t^k \in \sqrt{\langle t \rangle}$, where each $\bb_m$ is such that $\fv(\bb_m) = \v(m)$. Then for some $N \in \Z$ sufficiently large, we have $f^N \in \langle t \rangle = \bigoplus_{k \geq 1} \Gamma(\Aa_\MM, (k-1)\PP) \cdot t^k$. 
By taking a vertex of $\supp(f)$, we can find $m_0$ such that $c_{m_0,k} \neq 0$ and $c_{m_0,k}^N \bb_{m_0}^N t^{kN}$ does not cancel with any other monomial occurring in $f^N$. Then $\bb_{m_0}^{N} \in \Gamma(\Aa_\MM, (kN-1)\PP)$ and it follows that $c_{m_0,k} \bb_{m_0} t^k \in \sqrt{\langle t \rangle}$. Then $f-c_{m_0,k}\bb_{m_0} t^k \in \sqrt{\langle t \rangle}$, which has one fewer term in the convex adapted basis, and by repeating the argument we see that $\sqrt{\langle t \rangle}$ is spanned by $\bb_m t^k$'s as desired.
To complete the argument, it now suffices to show that if $\bb_m t^k \in Q_1 \cap \cdots \cap Q_\ell$ then $\bb_m t^k \in \sqrt{\langle t \rangle}$. 
So suppose $\bb_m t^k \in Q_1 \cap \cdots \cap Q_\ell$, i.e., $\v(m)(n_i) > k a_i$ for $i\in[ \ell]$. 
Let $c$ be a positive integer, sufficiently large so that $c\left( \v(m)(n_i) - k a_i\right) \geq -a_i > 0$ for $i\in[ \ell]$. Such an integer $c$ exists since the $-a_i$ for $i \in [\ell]$ is a finite list of positive integers. 
Then $c\v(m)(n_i) - (ck-1)a_i > 0$, which means $\bb_m^c \in \Gamma(\Aa_\MM, \mathcal{H}_{\w(n_i), (ck-1)a_i})$, hence $\bb_m^c t^{ck-1} \in \Aa^{\PP}_{\MM}$, so $(\bb_m t^k)^c = (\bb_m^c t^{ck-1}) t$ is divisible by $t$. Hence $(\bb_m t^k)^c \in \langle t \rangle$ and $\bb_m t^k \in \sqrt{\langle t \rangle}$ as desired. This completes the proof that  $\sqrt{\langle t \rangle} = Q_1 \cap \cdots \cap Q_\ell$. 
\end{proof}

\begin{remark}
    We described a detropicalization $(\Aa,\fv)$ of our running example in \Cref{example: detrop running example}. In \cite{CookEscobarHaradaManon2024} we give a compactification of $\Spec(\Aa)$ associated to a PL polytope. 
    We now briefly describe another method for obtaining compactifications of $\Spec(\Aa)$. The type $A_1$ cluster variety $\Spec(\Aa)$ is also known to be the braid variety associated to the braid word $\beta=\sigma_1\sigma_1\sigma_1$, and braid varieties have compactifications called brick manifolds, as introduced by the first author in \cite{Escobar}. See also \cite[\textsection 4.4]{CasalsGorskyGorskySimental}. The brick manifold corresponding to $\beta$ is $\mathbb{P}^1\times\mathbb{P}^1$.
    \exampleqed
\end{remark}

The next two results record some straightforward consequences of the construction of the compactification $X_{\Aa_\MM}^{\PP}$.

In the classical setting, a (classical) lattice polytope $\mathcal{Q}$ in a vector space $V \cong M \otimes \R$ (for $M$ a classical lattice $\cong \Z^r$) is \textbf{normal} if $(k\mathcal{Q}) \cap M + (\ell \mathcal{Q})\cap M = ((k+\ell)\mathcal{Q}) \cap M$ for all $k,\ell \in \Z_{\geq 0}$ \cite[Definition 2.2.9]{Cox_Little_Schenck}. It is also shown in \cite[Section 2.2]{Cox_Little_Schenck} that a (classical) lattice polytope $\mathcal{Q}$ is normal if and only if 
\begin{equation}\label{eq: normal}
\mathcal{Q} \cap M + \cdots + \mathcal{Q} \cap M = (k\mathcal{Q}) \cap M
\end{equation}
for all integers $k \geq 1$. In other words, for classical lattice polytopes, normality is equivalent to the condition that $\mathcal{Q}$ has enough lattice points to generate the lattice points in all integer multiples of $\mathcal{Q}$. 

In our setting, we make the following definition. 
\begin{definition}
Let $\MM$ be a polyptych lattice. 
    We say that an integral PL polytope $\PP$ is \textbf{normal} if $\pi_{\alpha}(\PP)$ is normal for all $\alpha \in \pi(\MM)$.
\end{definition}

\begin{lemma}\label{lemma: AMP normal if P normal}
    Following the notation and assumptions of this section, let $\PP$ be an integral PL polytope in $\MM_\R$. If $\PP$ is normal, then the PL polytope algebra $\mathcal{A}_\MM^{\PP}$ is generated in degree 1.
\end{lemma}

\begin{proof} 
Suppose that $\PP$ is normal. Choose $\alpha,\tilde{\rho}$ as in \Cref{lem-canonicaltototal} and consider the associated $\Z^r$-valued valuation $\fv_{\alpha,\tilde{\rho}}$. We first claim that the associated graded algebra $\mathrm{gr}_{\fv_{\alpha,\tilde{\rho}}}(\Aa_\MM^\PP)$ is generated in degree $1$. This follows because $\mathrm{gr}_{\fv_{\alpha,\tilde{\rho}}}(\Aa_\MM^\PP)$ is, by construction, the semigroup algebra of a cone over the lattice points contained in $\pi_\alpha(\PP)$. Since we are assuming $\pi_\alpha(\PP)$ is normal, then by the above characterization~\eqref{eq: normal} of normality it follows that the semigroup algebra of the cone over $\pi_\alpha(\PP)$ is generated in degree $1$. 

Now we claim that if $\mathrm{gr}(\Aa_\MM^\PP)$ is generated in degree $1$ then so is $\Aa_\MM^\PP$. Note that $\Aa_\MM^\PP$ is graded by $\Z_{\geq 0}$ (by the degree of the variable $t$) so we may work degree by degree. Let $\Aa^\PP_{\MM,k}$ denote the homogeneous degree $k$ piece of $\Aa^\PP_\MM$. Recall we are trying to show $\Aa_\MM^\PP$ is generated in degree $1$, i.e. by elements in $\Aa_{\MM,1}^\PP$. Let $f \in \Aa^\PP_{\MM,k}$ for $k>1$. 
We wish to show it can be written as a polynomial in elements $\mathbb{b} \in \B \cap \Aa^\PP_{\MM,1}$.
Consider $\fv_{\alpha,\tilde{\rho}}(f) =(\overline{a},k) \in \Z^r \times \Z$. From the definition of $\fv_{\alpha,\tilde{\rho}}$ and because the value semigroup $\Gamma(\Aa^\PP_\MM,\fv)$ is generated in degree $1$ by assumption, it follows that there exists a monomial $\prod_i \bb_i$, with $\bb_i \in \B \cap \Aa^\PP_{\MM,1}$, such that $\fv_{\alpha,\tilde{\rho}}(f) = \fv_{\alpha,\tilde{\rho}}(\prod_i \bb_i)$. WLOG we may assume $\prod_i \bb_i$ has the same homogeneous degree as $f$. The valuation $\fv_{\alpha,\tilde{\rho}}$ has one-dimensional leaves by \Cref{corollary: v alpha rho has one dim leaves} so it follows that there exists a constant $C \in K^*$ such that $\fv_{\alpha,\tilde{\rho}}(f - C\cdot \prod_i \bb_i) > \fv_{\alpha,\tilde{\rho}}(f)$. But $f - C \cdot \prod_i \bb_i$ is homogeneous of degree $k$ and there are only finitely many elements in the value semigroup at level $k$, so repeating this process, we are guaranteed that the algorithm terminates in finitely many steps. The claim follows. 
\end{proof}

 Even when the PL polytope $\PP$ is not integral or normal, we can still obtain results about finite generation, as follows.

\begin{proposition}\label{prop: APM finitely generated} 
Let the notation and assumptions be as in \Cref{compactifications}. 
The algebras $\Aa_\MM^\PP$ and $\Aa_\MM^{n_i,a_i}$ are finitely generated. 
\end{proposition}

\begin{proof} 
We begin by showing $\Aa_\MM^{n_i,a_i}$ is finitely generated. Let $\mathcal{C}$ be a cone in $\Sigma(\MM)$ and define $\Aa_\MM^{n_i,a_i,\mathcal{C}} := \Aa_\MM^{n_i,a_i} \cap \Aa_\MM^{\mathcal{C}}[t]$. This is the $\K$-subspace of $\Aa_\MM[t]$ spanned by $\bb_m \cdot t^k$ where $k\geq 0, k \in \Z$, and $m \in \mathcal{H}_{\w(n_i),ka_i} \cap \mathcal{C}$. By an argument similar to the proof of \Cref{lemma: cone subset of AaMM is subalg} it can be seen that $\Aa_\MM^{n_i,a_i,\mathcal{C}}$ is isomorphic to a semigroup algebra associated to a semigroup $S_{n_i,a_i,\mathcal{C}}$ in $\Z^r \times \Z$. Here $S_{n_i,a_i,\mathcal{C}}$ is defined by the inequalities specifying $\mathcal{C}$ in $\MM$ and the inequality $\tilde{\fv}_i \geq 0$ defining $\Aa_\MM^{n_i,a_i}$ as in~\eqref{eq: def A_niai}. In particular, $S_{n_i,a_i,\mathcal{C}}$ is the set of lattice points lying in a rational polyhedral cone, so by Gordan's Lemma $\Aa_\MM^{n_i,a_i,\mathcal{C}}$ is finitely generated. Since $\Sigma(\MM)$ is a complete fan, any element of $\Aa_\MM^{n_i,a_i}$ can be written as a linear combination of elements of $\Aa_\MM^{n_i,a_i,\mathcal{C}}$ as $\mathcal{C}$ varies over the finitely many cones in $\Sigma(\MM)$. Since each $\Aa_\MM^{n_i,a_i,\mathcal{C}}$ is finitely generated, it follows that $\Aa_\MM^{n_i,a_i}$ is finitely generated. A similar argument shows that $\Aa_\MM^\PP$ is finitely generated. 
\end{proof}

Next, if $\Aa_\MM$ is normal then we can also conclude normality for the compactification, as follows. 

\begin{proposition}\label{prop: normality of compactification} 
Let the notation and assumptions be as in~\Cref{compactifications}. If $\Aa_\MM$ is normal, then 
\begin{enumerate} 
\item $X_{\Aa_\MM}(\PP)$ is normal, and 
\item the valuation $\ord_{D_i}: \Aa_\MM\setminus\{0\} \to \Z$ coincides with the composition of $\fv$ with the point $n_i: \MM \to \Z$ in $\NN$. 
\end{enumerate}
\end{proposition}

\begin{proof}
 First we prove (1). If $\Aa_\MM$ is normal, then $\Aa_\MM[t]$ is normal \cite[Lemma 10.37.8]{StacksProject}. We have seen in the proof of \Cref{compactifications} that the fraction fields of $\Aa_\MM[t]$ and $\Aa_\MM^\PP$ coincide, and a similar argument shows that the discrete valuation rings $\O_{n_i,a_i}$ also has fraction field $\mathcal{K}(\Aa_\MM[t])$. 
Discrete valuation rings are normal \cite[Theorem 7 in Section 16.2]{DummitFoote}, and an intersection of normal domains with the same field of fractions is normal \cite[Exercise 1.0.7]{Cox_Little_Schenck}, so $\Aa_\MM^\PP = \Aa_\MM[t] \cap (\cap_{i \in [\ell]} \O_{n_i,a_i})$  is normal. This proves (1). 

We now prove (2).  By (1), we know that $\Aa_{\MM}^{\PP}$ is normal, so it makes sense to discuss an order of vanishing along a divisor. Let $i \in [\ell]$. We claim that the localization $\O_{P_i}$ of $\Aa_\MM^{n_i, a_i}$ at the prime ideal $P_i$ coincides with $\O_{n_i,a_i}$. 
To see this, we first claim that the localization $\O_{P_i}$ is a discrete valuation ring.  Indeed, by assumption, $\Aa_\MM$ is a Noetherian integral domain which is integrally closed, and hence so is $\Aa_\MM[t]$.  By \Cref{prop: APM finitely generated} we know that $\Aa_{\MM}^{n_i,a_i}$ is finitely generated, hence Noetherian, and the argument given above to show (1) yields that $\Aa_\MM^{n_i,a_i}$ is normal. It is also an integral domain since it is a subring of $\Aa_\MM[t]$. The ideal $P_i$ is prime and in fact a minimal non-zero prime since it is height $1$, as seen in the proof of \Cref{compactifications}. The above allows us to conclude that $\O_{P_i}$ is a discrete valuation ring \cite[Theorem 7 in Section 16.2]{DummitFoote}. 
We have that $\O_{P_i}$ and $\O_{n_i,a_i}$ have the same fraction field, and both are non-zero proper discrete valuation rings. It is also straightforward to see that $\O_{P_i}$ is a subring of $\O_{n_i,a_i}$. But discrete valuation rings are maximal proper subrings in their fraction fields, so the containment cannot be strict, and we conclude $\O_{P_i} = \O_{n_i,a_i}$, as was to be shown.  From the above, the valuation $\mathrm{ord}_{D_i}$ may therefore be identified with $\tilde{\fv}_{n_i,a_i}$, and for an element $f=\sum c_m b_m$ in $\Aa_\MM$ (viewed as a regular function on $U_\MM$) its order of vanishing along $D_i$ is given by the minimum of the quantities $\langle n_i, m\rangle$. This is precisely the composition of $\fv$ with $n_i$. 
\end{proof}

\subsection{Compactifications $X_{\Aa_\MM}(\PP)$ are arithmetically Cohen-Macaulay}\label{subsec: Cohen Macaulay}

The main result of this section is to show that the compactifications $X_{\Aa_\MM}(\PP)$ constructed in the previous section are arithmetically Cohen-Macaulay (\Cref{definition: ACM}). This generalizes an analogous result for normal projective toric varieties arising from integral polytopes \cite[Exercise 9.2.8]{Cox_Little_Schenck}. 

\begin{definition}\label{definition: ACM}
Let $X = \mathrm{Proj}(S)$ be a projective variety over $k$. We say that $X$ is \textbf{arithmetically Cohen-Macaulay} if there exists an ample divisor $D$ with respect to which the section ring $\mathcal{R}_D := \oplus_{k \geq 0} \Gamma(X, kD)$ is Cohen-Macaulay. 
\exampleqed
\end{definition}

We have the following. 

\begin{theorem}\label{theorem: Cohen Macaulay}
Let the notation and assumptions be as in \Cref{compactifications}. 
In particular, let $\Aa_\MM^{\PP}$ denote the PL polytope ring associated to $\PP$ and $X_{\Aa_\MM}(\PP) := \mathrm{Proj}(\Aa_\MM^\PP)$ be the compactification of $\Spec(\Aa_\MM)$ with respect to $\PP$. Then $X_{\Aa_\MM}(\PP)$ is arithmetically Cohen-Macaulay. 
\end{theorem}

\begin{proof} 
We have seen in the setting of, and in the proof of, \Cref{lemma: rank 1 modify full rank} that the associated graded algebra $\mathrm{gr}_{\fv_{\alpha,\tilde{\rho}}}(\Aa_\MM^\PP)$ is isomorphic to a semigroup algebra $\K[S]$ for $S$ a saturated affine semigroup. Moreover, again by the proof of \Cref{lemma: rank 1 modify full rank}, we know there exists a $\overline{\Z}$-valued valuation $\bar{\fv}$ with $\mathrm{gr}_{\bar{\fv}}(\Aa_\MM^\PP)$ isomorphic to $\mathrm{gr}_{\fv_{\alpha,\tilde{\rho}}}(\Aa_\MM^\PP) \cong \K[S]$. 
Recall that by a theorem of Hochster, $\mathbb{K}[S]$ is Cohen-Macaulay \cite[Theorem 1]{Hochster}; this implies $\mathrm{gr}_{\fv_{\alpha,\tilde{\rho}}}(\Aa_\MM^\PP)$ is Cohen-Macaulay. 
Now we claim that if $\Aa_\MM^\PP$ has an associated graded $\mathrm{gr}_{\bar{\fv}}(\Aa_\MM^\PP)$ which is Cohen-Macaulay, then $\Aa_\MM^\PP$ is itself Cohen-Macaulay. To see this, consider the decreasing filtration $\mathcal{F}_{\bar{\fv}}$ of $\Aa_\MM^\PP$ associated to $\bar{\fv}$ and the (extended) Rees algebra $\mathcal{R} := \bigoplus_{k \in \Z} F_{\bar{\fv} \geq k} \cdot \tau^{-k}$ corresponding to $\mathcal{F}_{\bar{\fv}}$, viewed as a subalgebra of $\Aa_\MM^\PP[\tau,\tau^{-1}]$. By construction, the associated graded $\mathrm{gr}_{\mathcal{F}_{\bar{\fv}}}(\mathcal{R})$ as defined in \cite[\textsection 4.5]{Bruns_Herzog}
is isomorphic to $\mathrm{gr}_{\bar{\fv}}(\Aa_\MM^\PP)$. 
It is straightforward to see that $\mathcal{F}_{\bar{\fv}}$ has the property that it is strongly separated in the sense of \cite[\textsection 4.5]{Bruns_Herzog}, i.e., $\bigcap_{k \geq 0} \left(I + F_{\bar{\fv}\geq k}\right) = I$ for all ideals $I \subset \Aa_\MM^\PP$. 
Moreover, we have already seen that $\mathrm{gr}_{\bar{\fv}}(\Aa_\MM^\PP) \cong \mathbb{K}[S]$ and is hence finitely generated over $\mathbb{K}$. Finally, $\Aa_\MM^\PP/F_{\bar{\fv}\geq 1}$ is isomorphic to $\mathbb{K}$ by construction of $\bar{\fv}$. 
Thus we may apply \cite[Proposition 4.5.4]{Bruns_Herzog} to conclude that $\mathcal{F}_{\bar{\fv}}$ is a Noetherian filtration of $\Aa_\MM^\PP$. Then by \cite[Theorem 4.5.7]{Bruns_Herzog} applied to $R=\Aa_\MM^\PP$ and $F = \mathcal{F}_{\bar{\fv}}$, and taking $\mathfrak{p}$ to be the ${}^{*}$-maximal ideal $\mathfrak{m} := F_{\bar{\fv}\geq 1}$, we may conclude that the localization $(\Aa_\MM^\PP)_{\mathfrak{m}}$ is Cohen-Macaulay, since we saw  $\mathrm{gr}_{\bar{\fv}}(\Aa_\MM^\PP) \cong \mathbb{K}[S]$ is Cohen-Macaulay above. We know $\Aa_\MM^\PP$ is a Noetherian graded ${}^*$-local ring with ${}^*$-maximal ideal $\mathfrak{m}$, so by \cite[Exercise 2.1.27]{Bruns_Herzog}, we conclude that $\Aa_\MM^\PP$ is Cohen-Macaulay. 

 The ring $\Aa^\PP_\MM$ is the section ring of the divisor $D=V(t)$ as described in the proof of \Cref{compactifications}. In fact, this divisor is ample; indeed, $\Aa^\PP_\MM$ is graded and is finitely generated (by homogeneous elements) so it follows that for some large enough $K \in \Z, K>0$, the $K$-th Veronese subring of $\Aa^\PP_\MM$ is generated in degree $1$ \cite[Exercise 7.4.G]{Vakil}. Moreover, $K \cdot D$ is basepoint-free since $X_{\Aa_\MM}(\PP)$ is by definition the $\mathrm{Proj}$ of $\Aa_\MM^\PP$. This proves the claim.  
 \end{proof}

\subsection{Finite generation of Cox rings of $X_{\Aa_\MM}(\PP)$}\label{subsec: Cox rings}

We keep the assumptions and notation of Sections~\ref{subsec: compactification definition} to~\ref{subsec: Cohen Macaulay}. The purpose of this section is to show that, if a detropicalization $\Aa_\MM$ is a unique factorization domain (UFD), then for any PL polytope $\PP$, the compactification $X_{\Aa_\MM}(\PP)$ will have both a finitely generated class group and a finitely generated Cox ring. This result is motivated by the theory of Mori dream spaces, since it is a well-known theorem of Hu and Keel \cite{HuKeel} that a normal projective $\Q$-factorial variety with finitely generated class group is a Mori dream space if and only if its Cox ring is finitely generated. In general, the question of finite generation is challenging; for example, the problem of finding Mori dream spaces among toric vector bundles has been studied for many years (see e.g.\ \cite{HMP10, HS10, Gon12, GHPS12}). Our result can be viewed as a first step in a search for Mori dream spaces among our compactifications $X_{\Aa_\MM}(\PP)$. For this discussion we additionally assume that \textbf{$\K$ is of characteristic $0$}.

\begin{theorem}\label{theorem: finitely generated Cox rings} 
Let the notation and assumptions be as in \Cref{compactifications}. 
Let $\Aa_\MM^{\PP}$ denote the PL polytope ring associated to $\PP = \cap_{i=1}^{\ell} \HH_{\w(n_i), a_i} \subset \MM_\R$ and $X_{\Aa_\MM}(\PP) := \mathrm{Proj}(\Aa_\MM^\PP)$. 
If $\Aa_\MM$ is a unique factorization domain, then $X_{\Aa_\MM}(\PP)$ has a finitely generated class group and a finitely generated Cox ring. 
\end{theorem}

\begin{proof} 
We begin by showing that $X_{\Aa_\MM}(\PP)$ has a finitely generated class group. 
Since $U_\MM = \Spec(\Aa_\MM)$ is an open dense subset of $X_{\Aa_\MM}(\PP)$ and the complement $X_{\Aa_\MM}(\PP) \setminus U_\MM$ is the union of the prime divisors $D_i$ for $i \in [\ell]$, we have an exact sequence \cite[Theorem 4.0.20]{Cox_Little_Schenck} 
$$ \bigoplus_{i = 1}^\ell \Z D_i \to \Cl(X_{\Aa_\MM}(\PP)) \to \Cl(\Spec(\Aa_\MM)) \to 0.$$\\
If $\Aa_\MM$ is a UFD, then $\Cl(\Spec(\Aa_\MM))=0$ by e.g.\ \cite[Theorem 4.0.18]{Cox_Little_Schenck}, and it follows that $\Cl(X_{\Aa_\MM}(\PP))$ is finitely generated by the images of the prime divisors $D_i$. 

We now show that the Cox ring is finitely generated.  First, we know from \cite[Construction 1.4.2.1]{CoxRingsBible} that the Cox ring $\textup{Cox}(X_{\Aa_\MM}(\PP))$ may be realized as a quotient algebra of the global sections algebra $\Gamma(X_{\Aa_\MM}(\PP), \mathcal{S})$ of a sheaf $\mathcal{S} = \bigoplus_{D \in K} \mathcal{O}_{X_{\Aa_\MM}(\PP)}(D)$ of divisorial algebras. In \cite{CoxRingsBible}, the direct sum is over $K$ a finitely generated subgroup of the group of Weil divisors, which surjects onto the class group.  In our setting we choose $K = \bigoplus_{i \in [\ell]}\Z D_i$ since we just saw above that this surjects onto $\Cl(X_{\Aa_\MM}(\PP))$. Thus, in our case, the global sections algebra under consideration is 
\begin{equation*}
    \mathcal{R} := \Gamma(X_{\Aa_\MM}(\PP), \mathcal{S}) = \bigoplus_{\overline{r} \in \Z^\ell} \Gamma(X_{\Aa_\MM}(\PP), \mathcal{O}(\sum_i r_i D_i))
.\end{equation*} 

We now describe $\mathcal{R}$ in a different way. Fix $\overline{r} \in \Z^{\ell}$. The global sections $\Gamma(X_{\Aa_\MM}(\PP), \mathcal{O}(\sum_i r_i D_i))$ can be identified with a subspace of the regular functions on $\Spec(\Aa_\MM)$ with restrictions on their zeroes and poles along the $D_i$, as specified by the parameters $r_i$. Note that a UFD is normal, so by \Cref{prop: normality of compactification} we know $\mathrm{ord}_{D_i}$ can be identified with pairing with $n_i$. Moreover, since the space of regular functions on $\Spec(\Aa_\MM)$ is $\Aa_\MM$, we see that 
$\Gamma(X_{\Aa_\MM}(\PP), \mathcal{O}(\sum_i r_i D_i)) \cong \mathrm{span}_\K\{ \bb_m \, \mid \, \langle n_i, m \rangle \geq -r_i \, \textup{ for all} \, i\in [\ell] \} \subset \Aa_\MM$.
By adjoining formal parameters $t_1^{\pm}, \cdots, t_\ell^{\pm}$ to record the $\Z^\ell$-grading, we may view $\mathcal{R}$ as a subalgebra of $\Aa_\MM[t_1^{\pm}, \cdots, t_\ell^{\pm}]$.
Now from \Cref{prop: direct product of detrop} we know $\Aa_\MM[t_1^{\pm},\cdots,t_\ell^{\pm}]$ is a detropicalization of $\MM \times \Z^\ell$ with convex adapted basis $\{\bb_m t_1^{r_1} \cdots t_\ell^{r_\ell} \, \mid \, \bb_m \in \B, \overline{r} \in \Z^\ell\}$ and valuation $\fv_{\MM \times \Z^\ell}$ with the property $\fv_{\MM \times \Z^\ell}(\bb_m \overline{t}^{\overline{r}}) = (m,\overline{r}) \in \MM \times \Z^\ell$. 
Next consider the points $(n_i, \varepsilon_i) \in \Sp(\MM \times \Z^\ell) \cong \Sp(\MM) \times \Z^\ell$ where the $\varepsilon_i$ denote standard basis vectors in $\Z^\ell$ and we have used \Cref{lemma: points of product PL} for the isomorphism. 
Let $\mathcal{C} := \bigcap_{i\in[\ell]} \HH_{(\w(n_i),\varepsilon_i),0}$ denote the intersection of the PL half-spaces in $\MM_\R \times \R^\ell \cong (\MM \times \Z^\ell)_{\R}$ defined by the points $(n_i,\varepsilon_i)$ above, and parameters $a_i=0$ for all $i$. 
From the explicit descriptions of $\Gamma(X_{\Aa_\MM}(\PP), \mathcal{O}(\sum_i r_i D_i))$ given above, it then follows that $\mathcal{R}$ can be identified with the subalgebra $\Aa_{\MM \times \Z^\ell}^{\mathcal{C}} := \bigoplus_{\overline{r} \in \Z^\ell} \Gamma(\Aa_{\MM \times \Z^\ell}, \HH_{(\w(n_i),\varepsilon_i),0})\cdot \overline{t}^{\overline{r}}$ of $\Aa_{\MM}[t_1^{\pm}, \cdots, t_\ell^{\pm}]$. 

With these preliminaries in place, the rest of the argument is similar to previous proofs in this section, so we keep it brief. First, choose a $\Z^r \times \Z^\ell$-valued valuation as in \Cref{lem-canonicaltototal} corresponding to a choice of coordinate chart $M_\alpha \times \Z^\ell$. This valuation has one-dimensional leaves, and its associated graded algebra is the semigroup algebra $\K[S]$ of the semigroup $S$ of lattice points contained in the cone $\pi_\alpha(\mathcal{C})$. In particular, since $\pi_\alpha(\mathcal{C})$ is an intersection of half-spaces (over $\Z$) and hence a rational polyhedral cone, by Gordan's Lemma $\K[S]$ is finitely generated. 
To see that the original algebra is finitely generated, we may use a subduction argument similar to that in the proof of \Cref{prop: APM finitely generated}, but this time using homogeneity with respect to the $\Z^\ell$-grading. We also use the fact, recounted in \Cref{lemma: once compact always compact}, that if the original PL polytope $\PP$ is compact then the slice of the cone $\mathcal{C}$ at any $\overline{r}$-level set for $\overline{r} \in \Z^\ell$ is compact (and hence contains only finitely many lattice points).  We leave details to the reader. From this we see that $\mathcal{R}$ is finitely generated, and hence the Cox ring is also finitely generated. 
\end{proof}

We will see a concrete example of a compactification $X_{\Aa_\MM}(\PP)$ with finitely generated Cox ring in \Cref{subsec: Adr compactification}.

\begin{remark} 
Note that $\Aa^C_{\MM\times \Z^\ell}$ can be interpreted as a detropicalization of the cone $C \subset \MM\times \Z^\ell$.  The Cox ring of $X_{\Aa_\MM}(\PP)$ can therefore be viewed as a detropicalization.
\end{remark}

%-------------------------------------%%%
\subsection{Families of toric degenerations and Newton-Okounkov bodies}\label{subsec: families of degenerations}

In this section we state some first results relating the constructions given in \Cref{subsec: compactification definition} to \Cref{subsec: Cox rings}, and, the theories of Newton-Okounkov bodies and toric degenerations. Experts will have already noted that both of these concepts are lurking in the background throughout the paper; here we make them explicit, for future reference. The results we state here are straightforward consequences of results already developed in the preceding sections so we keep discussion brief. 

We maintain the running hypotheses of \Cref{subsec: compactification definition} to \Cref{subsec: Cox rings}. In this situation, recall that we constructed in \Cref{subsec: compactification definition} (in the paragraph before \Cref{lemma: rank 1 modify full rank}) a valuation $\fv_{\alpha,\tilde{\rho}}: \Aa_\MM^\PP \to \Z^{\mathrm{rank}(\MM)} \times \Z$. By construction, $\fv_{\alpha,\tilde{\rho}}$ is a homogeneous valuation (see e.g.\ \cite[\textsection 2]{EscobarHarada} for a definition) on the homogeneous coordinate ring $\Aa_\MM^\PP$ of the compactification $X_{\Aa_\MM}(\PP) := \mathrm{Proj}(\Aa_\MM^\PP)$.  Thus we may compute the associated Newton-Okounkov body (see e.g.\ \cite[Definition 2.2]{EscobarHarada}) of $X_{\Aa_\MM}(\PP)$, or equivalently, of $\Aa_\MM^\PP$, with respect to this valuation. The following is an immediate consequence of the theory developed thus far (cf.~also the proof of \Cref{lemma: rank 1 modify full rank}). 

\begin{theorem}\label{theorem: NO body}
Let the assumptions and notation be as in \Cref{compactifications}. Let $\fv_{\alpha,\tilde{\rho}}: \Aa_\MM^\PP \to \Z^{\mathrm{rank}(\MM)} \times \Z$ denote the homogeneous valuation defined in \Cref{subsec: compactification definition} (paragraph before \Cref{lemma: rank 1 modify full rank}). 
Then the Newton-Okounkov body $\Delta(\Aa_\MM^\PP, \fv_{\alpha,\tilde{\rho}})$ is the (classical) polytope $\pi_\alpha(\PP)$. 
\end{theorem}

By construction, a PL polytope $\PP$ maps to a classical polytope $\pi_\alpha(\PP)$ for each choice of coordinate chart $\alpha \in \pi(\MM)$, and these polytopes $\pi_\alpha(\PP)$ satisfy the relation $\mu_{\alpha,\beta}(\pi_\alpha(\PP)) = \pi_\beta(\PP)$. Thus \Cref{theorem: NO body} asserts that there is a family of valuations on $\Aa^\PP_\MM$ which yield Newton-Okounkov bodies which are mutation-related. In this sense, the theory of polyptych lattices, and PL polytope algebras, as developed in this paper systematizes the wall-crossing phenomena observed in \cite{EscobarHarada}, and we can think of a PL polytope as a ``global object'' encoding a family of mutation-related Newton-Okounkov bodies. 

The theorem below also follows immediately from the arguments given in \Cref{subsec: compactification definition} and \Cref{subsec:compactification}, as well as the identification given in \Cref{theorem: NO body} between $\Delta(\Aa_\MM^\PP, \fv_{\alpha,\tilde{\rho}})$ and $\pi_\alpha(\PP)$. 

\begin{theorem}\label{theorem: family of degenerations}
Let the assumptions and notation be as in \Cref{compactifications}. Assume that $\PP$ is an integral PL polytope. Then, for each choice of homogeneous valuation $\fv_{\alpha,\tilde{\rho}}$ as constructed in \Cref{subsec: compactification definition}, there exists a toric degeneration $\mathcal{X}_{\alpha,\tilde{\rho}} \to \Spec(\K[t])$ with generic fiber isomorphic to the compactification $X_{\Aa_\MM}(\PP) := \mathrm{Proj}(\Aa_\MM^\PP)$ and central fiber $X(\Delta(\Aa_\MM,\fv_{\alpha,\tilde{\rho}})) = X(\pi_\alpha(\PP))$ the toric variety associated to the (classical) integral polytope $\pi_\alpha(\PP) = \Delta(\Aa_\MM,\fv_{\alpha,\tilde{\rho}})$. 
\end{theorem} 

\begin{proof} 
The associated graded algebra of $\Aa_\MM^\PP$ corresponding to $\fv_{\alpha,\tilde{\rho}}$ is the semigroup algebra $\K[S(\Aa^\PP_\MM,\fv_{\alpha,\tilde{\rho}})]$ where $S(\Aa^\PP_\MM,\fv_{\alpha,\tilde{\rho}})$ is the semigroup of lattice points in the cone over the polytope $\pi_\alpha(\PP)$, as explained in the proof of \Cref{lemma: rank 1 modify full rank}. The flatness of the associated family is shown in the proof of \Cref{compactifications}. 
\end{proof}

From \Cref{theorem: NO body} we saw that our theory gives us a combinatorial family of mutation-related Newton-Okounkov bodies. \Cref{theorem: family of degenerations} develops this further, giving us a family of toric degenerations of the single geometric object, namely $X_{\Aa_\MM}(\PP)$, where the toric varieties arising as the central fibers are associated to these Newton-Okounkov bodies. Thus we may interpret \Cref{theorem: family of degenerations} as a systematic geometric realization of the combinatorial data  obtained in \Cref{theorem: NO body}, and the compactification $X_{\Aa_\MM}(\PP)$ as a ``global object'' which geometrically interpolates between the family of toric varieties $\{X(\pi_\alpha(\PP))\}_{\alpha \in \pi(\MM)}$.

\section{Example: the polyptych lattices $\MM_{d,r}$}\label{sec-Example}

In this section, we give a concrete family, denoted $\MM_{d,r}$, of examples of polyptych lattices that serve to illustrate the theory developed in the previous sections. In \Cref{subsec: Gorenstein Fano for Mdr} we also give explicit examples of Gorenstein-Fano PL polytopes in $\MM_{d,r} \otimes \R$.

\subsection{Definition of $\MM_{d,r}$.}\label{subsec: def MMdr}

Let $d,r$ be positive integers, $d \geq 2, r \geq 2$. We now define a polyptych lattice $\MM_{d,r}$ of rank $d+r-1$ over $\Z$, associated to $(d,r)$, and describe some of its basic properties. 
Throughout this discussion, we use coordinates $\bu = (u_1,\cdots,u_d) \in \Z^d, \bw=(w_1,\cdots,w_r) \in \Z^r$. 
We begin by defining the $r$ many coordinate charts of $\MM_{d,r}$. Specifically, for $i \in [r]$ we let $M_{d,r}^{(i)}$ denote the subgroup of $\Z^{d}\times\Z^{r}$ of rank $d+r-1$ given by
\begin{equation*}
    M_{d,r}^{(i)}:=\left\{(\bu,\bw)\in\Z^{d}\times\Z^{r}\mid w_i=0\right\}.
\end{equation*}
The mutation maps between the charts $M^{(i)}_{d,r}$ are defined as follows. For each $i\in [r-1]$ we define $\mu_{i,i+1}:M_{d,r}^{(i)}\to M_{d,r}^{(i+1)}$ by 
    \begin{equation*}
        \mu_{i,i+1}(\bu,\bw):=(\bu,w_1,\ldots,w_{i-1},\min\{u_1,\ldots,u_d\}-\sum w_k,0,w_{i+2},\ldots,w_r).
    \end{equation*}
The mutation map $\mu_{i,i+1}$ is invertible, with inverse given by 
\begin{equation*}
        \mu_{i,i+1}^{-1}(\bu,\bw)=(\bu,w_1,\ldots,w_{i-1},0, \min\{u_1,\ldots,u_d\}-\sum w_k,w_{i+2},\ldots,w_r).
    \end{equation*}
Note that $\mu_{i,i+1}$ has $d$ regions of linearity, where the $k$-th region, for $k \in [d]$, is given by $u_k=\min\{u_1,\ldots,u_d\}$, and this is independent of $i$. 
By composing the maps $\mu_{i,i+1}$ and their inverses, we obtain compatible mutation maps $\mu_{i,j}: M^{(i)}_{d,r} \to M^{(j)}_{d,r}$ for all pairs $i,j \in [r]$; it is straightforward that these satisfy the conditions in \Cref{definition: polyptych lattice}. 

As noted above, the regions of linearity for all mutation maps $\mu_{i,j}$ are the same, and it follows that the maximal cones of $\Sigma(\MM_{d,r})$ are those specified by the equations $u_k=\min\{u_1,\cdots,u_d\}$, for $k \in [d]$. For concreteness we may identify $\Sigma(\MM_{d,r})$ with its image in the first coordinate chart $M^{(1)}_{d,r} \otimes \R$, i.e., the classical fan with maximal cones $C_k:=\{(\bu,\bw)\in M_{d,r}^{(1)}\otimes \R\mid u_k=\min(u_1,\ldots,u_d)\}$ for $k\in[d]$. (In fact the same equation $u_k=\min\{u_1,\cdots,u_d\}$ specifies the cones in all charts $M_{d,r}^{(i)}$ since the mutation maps leave the $\bu$ coordinates unchanged.)

We will find it useful to identify $\MM_{d,r}$ with the following subset of $\Z^{d}\times\Z^{r}$: 
\begin{equation*}\label{eq: def Mdr}
    \M_{d,r} := \{(\bu,\bw)\in \Z^{d}\times\Z^{r}\mid \min\{u_1,\ldots,u_d\}=0\} \subset \Z^d \times \Z^r. 
\end{equation*}
Note that $\MM_{d,r}$ is not a subgroup of $\Z^d \times \Z^r$. However, we can define for $k \in [d]$ the subset
\begin{equation*}\label{eq: def Mdr k}
\M_{d,r}(k) := \{(\bu,\bw) \in \M_{d,r} \, \mid \, u_k=0\}
\end{equation*}
which is closed under addition, and it is straightforward that $\M_{d,r} = \cup_{k \in [d]} \M_{d,r}(k)$. 
To see the identification of $\M_{d,r}$ with $\MM_{d,r}$, we first define the notation $\bo:=\varepsilon_1+\cdots+\varepsilon_d = (1,1,\cdots,1) \in \Z^d$ (here $\{\varepsilon_1,\cdots,\varepsilon_d\}$ denote the standard basis of $\Z^d$) and define $\pi_i: \Z^d \times \Z^r \to \Z^d \times \Z^r$ to be the projection that sets the $i$-th coordinate of $\Z^r$ to be $0$, i.e. $\pi_i(\bu, w_1, \cdots, w_r) = (\bu, w_1,\cdots,w_{i-1}, 0, w_{i+1}, \cdots, w_r)$. 
For $i \in [r]$, we then define maps $\varphi_i$ as follows: 
\begin{equation*}
    \varphi_i:\M_{d,r}\to M_{d,r}^{(i)},\qquad
    \varphi_i(\bu,\bw):=\pi_i\left(\bu + \langle \bo,\bw\rangle\bo, \bw\right). 
\end{equation*}
These maps $\varphi_i$ behave well with respect to the mutation maps $\mu_{i,i+1}$, as recorded in the following lemma.

\begin{lemma}\label{lem_Mdr_identified}
    For all $i \in [r]$, the map $\varphi_i$ is a bijection, and for all $i \in [r-1]$, $\mu_{i,i+1}\circ \varphi_i=\varphi_{i+1}$.
\end{lemma}

\begin{proof}
    We first show that $\varphi_i$ is one-to-one. 
    Note that if $\varphi_i(\bu,\bw)=\varphi_i(\bu',\bw')$, then $w_k=w'_k$ for all $k\neq i$. 
    Moreover, since $u_j+\sum_{k\in[r]}w_k=u'_j+\sum_{k\in[r]}w'_k$ for all $j \in [d]$, and since $\min\{u_1,\cdots,u_d\} = 0$ (so at least one $u_j$ is equal to $0$) we can conclude $\sum_{k \in [r]} w_k = \sum_{k \in [r]} w'_k$. 
       From this it follows that $(\bu,\bw)=(\bu',\bw')$, so $\varphi_i$ is one-to-one. 
    In fact, $\varphi_i$ is also onto, since an inverse is given by 
\begin{equation}\label{eq: varphi inverse}
    \varphi_i^{-1}(\bu,\bw)
    =
    \left(\bu-\min(u_1,\ldots,u_d)\bo,\ \bw + (\min(u_1,\ldots,u_d)-\langle \bo,\bw\rangle)\varepsilon_i\right)
.\end{equation}
The claim $\mu_{i,j}\circ \varphi_i=\varphi_j$ is a straightforward computation which we leave to the reader. 
\end{proof}

It follows immediately from \Cref{lem_Mdr_identified} that the map
$\M_{d,r} \to \MM_{d,r}$ defined by 
\begin{equation*}\label{eq: identify Mdr MMdr}
(\bu, \bw) \mapsto ( \varphi_i(\bu, \bw))_{i \in [r]} 
\end{equation*}
is well-defined, and bijective. We may therefore make computations for $\MM_{d,r}$ in terms of $\M_{d,r}$.

%-------------------------------------%%%%%
\subsection{Points of $\MM_{d,r}$}\label{subsec: points Mdr}
%-------------------------------------%%%%%

Our next goal is to describe the set $\Sp(\MM_{d,r})$ of points of $\MM_{d,r}$. As we saw in \Cref{lem_Mdr_identified}, the elements of $\MM_{d,r}$ may be identified with $\M_{d,r}$, so a point $p \in \Sp(\MM_{d,r})$ may be thought of as a function on $\M_{d,r}$. 
With this in mind, we define a subset, denoted $\Sp(\M_{d,r})$, of functions on $\M_{d,r}$ as follows: 
\begin{equation}\label{eq: def SpMdr}
    \Sp(\M_{d,r}):=
    \left\{f: \M_{d,r}\to \Z \mid f(\bu,\bw)+f(\bu',\bw')
    =\min_{i\in[r]}\{f\circ\varphi_i^{-1}(\varphi_i(\bu,\bw)+\varphi_i(\bu',\bw'))\}\right\}. 
\end{equation}
The following is immediate using the identification given in \Cref{subsec: def MMdr}. We leave details to the reader. 

\begin{lemma}\label{lem_pts_Bdr}
    The space of points $\Sp(\MM_{d,r})$ is in bijection with $\Sp(\M_{d,r})$. For $i \in [r]$, 
$\Sp(\MM_{d,r},i)$ is in bijection with
  $  \Sp(\M_{d,r},i):=
    \left\{f\in \Sp(\M_{d,r}) \mid f(\bu,\bw)+f(\bu',\bw')
    =f\circ\varphi_i^{-1}(\varphi_i(\bu,\bw)+\varphi_i(\bu',\bw'))\right\}.$
\exampleqed
\end{lemma}

We now describe another set which we will identify with $\Sp(\MM_{d,r}) \cong \Sp(\M_{d,r})$. We define 
\[T_{d,r} := \{(\ba, \b)\in \Z^{d}\times\Z^{r} \mid a_1 + \cdots + a_d = \min\{b_1, \ldots, b_r\} \} \subset \Z^d \times \Z^r\]
and 
$$T_{d, r}(i) := \{(\ba, \b)\in T_{d,r} \mid a_1 + \cdots +a_d = b_i\}$$
for $i \in [r]$. Note that $T_{d,r} = \cup_{i \in [r]} T_{d,r}(i)$.
Moreover, given $(\ba,\b)\in T_{d,r}$, we define a function $f_{\ba,\b}$ as follows:
\begin{equation*}\label{eq_Mdr_fab}
    f_{\ba,\b}:\M_{d,r}\to \Z,\qquad
    f_{\ba,\b}(\bu,\bw):=\langle \ba,\bu\rangle+\langle \b,\bw\rangle
    =a_1u_1+\cdots+a_du_d+b_1w_1+\cdots+b_rw_r
.\end{equation*}
Note that $f_{\ba,\b}$ is simply the restriction of the usual inner product pairing with $(\ba,\b)$, which is linear on $\Z^d \times \Z^r$, to the subset $\M_{d,r}$. However, since $\M_{d,r}$ is not an abelian subgroup of $\Z^d \times \Z^r$, it does not make sense to discuss linearity on $\M_{d,r}$. Now define the map 
$$
\Psi_{d,r}: T_{d,r} \to \{f: \M_{d,r} \to \Z\}
$$
by $\Psi_{d,r}(\ba,\b) := f_{\ba,\b}$. The following proposition shows that $\Psi$ identifies $\Sp(\MM_{d,r}) \cong \Sp(\M_{d,r})$ with $T_{d,r}$.

\begin{proposition}\label{prop_Mdr_pts}
The map $\Psi_{d,r}$ defined above 
has image $\Sp(\M_{d,r})$ and is a bijection between $T_{d,r}$ and $\Sp(\M_{d,r})$. 
Moreover, for all $i \in [r]$, the map $\Psi_{d,r}$ restricts to give a bijection between $T_{d,r}(i)$ and $\Sp(\M_{d,r},i)$. 
In particular, $\MM_{d,r}$ is a full polyptych lattice.
\end{proposition}

\begin{proof} 
We first show that $\Psi_{d,r}(\ba,\b) = f_{\ba,\b} $ is an element of $\Sp(\M_{d,r})$, i.e., $f_{\ba,\b}$ satisfies the condition given in the RHS of \Cref{eq: def SpMdr}. 
A straightforward computation shows that for all $i \in [r]$ and all $(\bu,\bw),(\bu',\bw') \in \M_{d,r}$, we have 
\begin{equation}\label{eq_ptsMdr_eval}
    \varphi_i^{-1}(\varphi_i(\bu,\bw)+\varphi_i(\bu',\bw'))
    =
    \left(\bu+\bu'-\min_{j\in[d]}\{u_j+u'_j\}\bo,\ \bw+\bw'+\min_{j\in[d]}\{u_j+u'_j\}\varepsilon_i\right)
,\end{equation}
so that
\begin{equation*}
    f_{\ba,\b}(\varphi_i^{-1}(\varphi_i(\bu,\bw)+\varphi_i(\bu',\bw')))
    =
    \langle \ba,\bu+\bu'\rangle+\langle \b,\bw+\bw'\rangle+\left(b_i-\sum_{j\in[d]}a_j\right)\min_{j\in[d]}\{u_j+u_j'\}
.\end{equation*}
By the hypothesis that $(\ba,\b)\in T_{d,r}$, we have that $\min_{i\in[d]}\left\{b_i-\sum_{j\in[d]}a_j\right\}=0$,
so 
\begin{equation*}
    \min_{i\in[r]}\{f_{\ba,\b}(\varphi_i^{-1}(\varphi_i(\bu,\bw)+\varphi_i(\bu',\bw')))\}
    =\langle \ba,\bu+\bu'\rangle+\langle \b,\bw+\bw'\rangle\\
    =f_{\ba,\b}(\bu,\bw)+f_{\ba,\b}(\bu',\bw')
,\end{equation*}
as desired. Thus the image of $\Psi_{d,r}$ is in $\Sp(\M_{d,r})$. Next observe that the standard basis vectors $e_1,…, e_d$ of $\Z^d$ and $f_1,\cdots,f_r$ of $\Z^r$ are contained in $\M_{d, r}$. Evaluating $f_{\ba,\b}$ against these basis vectors yields the coefficients $a_1,\cdots,a_d, b_1,\cdots, b_r$. It follows that $\Psi_{d,r}$ is injective. To see that $\Psi_{d,r}$ is surjective, we wish to show that given $f\in \Sp(\M_{d,r})$ there exists $(\ba,\b)\in T_{d,r}$ such that $f=f_{\ba,\b}$.
Let $k \in [d]$. For all $(\bu,\bw)\in\M_{d,r}(k)$, we have $\min_{j\in[d]}\{u_j+u_j'\}=0$
and thus, by \Cref{eq_ptsMdr_eval}, we may conclude that for all $i\in[r]$, $k\in[d]$, and $(\bu,\bw), (\bu',\bw') \in\M_{d,r}(k)$ we have
\begin{equation*}
    f(\bu,\bw)+f(\bu',\bw')
    =
    \min_{i\in[r]}\{f\circ\varphi_i^{-1}(\varphi_i(\bu,\bw)+\varphi_i(\bu',\bw'))\}
    =
    f\left(\bu+\bu',\ \bw+\bw'\right).
\end{equation*}
From this we see that $f$ is completely determined by its values $f(\varepsilon_k,\bze)$ and $f(\bze,\varepsilon_i)$ on the standard basis vectors. Given $f \in \Sp(\M_{d,r})$ we may set $a_j:=f(\varepsilon_j,\bze)$ and $b_i:=f(\bze,\varepsilon_i)$
for $i\in[r]$ and $j\in[d]$. Then by construction it follows that for all $k \in [d]$ and all $(\bu,\bw)\in\M_{d,r}(k)$, we have $f(\bu,\bw)
= \langle \ba,\bu\rangle+\langle \b,\bw\rangle$.
It then follows that $f=f_{\ba,\b}$.
To complete the argument, we must check that this vector $(\ba,\b)$ lies in $T_{d,r}$.
To see this, note that since $f|_{\M_{d,r}(1)}$ is linear, we have that 
   $ \sum_{j\in[d]}a_j
    =
    \sum_{j\in[d]}f(\varepsilon_j,\bze)
    =
    f(\varepsilon_1,\bze)+f(\varepsilon_2+\cdots+\varepsilon_d,\bze)$
and a straightforward computation shows that
    $\varphi_i^{-1}(\varphi_i(\varepsilon_1,\bze)+\varphi_i(\varepsilon_2+\cdots+\varepsilon_d,\bze))
    =
    (\bze,\varepsilon_i)$.
It follows that
\begin{align*}
    f(\varepsilon_1,\bze)+f(\varepsilon_2+\cdots+\varepsilon_d,\bze)
    &=
     \min_{i\in[r]}\{f\circ\varphi_i^{-1}(\varphi_i(\varepsilon_1,\bze)+\varphi_i(\varepsilon_2+\cdots+\varepsilon_d,\bze))\}\\
    &=
     \min_{i\in[r]}\{f(\bze,\varepsilon_i)\}\\
    &=
     \min_{i\in[r]}\{b_i\},
\end{align*}
so $(\ba,\b)\in T_{d,r}$ as required. 
Finally we note that the computations above imply that $(\ba,\b)\in T_{d,r}(i)$ if and only if $f_{\ba,\b}\in \Sp(\M_{d,r},i)$, so $\Psi_{d,r}$ restricts to give bijections between them. Since $T_{d,r}=\cup_{i \in [r]} T_{d,r}(i)$ it follows that $\Sp(\M_{d,r}) = \cup_i \Sp(\M_{d,r},i)$, showing that $\MM_{d,r}$ is full. 
\end{proof}

%-------------------------------------%%%%%
\subsection{A strict dual pairing between $\MM_{d,r}$ and $\MM_{r,d}$ }

For fixed integers $d \geq 2$ and $r \geq 2$, we now show that there is a strict dual pairing of $\MM_{d,r}$ and $\MM_{r,d}$. When $r=d$, this implies that $\MM_{d,d}$ is strictly self-dual. 
More precisely, following notation in previous sections, we have the following. 

\begin{proposition}\label{prop: Mdr Mrd strict dual pair}
Let $d, r$ be integers, $d \geq 2, r \geq 2$. 
Let $\v_{d,r}:\Z^d\times \Z^r\to \Z^r\times \Z^d$ be the map defined by
\begin{equation*}\label{eq_Mdr_vdr}
    \v_{d, r}(\bu,\bw) = \left(\bw,\bu + \langle \bo,\bw\rangle\bo\right). 
\end{equation*}
Then $\v_{d,r}$ restricts to $\M_{d,r}$ to be a bijection $\v_{d,r}: \M_{d,r}\to T_{r,d}$, and, $\v_{d,r}$ induces a strict dual pairing between $\MM_{d,r}$ and $\MM_{r,d}$. 
In particular, 
$\MM_{d, r}^\vee \cong \MM_{r,d}$, and $\MM_{d,d}$ is strictly self-dual. 
 \end{proposition}

\begin{proof}
We first show that the restriction of $\v_{d,r}$ to $\M_{d,r}$ takes values in $T_{r,d}$. Indeed, 
    given $(\bu,\bw)\in\M_{d,r}$, note that 
        $\min\left\{u_1+\sum_{i\in[r]}w_i,\ldots, u_d+\sum_{i\in[r]}w_i\right\}
        =
        \sum_{i\in[r]}w_i$.
    This shows that $\v_{d, r}(\bu,\bw)\in T_{r,d}$, by definition of $T_{r,d}$. The proof that $\v_{d,r}$ is a bijection is straightforward, and similar to the proof that $\varphi_i$ is a bijection in \Cref{lem_Mdr_identified}.
Define $\v_{r,d}$ in the same way as $\v_{d,r}$ with the roles of $d$ and $r$ reversed. Throughout, we identify $\MM_{d,r}$ with $\M_{d,r}$ and $\Sp(\MM_{d,r})$ with $T_{d,r}$ (and similarly for $\MM_{r,d}, \M_{r,d}$, and $T_{r,d})$. We wish to show that $(\MM_{d,r}, \MM_{r,d}, \v_{d,r}, \v_{r,d}$) form a strict dual pair.  To check that $\v_{d,r},\v_{r,d}$ give a pairing in the sense of \Cref{def_dual}, first let $(\bu,\bw)\in\M_{d,r}$ and $(\mathbf{y}, \mathbf{z})\in \M_{r,d}$.  The axiom (1) holds by construction. For axiom (2) we need to check that 
    \begin{equation}\label{eq: vdr vrd dual pair req}
    \v_{d,r}(\bu, \bw)(\mathbf{y},\mathbf{z}) = \v_{r,d}(\mathbf{y}, \mathbf{z})(\bu,\bw)
    \end{equation}
    where the pairing between $\M_{r,d}$ and $T_{r,d}$ (respectively $\M_{d,r}$ and $T_{d,r}$) is given via $\Psi$ of \Cref{prop_Mdr_pts}. 
We prove \Cref{eq: vdr vrd dual pair req} by computing both sides. We have that the LHS is 
    \begin{equation*}\label{eq: lhs vdr}
         \v_{d,r}(\bu, \bw)(\mathbf{y},\mathbf{z})  = f_{\v_{d,r}(\bu,\bw)}(\mathbf{y}, \mathbf{z}) = f_{\left(\bw,\bu + \langle \bo,\bw\rangle\bo\right)}(\mathbf{y}, \mathbf{z})
        = \langle \bw,\mathbf{y} \rangle+\langle \bu,\mathbf{z} \rangle+\langle \bo,\mathbf{z} \rangle\langle \bo,\bw\rangle.
    \end{equation*}
    Similarly we compute the RHS to be
    \begin{equation*}\label{eq: rhs vrd}
         \v_{r,d}(\mathbf{y}, \mathbf{z})(\bu,\bw)  = f_{\v_{r,d}(\mathbf{y}, \mathbf{z})}(\bu,\bw) 
       = f_{(\mathbf{z}, \mathbf{y}+\langle \mathbf{1}, \mathbf{z} \rangle \mathbf{1})}(\bu,\bw) 
       = \langle \mathbf{z}, \bu \rangle + \langle \mathbf{y}, \bw \rangle + \langle \mathbf{1}, \mathbf{z} \rangle \langle \mathbf{1}, \bw \rangle. 
       \end{equation*}
   Thus by symmetry of the usual inner product the LHS and RHS agree and 
    so \Cref{eq: vdr vrd dual pair req} holds as claimed. For axiom (3), we already noted above that $\v_{d,r}$ and $\v_{r,d}$ are bijections. 
   Finally, we prove axiom (4).  It suffices to show that for all $i \in [d]$, the preimage $\v_{d,r}^{-1}(T_{r,d}(i))$ is one of the maximal-dimensional cones of linearity in $\Sigma(\MM_{d,r})$, which we saw above can be identified with the cones $C_k := \{u_k = \min\{u_1,\cdots,u_d\}\}$ in $M^{(1)}_{d,r}$ for $k \in [d]$. A straightforward computation using \Cref{eq: varphi inverse} shows that $\varphi_1^{-1}(C_k) \subset \M_{d,r}$ can be expressed as 
$C'_k := \varphi_1^{-1}(C_k) = \{(\bu,\bw) \in \M_{d,r} \, \mid \, u_k=0\}$.
Together with the discussion above, what we need to show is that $\v_{d,r}^{-1}(T_{r,d}(i))$ is precisely $C'_k$ for some $k$. To see this, it is useful to note that an inverse to $\v_{d,r}$ is given by 
$\v_{d,r}^{-1}(\ba, \b) = (\b - \min\{b_1,\cdots, b_d\} \mathbf{1}, \ba)$.
From this formula it is straightforward to compute that $\v_{d,r}^{-1}(T_{r,d}(k)) = C'_k$, and as $k$ varies, we see that we obtain precisely all of the maximal-dimensional cones of $\Sigma(\MM_{d,r})$, as desired. 
\end{proof}

When $d=r$, we immediately obtain the following. 

\begin{corollary}
Let $d$ be an integer, $d \geq 2$. 
Then $\MM_{d,d}$ is strictly self-dual. 
\end{corollary}

%-------------------------------------%%%%%
\subsection{A chart-Gorenstein-Fano PL polytope in $\MM_{d,r} \otimes \R$}\label{subsec: Gorenstein Fano for Mdr}

We now give an example of a PL polytope in $\MM_{d,r} \otimes \R$ and show that it is chart-Gorenstein-Fano in the sense of \Cref{definition: PL Gorenstein Fano}.  Throughout, we use the identifications $\MM_{r,d} \cong \M_{r,d}$ and $\Sp(\MM_{d,r}) \cong T_{d,r}$; we also use that $\MM_{d,r}$ and $\MM_{r,d}$ are a strict dual pair, so $\Sp(\MM_{d,r}) \cong T_{d,r}$ is identified with $\M_{r,d}$.

A PL polytope in $\MM_{d,r} \otimes \R$ is defined as an intersection of PL half-spaces, specified by a pair of data: a point in $\Sp(\MM_{d,r}) \cong T_{d,r}$, and a scalar parameter. We begin our construction by specifying a set $S$ of elements of $\M_{r,d}$, which -- by strict duality and the map $\v_{r,d}$ of \Cref{prop: Mdr Mrd strict dual pair} --- we can think of as a subset of $T_{d,r} \cong \Sp(\MM_{d,r})$.  Consider the following subset of $M_{r,d}^{(1)}$
    $$S^{(1)}=\{(\varepsilon_1,\bze),\ldots, (\varepsilon_r,\bze),\pm(\bo,\bze),\pm(\bo,\varepsilon_2),\ldots,\pm(\bo,\varepsilon_{d})\}.
    $$
and let $S:=\varphi_1^{-1}(S^{(1)})\subset \M_{r,d}$. 
Then we can define a PL polytope $\PP$ in $\MM_{d,r} \otimes \R$ as 
\begin{equation}\label{eq: def Mdr GF polytope}
\PP = \bigcap_{n\in S} \HH_{\v_{r,d}(n), -1} \subset \MM_{d,r} \otimes \R.
\end{equation}
We claim that $\PP$ is indeed a PL polytope, and that it is integral. 

\begin{proposition} \label{prop_Mdr_lattice}
The subset $\PP$ of $\MM_{d,r} \otimes \R$ in \Cref{eq: def Mdr GF polytope} is an integral PL polytope.  In particular, $\PP$ is chart-Gorenstein-Fano. 
\end{proposition}

To prove \Cref{prop_Mdr_lattice} we need to verify that $\PP$ is compact, and, that $\pi_i(\PP)$ is an integral polytope for all $i\in[r]$. 
To do so, we will utilize \textbf{totally unimodular matrices}, which are integer matrices such that all their minors equal $0$, $1$, or $-1$. 
The following result is well known in the context of combinatorial optimization, see e.g.\ \cite[Theorem 19.1]{Sch86}. In the statement of the lemma below, for two vectors $\bv_1, \bv_2 \in \R^p$, we say $\bv_1 \geq \bv_2$ if each component of $\bv_1$ is greater than or equal to the corresponding component of $\bv_2$. 

\begin{lemma}\label{lem_TUM}
Let $A$ be a $q \times p$ matrix which is totally unimodular. Let $\bc \in \Z^q$. If the set $\{\bv \in \R^p \, \mid \, A\bv \geq \bc\}$ is bounded, then it is an integral polytope. 
  \exampleqed
\end{lemma}

The following statement is also useful; the proof is straightforward. 

\begin{lemma}\label{lemma: TUM preserved}
Suppose $A$ is a totally unimodular matrix. Then a matrix $A'$ obtained from $A$ by inserting either a row or a column of the form $\varepsilon_j$ or $\varepsilon_j^t$ (i.e. a standard basis vector of all $0$'s except for a single $1$, or, its transpose) is also totally unimodular. 
\exampleqed
\end{lemma} 

We can now prove \Cref{prop_Mdr_lattice}. 

\begin{proof}[Proof of \Cref{prop_Mdr_lattice}]
To prove compactness, it suffices to prove that $\pi_1(\PP)$ is compact, and to show that $\PP$ is integral, it suffices by \Cref{lem_lattice_intercrit} to prove that $\pi_i(\PP)\cap C_k\subseteq M_{d,r}^{(1)}$ is an integral polytope for all $k\in[d]$ and $i \in [r]$. For concreteness we first take $i=1$ in the discussion that follows. 
We begin by analyzing the inequalities that specify $\pi_1(\PP)$. Given $n\in\M_{r,d}$, and using strict duality, we have 
\begin{equation}\label{eq_Mdr_halfspace}
    \pi_1(\HH_{\v_{r,d}(n),-1})=
    \{(\bu,\bw)\in M_{d,r}^{(1)} \otimes \R \mid f_{\v_{r,d}(n)}(\varphi_1^{-1}(\bu,\bw))\ge -1\} \subset M_{d,r}^{(1)} \otimes \R 
,\end{equation}
where $(\ba,\b)=\v_{r,d}(\varphi_1^{-1}(\varphi_1(n)))$.
Fix $k\in[d]$. For $(\bu,\bw)\in C_k$, we have $u_k = \min\{u_1,\cdots,u_d\}$, so \Cref{eq: varphi inverse} simplifies and we have
    $ \varphi_1^{-1}(\bu,\bw)
    =
    (\bu-u_k\bo,\ \bw+(u_k-\langle \bo,\bw\rangle)\varepsilon_1)$.
We can also compute that $ \v_{r,d}(S)
    =
    \{ (\bze,\varepsilon_1),\ldots,(\bze,\varepsilon_r), \pm(\varepsilon_1,\bo),\ldots, \pm(\varepsilon_d,\bo) \}$
so that the inequalities on the RHS of \Cref{eq_Mdr_halfspace} as we vary $n \in S$ can be computed to be 
\begin{align*}
    w_i&\ge -1,\qquad  i=2,\ldots,r,
    \\
    u_k-w_2-\cdots-w_r&\ge -1,
    \\
    \pm u_j&\ge -1,\qquad j\in[d]
.\end{align*}
Since we are considering $(\bu,\bw)\in M_{d,r}^{(1)}$, we need also to impose the condition $w_1=0$.
Lastly, the inequalities defining $C_k$ are $u_k\le u_i$ for all $i\in[d]$. Putting this together, we have that $ \pi_1(\PP)\cap C_k$ is the set of $(\bu,\bw)\in \R^d\times\R^r$ such that
\begin{equation*}\label{eq_TUM_poly}
    \underbrace{\begin{bmatrix}
        I_{d} & 0 & 0\\
        -I_{d} & 0 & 0\\
        \varepsilon_k & 0 & -\bo\\
        0 & 0 & I_{r-1}\\
        \varepsilon_1-\varepsilon_k & 0 & 0\\
        \vdots & \vdots & \vdots\\
        \varepsilon_d-\varepsilon_k & 0 & 0\\
        0 & 1 & 0 \\
        0 & -1 & 0 \\ 
    \end{bmatrix}}_{A}
    \begin{bmatrix}
        \bu^t \\ \bw^t
    \end{bmatrix}
    \ge
    \begin{bmatrix}
        -\bo^t\\-\bo^t\\ -1 \\-\bo^t\\ 0 \\ \vdots \\ 0 \\ 0
    \end{bmatrix}
,\end{equation*}
where for $s \in \Z_{>0}$, the notation $I_s$ denotes the $s\times s$ identity matrix. 

We now claim that $\PP$ is compact. To see this, it suffices to show that each $\pi_1(\PP) \cap C_k$ is compact, since there are finitely many cones $C_k$. It suffices in turn to show that the pair $(\bu,\bw)$ satisfying the inequalities above lies in a bounded set within $M^{(1)}_{d,r} \otimes \R \cong \R^{d+r-1}$. To see this, first observe that the inequalities imply that each $u_i \in [-1,1]$, a bounded interval. From this it is straightforward to see that each $w_j$ lies in $[-1,r]$ for all $2 \leq j \leq r$. We already know $w_1=0$. Thus $(\bu,\bw)$ lies in a bounded region and $\pi_1(\PP) \cap C_k$ is compact, as required. Next we claim that $A$ is totally unimodular, which by \Cref{lem_TUM} and the fact that $\pi_1(\PP)$ is bounded then implies $\pi_1(\PP)$ is integral. Note that the $1 \times r$ matrix $[1 \, - \mathbf{1}]$, where $-\mathbf{1}$ is a $1 \times (r-1)$ vector of all $-1$'s, is totally unimodular. Now a series of applications of \Cref{lemma: TUM preserved} yields the matrix $A$ above, which implies $A$ is totally unimodular as required. By \Cref{definition: PL Gorenstein Fano} it follows that $\PP$ is chart-Gorenstein-Fano. 
\end{proof}

%-------------------------------------%%%%%
\subsection{Detropicalization of $\MM_{d,r}$}\label{subsec: detrop MMdr}
%-------------------------------------%%%%%%

In this section, we construct a detropicalization $\Aa_{d,r}$ of the polyptych lattice $\MM_{d,r}$. We then construct compactifications $X_{\Aa_{d,r}}(\PP)$ of $\mathrm{Spec}(\Aa_{d,r})$ in \Cref{subsec: Adr compactification}. For concreteness we take $\K=\C$.
First, we define the algebra $\Aa_{d,r}$ as follows: 
\begin{equation*}\label{eq: def Adr detrop}
\Aa_{d, r} := \C[x_1, \ldots, x_d, t_1, \ldots, t_r, t_1^{-1}, \cdots, t_r^{-1}]/\langle \left( \prod_{i=1}^d x_i \right) - t_1 - \cdots - t_r\rangle.
\end{equation*}
It is straightforward to see that $\Aa_{d,r}$ is a Noetherian $\C$-algebra which is an integral domain. We briefly remark that for $d=r=2$ we have that $\Aa_{d,r}$ is a type $A_1$ cluster algebra with two frozen variables.

We specify an additive basis of $\Aa_{d,r}$. Since there is a monomial ordering $<$ such that the initial term of the single defining relation of $\Aa_{d,r}$ is $\prod_{i=1}^d x_i$, the following lemma is immediate from standard results of Gr\"obner bases, i.e., the monomials not in the monomial ideal $in_{<}(I)$ form a vector space basis of $R/I$ for $R$ a polynomial ring and $I$ an ideal (see \cite[Proposition 1.1]{Sturmfels-Grobner}).  

\begin{lemma} 
The set of monomials 
\begin{equation*}\label{eq: def Bdr}
\B_{d, r} := \{x_1^{u_1}\cdots x_d^{u_d}t_1^{w_1}\cdots t_r^{w_r} \mid (\bu,\bw)\in \M_{d,r}\}
\end{equation*}
maps under the projection $\C[x_1, \ldots, x_d, t_1, \ldots, t_r, t_1^{-1}, \cdots, t_r^{-1}] \to \Aa_{d,r}$ to an additive basis of $\Aa_{d,r}$. In particular, there is a one-to-one correspondence between $\B_{d,r}$ and $\M_{d,r}$. 
\end{lemma}

We saw above that $\MM_{d,r}$ and $\MM_{r,d}$ are a strict dual pair. To construct a detropicalization of $\MM_{d,r}$, we wish to construct a valuation $\fv: \Aa_{d,r} \to \P_{\MM_{r,d}}$ (note the switch of indices); we do so using the basis $\B_{d,r}$. 
By slight abuse of notation we will denote also by $\Psi_{d,r}$ the composition of the identifications $\Psi_{d,r}: T_{d,r} \to \Sp(\M_{d,r})$ and the bijection $\Sp(\M_{d,r}) \cong \Sp(\MM_{d,r})$ of \Cref{lem_pts_Bdr}.
Thus we have $\Psi_{d,r}: T_{d,r} \stackrel{\cong}{\longrightarrow}  \Sp(\MM_{d,r})$. We define $\fv$ on elements of $\B_{d,r}$ as the composition of $\v_{d,r}$ from \Cref{prop: Mdr Mrd strict dual pair} with $\Psi_{r,d}: T_{r,d} \to \Sp(\MM_{r,d})$. Concretely, we define $\fv_{d,r}:\B_{d,r}\to \Sp(\MM_{r, d})$ as 
\begin{equation}\label{eq: def val on Bdr}
    \fv_{d,r}(\bx^\bu\bt^\bw)
    :=\Psi_{r,d} \circ \v_{d,r}(\bu, \bw) = \Psi_{r,d}(\bw,\bu+\langle \bo,\bw\rangle\bo)
.\end{equation}
We then extend this definition to all of $\Aa_{d,r}$ as follows: 
\begin{equation}\label{eq: def val on Mdr}
\fv_{d,r}:\Aa_{d,r} \to \P_{\MM_{r, d}},\qquad \fv_{d,r}\left(\sum c_m \bb_m \right) := \bigoplus \fv_{d,r}(\bb_m)
\end{equation}
where $c_m \in \C, \bb_m \in \B_{d,r}$. 

\begin{lemma}\label{lem_Mdr_val}
    The function $\fv_{d,r}:\Aa_{d,r} \to \P_{\MM_{r, d}}$ defined in \Cref{eq: def val on Bdr} and \Cref{eq: def val on Mdr} is a valuation with values in the idempotent semialgebra $\P_{\MM_{r,d}}$. Moreover, $\B_{d,r}$ is a convex adapted basis for $\fv_{d,r}$.
\end{lemma}

\begin{proof}
We must check the conditions (1)-(4) in \Cref{def_valuation}. 
We begin with (1), i.e., $\fv_{d,r}(fg)=\fv_{d,r}(f)\odot\fv_{d,r}(g)$ for all $f, g \in \Aa_{d,r}$. As a first step,  
we claim that 
\begin{equation}\label{eq_Mdr_val_basis}
    \fv_{d,r}(\bb \bb')=\fv_{d,r}(\bb)\odot\fv_{d,r}(\bb')
\end{equation} 
for all $\bb,\bb'\in\B_{d,r}$. To prove \Cref{eq_Mdr_val_basis}, we take cases. First suppose that $\bb=\bx^\bu\bt^\bw$ and  $\bb'=\bx^{\bu'}\bt^{\bw'}$, and that the pair $(\bu,\bw),(\bu',\bw')\in\M_{d,r}$ satisfies $(\bu+\bu',\bw+\bw')\in\M_{d,r}$, i.e., $\bb \bb' \in \B_{d,r}$.
Under this assumption, it follows from the definitions of $\v_{d,r}$ and $\Psi_{r,d}$ that 
\begin{equation*}
    \v_{d,r}(\bu+\bu',\bw+\bw')=\v_{d,r}(\bu,\bw)+\v_{d,r}(\bu',\bw')
\end{equation*}
and
\begin{equation*}
    \Psi_{r,d}(\v_{d,r}(\bu+\bu',\bw+\bw') = \Psi_{r,d}(\v_{d,r}(\bu,\bw))+\Psi_{r,d}(\v_{d,r}(\bu',\bw'))
\end{equation*}
so \Cref{eq_Mdr_val_basis} holds.
Next, we consider the case that $\bb \bb'\notin\B_{d,r}$, i.e., $\widetilde{u}:=\min_{i\in [d]}\{u_i+u_i'\}>0$.
Using the relation $x_1 x_2 \cdots x_d = t_1+\cdots+t_r$ in $\Aa_{d,r}$, we see that
\begin{equation}\label{eq: bbprime rewrite in Adr}
    \bx^{\bu+\bu'}\bt^{\bw+\bw'}
    =
    \bx^{\bu+\bu'-\widetilde{u}\bo}\bt^{\bw+\bw'}(t_1+\cdots+t_r)^{\widetilde{u}}.
\end{equation}
Each monomial in this expression is an element of $\B_{d,r}$, so by \Cref{eq: def val on Mdr} 
it follows that
\begin{align*}
    \fv_{d,r}(\bb \bb')
    &=
    \min\{\fv_{d,r}(\bx^{\bu+\bu'-\widetilde{u}\bo}\bt^{\bw+\bw'+\bv}) \mid \bv\ge0,\ v_1+\cdots+v_r=\widetilde{u}\}
    \\&=
    \min\{\Psi_{r,d}({\bw+\bw'+\bv},{\bu+\bu'-\widetilde{u}\bo+\langle \bo,\bw+\bw'+\bv\rangle\bo}) \mid \bv\ge0,\ v_1+\cdots+v_r=\widetilde{u}\}    
    \\&=
    \min\{\Psi_{r,d}({\bw+\bw'+\bv},{\bu+\bu'+\langle \bo,\bw+\bw'\rangle\bo}) \mid \bv\ge0,\ v_1+\cdots+v_r=\widetilde{u}\}  
.\end{align*}
Then by definition of $\Psi_{r,d}$ it suffices to show that 
\begin{equation}\label{eq_Mdr_val_fexp}
    \min\{f_{{\bw+\bw'+\bv},{\bu+\bu'+\langle \bo,\bw+\bw'\rangle\bo}} \mid \bv\ge0,\ v_1+\cdots+v_r=\widetilde{u}\}
    =
    f_{\bw,\bu+\langle \bo,\bw\rangle\bo}
    +
    f_{\bw',\bu'+\langle \bo,\bw'\rangle\bo}
\end{equation}
as functions on $\M_{r,d}$.
Let $g_\bv:\M_{r,d}\to \Z$ denote the function $g_\bv(\ba,\b) :=\langle \bv,\ba\rangle$. From the definition of $f_{\ba,\b}$ for $(\ba,\b) \in T_{r,d}$ it can be computed that 
\begin{equation}\label{eq: 1}
    f_{{\bw+\bw'+\bv},{\bu+\bu'+\langle \bo,\bw+\bw'\rangle\bo}}
    =
    g_\bv
    +
    f_{\bw,\bu+\langle \bo,\bw\rangle\bo}
    +
    f_{\bw',\bu'+\langle \bo,\bw'\rangle\bo}
.\end{equation}
For any $(\ba,\b)\in\M_{r,d}$, we have that 
\begin{equation*}
    \min\{g_{\bv}(\ba,\b) = \langle \bv,\ba\rangle \mid \bv\ge0,\ v_1+\cdots+v_r=\widetilde{u}\}=0
\end{equation*}
since $\min\{a_1,\ldots,a_r\}=0$. 
Hence the minimum of the RHS of \Cref{eq: 1} as $\bv$ ranges over $\bv \geq 0, v_1 + \cdots + v_r = \tilde{u}$ is equal to $f_{\bw,\bu+\langle \bo,\bw\rangle\bo}
    +
    f_{\bw',\bu'+\langle \bo,\bw'\rangle\bo}$.
    Since this last expression is independent of $\bv$, it follows that \Cref{eq_Mdr_val_fexp} and hence \eqref{eq_Mdr_val_basis} holds for this case, as desired.

Now we claim that (1) holds in general, for any $f = \sum_m c_m \bb_m, g = \sum_{m'} c_{m'} \bb_{m'} \in \Aa_{d,r}$, where $\bb_m, \bb_{m'} \in \B_{d,r}$ and we may assume $c_m \neq 0, c_{m'} \neq 0$. We wish to prove $\fv_{d,r}(fg)=\fv_{d,r}(f)+\fv_{d,r}(g)$ as functions on $\MM_{r,d}$. We know by definition of $\fv_{d,r}$ that $\fv_{d,r}(fg),\fv_{d,r}(f),\fv_{d,r}(g) \in \PP_{\MM_{r,d}}$ are piecewise-linear, so there exists a (finite) complete fan $\Sigma$ in $\MM_{r,d} \otimes \R$ such that they are linear on any full-dimensional cone $C$ in $\Sigma$.
Moreover, we may without loss of generality assume that $\Sigma$ is a refinement of $\Sigma(\MM_{r,d})$ since elements of $\P_{\MM_{r,d}}$ are min-combinations of points in $\Sp(\MM_{r,d})$, which are each linear on any cone in $\Sigma(\MM_{r,d})$. With this in mind, it suffices to show $\fv_{d,r}(fg)=\fv_{d,r}(f)+\fv_{d,r}(g)$, as functions, on each full-dimensional cone $C \in \Sigma$. Thus for the remainder of the argument we restrict to such a cone $C$. 

 To see the equality, we compute both sides and compare. For the RHS, by definition $\fv_{d,r}(f)+\fv_{d,r}(g)=\min\{\fv_{d,r}(\bb_m)\} + \min\{\fv_{d,r}(\bb_{m'})\}$. Since $\fv_{d,r}(\bb_m)$ is a point, it is linear on $C$ and also uniquely determined by its values on $C$, so it follows that there exists a unique $m_0$ with $\fv_{d,r}(\bb_{m_0})$ achieving the minimum in $\fv_{d,r}(f)$, and similarly for $\fv_{d,r}(g)$. Thus (as functions on $C$) the RHS is $\fv_{d,r}(f)+\fv_{d,r}(g)=\fv_{d,r}(\bb_{m_0}) + \fv_{d,r}(\bb_{m'_0})$. Next we compute the LHS. We have $\fv_{d,r}(fg)=\fv_{d,r}\left(\sum_{m,m'} c_m c_{m'} \bb_m \bb_{m'}\right)$. To compute further, we must expand each $\bb_m \bb_{m'}$ as a linear combination of elements of $\B_{d,r}$ and then take the sum. This process could result in cancellations. However, we now show that the basis element achieving the minimum cannot be cancelled. More precisely, we have the following. We saw already that $\fv_{d,r}(\bb_{m_0}\bb_{m'_0}) = \fv_{d,r}(\bb_{m_0})+\fv_{d,r}(\bb_{m'_0})$. Expand $\bb_{m_0}\bb_{m'_0} = \sum_k d_k \bb_{m_k}$ as a linear combination of elements of $\B_{d,r}$. In the arguments in the previous paragraphs we already saw that each $\bb_{m_k}$ appearing in the sum must satisfy $\fv_{d,r}(\bb_{m_k}) \geq \fv_{d,r}(\bb_{m_0})+\fv_{d,r}(\bb_{m'_0})$ and that there exists a unique $\bb_{m_{k_0}}$ achieving the minimum, i.e. $\fv_{d,r}(\bb_{m_0}) + \fv_{d,r}(\bb_{m'_0}) = \fv_{d,r}(\bb_{m_{k_0}})$. In particular, for all other $k \neq k_0$ we have $\fv_{d,r}(\bb_{m_k}) > \fv_{d,r}(\bb_{m_{k_0}})$ almost everywhere on $C$. Moreover, for all $m \neq m_0$ and $m' \neq m'_0$ we also have $\fv_{d,r}(\bb_m) > \fv_{d,r}(\bb_{m_0}), \fv_{d,r}(\bb_{m'}) > \fv_{d,r}(\bb_{m'_0})$ almost everywhere, so by the same argument, any basis element $\bb$ appearing in the expansion of $\bb_{m}\bb_{m'}$ for $(m,m') \neq (m_0,m'_0)$ must satisfy $\fv_{d,r}(\bb) > \fv_{d,r}(\bb_{m_0})+\fv_{d,r}(\bb_{m'_0}) = \fv_{d,r}(\bb_{m_{k_0}})$. This proves that nothing can cancel the $\bb_{m_{k_0}}$ term, and the minimum of $\fv_{d,r}(fg)$ is achieved by $\bb_{m_{k_0}}$. Thus $\fv_{d,r}(fg)=\fv_{d,r}(f)+\fv_{d,r}(g)$ on $C$, as claimed, and this argument is valid for all full-dimensional cones $C$. 

 The properties (2)-(4) are straightforward, and $\B_{d,r}$ is a convex adapted basis by construction. 
 \end{proof}

\begin{proposition}
    With notation as established above, the pair $(\Aa_{d,r},\fv_{d,r})$ is a detropicalization of $\MM_{d,r}$.
\end{proposition}

\begin{proof} 
Since the ideal $\langle x_1\cdots x_d - t_1 - \cdots - t_r\rangle$ is prime and principal, $\Aa_{d,r}$ is an integral domain of Krull dimension $d+r-1$.
We saw above that  $\fv_{d,r}|_{\B_{d,r}}$ is a bijection onto $\Sp(\MM_{r,d})$.
The result follows.
\end{proof}

\subsection{Compactifications of $\Aa_{d,r}$}\label{subsec: Adr compactification}

In \Cref{subsec: Gorenstein Fano for Mdr} we constructed an explicit example of a chart-Gorenstein-Fano PL polytope $\PP$ in $\MM_{d,r} \times \R$ and in \Cref{subsec: detrop MMdr} we built a detropicalization $(\Aa_{d,r}, \fv_{d,r})$ of $\MM_{d,r}$. Using this data we can now compactify $\mathrm{Spec}(\Aa_{d,r})$. Concretely, let $\PP = \cap_{n \in \mathcal{S}} \HH_{\v_{r,d}(n),-1}$ be as defined in \Cref{eq: def Mdr GF polytope}. Following the construction in \Cref{subsec:compactification}, we define 
$$
X_{\Aa_{d,r}}(\PP) = \mathrm{Proj}(\Aa_{d,r}^{\PP}). 
$$
The theory developed in this manuscript allows us to conclude several results about the geometry of $X_{\Aa_{d,r}}(\PP)$. For instance, it has a finitely generated Cox ring. 

\begin{proposition}
Let $d,r$ be positive integers, $d \geq 2, r \geq 2$. Then $\Aa_{d,r}$ is a unique factorization domain, and $X_{\Aa_{d,r}}(\PP)$ has a finitely generated Cox ring. 
\end{proposition}

\begin{proof}
Note that the ring 
$\C[x_1, \ldots, x_d, t_1, \ldots, t_r]/\langle \left( \prod_{i=1}^d x_i \right) - t_1 - \cdots - t_r\rangle$ is isomorphic to a polynomial ring, and is therefore a unique factorization domain. The ring $\Aa_{d,r}$ is obtained as a localization of this ring, and since localization preserves UFD-ness, $\Aa_{d,r}$ is also a UFD. The second statement follows from \Cref{theorem: finitely generated Cox rings}. 
\end{proof}

Since $\PP$ is a chart-Gorenstein-Fano PL polytope, the projection of $\PP$ to each coordinate chart $M_{d,r}^{(i)}$ for $i \in [r]$ is a classical Gorenstein-Fano polytope. Thus from \Cref{theorem: family of degenerations} it follows that there are $r$ many toric degenerations of $X_{\Aa_{d,r}}(\PP)$ to the Gorenstein-Fano toric varieties associated to these $r$ many Gorenstein-Fano polytopes. We leave further exploration of these ideas to future work.

%-------------------------------------%%%%%%%%%%%%%%%
\appendix

\section{Constructions with valuations and filtrations}\label{sec: appendix valuations}

In the course of our arguments, particularly in \Cref{sec_detrop} and \Cref{section: compactifications}, we will at times need to establish connections between different types of valuations or filtrations, and their associated graded algebras.  Specifically, in \Cref{sec_detrop} we want to construct a $\Z^r$-valued valuation starting with a collection of $r$ many $\Z$-valued valuations, and in \Cref{section: compactifications}, we wish to compare the associated graded algebra with respect to a $\Z^r$-valuation with the associated graded of a closely related $\Z$-valued valuation. It turns out that these issues can be handled in a systematic manner in a general setting. We are not aware of an exposition of these results in the literature (though we suspect they are well-known to experts) so we record them here. The proofs are not difficult so for the sake of brevity we provide sketches only. 

Let $\Aa$ be a $\K$-algebra. Let $\fv_1,\cdots,\fv_r: \Aa \to \Z$ be a collection of $\Z$-valued valuations on $\Aa$. We assume that there exists a $\K$-vector space basis $\B$ of $\Aa$ which is simultaneously an adapted basis for $\fv_i$ for all $i \in [r]$ in the sense of \cite[Definition 2.27]{KavehManon-Siaga}, i.e., $\B \cap F_{\fv_i \geq a}$ is a basis of $F_{\fv_i \geq a}$ for all $a \in \Z$ and for all $i \in [r]$. We now define two valuations on $\Aa$, denoted $\fv_1 \boxplus \fv_2 \boxplus \cdots \boxplus \fv_r: \Aa \to \Z$ and $\fv_1 \circledast \fv_2 \circledast \cdots \circledast \fv_r: \Aa \to \Z^r$ respectively. Here we view $\Z^r$ as a totally ordered group with respect to the usual lex order. 

It is well-known that a quasivaluation can be constructed from a decreasing multiplicative filtration $\mathcal{F}$ of $\Aa$ with the property that for all $f \in \Aa \setminus \{0\}$, there exists a parameter $a$ such that $f \in F_{\geq a} \setminus \cup_{a' > a} F_{\geq a'}$ (cf.\ for example \cite[\textsection 2.4]{KavehManon-Siaga}). In particular, the quasivaluation $\nu$ corresponding to such a filtration $\mathcal{F}$ is defined in such a way that $\mathrm{gr}_{\fv}(\Aa) = \mathrm{gr}_{\mathcal{F}}(\Aa)$. The quasivaluation is a valuation exactly when the associated graded algebra $\mathrm{gr}_{\mathcal{F}}(\Aa)$ is a domain. We will define both $\fv_1 \boxplus \fv_2 \boxplus \cdots \boxplus \fv_r$ and $\fv_1 \circledast \fv_2 \circledast \cdots \circledast \fv_r$ by constructing the corresponding decreasing filtrations of $\Aa$, which we denote $\mathcal{F}$ and $\mathcal{G}$ respectively. In fact, we will also use a third filtration $\mathcal{E}$ which allows us to interpolate between the associated graded algebras $\mathrm{gr}_{\mathcal{F}}(\Aa) = \mathrm{gr}_{\fv_1 \boxplus \fv_2 \boxplus \cdots \boxplus \fv_r}(\Aa)$ and $\mathrm{gr}_{\mathcal{G}}(\Aa) = \mathrm{gr}_{\fv_1 \circledast \fv_2 \circledast \cdots \circledast \fv_r}(\Aa)$. We begin by defining $\mathcal{F}, \mathcal{G}$, and $\mathcal{E}$. For what follows, we use the notation $F^i_{\geq a} := F_{\fv_i \geq a}$. The filtration $\mathcal{F}$ is indexed by $\Z$ and we define 
\begin{equation*}\label{eq: definition calF filtration} 
F_{\geq a} := \bigoplus_{s_1 + \cdots+ s_r=a} F^1_{\geq s_1} \cap F^2_{\geq s_2} \cap \cdots \cap F^r_{\geq s_r}.
\end{equation*} 
Since each $\{F^i_{\geq s}\}$ is a decreasing multiplicative filtration, and for any $f \in \Aa \setminus \{0\}$ there exists $a \in \Z$ with $f \in F^i_{\geq a} \setminus \cup_{a' > a} F^i_{\geq a'}$, it is clear that $\mathcal{F}$ has the same properties. 

The definition of the filtration $\mathcal{G}$ requires an inductive process. We begin by considering $\mathrm{gr}_{\fv_1}(\Aa) = \oplus_{a \in \Z} F^1_{\geq a}/F^1_{>a}$. If $r=1$, we stop; otherwise, we claim that the filtration $\mathcal{F}_{\fv_2} := \{F^2_{\geq b}\}_{b \in \Z}$ induces a filtration $\overline{\mathcal{F}}_2$ on $\mathrm{gr}_{\fv_1}(\Aa)$, by defining $(\overline{\mathcal{F}}_2)_{\geq b} := \bigoplus_{a \in \Z} (F^2_{\geq b} \cap F^1_{\geq a})/F^1_{> a}$. It is straightforward to verify that this is a decreasing multiplicative filtration. Thus we may consider the associated graded $\mathrm{gr}_{\overline{\mathcal{F}}}(\mathrm{gr}_{\fv_1}(\Aa))$, now equipped with a grading indexed by $\Z \times \Z$. We may proceed to consider the filtration induced on this algebra by $\mathcal{F}_{\fv_3}$, defined in an analogous manner, and so on. It is again easy to check that, at each step, the result is a decreasing multiplicative filtration yielding a corresponding associated graded algebra with a $\Z^3$ grading. Continuing in this manner we obtain a graded algebra -- which by slight abuse of notation we denote $\gr_{\fv_r}\gr_{\fv_{r-1}}\cdots \gr_{\fv_1}(\Aa)$ -- with a grading indexed by $\Z^r$. Moreover, using successively the fact that for any algebra $\mathcal{R}$ with decreasing multiplicative filtration $\mathcal{F}$ there is a map $\mathcal{R} \to \mathrm{gr}_{\mathcal{F}}(\mathcal{R})$, there is a natural map $\pi: \Aa \to \gr_{\fv_r}\gr_{\fv_{r-1}}\cdots \gr_{\fv_1}(\Aa)$. We may then ask for the preimage under $\pi$ of the $(a_1,a_2,\cdots,a_r)$-th graded piece of $\gr_{\fv_r}\gr_{\fv_{r-1}}\cdots \gr_{\fv_1}(\Aa)$, which we denote $\mathcal{G}_{(a_1,\cdots,a_r)}$. A calculation then shows that 
\begin{equation*}\label{eq: definition calG filtration}
\mathcal{G}_{(a_1,\cdots,a_r)} = F^1_{> a_1} + F^1_{\geq a_1}\cap F^2_{> a_2} + F^1_{\geq a_1} \cap F^2_{\geq a_2} \cap F^3_{> a_3} + \cdots + F^1_{\geq a_1} \cap F^2_{\geq a_2} \cap \cdots \cap F^r_{\geq a_r}.
\end{equation*}
It is again straightforward to check that $\mathcal{G}$ is a decreasing multiplicative filtration indexed by $\Z^r$ with respect to the standard lex order on $\Z^r$. 

The filtration $\mathcal{E}$ differs from the two others in that its indexing set is partially ordered, not totally ordered. Specifically, equip $\Z^r$ with the partial order $(a_1,a_2,\cdots,a_r) \succeq (b_1,b_2,\cdots,b_r)$ if $a_i \geq b_i$ (with respect to the standard order on $\Z$) for all $i \in [r]$. We define 
\begin{equation*}\label{eq: definition calE filtration}
\mathcal{E}_{(a_1,\cdots,a_r)} := F^1_{\geq a_1} \cap F^2_{\geq a_2} \cap \cdots \cap F^r_{\geq a_r}
\end{equation*}
and it is straightforward to see that this filtration is decreasing with respect to the partial order and is multiplicative. 

We may now consider the associated graded algebras of each of the filtrations, denoted $\gr_{\mathcal{F}}(\Aa)$, $\gr_{\mathcal{G}}(\Aa)$, and $\gr_{\mathcal{E}}(\Aa)$ respectively. These are graded with respect to $(\Z,>), (\Z^r,>_{lex})$ and $(\Z^r, \succ)$ respectively. Our goal in this appendix is to show that all three are isomorphic as $\K$-algebras. (Note that we are not claiming isomorphisms as graded algebras, since the gradings are clearly different.) The following lemma is simple linear algebra. 

\begin{lemma} 
Let $(a_1,\cdots,a_r) \in \Z^r$. Then $\mathcal{G}_{(a_1,\cdots,a_r)}$ and $\mathcal{E}_{(a_1,\cdots,a_r)}$ are isomorphic as $\K$-vector spaces. 
\end{lemma} 

In fact there is a natural isomorphism $\Psi_{(a_1,\cdots,a_r)}: \mathcal{E}_{(a_1,\cdots,a_r)} \to \mathcal{G}_{(a_1,\cdots,a_r)}$ induced by inclusion of vector spaces. Let $\Psi: \gr_{\mathcal{E}}(\Aa) \to \gr_{\mathcal{G}}(\Aa)$ be defined by $\bigoplus_{(a_1,\cdots,a_r) \in \Z^r} \Psi_{(a_1,\cdots,a_r)}$, the direct sum of the $\Psi_{(a_1,\cdots,a_r)}$ over the graded pieces of both algebras. It is clear this is a $\K$-vector space isomorphism.  The following lemma asserts that $\Psi$ also respects multiplication, which can be seen easily by computing what happens to adapted basis elements. 

\begin{lemma}\label{lemma: grE and grG iso}
The map $\Psi: \gr_{\mathcal{E}}(\Aa) \to \gr_{\mathcal{G}}(\Aa)$ defined above is a $\K$-algebra isomorphism. 
\end{lemma} 

Now we wish to make a connection with $\gr_{\mathcal{F}}(\Aa)$. The following is simple linear algebra. 

\begin{lemma} 
Let $(a_1,\cdots,a_r) \in \Z^r$ and let $s:=a_1+a_2+\cdots+a_r \in \Z$. There is a natural $\K$-linear map $\Phi_{(a_1,\cdots,a_r)}: \mathcal{E}_{(a_1,\cdots,a_r)} \to \mathcal{F}_{s}$ which is injective. For $(a_1,\cdots,a_r) \neq (b_1,\cdots,b_r)$ with $\sum_i a_i = \sum_i b_i=s$, the corresponding images satisfy $\Phi_{(a_1,\cdots,a_r)}(\mathcal{E}_{(a_1,\cdots,a_r)}) \cap \Phi_{(b_1,\cdots,b_r)}(\mathcal{E}_{(b_1,\cdots,b_r)}) = \{0\}$. 
Moreover, $\bigoplus_{a_1+a_2+\cdots+a_r=s} \Phi_{(a_1,\cdots,a_r)}(\mathcal{E}_{(a_1,\cdots,a_r)}) = \mathcal{F}_s$. 
\end{lemma} 

It follows immediately from the above lemma that, as $\K$-vector spaces, we have an isomorphism 
$$
\mathcal{F}_s \cong \bigoplus_{a_1+\cdots+a_r=s} \mathcal{E}_{(a_1,\cdots,a_r)}
$$
for $s \in \Z$. Putting together the individual $\Phi_{(a_1,\cdots,a_r)}$ as we did for $\Psi$ above we immediately obtain a $\K$-vector space isomorphism $\Phi: \gr_{\mathcal{E}}(\Aa) \to \gr_{\mathcal{F}}(\Aa)$. As above, computing with respect to adapted basis elements yields the following. 

\begin{lemma}\label{lemma: grE and grF iso}
The map $\Phi: \gr_{\mathcal{E}}(\Aa) \to \gr_{\mathcal{F}}(\Aa)$ is a $\K$-algebra isomorphism. 
\end{lemma} 

It is also not difficult to see by an explicit computation with $\B$ that both $\gr_{\mathcal{F}}(\Aa)$ and $\gr_{\mathcal{G}}(\Aa)$ are integral domains, implying that the quasivaluations associated to $\mathcal{F}$ and $\mathcal{G}$, which we denote $\fv_1 \boxplus \fv_2 \boxplus \cdots \boxplus \fv_r$ and $\fv_1 \circledast \fv_2 \circledast \cdots \circledast \fv_r$ respectively, are actually valuations. Putting together \Cref{lemma: grE and grF iso} and \Cref{lemma: grE and grG iso} then yields the desired result. 

\begin{lemma}\label{lemma: grF and grG iso}
The valuations $\fv_1 \boxplus \fv_2 \boxplus \cdots \boxplus \fv_r: \Aa \to \Z$ and $\fv_1 \circledast \fv_2 \circledast \cdots \circledast \fv_r: \Aa \to \Z^r$ have the property that $\gr_{\fv_1 \boxplus \fv_2 \boxplus \cdots \boxplus \fv_r}(\Aa) \cong \gr_{\fv_1 \circledast \fv_2 \circledast \cdots \circledast \fv_r}(\Aa)$ as $\K$-algebras. 
\end{lemma}

\section{Results on classical convex geometry}

In this appendix we collect some lemmas and well-known facts from classical convex geometry which we need in our arguments.

\begin{lemma}\label{lemma: once compact always compact}
    Suppose that the polyhedron $P = \cap_{i=1}^{n} H_{n_i, a_i}$ is bounded and nonempty. Then  $Q(b_1,...,b_n) = \cap_{i=1}^{n} H_{n_i, b_i}$ is bounded for any choice of $b_i$'s.
\end{lemma}

\begin{proof}
    This is a straightforward consequence of (5) part (iii) in \cite[\textsection 8.2]{Sch86}.
\end{proof}

We also recall the following classical fact about convex piecewise linear functions. 

\begin{lemma}\label{lemma: PL is a min of linears}
Let $V$ be a finite-dimensional real vector space and $p: V \to \R$ a convex piecewise linear function in the sense given in Section~\ref{sec: PL}. 
Then: 
\begin{enumerate} 
\item[(1)] there is a unique minimal representation of $p$ as $p = \min\{f_1,\cdots,f_\ell\}$ where $f_k: V \to \R$ is a linear function for all $k$, and, 
\item[(2)] for $C$ a maximal-dimensional cone of linearity of $p$ in $V$, if $f: V \to \R$ is a linear function such that the restrictions of $p$ and $f$ agree on $C$, i.e. $p \vert_C = f \vert_C$, then $f$ must appear as one of the $f_k$ in (1) above. 
\end{enumerate} 
\end{lemma}

We will use the following straightforward consequence of \Cref{lemma: PL is a min of linears}.

\begin{lemma}\label{lemma: PL half space in classical half space}
Let $\MM$ be a finite polyptych lattice of rank $r$ over $F$. Let $p \in \Sp(\MM)$ and $\alpha \in \pi(\MM)$ and $\mathcal{C}$ a full-dimensional cone in $\Sigma(\MM)$. Let $f: M_\alpha \to F$ be an $F$-linear function such that $f \vert_{\pi_\alpha(\mathcal{C})} = p_\alpha \vert_{\pi_\alpha(\mathcal{C})}$. Then for all $u \in M_\alpha$, we have $f(u) \geq p_\alpha(u)$. 
\end{lemma} 

\begin{proof} 
By \Cref{lemma: points linear on faces of SigmaM} we know $p$ is piecewise linear, and more specifically, $p$ is linear on each full-dimensional cone $\mathcal{C}$ of $\Sigma(\MM)$. Moreover, by \Cref{lemma: point PL concave} we know the natural extension $p: \MM_\R \to \R$ is convex. It follows that $p_\alpha: M_\alpha \otimes \R \to \R$ is also piecewise linear, with cones of linearity $\pi_\alpha(\mathcal{C})$ for $\mathcal{C}$ in $\Sigma(\MM)$, and also convex. Then by \Cref{lemma: PL is a min of linears}  we know $p_\alpha = \min\{f, f_2, \cdots, f_\ell\}$ for $f$ as given in the hypothesis of the lemma and $f_2,\cdots,f_\ell$ some (finite) list of linear functions. In particular, $f \geq p_\alpha$ as functions on $M_\alpha$, as desired. 
\end{proof}

The following is immediate from \Cref{lemma: PL half space in classical half space}. 

\begin{corollary}\label{corollary: PL half space in classical half space}
Under the hypotheses of \Cref{lemma: PL half space in classical half space}, and for any $a \in F$, we have $\pi_\alpha(\HH_{p,a}) \subseteq \{u \in M_\alpha \, \mid \, f(u) \geq a\}$. 
\end{corollary}

The following lemma is used to construct full rank valuations.

\begin{lemma}\label{lemma: min of functions and lex min}
Let $\Z \subseteq F \subseteq \R$. Let 
    $S,S'\subset\Hom(F^r_{\ge 0},F)$ be finite subsets of $\Hom(F^r_{\geq 0},F)$. Let $g :=\min\{f\mid f\in S\}$ and $h := \min\{f\mid f\in S'\}$, where the minimum is taken as functions. 
    Let $S'' \subseteq S \cup S'$ be such that the unique minimal expression for $\min\{g,h\}$ as a min-combination is $\min\{g,h\}=\min\{f\mid f\in S''\}$.
    Denote by $\varepsilon_i$ the $i$-th standard basis vector in $F^r_{\geq 0}$.
    If $f_* \in S \cup S'$  has the property 
    \begin{equation*}
        (f_*(\varepsilon_1),\ldots,f_*(\varepsilon_r))=\min\{(f(\varepsilon_1),\ldots,f(\varepsilon_r))\mid f\in S\cup S'\}\in F^r,
    \end{equation*}
    where the minimum on the RHS is taken with respect to the standard lex order, then $f_*\in S''$.
\end{lemma}

\begin{proof}
    We start by showing that there exists a nonempty open set $U\subset \R^r_{\ge 0}$ such that $f_*(x)< f(x)$ for all $x\in U$ and all $f\in S\cup S', f \neq f_*$. Indeed, since $f_*$ and $f \in S \cup S'$ are linear on $F^r_{\geq 0}$ we have $f_*(x_1,\cdots,x_r)=\sum_i x_i f(\varepsilon_i)$ and similarly for $f(x_1,\cdots,x_r)$. By assumption on $f$ we know that for any $f \in S \cup S'$ with $f_* \neq f$ we must have $f(\varepsilon_i)=f_*(\varepsilon_i)$ for $1 \leq i<k$ and $f(\varepsilon_k) > f_*(\varepsilon_i)$ for some $k \in [r]$. By choosing $U \subset \R^r_{\geq 0}$ to be a product of intervals with the property that the range of possible $x_j$ is always (sufficiently) large in comparison to $x_\ell$ for $\ell > j$, it can be seen that on $U$, we have $f(x) > f_*(x)$ for any $f \in S \cup S', f \neq f_*$. (Such a $U$ can be found since $S,S'$ are finite.) 
     Next, since the fan with cones the linear regions of $g\oplus h$ is complete, then one of its cones must intersect $U$.
    Together with the equation $\min\{{g,h}\} = \min\{f\mid f\in S\cup S'\}$, this implies that $f_*\in S''$ as desired. 
\end{proof}

\end{document}